\documentclass[12pt]{amsbook}

\usepackage[T1]{fontenc}
\usepackage[utf8]{inputenc}
\usepackage[english]{babel}

\usepackage{lmodern}
\usepackage{amssymb}
\usepackage{empheq}
\usepackage{eucal}

\usepackage{verbatim}
\usepackage{multirow}
\usepackage{multicol}
\usepackage{color}
\usepackage[normalem]{ulem}

\usepackage[toc]{appendix}
\usepackage{booktabs}
\usepackage{tabularx}
\usepackage{setspace}
\usepackage{graphicx}

\newif\ifarxiv
\arxivtrue % set \arxivfalse for local drafting

\ifarxiv
  \usepackage[disable]{todonotes}
\else
  \usepackage{todonotes}
  \usepackage{refcheck}
\fi

\ifarxiv\else
  \usepackage[nocfg]{nomencl}
  \makenomenclature
\fi

\usepackage{hyperref}

\newtheorem{theorem}{Theorem}
\newtheorem{proposition}{Proposition}
\newtheorem{lemma}{Lemma}
\newtheorem{definition}{Definition}
\newtheorem{corollary}{Corollary}

\theoremstyle{remark}
\newtheorem{remark}{Remark}
\newtheorem{example}{Example}

\numberwithin{equation}{section}
\numberwithin{proposition}{section}
\numberwithin{lemma}{section}
\numberwithin{corollary}{section}
\numberwithin{theorem}{section}
\numberwithin{definition}{section}
\numberwithin{remark}{section}



\title{\Huge\textbf{Nonautonomous Linear Systems: Exponential Dichotomy and its Applications}}
\author{\textsc{\'Alvaro Casta\~neda \& Gonzalo Robledo}}

\begin{document}

\frontmatter
\maketitle
\tableofcontents
\mainmatter

\chapter*{Preface}

The first purpose of this work is to 
provide a friendly introduction to the theory of nonautonomous linear systems of ordinary differential equations, the property of exponential dichotomy and its corresponding spectral theory. The second purpose of this work is disseminate the linearization results carried out by the authors in a nonautonomous framework.

The actual structure of this work is a consequence of several elective courses (2014, 2016, 2019, 2021 and 2023) carried out by the authors for undergraduate and graduated students at the Department of Mathematics of the Universidad de Chile. The monography assumes a good knowledge of multivariate calculus, linear algebra and ordinary differential equations.

%%%%%%%%%%%%%%%%%%%%%%%%%%%%%%%%%%%%
% Give credit where credit is due. %
% Say thanks!                      %
%%%%%%%%%%%%%%%%%%%%%%%%%%%%%%%%%%%%
\section*{Acknowledgements}
\begin{itemize}
\item  We want to express our appreciation and gratitude to our students Basti\'an Nu\~nez, Karla Catal\'an, N\'estor Jara, Valdo Lagos, Fernanda Torres, Ricardo Torres, Sebasti\'an Rivera and David Urrutia, who
read preliminary versions and made important remarks.
\item A special word of thanks goes to the \textit{joyeuse bande} of postdoctorants: Ignacio Huerta, Claudio Gallegos
and Heli Elorreaga, who read several intermediate versions of this work.  

\item The hospitality of Pablo Monz\'on (Universidad de la Rep\'ublica, Montevideo -- Uruguay), Christian P\"otzsche
(Alpen--Adria Universit\"at, Klagenfurt -- Austria) and Pablo Amster (Universidad de Buenos Aires -- Argentina) was fundamental
to work on several parts of this monography, we are deeply grateful with our colleagues.

%\item I'll also like to thank Gummi\footnote{\url{http://gummi.midnightcoding.org/}} developers and LaTeXila\footnote{\url{http://projects.gnome.org/latexila/}} development team for their awesome \LaTeX{} editors.
%\item I'm deeply indebted my parents, colleagues and friends for their support and encouragement.
\end{itemize}
%\mbox{}\\
%\noindent Amber Jain \\
%\noindent %\url{http://amberj.devio.us/}

%%%%%%%%%%%%%%%%
% NEW CHAPTER! %
%%%%%%%%%%%%%%%%
\

%\nomenclature[ainf]{$\inf \varnothing$}{$\inf \varnothing=+\infty$}%
%\onehalfspacing
%\nomenclature[ainf]{$BC(\mathbb{R},\mathbb{R}^{n})$}{$$}%
%\nomenclature[ainf]{$C^{1}(\Omega,\mathbb{R}^{n})$}{$$}%
%\printnomenclature

%\nomenclature[ainf]{$P^{*}$}{$$}%
%\nomenclature[ainf]{$C^{1}(\Omega,\mathbb{R}^{n})$}{$$}%

%\nomenclature[ainf]{$C_{\omega}(\mathbb{R},\mathbb{R}^{n})$}{$$}
%%%%%%%%%%%%%%%%%%%%%%%%%%

\section*{Notation}

\subsection*{Chapter 1}
\noindent\begin{tabularx}{\textwidth}{@{} l X @{}}
\toprule
$J$ & a time interval, $J\subseteq \mathbb{R}$.\\
$M_n(\mathbb{R})$ & the space of $n\times n$ real matrices.\\
$|x|$ & a norm of $x\in \mathbb{R}^n$.\\
$\|A\|$ & the operator norm of $A\in M_n(\mathbb{R})$ induced by $|\cdot|$.\\
$\det A$ & the determinant of $A\in M_n(\mathbb{R})$.\\
$\textnormal{Tr} A$ & the trace of $A\in M_n(\mathbb{R})$.\\
$I$ & the identity matrix in $M_n(\mathbb{R})$.\\
$\mathfrak{M}_n$ & the set of continuous and bounded matrix functions $A(\cdot)\colon J\to M_n(\mathbb{R})$.\\
$\dot x = A(t)x$ & a linear nonautonomous system on $J$.\\
$X(t)$ & a fundamental matrix of $\dot x = A(t)x$.\\
$X(t,s)$ & the transition matrix $X(t,s):=X(t)X^{-1}(s)$, for $t,s\in J$.\\
$K,\alpha$ & positive constants appearing in growth/decay estimates.\\
\bottomrule
\end{tabularx}

\subsection*{Chapter 2}
\noindent\begin{tabularx}{\textwidth}{@{} l X @{}}
\toprule
$BC(\mathbb{R},\mathbb{R}^n)$ & the space of bounded continuous functions $\mathbb{R}\to \mathbb{R}^n$.\\
$C_{\omega}(\mathbb{R},\mathbb{R}^n)$ & the space of $\omega$--periodic continuous functions $\mathbb{R}\to \mathbb{R}^n$.\\
$|x|$ & a norm of $x\in \mathbb{R}^n$.\\
$\|A\|$ & the operator norm of $A\in M_n(\mathbb{R})$ induced by $|\cdot|$.\\
$\operatorname{rank}A$ & the rank of a matrix $A$.\\
$\textnormal{Tr}A $ & the trace of $A\in M_n(\mathbb{R})$.\\
$\mathcal{B},\mathcal{D}$ & function spaces used in the notion of admissibility on $J$.\\
$P$ & a (constant) projection, $P^2=P$.\\
$P(t)$ & a family of projections (time--dependent projector).\\
$K,\alpha$ & constants in the exponential dichotomy estimates.\\
$\mathcal{V}(t_0)$ & forward--bounded set: 
$$
\{\xi:\ t\mapsto |X(t,t_0)\xi|\ \text{is bounded on }[t_0,+\infty)\}.
$$\\
$\mathcal{W}(t_0)$ & backward--bounded set: 
$$
\{\xi:\ t\mapsto |X(t,t_0)\xi|\ \text{is bounded on }(-\infty,t_0]\}.
$$
\\
$\mathcal{G}(t,s)$ & Green function associated with an exponential dichotomy.\\
$E^{s},E^{u}$ & stable and unstable subspaces in the hyperbolic (autonomous) setting.\\
\bottomrule
\end{tabularx}

\subsection*{Chapter 3}
\noindent\begin{tabularx}{\textwidth}{@{} l X @{}}
\toprule
$\mathbb{R}_{0}^{+}$ & the half--line $[0,+\infty)$.\\
$\mathbb{R}_{0}^{-}$ & the half--line $(-\infty,0]$.\\
$T,\theta$ & constants in the uniform noncriticality condition.\\
$P_{+},P_{-}$ & constant projections associated with exponential dichotomies on $\mathbb{R}_{0}^{+}$ and $\mathbb{R}_{0}^{-}$.\\
$i(A)$ & the index of a linear system (defined via $P_{+}$ and $P_{-}$).\\
$\mathcal{V}$ & forward--bounded set: $\{\xi:\ \sup_{t\ge 0}|X(t,0)\xi|<+\infty\}$.\\
$\mathcal{W}$ & backward--bounded set: $\{\xi:\ \sup_{t\le 0}|X(t,0)\xi|<+\infty\}$.\\
\bottomrule
\end{tabularx}

\subsection*{Chapter 4}
\noindent\begin{tabularx}{\textwidth}{@{} l X @{}}
\toprule
$\dot x = A(t)x$ & a linear system on $J$.\\
$\dot y = B(t)y$ & a linear system kinematically similar to $\dot x = A(t)x$.\\
$S(t)$ & an invertible (Lyapunov) change of variables relating two systems.\\
$A^{*}$ & Hermitian transpose (conjugate transpose) of a complex matrix.\\
$A^{T}$ & transpose of a real matrix.\\
$A>0$ (resp.\ $A\ge 0$) & positive definite (resp.\ positive semidefinite) matrix.\\
\bottomrule
\end{tabularx}

\subsection*{Chapter 5}
\noindent\begin{tabularx}{\textwidth}{@{} l X @{}}
\toprule
$\lambda$ & spectral parameter in shifted systems.\\
$\Sigma(A)$ & exponential dichotomy spectrum on $\mathbb{R}$.\\
$\Sigma^{+}(A)$ & future exponential dichotomy spectrum on $\mathbb{R}_{0}^{+}$.\\
$\Sigma^{-}(A)$ & past exponential dichotomy spectrum on $\mathbb{R}_{0}^{-}$.\\
$\rho(A)$ & resolvent set $\rho(A)=\mathbb{R}\setminus\Sigma(A)$.\\
$\rho^{+}(A)$ & resolvent set $\rho^{+}(A)=\mathbb{R}\setminus\Sigma^{+}(A)$.\\
$\rho^{-}(A)$ & resolvent set $\rho^{-}(A)=\mathbb{R}\setminus\Sigma^{-}(A)$.\\
$P_{\lambda}$ & projection associated with the exponential dichotomy of the shifted system at $\lambda$.\\
$C_{\lambda}$ & maximal open interval (component) of the resolvent containing $\lambda$.\\
$\mathcal{S}^{\lambda}$ & stable set for the $\lambda$--shifted dynamics (forward bounded with weight $e^{-\lambda t}$).\\
$\mathcal{U}^{\lambda}$ & unstable set for the $\lambda$--shifted dynamics (backward bounded with weight $e^{-\lambda t}$).\\
\bottomrule
\end{tabularx}

\subsection*{Chapter 6}
\noindent\begin{tabularx}{\textwidth}{@{} l X @{}}
\toprule
$H(t,u)$ & homeomorphism defining topological equivalence (Palmer).\\
$G(t,u)$ & inverse map, $G(t,\cdot)=H(t,\cdot)^{-1}$.\\
$f(t,y)$ & nonlinear perturbation in $y' = A(t)y + f(t,y)$.\\
$M,c$ & constants in bounded growth: $\|X(t,s)\|\le Me^{c|t-s|}$.\\
$\mu$ & bound on the perturbation: $|f(t,y)|\le \mu$.\\
$\gamma$ & Lipschitz constant: $|f(t,y)-f(t,y')|\le \gamma|y-y'|$.\\
\bottomrule
\end{tabularx}

\subsection*{Chapter 7}
\noindent\begin{tabularx}{\textwidth}{@{} l X @{}}
\toprule
$J$ & a (typically half--line) interval where equivalence is studied.\\
$C^{r}$ topological equivalence & $H(t,\cdot)$ is a $C^{r}$--diffeomorphism and derivatives are continuous in $(t,u)$.\\
$D^{k}H,\ D^{k}G$ & partial derivatives up to order $k\le r$ (with respect to the space variable).\\
\bottomrule
\end{tabularx}

\chapter{Preliminaries}

\section{Linear nonautonomous systems}
This chapter is devoted to briefly recalling basic results about the linear system of ordinary differential equations
\begin{equation} 
\label{LA}
\dot{x}=A(t)x
\end{equation}
where $x$ is a column vector of $\mathbb{R}^n$, $J\ni t\mapsto A(t)\in M_{n}(\mathbb{R})$ is a continuous valued matrix function and $J\subseteq \mathbb{R}$ is a no degenerate interval. In general, but not always, we will consider intervals as $J=(-\infty,t_{0}]$, $J=[t_{0},+\infty)$ or $J=\mathbb{R}$.

This system has been deeply studied in a myriad of classical \textit{cold war} works from both sides of the iron courtain as \cite{Adrianova,Dalecki,Demidovich,Erugin,Halanay,Nemit} and \cite{Bellman1,Bellman2,Brauer,Coddington,Coppel1,Conti,Dangelo,Hahn,Hartman}. There also exists more recent monographs as \cite{Berthelin,Graef,Hespanha,Hinrichsen,Khalil,Sideris}. Nevertheless, we will review a short list of essential topics and results in order to self contain this work.
\subsection{Fundamental matrix and transition matrix}

\begin{definition}
A \textbf{fundamental matrix} of \eqref{LA} is a matrix valued function
$t\mapsto X(t)\in M_{n}(\mathbb{R})$ whose columns $\{x_{1}(t),\ldots,x_{n}(t)\}$ are a basis of solutions for the system \eqref{LA} for any $t\in J$.
\end{definition}

The study of fundamental matrices implies recalling facts 
about matrix valued functions and its derivatives. In fact,
for any matrix function $J\in t\mapsto U(t)$, not necessarily a 
fundamental matrix, we will say that $U$ is differentiable at $t$ if and only if the following limit exists 
\begin{displaymath}
\dot{U}(t)=\frac{d}{dt}U(t):=\lim\limits_{h\to 0}\frac{1}{h}[U(t+h)-U(t)],
\end{displaymath}
which must be understood as a componentwise limit.

\begin{lemma}
\label{dimf}
If a matrix function $J\ni t\to U(t)$ is derivable and invertible for any $t\in J$, then its inverse is derivable on $J$ with derivate given by
\begin{equation}
\label{dimf2}
\frac{d}{dt}U^{-1}(t)= -U^{-1}(t)\dot{U}(t)U^{-1}(t).
\end{equation}
\end{lemma}

\begin{proof}
If $U$ is derivable on $J$, it follows that $U$ is continuous on $J$.
This fact, combined with the invertibility of $U$, also implies the continuity
of $U^{-1}$ on $J$. Now, we can see that the following limit exists
\begin{displaymath}
\begin{array}{rcl}
\displaystyle \frac{d}{dt}U^{-1}(t) & = & \displaystyle \lim\limits_{h\to 0}\frac{1}{h}\left[U^{-1}(t+h)-U^{-1}(t)\right]\\\\
&=&\lim\limits_{h\to 0}\frac{1}{h}U^{-1}(t+h)\left[I-U(t+h)U^{-1}(t)\right]\\\\
&=& \displaystyle \lim\limits_{h\to 0}U^{-1}(t+h)\frac
{1}{h}\left[U(t)-U(t+h)\right]U^{-1}(t)
\end{array}
\end{displaymath}
and the identity (\ref{dimf2}) follows.
\end{proof}

\begin{remark}
\label{R2}
For any $t\in J$, it follows that any fundamental matrix $X(t)$ of (\ref{LA}) verifies the identity:
\begin{equation}
\label{DMF}
\dot{X}(t)=A(t)X(t).
\end{equation}

In fact, we can write $X(t)=[x_{1}(t)\, x_{2}(t)\, \cdots \, x_{n}(t)]$ 
and observe that:
\begin{displaymath}
\begin{array}{rcl}
\dot{X}(t)&=& [\dot{x}_{1}(t)\, \dot{x}_{2}(t)\, \cdots \, \dot{x}_{n}(t)]\\\\
&=&[A(t)x_{1}(t)\, A(t)x_{2}(t)\, \cdots \, A(t)x_{n}(t)]\\\\
&=&A(t)[x_{1}(t)\, x_{2}(t)\, \cdots \, x_{n}(t)]\\\\
&=&A(t)X(t).
\end{array}
\end{displaymath}
\end{remark}

For example, let us consider the linear system
\begin{equation}
\label{EXMY}
\dot{x}=\left [\begin{array}{cc}
-1 + \frac{3}{2} \cos^2(t) & 1 - \frac{3}{2} \cos(t) \sin(t)\\
-1 - \frac{3}{2}\cos(t) \sin(t) & -1 + \frac{3}{2} \sin^2(t)
\end{array}
\right]x,
\end{equation}
which have been studied in \cite{MY}. It straightforward to verify that
\begin{displaymath}
X(t)= \left[\begin{array}{cc}
e^{\frac{t}{2}} \cos(t) & e^{-t}\sin(t)\\
-e^{\frac{t}{2}} \sin(t) & e^{-t} \cos(t)
\end{array}\right]
\end{displaymath}
is a fundamental matrix of the above linear system.

In the autonomous case, any course of ordinary differential equations provides a method to construct
a fundamental matrix $X(t)=e^{At}$ by using the eigenvalues and eigenvectors of
$A$. Nevertheless, and contrarily to the autonomous case, there is no algorithm applicable to every 
linear system (\ref{LA}) that allows finding a basis of solutions. Furthermore, the system (\ref{EXMY}) is an iconic example illustrating the subtleties and technicalities of the nonautonomous case. In fact, notice that the eigenvalues of its corresponding matrix $A(t)$ have negative real part but the first column of $X(t)$ is a unstable solution.

\begin{remark}
\label{R3}
If $X(t)$ is a fundamental matrix associated to (\ref{LA}), for any $t\in J$ it is easy to verify the following facts:
\begin{itemize}
\item[a)] If $C\in M_{n}(\mathbb{R})$ is nonsingular, then $Y(t)=X(t)C$ is also a fundamental matrix associated to (\ref{LA}).
\item[b)] We always can choose a fundamental matrix $Y(t)$ satisfying $Y(0)=I$. In fact, we can use $C=X^{-1}(0)$ in the above property.  This matrix $Y(t)$ is usually called a \textit{normalized} fundamental matrix.
\item[c)] If $J=\mathbb{R}$ and $A(-t)=-A(t)$ for any $t\in J$, then $Y(t)=X(-t)$ is also a fundamental matrix associated to (\ref{LA}).
%\item[d)] If $J=\mathbb{R}$ and $A(-t)=A(t)$ for any $t\in J$, then $Y(t)=-X(-t)$ is also a fundamental matrix associated to (\ref{LA}),
\item[d)] By its own definition we have that
for any $t\in J$, the fundamental matrix $X(t)$ is invertible since its columns are linearly independent vectors of $\mathbb{R}^{n}$. As $X(t)$ is invertible and derivable, it is followed by Lemma \ref{dimf} that $X^{-1}(t)$ is also derivable
and (\ref{dimf2}) combined with (\ref{DMF}) implies the identity:
\begin{equation}
\label{DMF2}
\frac{d}{dt}X^{-1}(t)=-X^{-1}(t)A(t).    
\end{equation}
\end{itemize}
\end{remark}

The next result is a direct consequence of the identity (\ref{DMF2}):
\begin{lemma}
\label{pitaron}
The map $J\ni t\mapsto X(t)X^{-1}(t_{0})c$ with $t_{0}\in J$ and $c\in \mathbb{R}^{n}$ is  being the unique solution of the
initial value problem:
\begin{equation}
\label{Cauchy}
\left\{\begin{array}{rcl}
\dot{x}&=&A(t)x\\
x(t_{0})&=&c.
\end{array}\right.
\end{equation}
\end{lemma}

\begin{proof}
It is straightforward to verify that $t\mapsto X(t)X^{-1}(t_{0})c$ is a solution of the initial value problem (\ref{Cauchy}). In order to prove its uniqueness let $t\mapsto y(t)$ also be a solution of the above problem. Now, let us define the auxiliar function $u(t)=X(t)X^{-1}(t_{0})c-y(t)$ and note that $u(t)$ verifies:
$$
\dot{u}(t)=A(t)u(t) \quad \textnormal{with $u(t_{0})=0$}.
$$

First of all, let us assume that $t\geq t_{0}$ and
integrates the above equation between $t_{0}$ and $t$. By using the fact that $u(t_{0})=0$ it follows that
$$
u(t)=\int_{t_{0}}^{t}A(s)u(s)\,ds,
$$
and consequently

\begin{displaymath}
\begin{array}{rcl}
|u(t)| &=&  \displaystyle \left |\int_{t_{0}}^{t}A(s)u(s)\,ds \right|\\\\
&\leq &  \displaystyle \int_{t_{0}}^{t}|A(s)u(s)|\,ds \\\\
&\leq & \displaystyle  \int_{t_{0}}^{t}||A(s)||\,|u(s)|\,ds.
\end{array}
\end{displaymath}

Secondly, let us assume that $t\leq t_{0}$ and
integrates the above equation between $t$ and $t_{0}$. By reiterating the fact that $u(t_{0})=0$ it follows that:
$$
-u(t)=\int_{t}^{t_{0}}A(s)u(s)\,ds,
$$
and consequently
\begin{displaymath}
\begin{array}{rcl}
|u(t)|&\leq & \displaystyle  \int_{t}^{t_{0}}||A(s)||\,|u(s)|\,ds \\\\
\end{array}
\end{displaymath}

Consequently, we have that
$$
|u(t)|\leq \left|\int_{t_{0}}^{t}||A(s)||\,|u(s)|\,ds\right| \leq c_{n}+ \left|\int_{t_{0}}^{t}||A(s)||\,|u(s)|\,ds\right|
$$
for any positive sequence $\{c_{n}\}_{n\in \mathbb{N}}$ convergent to zero. 

By the variation of Gronwall's Lemma (see Appendix C for details), we have that
$$
|u(t)|\leq c_{n}\exp\left(\left|\int_{t_{0}}^{t}||A(s)||\,ds\right|\right).
$$

Finally, the uniqueness follows by letting $c_{n}\to 0$ when $n\to +\infty$.
\end{proof}

\begin{definition}
%\label{MT}
Let $X(t)$ be a fundamental matrix of \eqref{LA}. The \textbf{transition matrix} associated to $X(t)$ is defined by 
$$
X(t,s):=X(t)X^{-1}(s) \quad \textit{where $t,s\in J$}.
$$
\end{definition}

We point out that, in the russian literature, the transition matrix is also called as \textit{Cauchy operator} or \textit{matriciant}
while in PDE and functional analysis research it is usually called as evolution operator.

\begin{remark}
\label{remark-MT}
By using Remarks \ref{R2} and \ref{R3}, we can see that:
\begin{equation}
\label{MT1}
\frac{\partial X}{\partial t}(t,s)=A(t)X(t,s) \quad
\textnormal{and} \quad \frac{\partial X}{\partial s}(t,s)=-X(t,s)A(s).
\end{equation}
\end{remark}

\begin{remark}
\label{remark-MT2}
Any transition matrix $X(t,s)$ associated to (\ref{LA}) verifies the following properties:
\begin{itemize}
\item[a)] For any $t,t_{0},s \in J$ it follows that: 
$$
X(t,t)=I \quad \textnormal{and} \quad X(t,t_{0})X(t_{0},s)=X(t,s).
$$ 
%\item[2.-] Given $t,t_{0}\in J$ and a matrix norm $||\cdot||$, the Gronwall inequality implies that
%\begin{equation}
%\label{BG0}
%||X(t,t_{0})||\leq ||I||\,e^{\Big|\int_{t_{0}}^{t}||A(s)||\,ds\Big|}.
%\end{equation}
\item[b)] Any solution $t\mapsto x(t)$ of the initial value problem (\ref{Cauchy}) can be written as 
$x(t)=X(t,t_{0})c$. By using a) we can see that, for any $t,s\in J$, we have the identity $x(t)=X(t,s)x(s)$, which
is behind the adjective \textit{transition}.
\item[c)] If $A\in M_{n}(\mathbb{R})$ is a constant matrix, it can be seen that $e^{At}$ is a fundamental matrix of the system (\ref{LA}) when $A(t)$ is constant for any $t\in J$. By using Lemma \ref{pitaron}, we have that:
\begin{displaymath}
X(t,t_{0})=e^{A(t-t_{0})}.
\end{displaymath}
\end{itemize}
\end{remark}

The next result provides estimations for the transition matrix and can be found in
\cite[Lemma 3.3.4]{Hinrichsen}:
\begin{lemma}
\label{bgpg}
Given $t,t_{0}\in J$ it follows that
\begin{equation}
\label{BGO-I}
\left\{\begin{array}{l}
e^{-\left|\int_{t_{0}}^{t}||A(s)||\,ds\right|}\leq ||X(t,t_{0})||\,\,\,\, \leq e^{\left|\int_{t_{0}}^{t}||A(s)||\,ds\right|},\\\\
e^{-\left|\int_{t_{0}}^{t}||A(s)||\,ds\right|}\leq ||X(t_{0},t)||^{-1} \leq e^{\left|\int_{t_{0}}^{t}||A(s)||\,ds\right|},
\end{array}\right.
\end{equation}
where $||\cdot||$ is a matrix norm satisfying $||I||=1$.
\end{lemma}
\begin{proof}
By  (\ref{remark-MT}) we know that
$$
\frac{\partial X}{\partial t}(t,t_{0})=A(t)X(t,t_{0}) \quad \textnormal{and} \quad
\frac{\partial X}{\partial t}(t_{0},t)=-X(t_{0},t)A(t). 
$$

If $t>t_{0}$, we integrate the above equations, obtaining
$$
X(t,t_{0})=I+\int_{t_{0}}^{t}A(s)X(s,t_{0})\,ds,  \quad X(t_{0},t)=I-\int_{t_{0}}^{t}A(s)X(t_{0},s)\,ds,
$$
which leads to
$$
||X(t,t_{0})||\leq 1+\int_{t_{0}}^{t}||A(s)||\,||X(s,t_{0})||\,ds
$$
and
$$
||X(t_{0},t)||\leq 1+\int_{t_{0}}^{t}||A(s)||\,||X(t_{0},s)||\,ds.
$$

By using Gronwall's Lemma (see Appendix B) we obtain
$$
||X(t,t_{0})||\leq e^{\int_{t_{0}}^{t}||A(s)||\,ds} \quad \textnormal{and} \quad
 e^{-\int_{t_{0}}^{t}||A(s)||\,ds} \leq ||X(t_{0},t)||^{-1}.
$$

By statement 1 from Remark \ref{remark-MT2} we have that
$$
1=||I||=||X(t,t_{0})X(t_{0},t)||\leq ||X(t,t_{0})||\,||X(t_{0},t)||,
$$
and by using the above estimations, we can deduce that
$$
e^{-\int_{t_{0}}^{t}||A(s)||\,ds}\leq ||X(t,t_{0})|| \quad \textnormal{and} \quad ||X(t_{0},t)||^{-1}\leq e^{\int_{t_{0}}^{t}||A(s)||\,ds}
$$
and the Lemma follows for $t>t_{0}$. The proof for the case $t<t_{0}$ can be carried out in a similar way and is left for the reader.
\end{proof}

\begin{lemma}
\label{SNHC1}
The initial value problem:
\begin{equation}
\label{No-homogeneo}
\left\{\begin{array}{rcl}
\dot{x} & = & A(t)x+f(t)\\
x(t_{0}) &=& x_{0}   
\end{array}\right.
\end{equation}
where $J\ni t\mapsto f(t)\in \mathbb{R}$, $t_{0}\in J$ and $x_{0}\in \mathbb{R}^{n}$ has a unique solution $t\mapsto x(t)$ given by:
\begin{equation}
\label{solucion-NH}    
\displaystyle
x(t)=X(t,t_{0})x_{0}+\int_{t_{0}}^{t}X(t,s)f(s)\,ds.
\end{equation}
\end{lemma}
\begin{proof}
Let $J\ni s\mapsto x(s)$ be a solution of (\ref{No-homogeneo}), then we multiply $\dot{x}(s)$ by $X^{-1}(s)$ obtaining that
$$
X^{-1}(s)\dot{x}(s)-X^{-1}(s)A(s)x(s)=X^{-1}(s)f(s).
$$

Then, by using Remark \ref{R3} we have that the above identity is equivalent to
\begin{displaymath}
\frac{d}{ds}\Big\{X^{-1}(s)x(s)\Big\}=X^{-1}(s)f(s),
\end{displaymath}
thereafter, we integrate between $t_{0}$ and $t$, then  we can deduce that
\begin{displaymath}
X^{-1}(t)x(t)=X^{-1}(t_{0})x_{0}+\int_{t_{0}}^{t}X^{-1}(s)f(s)\,ds
\end{displaymath}
and the identity (\ref{solucion-NH}) can be deduced easily.

In order to prove the uniqueness, let us consider the solution $J\ni t\mapsto z(t)$ and define $u(t)=x(t)-z(t)$. Then, we can see that $t\mapsto u(t)$ is the solution of the initial value problem (\ref{Cauchy}) with $u(t_{0})=c=0$. Hereafter, by using Lemma \ref{pitaron} we have that $u(t)=0$ for any $t\in J$ and the uniqueness follows.
\end{proof}

From the starting point that fundamental and transition matrices are solutions of matrix differential equations such as (\ref{remark-MT}), we can see that it is possible to study 
matrix differential equations generalizing the above matrix equation. The following results
are adapted from \cite[Ch.3]{Dalecki}.

\begin{lemma}
%\label{SEDM0}
Let us consider the continuous matrix valued functions $J\ni t\mapsto A(t),B(t),F(t)\in M_{n}(\mathbb{R})$. The matrix differential equation
\begin{equation}
\label{EDM00}
\begin{array}{rcl}
\dot{Q}(t) & = & A(t)Q(t)-Q(t)B(t)+F(t)\\
\end{array}
\end{equation}
has a solution given by
\begin{equation}
\label{we}    
\displaystyle
Q(t)=X(t,s)Q(s)Y(s,t)+\int_{s}^{t}X(t,\tau)F(\tau)Y(\tau,t)\,d\tau,
\end{equation}
where $X(t)$ and $Y(t)$ are fundamental matrices associated to the linear systems
$\dot{x}=A(t)x$ and $\dot{y}=B(t)y$ respectively.
\end{lemma}

\begin{proof}
We multiply (\ref{EDM00}) by $X(s,\tau)$ on the left. Then, by using (\ref{MT1}) we obtain
\begin{displaymath}
\begin{array}{rcl}
X(s,\tau)F(\tau) &=& X(s,\tau)\dot{Q}(\tau)-X(s,\tau)A(\tau)Q(\tau)+X(s,\tau)Q(\tau)B(\tau),\\\\
&=& \displaystyle X(s,\tau)\dot{Q}(\tau)+\frac{d}{d\tau}\left\{X(s,\tau)\right\}Q(\tau)+X(s,\tau)Q(\tau)B(\tau),\\\\
&=&\displaystyle \frac{d}{d\tau}\left\{ X(s,\tau)Q(\tau)\right\}+X(s,\tau)Q(\tau)B(\tau).
\end{array}
\end{displaymath}

In this manner, we multiply the above identity by $Y(\tau,s)$ on the right, then by using again (\ref{MT1}) we obtain
\begin{displaymath}
\begin{array}{rcl}
X(s,\tau)F(\tau)Y(\tau,s) &=&
\displaystyle \frac{d}{d\tau}\left\{ X(s,\tau)Q(\tau)\right\}Y(\tau,s)+X(s,\tau)Q(\tau)B(\tau)Y(\tau,s),\\\\
&=&
\displaystyle \frac{d}{d\tau}\left\{ X(s,\tau)Q(\tau)\right\}Y(\tau,s)+X(s,\tau)Q(\tau)\frac{d}{d\tau}\left\{Y(\tau,s)\right\},\\\\
&=&\displaystyle \frac{d}{d\tau}\left\{X(s,\tau)Q(\tau)Y(\tau,s)\right\}.
\end{array}
\end{displaymath}

By integrating between $s$ and $t$ we obtain
$$
X(s,t)Q(t)Y(t,s)=Q(s)+\int_{s}^{t}X(s,\tau)F(\tau)Y(\tau,s)\,d\tau,
$$
and (\ref{we}) can be deduced by multiplying the above identity
by $X(t,s)$ and $Y(s,t)$ from the left and right respectively.
\end{proof}

The above result has limit cases, which are important by themselves:
\begin{corollary}
\label{SEDM}
The matrix differential equations
\begin{subequations}
  \begin{empheq}{align}
    & \dot{Q}(t) = A(t)Q(t)+F(t) \label{EDM001}, \\
    & \dot{Q}(t) = A(t)Q(t)-B(t)Q(t)  \label{EDM002},
  \end{empheq}
\end{subequations}
where 
$J\ni t\mapsto A(t),B(t),F(t)\in M_{n}(\mathbb{R})$ have solutions respectively given by:
\begin{subequations}
  \begin{empheq}{align}
    & Q(t) = X(t,s)Q(s)+\int_{s}^{t}X(t,\tau)F(\tau)\,d\tau,\label{we1} \\
    & Q(t) = X(t,s)Q(s)Y(s,t)     \label{we2}
  \end{empheq}
\end{subequations}
where $X(t)$ and $Y(t)$ are fundamental matrices associated to the linear systems
$\dot{x}=A(t)x$ and $\dot{y}=B(t)y$.
\end{corollary}

\begin{remark}
\label{bwks}
The equation (\ref{EDM002}) with initial condition $Q(t_{0})=I$ at $t_{0}\in J$ has a solution 
$Q(t)=X(t)Y^{-1}(t)$, where these fundamental matrices verify $X(t_{0})=Y(t_{0})=I$. Moreover, as $Q(t)$ 
is non singular, this allows to write $Y(t)=Q^{-1}(t)X(t)$ and $X(t)=Q(t)Y(t)$. 
\end{remark}

The following Remark can be seen as a converse result of the previous ones:
\begin{remark}
\label{KS-Preh}
If a fundamental matrix $X(t)$ of $\dot{x}=A(t)x$ has the decomposition $X(t)=Q(t)V(t)$, where $Q$ and $V$ are derivable an invertible then
$Q$ satisfies the matrix differential equation
\begin{displaymath}
\dot{Q}(t)=A(t)Q(t)-Q(t)B(t) \quad \textnormal{with} \quad B(t)=\dot{V}(t)V^{-1}(t).
\end{displaymath}
\end{remark}

\subsection{Adjoint systems}
In order to introduce the adjoint of the linear system (\ref{LA}),
let us consider 
a particular case of the matrix differential equation (\ref{EDM002}) with $B(t)=-A^{T}(t)$: 
\begin{equation}
\label{MDEADJ}
\dot{Q}(t) = A(t)Q(t)+Q(t)A^{T}(t).
\end{equation}

A consequence of Corollary \ref{SEDM} and Remark \ref{bwks} is that the matrix differential equation 
(\ref{MDEADJ}) with initial condition $Q(t_{0})=I$ at some $t_{0}\in J$ has a solution 
$Q(t)=X(t)Y^{-1}(t)$, where $X(t)$ and $Y(t)$ are respectively fundamental matrices
(normalized at $t_{0}\in J$) of (\ref{LA}) and 
\begin{equation}
\label{adj}
\dot{y}=-A^{T}(t)y,
\end{equation}
which is known as the \textit{adjoint system} of (\ref{LA}).

\medskip
In order to study the relation between $X(t)$ and $Y(t)$, notice that:
$$
Y^{T}(t)\dot{X}(t)=Y^{T}(t)A(t)X(t).
$$
Moreover, we integrate by parts between $t_{0}$ and $t$ with $t,t_{0}\in J$ and, by using the identity $\frac{d}{dt}[U(t)]^{T}=\left[\frac{d}{dt}U(t)\right]^{T},$
we can deduce that
\begin{displaymath}
Y^{T}(t)X(t)= \displaystyle I+\int_{t_{0}}^{t}[\dot{Y}^{T}(\tau)+Y(\tau)^{T}A(\tau)]X(\tau)\,d\tau= I,
\end{displaymath}
since $Y(t)$ verifies $\dot{Y}^{T}(t)+Y(t)A(t)=0$ for any
$t\in J$ and we conclude that:  
\begin{theorem}
\label{Tadj0}
If $X(t)$ and $Y(t)$ are fundamental matrices of \eqref{LA} and its adjoint respectively, both normalized
at $t_{0}$, it follows that
\begin{equation*}
Y(t)=[X^{-1}(t)]^{T}  \quad \textnormal{for any $t\in J$}.
\end{equation*}
\end{theorem}

We stress that the above result can be obtained by direct computation by using the statement d) from
Remark \ref{R3}. Nevertheless, the Remark \ref{bwks} provides us a complementary point of view: to see
the fundamental matrix of (\ref{adj}) as $Y(t)=Q^{-1}(t)X(t)$ where $Q(\cdot)$ is a solution of (\ref{MDEADJ})
with $Q(t_{0})=X(t_{0})=Y(t_{0})=I$.

\medskip

The following result is a direct consequence from Theorem \ref{Tadj0}:
\begin{corollary}
%\label{Tadj}
Let $X(t)$ be a fundamental matrix of the system \eqref{LA}. Then $Y(t)$ is a fundamental matrix of its adjoint system \eqref{adj} if and only if
the following identity is satisfied
\begin{equation}
\label{identidad-adj}
Y^{T}(t)X(t)=C  \quad \textnormal{for any $t\in J$, where $C$ is invertible.}
\end{equation}
\end{corollary}

\begin{proof}
We first assume that $X(t)$ is a fundamental matrix of (\ref{LA}), then
$X(t)X^{-1}(t_{0})$ is also a fundamental matrix normalized at $t_{0}\in J$.
By Theorem \ref{Tadj0} we know that 
$$
Y(t)=[X(t_{0})X^{-1}(t)]^{T}=[X^{-1}(t)]^{T}X^{T}(t_{0}),
$$
which implies that $Y^{T}(t)X(t)=X(t_{0})$ and the identity (\ref{identidad-adj}) is verified with $C=X^{T}(t_{0})$. 

Now, let us assume that the identity (\ref{identidad-adj}) is verified. Given $t_{0}\in J$ we have that
$C=Y^{T}(t_{0})X(t_{0})$ and
$$
\tilde{X}(t)=X(t)X^{-1}(t_{0})=[Y^{T}(t)]^{-1}CX^{-1}(t_{0})=[Y^{T}(t)]^{-1}Y^{T}(t_{0})
$$
is a fundamental matrix of (\ref{LA}) normalized at $t_{0}\in J$. By applying Theorem \ref{Tadj0} we can deduce 
that
$$
[\tilde{X}^{-1}(t)]^{T}=Y(t)Y^{-1}(t_{0})
$$
is also a fundamental matrix of of (\ref{adj}) normalized at $t_{0}$. Finally, as $Y(t_{0})$ is invertible by using Remark \ref{R2}
we an conclude that $Y(t)$ is a fundamental matrix of (\ref{adj}) and the result follows. 
\end{proof}

A particular but interesting case of the above Theorems is given by the following result:
\begin{corollary}
\label{antisim}
If a linear system \eqref{LA} is such that $A(t)$ is antisymmetric for any $t\in J$, then any fundamental matrix $X(t)$ is orthogonal for any $J$.
\end{corollary}

\begin{proof}
Let $X(t)$ be a fundamental matrix of the system (\ref{LA}). As $A(t)=-A^{T}(t)$ for any $t\in J$, then the systems (\ref{LA}) and (\ref{adj}) coincides, which implies that $X(t)=[X^{-1}(t)]^{T}$ and the orthogonality of $X(t)$ follows.
\end{proof}

\begin{remark}
\label{ejemplo-reducible}
An interesting example of the above Corollary is given by the linear systems
\begin{displaymath}
\dot{x}=\left[\begin{array}{cc}
0 & e^{t} \\
-e^{t} & 0
\end{array}\right]x \quad \textnormal{and} \quad
\dot{y}=\left[\begin{array}{cc}
0 & 1 \\
-1 & 0
\end{array}\right]y 
\end{displaymath}
which have fundamental matrices
\begin{displaymath}
X(t)=\left[\begin{array}{rc}
\cos(e^{t}) & \sin(e^{t}) \\
-\sin(e^{t}) & \cos(e^{t})
\end{array}\right] \quad \textnormal{and} \quad
Y(t)=\left[\begin{array}{rc}
\cos(t) & \sin(t) \\
-\sin(t) & \cos(t)
\end{array}\right].
\end{displaymath}

We can easily see that $X^{T}(t)=X^{-1}(t)$ and $Y^{-1}(t)=Y^{T}(t)$. Note also that $t\mapsto \det X(t)$ and $t\mapsto \det Y(t)$
are equal to $1$ for any $t\in \mathbb{R}$. Finally, as any antisymmetric matrix has null trace, the above examples
provide a good illustration for the topic of the next subsection.
\end{remark}

\subsection{The Liouville's formula}
Given a linear system (\ref{LA}), the Liouville's formula relates the determinant of 
any fundamental matrix $X(t)$ with the trace of its corresponding matrix as follows:
\begin{theorem}
\label{liouville}
For any $t,t_{0}\in J$, the determinant of any fundamental matrix $X(t)$ of \eqref{LA}
verifies the identity:
\begin{equation}
\label{FL}
\det X(t) =\det X(t_{0})e^{\int_{t_{0}}^{t}\textnormal{Tr}A(s)\,ds}.
\end{equation}
\end{theorem}

\begin{proof}
Note that 
\begin{displaymath}
\frac{d}{dt}\Big(\det X(t_{0})e^{\int_{t_{0}}^{t}\textnormal{Tr} A(s)\,ds}\Big)=\textnormal{Tr} A(t)\det X(t_{0})e^{\int_{t_{0}}^{t}\textnormal{Tr}A(s)\,ds}
\end{displaymath}
and we can verify that 
$t\mapsto \det X(t_{0})e^{\int_{t_{0}}^{t}\textnormal{Tr} A(s)\,ds}$ is solution of the scalar initial value problem
\begin{equation}
\label{escalar}
\left\{\begin{array}{rcl}
\dot{u}&=&\textnormal{Tr} A(t)u\\
u(t_{0})&=&\det X(t_{0}).
\end{array}\right.
\end{equation}

Now, let us calculate the derivative of the right part of (\ref{FL}):
\begin{displaymath}
\begin{array}{rcl}
\frac{d}{dt}\left[\det X(t)\right]&=& \displaystyle \lim\limits_{h\to 0} \frac{\det X(t+h)-\det X(t)}{h}\\\\
&=&\displaystyle \lim\limits_{h\to 0} \frac{\det [X(t+h)X^{-1}(t)X(t)]-\det X(t)}{h}\\\\
&=&\displaystyle \lim\limits_{h\to 0} \frac{\det [X(t+h)X^{-1}(t)]\det X(t)-\det X(t)}{h}\\\\
&=&\displaystyle \left( \lim\limits_{h\to 0} \frac{\det [X(t+h)X^{-1}(t)]-\det I}{h}\right)\det X(t).
\end{array}
\end{displaymath}

Now, notice that
\begin{displaymath}
\begin{array}{rcl}
\displaystyle
X(t+h)X^{-1}(t)&=& \displaystyle  I+\int_{t}^{t+h}A(s)X(s)X^{-1}(t)\,ds\\\\
&=& \displaystyle I+h\theta(t,h) \quad \textnormal{with} \quad  \theta(t,h)=\frac{1}{h}\int_{t}^{t+h}A(s)X(s,t)\,ds.
\end{array}
\end{displaymath}

By the above identity, the computation of $\det X(t+h)X^{-1}(t)$ can be made in a similar way as the computation of the characteristic polynomial of a matrix. In fact, note that
\begin{displaymath}
\begin{array}{rcl}
\det X(t+h)X^{-1}(t) & = & \det \left(hh^{-1}I+h\theta(t,h)\right)\\\\
&= &h^{n}\det\left(h^{-1}I+\theta(t,h)\right)\\\\
&= &h^{n}\det\left(h^{-1}I-\sigma(t,h)\right) \,\,\, \textnormal{with $\sigma(t,h)=-\theta(t,h)$}.
\end{array}
\end{displaymath}

By using the Leverrier--Faddeev algorithm \cite{Hou}, we have that
\begin{displaymath}
\begin{array}{rcl}
\det\left(h^{-1}I-\sigma(t,h)\right)&=& h^{-n}-h^{1-n}\textnormal{Tr}(\sigma_{0})-\sum\limits_{k=2}^{n}\textnormal{Tr}(\sigma_{k-1})\frac{h^{k-n}}{k},
\end{array}
\end{displaymath}
where $\sigma_{k}$ (with $k=0,\ldots,n-1$) is defined recursively as
\begin{displaymath}
\left\{\begin{array}{rcl}
\sigma_{0}&=&-\theta(t,h)\\\\
\sigma_{k}&=&-\theta(t,h)\left(\sigma_{k-1}-\frac{1}{k}\textnormal{Tr}(\sigma_{k-1}I)\right).
\end{array}\right.
\end{displaymath}

Then, it follows that
$$
\det X(t+h)X^{-1}(t)=1+h\,\textnormal{Tr}[\theta(t,h)]+\sum\limits_{k=2}^{n}\frac{h^{k}}{k}\textnormal{Tr}(\sigma_{k-1}).
$$

By  using the above identity combined with $\det I=1$, we can finish the computation:
\begin{displaymath}
\begin{array}{rcl}
\frac{d}{dt}\det X(t)&=&
\lim\limits_{h\to 0}\left(\textnormal{Tr}(\theta(t,h))+\sum\limits_{k=2}^{n}\frac{h^{k-1}}{k}\textnormal{Tr}(\sigma_{k-1})\right)\det X(t)\\\\
&=& \displaystyle\lim\limits_{h\to 0}\textnormal{Tr}\left[\frac{1}{h}\int_{t}^{t+h}A(s)X(s)X^{-1}(t)\right]\det X(t)\\\\
&=&\textnormal{Tr} A(t)\det X(t), 
\end{array}
\end{displaymath}
which implies that $t\mapsto \det X(t)$ is also solution of the initial value problem (\ref{escalar}) and the identity (\ref{FL}) follows from Lemma \ref{pitaron}.
\end{proof}

The Liouville's formula is also known as Abel--Jacobi--Liouville in the literature,
we refer the reader to \cite{Adrianova} for details.

\section{Bounded growth}
To the best of our the property of bounded growth was introduced by W.A. Coppel in \cite[pp. 8--9]{Cop}. Nevertheless,
we will consider a pair of definitions proposed by K.J. Palmer in \cite[p.s172]{Palmer2006}:
\begin{definition}
\label{BG-C}
The system \eqref{LA} has a \textbf{bounded growth} on $J$ if and only if there exists constants $K>0$ and $\alpha \geq 0$ such that its transition matrix satisfies
\begin{equation}
\label{BG-S}
||X(t,s)||\leq Ke^{\alpha(t-s)}\quad \textnormal{for any $t\geq s$ with $t,s\in J$.}
\end{equation}
\end{definition}

\begin{definition}
%\label{BD-C}
The system \eqref{LA} has a \textbf{bounded decay} on $J$ if and only if there exists constants $K>0$ and $\alpha \geq 0$ such that its transition matrix satisfies
\begin{equation*}
%\label{BG-S1}
||X(t,s)||\leq Ke^{\alpha(s-t)}\quad \textnormal{for any $s\geq t$ with $t,s\in J$.}
\end{equation*}
\end{definition}

It is straightforward to see that the above definitions describe properties which are contained in the following definition
stated by S. Siegmund in \cite[p.253]{Siegmund-2002}:
\begin{definition}
\label{BGBD-C}
The system \eqref{LA} has a \textbf{bounded growth $\&$ decay} on $J$ if and only if there exists constants $K>0$ and $\alpha \geq 0$ such that its transition matrix satisfies
\begin{equation*}
%\label{BG-S3}
||X(t,s)||\leq Ke^{\alpha|t-s|}\quad \textnormal{for any $t,s\in J$.}
\end{equation*}
\end{definition}

An alternative characterization of the property of bounded growth on $J$ is stated by W. Coppel in \cite{Cop}:

\begin{proposition}
\label{HBGa}
The system \eqref{LA} has a bounded growth on $J$ if and only if for each $h>0$ there exists $C_{h}\geq 1$ such that any solution $t\mapsto x(t)$ of \eqref{LA} verifies
\begin{equation}
\label{CA}
|x(t)|\leq C_{h}|x(s)| \quad \textnormal{for any $t\in [s,s+h] \cap J$}.
\end{equation}
\end{proposition}
\begin{proof}
First, assume that (\ref{BG-C}) is verified. Then, given $h>0$
and $0\leq t-s\leq h$ it is straightforward to verify that (\ref{CA}) is satisfied with 
$C_{h}=\max\{1,Ke^{\alpha h}\}$.

Secondly, assume that (\ref{CA}) is verified for a given $h>0$ and consider a partition of $J$ in intervals of large
$h$ such that 
$t\in [s+(n-1)h,s+nh] \cap J$. It is straightforward to deduce that
$|x(t)|\leq C_{h}^{n}|x(s)|$. On the other hand, it is easy to deduce 
that $n-1\leq \frac{t-s}{h}\leq n$ and
$$
|x(t)|=|X(t,s)x(s)|=C_{h}^{n}|x(s)|\leq C_{h}\,e^{\frac{\ln(C_{h})}{h}(t-s)}|x(s)|,
$$
and (\ref{BG-C}) is verified with $K=C_{h}$ and $\alpha=\frac{\ln(C_{h})}{h}$.
\end{proof}

The next result can be proved in a similar way and its proof is left as an exercise.
\begin{proposition}
\label{HBGa2}
The system \eqref{LA} has:
\begin{itemize}
\item[a)] A bounded decay on $J$ if and only if for each $h>0$, there exists $C_{h}\geq 1$ such that any solution $t\mapsto x(t)$ of \eqref{LA} verifies
\begin{equation*}
%\label{CA2}
|x(t)|\leq C_{h}|x(s)| \quad \textnormal{for any $t\in [s-h,s] \cap J$},
\end{equation*}
\item[b)] A bounded growth $\&$ decay on $J$ if and only if for each $h>0$, there exists $C_{h}\geq 1$ such that any solution $t\mapsto x(t)$ of \eqref{LA} verifies
\begin{equation*}
%\label{CA3}
|x(t)|\leq C_{h}|x(s)| \quad \textnormal{for any $t\in [s-h,s+h] \cap J$}.
\end{equation*}
\end{itemize}
\end{proposition}

\begin{remark}
%\label{hBG}
As pointed out by W.A. Coppel in \cite[p.9]{Cop}, a careful reading of the proof of \eqref{HBGa} shows that if \eqref{BG-S} is satisfied then the property \eqref{CA}
follows for any $h>0$. That is, the property of bounded growth is independent of the choice of $h$. This fact also applies to the properties of bounded decay and bounded growth $\&$ 
decay.
\end{remark}

From a practical point of view, it will be useful to know sufficient conditions ensuring the bounded growth properties for a linear system. The Lemma \ref{bgpg} provides a set of sufficient conditions, namely, when $A(t)$ is bounded  or integrally bounded:
\begin{corollary}
\label{Cota-BG1}
If $t\mapsto A(t)$ verifies $\sup\limits_{t\in J}||A(t)||=M$, then the system \eqref{LA} has a bounded growth $\&$ decay on $J$.
\end{corollary}

\begin{proof}
By using the inequality (\ref{BGO-I}) from Lemma \ref{bgpg} it follows that
$$
||X(t,s)||\leq ||I||e^{M|t-s|}
$$
and the result follows.

\begin{corollary}
If $t\mapsto A(t)$ satisfies the property:
$$
A_{h}:=\sup\limits_{t\in J}\int_{t}^{t+h}||A(r)||\,dr < +\infty,
$$
for some $h>0$, then the system \eqref{LA} has a bounded growth on $J$.
\end{corollary}

\begin{proof}
Let $s\in J$ and $t\mapsto x(t)$ a solution of (\ref{LA}). Notice that
\begin{displaymath}
|x(t)|\leq |x(s)|+\int_{s}^{t}||A(\tau)||\,|x(\tau)|\,d\tau,    
\end{displaymath}
for any $t>s$. By Gronwall's Lemma (See Appendix B), we have that
$$
|x(t)|\leq |x(s)|e^{\int_{s}^{t}||A(\tau)||\,d\tau}<|x(s)|e^{\int_{s}^{s+h}||A(\tau)||\,d\tau} \quad \textnormal{for any
$t\in [s,s+h]$}
$$
and we can conclude that (\ref{CA}) is satisfied with $C_{h}=e^{A_{h}}$.
\end{proof}

An interesting example of differential equation which does not has the property of bounded growth $\&$ decay on $\mathbb{R}$ is given by $\dot{x}=2tx$.  In fact, notice that its \textit{fundamental matrix} is given by $X(t)=e^{t^{2}}$, then we have that
$X(t,s)=e^{(t-s)(t+s)}$ and it is straightforward to see that the property (\ref{BG-S}) cannot be verified.

The following result can be found in \cite{Cop,Siegmund-2002} and describes an interesting consequence of the 
bounded growth properties:
\begin{proposition}
\label{CDCI}
If the system \eqref{LA} has the property of bounded growth $\&$ decay
on $J$, provided that $\alpha>0$, then the solutions of \eqref{LA} have the property of uniform continuous dependence with respect to the initial conditions on any compact interval contained on $J$.
\end{proposition}
\begin{proof}
Let $t\mapsto x_{1}(t)$ and $t\mapsto x_{2}(t)$
be two solutions of (\ref{LA}) respectively passing through by $\xi_{1}$ and $\xi_{2}$
at $t=t_{0}$, then for any $t\in [t_{0}-h,t_{0}+h]\subset J$ it follows that
$$
|x_{1}(t)-x_{2}(t)|=\left|X(t,t_{0})\xi_{1}-X(t,t_{0})\xi_{2}\right|\leq |\xi_{1}-\xi_{2}|Ke^{\alpha|t-t_{0}|}
$$
and for any $\varepsilon>0$ there exists $\delta(\varepsilon,h)=\varepsilon/Ke^{h}$ such that
if $|\xi_{1}-\xi_{2}|<\delta$ then $|x_{1}(t)-x_{2}(t)|<\varepsilon$ for any 
$|t-t_{0}|\leq h$. 
\end{proof}
We point out that the term uniform, in order to emphasize that $\delta$ is not dependent of $t_{0}$.

\bigskip

\end{proof}

The following result shows that if the linear system (\ref{LA}) has
the property of bounded growth, then the property is preserved
for bounded perturbed systems
\begin{equation}
\label{LA-BP}
\dot{z}=[A(t)+B(t)]z.    
\end{equation}

\begin{theorem}
\label{RoughBG}
If \eqref{LA} has a bounded growth on $J$ and $B(\cdot)$ verifies $||B(t)|| \leq M$ for any $t\in J$,
then this property is also verified for the system \eqref{LA-BP}.
\end{theorem}

\begin{proof}
Any fundamental matrix of (\ref{LA-BP}) satisfies the matrix differential equation
\begin{displaymath}
\dot{Z}(t)=A(t)Z(t)+B(t)Z(t)
\end{displaymath}
and by using Corollary \ref{SEDM} with $F(t)=B(t)Z(t)$ we can deduce that
\begin{displaymath}
Z(t)=X(t,s)Z(s)+\int_{s}^{t}X(t,\tau)B(\tau)Z(\tau)\,d\tau,
\end{displaymath}
where $X(t,s)$ is the transition matrix associated to (\ref{LA}). The bounded growth of
(\ref{LA}) combined with Proposition
\ref{HBGa} implies the existence of constants $K>0$ and $\alpha\geq 0$ such that 
$$
||X(t,s)||\leq Ke^{\alpha(t-s)} \quad \textnormal{for any $t\geq s$ such that $t,s\in J$}.
$$
Let $t\mapsto z(t)=Z(t,s)\xi$ be a solution of (\ref{LA-BP}) passing through
$\xi\neq 0$ at $t=s$ and $\sup\limits_{t\in J}||B(t)||\leq M$, then notice that
$$
|z(t)|\leq Ke^{\alpha(t-s)}|z(s)|+\int_{s}^{t}Ke^{\alpha(t-\tau)}M|z(\tau)|\,d\tau \quad \textnormal{for any $t\geq s$ with $t,s\in J$}.
$$

By using Gronwall's Lemma (see Appendix B) we can deduce that
$$
|z(t)|\leq K |z(s)|e^{(KM+\alpha)(t-s)} \quad \textnormal{for any $t\geq s$ such that $t,s\in J$}
$$
or equivalently
$$
\frac{|Z(t,s)z(s)|}{|z(s)|}\leq Ke^{(KM+\alpha)(t-s)} \quad \textnormal{for any $t\geq s$ such that $t,s\in J$},
$$
which directly implies that $||Z(t,s)||\leq  Ke^{(KM+\alpha)(t-s)}$
for any $t\geq s$ with $t,s\in J$ and the bounded growth property is verified.

\end{proof}

\section{Kinematical Similarity}
In order to state the property of kinematical similarity, let us denote the set of complex valued square matrices functions $\mathfrak{M}_{n}$ defined as follows: 
\begin{displaymath}
\mathfrak{M}_{n}=\left\{A\colon J\to M_{n}(\mathbb{C}) \mid \textnormal{$A$ is continuous and $\sup\limits_{t\in J}||A(t)||<\infty$}\right\}.
\end{displaymath}

\begin{definition}
\label{LyapTra}
A matrix function $t\mapsto Q(t)\in M_{n}(\mathbb{C})$ is said to be a \textbf{Lyapunov transformation} if and only if satisfies the following properties:
\begin{enumerate}
    \item[(i)] $t\mapsto Q(t)$ is invertible for any $t\in J$,
    \item[(ii)] $t\mapsto Q(t)$ is differentiable for any $t\in J$,
    \item[(iii)] $Q,Q^{-1},\dot{Q}\in \mathfrak{M}_{n}$, that is:
    \begin{displaymath}
    \sup\limits_{t\in J}\left[||Q(t)||+||Q^{-1}(t)||+||\dot{Q}(t)||\right]<+\infty.
    \end{displaymath}
\end{enumerate}
\end{definition}

The classical definition of kinematical similarity was introduced by L. Markus in \cite{Markus} and considers the relation between two linear systems with bounded matrices as follows:
\begin{definition}
\label{similarite}
Let us assume that $A,B\in \mathfrak{M}_{n}$. The systems
\begin{subequations}
  \begin{empheq}{align}
    & \dot{x} = A(t)x \quad \textnormal{for any $t\in J$} \label{LKS}, \\
    & \dot{y} = B(t)y  \quad \textnormal{for any $t\in J$}\label{LKS1}
  \end{empheq}
\end{subequations}
are \textbf{kinematically similar} if there exists a Lyapunov transformation $Q(t)$ such that 
\begin{equation}
\label{LT-KS}
\dot{Q}(t)=A(t)Q(t)-Q(t)B(t) \quad \textnormal{for any $t\in J$}.
\end{equation}

In addition, if $A$,$B$ and $Q$ are real valued it is said that \eqref{LKS} and \eqref{LKS1} are kinematically similar over the real field.
\end{definition}

\begin{remark}
In \cite{Markus}, Markus considers the case $J=[0,+\infty)$. The case $J=\mathbb{R}$ is considered in \cite{Harris-Miles},
which is also called \textit{complete kinematical similarity}.
\end{remark}

\begin{remark}
%\label{RE}
It is easy to verify that the property of kinematical similarity defines an equivalence relation.
\end{remark}

\begin{remark}
\label{SCR}
By using the identity (\ref{LT-KS}) it is a straightforward exercise to verify that:
\begin{itemize}
\item[i)] The change of coordinates $y=Q^{-1}(t)x$ transforms the linear system (\ref{LKS}) into (\ref{LKS1}). 
\item[ii)] The previous property prompts that if $X(t)$ is a fundamental matrix of (\ref{LKS})
then $Y(t)=Q^{-1}(t)X(t)$ is a fundamental matrix of (\ref{LKS1}).
\item[iii)] The Corollary \ref{SEDM} provides a way to solve the matrix differential equation (\ref{LT-KS}). If $0\in J$ and $Q(0)=I$, we recover the identity of the previous statement.
\end{itemize}
\end{remark}

In certain works, such as \cite[p.158]{Dalecki} and \cite{Cop,Palmer80,Palmer82}
the authors consider a different, and more general, definition of kinematical similarity which does not assumes boundedness  properties for the derivative of $Q$\footnote{For example: in the Encyclopaedia of Mathematics, when the Lyapunov transformation is defined, it is pointed out that in many cases the requierement $\sup\limits_{t\in J}||\dot{Q}(t)||<+\infty$ can be discarded. \href{https://encyclopediaofmath.org/wiki/Lyapunov_transformation}{See link}.}:
\begin{definition}
\label{wsimilarite}
Given two continuous matrix valued functions $J\in t\mapsto A(t),B(t)\in M_{n}(\mathbb{C})$, the linear systems
\begin{subequations}
  \begin{empheq}{align}
    & \dot{x} = A(t)x \quad \textnormal{for any $t\in J$} \label{WLKS}, \\
    & \dot{y} = B(t)y  \quad \textnormal{for any $t\in J$}\label{WLKS1}
  \end{empheq}
\end{subequations}
are \textbf{generally kinematically similar} if there exists a invertible and differentiable transformation $t\mapsto Q(t)$ such that, $Q$,$Q^{-1}$ and $\dot{Q}$ are continuous in $J$, 
\begin{displaymath}
    \sup\limits_{t\in J}\left[||Q(t)||+||Q^{-1}(t)||\right]<+\infty,
    \end{displaymath}
and \eqref{LT-KS} is satisfied.
In addition, if $A$,$B$ and $Q$ are real valued it is said that \eqref{WLKS} and \eqref{WLKS1} are generally kinematically similar over the real field.
\end{definition}

The kinematical similarity in the sense of Definition \ref{similarite}
will be named as \textit{classical} kinematical similarity in this monography. We are proposing Definition \ref{wsimilarite} because it will
be used in the next sections.

Both definitions of kinematic similarity, classical or general, are equivalent under the assumption that $A,B\in \mathfrak{M}_{n}$ combined with the boundedness of $Q(t)$ and $Q^{-1}(t)$ in $J$ since the matrix differential equation (\ref{LT-KS}) implies that $\dot{Q}(t)$ is also bounded in $J$. 

On the other hand, the general Definition \ref{wsimilarite} is useful when working with unbounded matrices $A$ and $B$.

The kinematical similarity, classical or general, can be seen as a ge\-ne\-ra\-li\-za\-tion to the nonautonomous framework of the classical reduction of autonomous linear systems to its canonical form. In fact, when $A$ and $B$ are constant matrices 
in $M_{n}(\mathbb{C})$, the systems (\ref{LKS}) and (\ref{LKS1}) becomes
\begin{subequations}
  \begin{empheq}{align}
    & \dot{x} = Ax  \label{CLKS}, \\
    & \dot{y} = By  \label{CLKS1}
  \end{empheq}
  \end{subequations}
and they are \textit{similar} or \textit{statically similar} if 
there exists a non singular and constant matrix $Q\in M_{n}(\mathbb{C})$ such that 
$Q^{-1}AQ=B$. Note that this identity is equivalent to $AQ-QB=0$ and then $Q$ is solution of (\ref{LT-KS}). The static similarity is a classic topic of linear algebra which allows to obtain a system (\ref{CLKS1}) having the simplest possible structure, either diagonal or triangular.

\begin{remark}
\label{importanciaQ}
The boundedness of $Q(t)$ and $Q^{-1}(t)$ is essential in the theory of kinematical similarity. In fact, by Corollary \ref{SEDM} we can see that
$$
Q(t)=X(t,0)Q(0)Y(0,t)
$$
satisfies the matrix differential equation (\ref{LT-KS}) and $y=Q^{-1}(t)x$ transforms
(\ref{WLKS}) into (\ref{WLKS1}). Then, if we drop the assumption of boundedness for $Q(t)$ and $Q^{-1}(t)$, we would have that any linear system would be similar to other one.
\end{remark}

In order to point out the importance of the above remark, we state the following result: 

\begin{lemma}
\label{BKL00}
If a pair of fundamental matrices $X(t)$ and $Y(t)$ of the linear systems \eqref{WLKS} and \eqref{WLKS1} and its inverses are bounded on $J$, then the systems
are generally kinematically similar.
\end{lemma}

\begin{proof}
We can consider normalized matrices, namely, $X(0)=Y(0)=I$, then
by Corollary \ref{SEDM} we can see that
$Q(t)=X(t)Y^{-1}(t)$ sa\-tis\-fy the assumptions of  Definition \ref{wsimilarite}
and the Lemma follows.
\end{proof}

\begin{corollary}
\label{BKL}
If a fundamental matrix $X(t)$ of the linear system \eqref{WLKS} and its inverse are bounded on $J$, then the system \eqref{WLKS} and
\begin{equation*}
%\label{CT-0}
\dot{y}=0,
\end{equation*}
are generally kinematically similar.
\end{corollary}

\begin{proof}
By Lemma \ref{BKL00} we have that $Q(t)=X(t)$ sa\-tis\-fy the assumptions of  Definition \ref{wsimilarite} and the general kinematical similarity follows.

In addition, if $A(t)$ is bounded on $J$, it follows that $\dot{Q}(t)=A(t)Q(t)$
is also bounded on $J$ then $Q(t)$ is a Lyapunov transformation and the classical kinematical similarity follows.
\end{proof}

\begin{corollary}
Any pair of antisymmetric linear systems \eqref{WLKS} and \eqref{WLKS1} is generally kinematically similar on $J$.
\end{corollary}

\begin{proof}
Let $X(t)$ and $Y(t)$ be fundamental matrices of (\ref{WLKS})
and (\ref{WLKS1}). By Corollary \ref{antisim} it follows that $X^{T}(t)=X^{-1}(t)$
and $Y^{T}(t)=Y^{-1}(t)$, which implies that
$$
||X(t)||_{2}=\sigma_{\max}(X(t))=\sqrt{\lambda_{\max}(X^{T}(t)X(t))}=1
$$
and similarly, we can deduce that $Y(t),X^{-1}(t)$ and $Y^{-1}(t)$ have
the same euclidean norm and the result follows.
\end{proof}

An interesting example of the above Corollary is given by the linear systems
\begin{displaymath}
\dot{x}=\left[\begin{array}{cc}
0 & e^{t} \\
-e^{t} & 0
\end{array}\right]x \quad \textnormal{and} \quad
\dot{y}=\left[\begin{array}{cc}
0 & 1 \\
-1 & 0
\end{array}\right]y,
\end{displaymath}
which are kinematically similar by the transformation
$y=Q^{-1}(t)x$, where $Q(t)=X(t)Y^{T}(t)$ with $X(t)$
and $Y(t)$ stated in Remark \ref{ejemplo-reducible}. 

\begin{theorem}
The properties of bounded growth, bounded decay and bounded $\&$ decay on $J$ are preserved by kinematical similarity, either classical or general. 
\end{theorem}

\begin{proof}
Without loss of generality, we will assume that  (\ref{WLKS}) and (\ref{WLKS1}) are generally kinematically similar by the transformation $y=Q^{-1}x$ with
$$
\max\left\{\sup\limits_{t\in J}||Q(t)||,\sup\limits_{t\in J}||Q^{-1}(t)||\right\}\leq M.
$$

Let $X(t)$ and $Y(t)$ be fundamental matrices of (\ref{WLKS}) and (\ref{WLKS1}) respectively. If (\ref{WLKS}) has bounded growth $\&$ decay on $J$, by Definition \ref{BGBD-C} there exist $K>0$ and $\alpha \geq 0$
such that 
$$
||X(t,s)||\leq Ke^{\alpha|t-s|} \quad \textnormal{for any $t,s\in J$. }
$$

By Remark \ref{SCR} we also know that general kinematical similarity implies that
$Y(t)=Q^{-1}(t)X(t)$ and
\begin{displaymath}
||Y(t,s)||=||Q^{-1}(t)X(t)X^{-1}(s)Q(s)||\leq M^{2}||X(t,s)||\leq KM^{2}e^{\alpha|t-s|},
\end{displaymath}
and the bounded growth $\&$ decay on $J$ is verified for (\ref{WLKS1}). Note that the properties of bounded growth 
and bounded decay are a consequence of the previous estimations. 
\end{proof}

In the rest of this section, we will
study two distinguished examples of both kinematical similarities where
the fundamental matrix $X(t)$ has a decomposition $X(t)=Q(t)V(t)$ and $V$ has special properties. Let us recall by Remark \ref{KS-Preh} that in this case it will follows that
$$
\dot{Q}(t)=A(t)Q(t)-Q(t)\dot{V}(t)V^{-1}(t)
$$
provided suitable conditions on $Q$ and $V$.
Firstly, we will study an abridged version of Floquet's theory which considers the $\omega$--periodic case $A(t+\omega)=A(t)$ for any $t\in \mathbb{R}$. Secondly, we will study the transformation of (\ref{LKS}) in an upper triangular linear system, which has been initiated by O. Perron and S.P. Diliberto.

\subsection{Floquet's Theory}
Let us consider the linear nonautonomous system (\ref{LKS}) under the assumptions of continuity on $\mathbb{R}$ and:
\begin{equation}
\label{Floquet1}
A(t+\omega)=A(t) \quad \textnormal{for any $t\in \mathbb{R}$.}    
\end{equation}

The seminal work on this topic was carried out by the french mathematician G. Floquet in \cite{Floquet}. Nevertheless, we will follow a modern approach 
developed by several works as \cite{Adrianova}.

\begin{lemma}
\label{TF1}
If \eqref{Floquet1} is verified then the transition matrix associated to the system \eqref{LA} satisfies the following property of $\omega$--biperiodicity:
\begin{equation*}
%\label{Floquet2}
X(t+\omega,s+\omega)=X(t,s) \quad \textnormal{for any $t,s\in \mathbb{R}$}.
\end{equation*}
\end{lemma}

\begin{proof}
Let $X(t)$ be a fundamental matrix of (\ref{LA}). Firstly, it can be proved that $X(t+\omega)$ is also a fundamental matrix of (\ref{LA}). In fact, note that (\ref{Floquet1}) implies that
$$
\dot{X}(t+\omega)=A(t+\omega)X(t+\omega)=A(t)X(t+\omega),
$$
and we can see that the columns of $X(t+\omega)$ are solutions of (\ref{LA}).

As the columns of $X(t)$ are linearly independent we have that $\det X(t+\omega)\neq 0$ and the columns of $X(t+\omega)$ are also a basis of solutions of (\ref{LA}).

As $X(t)$ and $X(t+\omega)$ are fundamental matrices, there exists an invertible matrix $B$ such that
\begin{equation}
\label{Floquet3}
X(t+\omega)=X(t)B \quad \textnormal{and} \quad X^{-1}(t+\omega)=B^{-1}X^{-1}(t).
\end{equation}

By using the definition of transition matrix, we have that
\begin{displaymath}
\begin{array}{rcl}
X(t+\omega,s+\omega)&=&X(t)BB^{-1}X^{-1}(s)\\
&=&X(t)X^{-1}(s)\\
&=&X(t,s)
\end{array}
\end{displaymath}
and the Lemma follows.
\end{proof}

\begin{definition}
%\label{TF2}
The matrix $B$ from \eqref{Floquet3} is known as the \textbf{monodromy matrix} of the system \eqref{LA}. The eigenvalues of $B$ are known as the \textbf{Floquet's multipliers} of the system \eqref{LA}. 
\end{definition}

Notice that, if we consider a normalized fundamental matrix, the equation (\ref{Floquet3}) with $t=0$ implies that
\begin{equation}
\label{Flo-Exp}
X(\omega)=B  \quad  \textnormal{and consequently} \quad X(t+\omega)=X(t)X(\omega).
\end{equation}

\begin{lemma}
%\label{TF3}
A number $\rho \in \mathbb{C}$ is a Floquet's multiplier of the system \eqref{LA} if and only if there exists a nontrivial solution $t\mapsto x(t)$ such that
\begin{equation}
\label{Floquet4}
x(t+\omega)=\rho \,x(t).
\end{equation}
\end{lemma}

\begin{proof}
Without loss of generality, it can be assumed that $X(0)=I$. In this case it follows by  equation (\ref{Flo-Exp}) that $X(\omega)=B$.

Firstly, we will assume that $\rho$ is an eigenvalue of the monodromy matrix $X(\omega)=B$. Then, there exists a vector $v\in \mathbb{C}^{n}\setminus \{0\}$ such that $X(\omega)v=\rho \, v$.
Now, let us consider the solution $t\mapsto x(t)$ of (\ref{LA}) passing through $v$ at $t=0$, then by Lemma \ref{pitaron} combined with $X(0)=I$ and the right equation of (\ref{Flo-Exp}), it follows that $x(t)=X(t)v$, and we can see that
$$
x(t+\omega)=X(t+\omega)v=X(t)X(\omega)v=\rho\, x(t)
$$
and the identity (\ref{Floquet4}) follows.

Secondly, let us assume that (\ref{Floquet4}) is verified. We evaluate this identity at $t=0$ obtaining that
$$
x(\omega)=\rho \, x(0).
$$

In addition, by Lemma \ref{pitaron} we know that any non trivial solution $t\mapsto x(t)$ is given by $x(t)=X(t)x(0)$. Then, we can see that
$$
x(\omega)=X(\omega)\,x(0).
$$

Now, considering the right parts of the above identities it follows that
$$
X(\omega)x(0)=\rho\, x(0) $$
and we have that $\rho$ is an eigenvalue of $B=X(\omega)$. 
\end{proof}

\begin{remark}
\label{EPS}
A direct consequence of the definition of Floquet's multiplier is that
the continuous $\omega$--periodic system (\ref{LA}) has a non trivial $\omega$--periodic solution if and only if $\rho=1$ is a Floquet's multiplier of (\ref{LA}) or equivalently $X(\omega,0)-I$ is non singular.
\end{remark}

We point out that Floquet's original work \cite{Floquet} is difficult to read due to its anachronistic language and, in our opinion, a careful reading is still pending despite the contributions made to it \cite{Rollet,Weikard}. The following result is a modern interpretation:

\begin{theorem}
\label{TeoFlo}
If $X(t)$ is a fundamental matrix of \eqref{LA}, then there exists a continuous and $\omega$--periodic matrix $Q(t)$ and a constant matrix $D\in M_{n}(\mathbb{C})$ such that:
\begin{equation}
\label{TF5}
X(t,0)=Q(t)e^{Dt} \quad \textnormal{where} \quad D=\frac{1}{\omega}\ln X(\omega,0).
\end{equation}
\end{theorem}

\begin{proof}
As $X(t,0)$ is an invertible matrix, we know that $D$ is well defined and we refer the reader to \cite[Appendix 3]{Brauer} for details. Then we have that
$$
X(\omega,0)=e^{D\omega} \quad \textnormal{and} \quad X(0,\omega)=e^{-D\omega}.
$$

Now, let us define $Q(t)=X(t,0)e^{-Dt}$ and note that
$$
X(t,0)=X(t,0)e^{-Dt}e^{Dt}=Q(t)e^{Dt}.
$$

To finish the proof, we need to verify that $Q(t)$ is continuous and $\omega$--periodic. As the continuity is straightforward, we will only verify the $\omega$--periodicity by using the above identities combined with the $\omega$--biperiodicity:
\begin{displaymath}
\begin{array}{rcl}
Q(t+\omega) &=&X(t+\omega,0)e^{-D\omega}e^{-Dt}\\
&=&X(t+\omega,0)X(0,\omega)e^{-Dt}\\
&=&X(t+\omega,\omega)e^{-Dt}\\
&=&X(t,0)e^{-Dt}=Q(t).
\end{array}
\end{displaymath}

\end{proof}

Without loss of generality, we always can consider a fundamental matrix $X(t)$ such that $X(0)=I$. Then by Floquet's theorem we have that 
$$
X(t)=Q(t)e^{Dt} \quad \textnormal{and} \quad Q(t)=X(t)e^{-Dt},
$$
where $D$ is related to the monodromy matrix.

\begin{lemma}
\label{QFLO}
The matrix $Q(t)=X(t)e^{-Dt}$ stated in Floquet's Theorem is a Lyapunov transformation
\end{lemma}

\begin{proof}
On one hand, the properties (i) and (ii) of Definition \ref{LyapTra} are consequences of the invertibility and derivability of $X(t)$ and $e^{-Dt}$. On the other hand, as $Q(t)$ is continuous and $\omega$--periodic, it can be proved that $Q^{-1}(t)$ and $\dot{Q}(t)$ have the same properties, which implies that
\begin{displaymath}
    \sup\limits_{t\in \mathbb{R}}\left[||Q(t)||+||Q^{-1}(t)||+||\dot{Q}(t)||\right]=\max\limits_{t\in [0,\omega]}\left[||Q(t)||+||Q^{-1}(t)||+||\dot{Q}(t)||\right]
    \end{displaymath}
and the property (iii) follows.  
\end{proof}

\begin{theorem}
\label{Teo-Floquet}
The $\omega$--periodic and continuous system
\eqref{LKS} and the autonomous linear system
$$
\dot{y}=Dy
$$
are clasically kinematically similar.
\end{theorem}

\begin{proof}
As $Q(t)=X(t)e^{-Dt}$, we note that
\begin{displaymath}
\begin{array}{rcl}
\dot{Q}(t)&=&\dot{X}(t)e^{-Dt}-X(t)De^{-Dt},\\
&=&A(t)X(t)e^{-Dt}-X(t)e^{-Dt}D,\\
&=&A(t)Q(t)-Q(t)D.
\end{array}
\end{displaymath}

Now, if $t\mapsto x(t)$ is solution of the linear system (\ref{LKS}) it is straightforward to verify that 
$y=Q^{-1}(t)x$ transforms (\ref{LKS}) into $\dot{y}=Dy$ and the Theorem follows since $Q(t)$ is a Lyapunov transformation.
\end{proof}

\subsection{Triangularization of linear systems}

The following result was proved by O. Perron in \cite{Perron} and is classical in the literature. Nevertheless, simpler proofs have been carried out by S.P. Diliberto in \cite[Th.1]{Diliberto} and L.Y. Adrianova in \cite[Th.3.3.1]{Adrianova}.

\begin{theorem}[Perron's Theorem of Triangularization]
Any linear system \eqref{WLKS} is generally kinematically similar over the real field via the linear transformation $y=Q^{-1}(t)x$, where $Q(t)$ is orthogonal, to an upper triangular system
\begin{equation}
\label{LA2}
\dot{y}=B(t)y,
\end{equation}
whose diagonal coefficients are real. Moreover, if $||A(t)||_{2}\leq M$ the systems \eqref{LKS} and \eqref{LA2} are classically kinematically similar  and $||B(t)||_{2}\leq (1+n)M$.
\end{theorem}

\begin{proof}
Let $X(t)=[x_{1}(t) \, \ldots \, x_{n}(t)]$ be a fundamental matrix of (\ref{LKS}). By using the QR factorization, we have that this fundamental matrix can be decomposed as follows
\begin{equation}
\label{QR1}
X(t)=Q(t)R(t),
\end{equation}
where $Q(t)\in M_{n}(\mathbb{R})$ is the orthogonal matrix:
$$
Q(t)=[e_{1}(t) \, \ldots \, e_{n}(t)]   \quad \textnormal{with $e_{i}(t)=\frac{\xi_{i}(t)}{||\xi_{i}(t)||_{2}}$ for $i=1,\ldots,n$}
$$
and $\{\xi_{i}(t)\}_{i=1}^{n}$ the orthogonal set obtained by applying the Gram--Schmidt process to $\{x_{1}(t),\ldots,x_{n}(t)\}$, that is
\begin{displaymath}
\begin{array}{rcl}
\xi_{1} & = & x_{1}\\
\xi_{2} & = & x_{2}-\langle x_{2},e_{1}\rangle e_{1}\\
\vdots  & \vdots & \vdots \\
\xi_{n} & = & x_{n}-\sum\limits_{j=1}^{n-1}\langle x_{n},e_{j}\rangle e_{j}. 
\end{array}
\end{displaymath}

We can see that $\{e_{i}(t)\}_{i=1}^{n}$ is an orthonormal basis of solutions of (\ref{LKS}).

On the other hand, $R(t)$ is the upper triangular matrix 
\begin{displaymath}
R(t)=\left[\begin{array}{cccc}
||\xi_{1}(t)||_{2} & \langle x_{2}(t),e_{1}(t)\rangle & \ldots &   \langle x_{n}(t),e_{1}(t)\rangle \\
0 & ||\xi_{2}(t)||_{2} & \ldots & \langle x_{n}(t),e_{2}(t)\rangle\\
\vdots & & & \vdots\\
0 & \ldots & \ldots & ||\xi_{n}(t)||_{2}
\end{array}\right].
\end{displaymath}

We refer the reader to Lemma 3.3.1 from \cite{Adrianova} for a detailed computation
of $Q(t)$ and $R(t)$. In addition, notice that the diagonal terms
of $R(t)$ verifies
$$
r_{ii}(t)=||\xi_{i}(t)||_{2}.
$$

We will prove that $y=Q^{-1}(t)x=Q^{T}(t)x$ tran\-sforms (\ref{LKS}) into the upper triangular system $\dot{y}=B(t)y$ with
\begin{equation*}
%\label{triangular-KS}
B(t)=\dot{R}(t)R^{-1}(t).
\end{equation*}

In order to prove that $B(t)$ is well defined, we point out some trivial consequences of (\ref{QR1}) are the identities:

\begin{subequations}
  \begin{empheq}{align}
     X^{-1}(t)&=R^{-1}(t)Q^{T}(t) \label{QR2}, \\
     Q(t)&=X(t)R^{-1}(t),  \label{QR3}\\
     Q^{T}(t)&=R(t)X^{-1}(t).
    \label{QR4}
  \end{empheq}
\end{subequations}

Note that $X(t)$ is derivable since is fundamental matrix of (\ref{LKS}) while $Q(t)$ is also derivable since its columns are constructed with the columns of $X(t)$. Moreover, the Remark \ref{R3} says that $X^{-1}(t)$ and $Q^{-1}(t)$ are also derivable. By using this fact combined with the identity 
\begin{equation}
\label{QR45}
R(t)=Q^{-1}(t)X(t)=Q^{T}(t)X(t).
\end{equation}
we have that $R(t)$ is derivable 
and by using Lemma \ref{dimf} it follows that $R^{-1}(t)$ is a derivable upper triangular matrix  and $B(t)$ is well defined.

It is easy to verify that $B(t)$ is upper triangular with real diagonal terms described by:
$$
b_{ii}(t)=\frac{r_{ii}'(t)}{r_{ii}(t)}=\frac{\langle \xi_{i}'(t),\xi_{i}(t)\rangle}{\langle \xi_{i}(t),\xi_{i}(t)\rangle}.
$$

By using (\ref{QR4}) combined with Remark \ref{R3} and the identity (\ref{QR2}) we have that:
\begin{displaymath}
\begin{array}{rcl}
\dot{Q}^{T}(t)&=&\dot{R}(t)X^{-1}(t)-R(t)X^{-1}(t)A(t) \\\\
&=&\dot{R}(t)X^{-1}(t)-Q^{T}(t)A(t)\\\\
&=& \dot{R}(t)R^{-1}(t)Q^{T}(t)-Q^{T}(t)A(t)\\\\
&=& B(t)Q^{T}(t)-Q^{T}(t)A(t).
\end{array}
\end{displaymath}

Moreover, the identity
$\dot{Q}^{T}(t)=-Q^{T}(t)\dot{Q}(t)Q^{T}(t)$ is equivalent to
$$
\dot{Q}(t)=A(t)Q(t)-Q(t)B(t)
$$
and we can see that $y=Q^{-1}(t)x=Q^{T}(t)x$ tran\-sforms (\ref{LKS}) into the upper triangular system $\dot{y}=B(t)y$. In addition, as $Q(t)$ is orthogonal, we have that
$$
||Q(t)||_{2}=||Q^{T}(t)||_{2}=\sqrt{\lambda_{\max}(I_{n})}=1 \quad \textnormal{for any $t\in J$}
$$
and the general kinematical similarity is verified. 

From now on, we will assume that $||A(t)||_{2}\leq M$ for any $t\in J$. By using (\ref{QR45}) we can note that
\begin{displaymath}
\begin{array}{rcl}
B(t)&=&\frac{d}{dt}\left\{Q^{T}(t)X(t)\right\}R^{-1}(t)\\\\
&=&\left\{\frac{d}{dt}Q^{T}(t)\right\}X(t)R^{-1}(t)+Q^{T}(t)A(t)X(t)R^{-1}(t)\\\\
&=&\underbrace{\dot{Q}^{T}(t)Q(t)}_{=V(t)}+\underbrace{Q^{T}(t)A(t)Q(t)}_{=W(t)}. 
\end{array}
\end{displaymath}

As $Q^{T}$ is orthogonal, it follows by a straightforward computation that $V(t)$ is antisymmetric.
Moreover, it easy to see that
$$
||W(t)||_{2}\leq ||Q^{T}(t)||_{2}||A(t)||_{2}||Q(t)||_{2}\leq ||A(t)||_{2}\leq M.
$$

Then, by combining the above properties, we have that:

\noindent $\bullet$ As $B(t)$ is upper triangular we have $B_{ij}(t)=0$ for $i>j$ and this
implies that 
$$
V_{ij}(t)=-W_{ij}(t) \quad \textnormal{for any $i>j$}.
$$

\noindent $\bullet$ As $V(t)$ is antisymmetric, we have that
$$
V_{ii}(t)=0 \quad \textnormal{for any $i=1,\ldots,n$ and $V_{ij}(t)=-V_{ji}(t)=W_{ji}(t)$ for any $i<j$}.
$$

We can see that 
$$
||V(t)||_{2}\leq \sqrt{n}||V(t)||_{1}\leq \sqrt{n}||W(t)||_{1}\leq n||W(t)||_{2}
$$
and we conclude that $||B(t)||_{2}\leq ||V||_{2}+||W||_{2}\leq M(1+n)$

Finally, the classical kinematical similarity  follows since $Q(t)$ is a 
Lyapunov transformation, in fact we already know that $||Q(t)||_{2}=||Q^{T}||_{2}=1$
and the identity $\dot{Q}(t)=A(t)Q(t)-Q(t)B(t)$ combined with the boundedness of $A(t)$ and $B(t)$ implies the boundedness of $\dot{Q}(t)$. 
\end{proof}

\section{Comments and References}

\noindent \textbf{1)} The fundamental and transition matrices are solutions of
a linear matrix differential equation. In this framework, in the subsection 1.1) 
we provided additional results about matrix differential equations results inspired in the variation of parameters. The reader interested in this topic is referred to
\cite{Coles,Palmer84a,Sternberg,Whyburn}.

\medskip  

\noindent \textbf{2)} There exists several definitions of bounded growth and we decide to use the Palmer's notions
of bounded growth and bounded decay in order to make some distinctions. The definition of bounded growth  has been introduced by W.A. Coppel monography \cite{Cop}. To the best of our knowledge, the property of bounded growth $\&$ decay has been firstly used in \cite{Siegmund-2002}
where is also called as \textit{bounded growth} and is used to prove the uniform continuity of the solution of (\ref{LA}) with respect
to the initial conditions.

We also provides additional details as Theorem \ref{RoughBG} and described connections with general Bohl's exponents. 

The property of bounded growth has been generalized in several ways: a first approach considers the existence of constants
$L>0$ and $\beta\geq 0$ such that
\begin{displaymath}
||X(t,s)||\leq L\left(\frac{h(t)}{h(s)}\right)^{\beta} \quad \textnormal{for any $t\geq s$ with $t,s\in J_{0}:=(a,+\infty)$},    
\end{displaymath}
where $h\colon J_{0} \to (0,+\infty)$ is an increasing homeomorphism encompassing the exponential growth. We refer the reader to
\cite[Sect.3]{Elorreaga} for a deeper treatment of this approach.

A second approach is given by the property of \textit{non uniform bounded growth} if there exist positive constants $K,\alpha$ and $\varepsilon$ such that
$$
||X(t,s)||\leq Ke^{\alpha |t-s|+\varepsilon |s|}  \quad \textnormal{for $t,s \in J$}
$$
and we refer the reader to \cite{Chu} for details. We point out that
the uniform continuity of the solutions of (\ref{LA}) with respect to the initial
conditions is not preserved and its dependent of the initial time.

\medskip

\noindent\textbf{3)} In spite that the topics of adjoint system and Liouville formula are standard in the literature, we point out the originality of our proof of the Liouville's formula and that a nice application of Corollary \ref{antisim} is made by T. Burton in \cite{Burton-PAMS}, where it is proved that if $A(t)$ is antisymmetric and $\omega$--periodic, the solutions of (\ref{LA}) are almost 
periodic.

\medskip

\noindent \textbf{4)} With respect to the kinematical similarity, we pointed out the nuances in the definition of kinematic similarity and make a distinction between 
a \textit{classical} and a \textit{general}
definition. This distinction has been noticed previously in the russian literature
(see for example the discussion carried out in \cite[pp.43--44]{Adrianova}) where the property of classic kinematic similarity is denoted as \textit{reducibility} and is noticed that some reduction can be achieved without Lyapunov transformations.

The Remark \ref{importanciaQ}, Lemma \ref{BKL00} and Corollary \ref{BKL} are based in the 
monography \cite{Adrianova}. Finally, a detailed review about the kinematical similarity
and its consequences can be founded in \cite[Ch.4]{Harris-Miles}.

\medskip

\noindent \textbf{5)} In \cite{Bylov} (see also \cite[p.37]{Adrianova}), B.F. Bylov coined the property of \textit{almost reducibility} to the denote
the classical kinematical similarity to an arbitrarily perturbed system, that is, the systems
\begin{displaymath}
\dot{x} = A(t)x \quad \textnormal{and} \quad \dot{y} = B(t)y  \quad \textnormal{for any $t\geq 0$},
  \end{displaymath}
are almost reducible if, for any $\delta>0$, there exists a Lyapunov transformation
$L_{\delta}(t)$ such that the change of variables $y=L_{\delta}^{-1}(t)x$ transforms the first
system into
$$
y'=[B(t)+\Phi(t)]y, \quad \textnormal{where $||\Phi(t)||<\delta$ for any $t\geq 0$}.
$$

\medskip
 
\noindent \textbf{6)} A basic summary of Floquet's theory can be founded in \cite{Adrianova,Berthelin,Hartman,Nemit} while a deep study of linear $\omega$--periodic systems is made in \cite{Yakubovitch}. The proof of Theorem \ref{TeoFlo} assumes a basic knowledge of complex matrix functions and the reader is refered to \cite[Ch.1]{Adrianova} and \cite[Ch.6]{Bellman3}.

The Floquet theory has been successfully applied in the study of the Mathieu's equation
\begin{equation*}
%\label{Hill}
x''+(a-2q\cos(2t))x =0 \quad \textnormal{where $a$ and $q$ are real parameters},
\end{equation*}
which arises in astronomic and physical problems and we refer the reader to \cite{Magnus,Yakubovitch} for details. On the other hand, there have been many efforts to extending Floquet's theory and we point out some directions:

\noindent $\bullet$ To construct a Floquet's type theory for linear nonautonomous equations where
$\mathbb{R}\ni t\mapsto A(t)$ is almost periodic in the Bohr--Bochner sense. In this context, we point out the work
of J. Lillo \cite{Lillo} devoted to linear almost periodic systems, which introduces the property:
\begin{definition}
%\label{ap-sim}
The almost periodic linear systems \eqref{LKS} and \eqref{LKS1} are \textbf{approximately similar} if for any $\varepsilon>0$ there exists a Lyapunov transformation matrix $Q(t,\varepsilon)$ such that
\begin{displaymath}
\sup\limits_{t\in \mathbb{R}}||Q^{-1}(t,\varepsilon)\{A(t)Q(t,\varepsilon)-\dot{Q}(t,\varepsilon)\}-B(t)||<\varepsilon,
\end{displaymath} 
\end{definition} 
\noindent and note that if $\varepsilon=0$, we recover (\ref{LT-KS}) with $J=\mathbb{R}$.
 
In \cite{Coppel67} W.A. Coppel  considers a linear almost periodic system (\ref{LKS})
and tries, unsuccessfully, to prove that if (\ref{LKS}) has an exponential dichotomy 
on $J=\mathbb{R}$ is kinematically similar to (\ref{LKS1}) by means
an almost periodic Lyapunov transformation $Q(t)$. Nevertheless, by following Lillo's ideas it is proved that $t\mapsto X(t)PX^{-1}(t)$ is almost periodic.

The work of J. Ben Slimene and J. Blot \cite{Blot} considers an almost periodic system
(\ref{LKS}) and proves that if its fundamental matrix is composed by $n$ almost periodic solutions then is kinematically similar to an almost periodic triangular system
(\ref{LKS1}) by means of an almost periodic Lyapunov transformation.

On the other hand, the works of D. Berkey \cite{Berkey} and M. Pinto \cite{Pinto}  does not follow an approach based in the kinematical similarity but are focused in to obtain a fundamental matrix emulating the structure of (\ref{TF5}) for the almost periodic case.
 
 \noindent $\bullet$ To generalize the Floquet's Theorem to a nonlinear framework described by the systems 
\begin{subequations}
  \begin{empheq}{align}
    & \dot{x} = f(t,x) \quad \textnormal{for any $t\in \mathbb{R}$} \label{NLKS}, \\
    & \dot{y} = g(t,y)  \quad \textnormal{for any $t\in \mathbb{R}$}\label{NLKS1}
  \end{empheq}
\end{subequations}
where $f,g\colon \mathbb{R}\times \mathbb{R}^{n}\to \mathbb{R}^{n}$ are such that 
the existence and uniqueness of solutions in ensured and $f(t+\omega,x)=f(t,x)$ 
and $g(t+\omega,y)=g(t,y)$ for any $t$. This task has been carried
out by researchers from the Fuzhou's school. In \cite{ZS}, C. Zou and J. Shi provide 
necessary conditions ensuring the existence of a nonlinear continuous transformation $H\colon \mathbb{R}\times \mathbb{R}^{n}\to \mathbb{R}^{n}$ such that:
\begin{itemize}
\item[i)] $H(t,\cdot)$ is an homeomorphism for any $t\in \mathbb{R}$,
\item[ii)] $H(t+\omega,x)=H(t,x)$ for any $(t,x)\in  \mathbb{R}\times \mathbb{R}^{n}$, 
\item[iii)] If $t\mapsto x(t)$ is solution of (\ref{NLKS}) then $t\mapsto H(t,x(t))$ is solution
of (\ref{NLKS1}).
\end{itemize}

The $H(t,x)$ can also be seen as a particular case of a generalization of kinematical similiarity to the nonlinear framework
 \begin{definition}
 \label{KSLIN}
 Given two nonlinear systems
\begin{subequations}
  \begin{empheq}{align}
    & \dot{x} = f(t,x) \quad \textnormal{for any $t\in \mathbb{R}$} \label{Fu1}, \\
    & \dot{y} = g(t,y)  \quad \textnormal{for any $t\in \mathbb{R}$}\label{Fu2}
  \end{empheq}
\end{subequations}
 where $f,g\colon J\to \mathbb{R}^{n}$ are such that the existence and uniqueness of solutions in ensured, we say that \eqref{Fu1} and \eqref{Fu2} are kinematically similar if there exists a continuously differentiable and invertible matrix $Q(t)$ such that
 \begin{equation*}
 %\label{FKS}
 G(t,x)=\dot{Q}(t)Q^{-1}(t)x+Q(t)F(t,Q^{-1}(t)x).
 \end{equation*}
 \end{definition}
 To the best of our knowledge, the above Definition has been stated in works of F. Lin \cite{Lin-2007}.

\noindent $\bullet$ Another generalization of Floquet theory is made by T.A. Burton and J.S. Muldowney in \cite{Burton} and extended in \cite{Burton76,Freedman}, which state that (\ref{LA}) is a generalized Floquet system with respect to $f$ (GFS--$f$) if 
\begin{equation}
\label{GFS1}
f'(t)A(f(t))=A(t).
\end{equation}
where $f\colon (\alpha,\infty)\to \mathbb{R}$ is absolutely continuous and
$f(t)>t$.

Note that $f(t)=t+\omega$ with $\omega>0$ satisfies the above properties for $f$
and (\ref{GFS1}) becomes $A(t+\omega)=A(t)$, which generalizes the $\omega$--periodic case. In addition, the identity 
(\ref{GFS1}) implies that the matrix $Y(t)=X(f(t))$ verifies $Y'(t)=A(t)Y(t)$.

\section{Exercises} 
 
\begin{enumerate}
    \item[1.-] Prove that if the matrix valued function $t\mapsto U(t)$
    is continuous and invertible on $J$, then $t\mapsto U^{-1}(t)$ is also continuous on $J$.
    \item[2.-] Prove the statements a) to d) from Remark \ref{R3}. 
    \item[3.-] Prove that any transition matrix satisfies the matrix differential equations (\ref{MT1}).
    \item[4.-] Prove that any transition matrix verifies the statements from Remark \ref{remark-MT2}.
    \item[5.-] Prove that if $X(t)\in M_{n}(\mathbb{R})$ and $Y(t)\in M_{m}(\mathbb{R})$ are fundamental matrices of the linear systems
    $$
    \dot{x}=A(t)x \quad \textnormal{and} \quad \dot{y}=B(t)y,
    $$
    then the $n+m$ dimensional system
    $$
    \dot{z}=\left[\begin{array}{cc}
    A(t) & C(t)\\
    0 & B(t)\end{array}\right]z
    $$
    has the following transition matrix
    $$
    \left[\begin{array}{cc}
    X(t,t_{0}) & W(t,t_{0})\\
    0 & Y(t,t_{0})\end{array}\right] \quad \textnormal{with} \quad W(t,t_{0})=\int_{t_{0}}^{t}X(t,\tau)C(\tau)Y(\tau,t_{0})\,d\tau.
    $$
    \item[6.-] The linear system $\dot{x}=A(t)x$ with $t\mapsto A(t)$ bounded continuous on $[t_{0},+\infty)$
        is \textit{uniformly exponentially stable} if and only if there exist $K>0$ and $\alpha>0$ such that its transition matrix verifies 
        $$
        ||X(t,s)||\leq Ke^{-\alpha (t-s)} \quad \textnormal{for any $t\geq t_{0}$}.
        $$
        Prove that the uniform exponential stability is preserved by kinematical similarity.
    \item[7.-] If the non homogeneous system (\ref{No-homogeneo}) from Lemma \ref{SNHC1} has a function $f$ such that $\lim\limits_{t\to +\infty}f(t)=0$
    and the corresponding linear system is uniformly exponentially stable, then prove that any solution $t\mapsto x(t)$ verifies $\lim\limits_{t\to +\infty}x(t)=0$. 
    
    \item[8.-] If $X(t,s)$ is a transition matrix of a linear system, prove that
    $[X(t,s)^{-1}]^{T}$ is a transition matrix for its adjoint.
    \item[9.-] Under the assumption that 
    $$
    A(t)\int_{\tau}^{t}A(s)\,ds = \left[\int_{\tau}^{t}A(s)\,ds\right]A(t),
    $$
    prove that  
    $$
    X(t,\tau)=e^{\int_{\tau}^{t}A(s)\,ds}
    $$
    is transition matrix of  $\dot{x}=A(t)x$.
    \item[10.-] Prove Proposition \ref{HBGa2}.
    \item[11.-] Under the assumption that $t\mapsto A(t)$ is continuous and antisymmetric for any $t\in \mathbb{R}$ prove that any fundamental matrix of
    the linear system $\dot{x}=A(t)x$ verifies $\det X(t)=\det X(0)$ for any
    $t\in \mathbb{R}$.
    \item[12.-] Prove that the classical kinematical similarity is an equivalence relation.
    \item[13.-] Prove the statements of Remark \ref{SCR}.
    \item[14.-] Under the assumption that the linear systems
    $$
    \dot{x}=A(t)x \quad \textnormal{and} \quad \dot{y}=B(t)y
    $$
    are classicaly kinematically similar on $J$ by means of the transformation $Q(t)$, prove that the adjoint systems
    $$
    \dot{x}=-A^{T}(t)x \quad \textnormal{and} \quad \dot{y}=-B^{T}(t)y
    $$
    are also kinematically similar on $J$. Note that the corresponding transformation is strongly related to $Q(t)$.
    \item[15.-] Under the assumption that the linear systems
    $$
    \dot{x}=A(t)x \quad \textnormal{and} \quad \dot{y}=B(t)y
    $$
    are classicaly kinematically similar on $J$ by means of the transformation $Q(t)$, prove that 
    \begin{displaymath}
    \det Q(t) = \frac{\det X(t_{0})}{\det Y(t_{0})}e^{\int_{t_{0}}^{t}\textnormal{Tr}[A(s)-B(s)]\,ds}.     
    \end{displaymath}
    
    \item[16.-] Let $\dot{x}=A(t)x$ be a system with solutions defined in all the real line and let $\dot{y}=-A^{T}(t)y$ be its adjoint. If any solution $t\mapsto x(t)$ of the initial system verifies $\lim\limits_{t\to +\infty}x(t)=0$ prove that any solution $t\mapsto y(t)$ of the adjoint system is such that the limit $\lim\limits_{t\to +\infty}|y(t)|$ diverges.
    \item[17.-] An autonomous system $\dot{y}=Dy$ is \textit{hyperbolic} if and only if the real part of all the eigenvalues of $D$ is different from zero. Prove that if
    $\dot{x}=A(t)x$ with $t\mapsto A(t)$ continuous and $\omega$--periodic is classically kinematically similar to an hyperbolic system if and only if its Floquet's multipliers are different from $1$.
    \item[18.-] Prove that if $A$ is continuous, $\omega$--periodic and antisymmetric, then any Floquet multiplier of $\dot{x}=A(t)x$ is $\rho=1$. 
    \item[19.-] In the context of Lemma \ref{QFLO}, prove that $Q^{-1}(t)$ and $\dot{Q}(t)$ are $\omega$--periodic.
    \item[20.-] Prove that the $n+m$--dimensional system
    $$
    \left(\begin{array}{c}
    \dot{x}_{1}\\
    \dot{x}_{2}\end{array}\right)
    =\left[\begin{array}{cc}
    A(t) & 0\\
    0 & C(t)\end{array}\right]
    \left(\begin{array}{c}
    x_{1}\\
    x_{2}\end{array}\right) 
    $$
    with $A(t)\in M_{n}(\mathbb{R})$ and $C(t)\in M_{m}(\mathbb{R})$ is generally kinematically similar to
    $$
    \left(\begin{array}{c}
    \dot{y}_{1}\\
    \dot{y}_{2}\end{array}\right)
    =\left[\begin{array}{cc}
    B(t) & 0\\
    0 & C(t)\end{array}\right]
    \left(\begin{array}{c}
    y_{1}\\
    y_{2}\end{array}\right),
    $$
    by a transformation
    $$
    Q(t)=\left[\begin{array}{cc}
    Q_{1}(t) & 0\\
    0 & I_{m}\end{array}\right]  
    $$  
    where $Q_{1}(t)\in M_{n}(\mathbb{R})$ and $I_{m}$ is the identity on $\in M_{m}(\mathbb{R})$, if and only if
    $$
    \dot{x}_{1}=A(t)x_{1} \quad \textnormal{and} \quad \dot{y}_{1}=B(t)y_{1},
    $$
    are generally kinematically similar by the transformation $Q_{1}$.
    \item[21.-] Let $X_{A}(t)$ and $X_{B}(t)$ be fundamental matrices of 
    $\dot{x}=A(t)x$ and $\dot{y}=B(t)y$ respectively. Prove that if 
    $$
    B(t)X_{A}(t)=X_{A}(t)B(t),
    $$
    then $X_{A}(t)X_{B}(t)$ is fundamental matrix of 
    $\dot{z}=[A(t)+B(t)]z$.
    \item[22.-] Let $X_{A}(t)$ and $X_{B}(t)$ be fundamental matrices of 
    $\dot{x}=A(t)x$ and $\dot{y}=B(t)y$ respectively. Prove that if 
    $$
    A(t)X_{B}(t)=X_{B}(t)A(t)
    $$
    then $X_{B}(t)X_{A}(t)$ is fundamental matrix of 
    $\dot{z}=[A(t)+B(t)]z$.
    \item[23.-] Prove that if we consider $F(t,x)=A(t)x$ and $G(t,x)=B(t)x$ in the Definition \ref{KSLIN}, we can recover the definitions of general Kinematical similarity for linear systems.
    \item[24.-] A linear control system is described by 
\begin{equation}
\label{control1}
x'=A(t)x+B(t)u \quad \textnormal{for any $t \in J$},
\end{equation}
where $t\mapsto u(t)$ is known as the \textit{input} or \textit{control} of the system. In addition
the uncontrolled part is given by:
\begin{equation*}
%\label{plant}
x'=A(t)x  \quad \textnormal{for any $t \in J$},
\end{equation*}
which is usually know as the \textit{plant} of \eqref{control1} and a fundamental
matrix will be denoted by $X_{A}(t)$.

The \textit{controllability gramians} associated to \eqref{control1} are defined by the matrices:
\begin{equation*}
W(a,b)=\displaystyle\int_{a}^{b}X_{A}(a,s)B(s)B^{T}(s)X_{A}^{T}(a,s)\;ds
\end{equation*}
and
\begin{equation*}
K(a,b)=\displaystyle\int_{a}^{b}X_{A}(b,s)B(s)B^{T}(s)X_{A}^{T}(b,s)\;ds.
\end{equation*}

    \begin{enumerate}
    \item[23.1)] Prove that above controllability gramians satisfy 
the following relations:
\begin{equation*}
%\label{GP}
\left\{\begin{array}{rcl}
K(a,b)&=&X_{A}(b,a)W(a,b)X_{A}^{T}(b,a) \\\\
W(a,b)&=&X_{A}(a,b)K(a,b)X_{A}^{T}(a,b).
\end{array}\right.
\end{equation*}
    \end{enumerate}

\end{enumerate}

\chapter{Exponential Dichotomy}
\medskip

 Hyperbolicity is a fundamental property of dynamic systems that enables the splitting of phase space into subspaces that contract and expand under map $f,$ respectively. When the space phase is a smooth manifold of finite dimension, and the considered function is smooth, determining the time evolution, then the evolution of an initial point $x$ is obtained by solving a system of ordinary differential equations. If the system is autonomous, the hyperbolicity of an equilibrium point $x_0$ is determined by the real parts of the eigenvalues of the Jacobian matrix of $f$ at $x_0;$ however, this condition cannot be replicated in a nonautonomous system as is stated in \cite[p. 310]{MY}. To study hyperbolicity effectively in that context, one might examine the property  of exponential dichotomy and the related spectrum, which emulates the work done by eigenvalues in a nonautonomous framework.

\section{Autonomous hyperbolicity}
\begin{definition}
Given a matrix $A\in M_{n}(\mathbb{R})$, the autonomous system of ordinary differential equations
\begin{equation}
\label{li-auto}
\dot{x}=Ax,
\end{equation}
is \textbf{hyperbolic} if $A$ has no eigenvalues having real part zero.
\end{definition}

The hyperbolicity condition has remarkable consequences and applications in the qualitative theory of differential equations, we will highlight the most relevant ones:

\bigskip

\noindent \textbf{a) Splitting between stable and unstable spaces.} Let
$\lambda_{j}=a_{j}+ib_{j}$ be a (possible multiple) eigenvalue of $A$ with
generalized eigenvector $W_{j}=\bf{u}_{j}+i\bf{v}_{j}$. Then we can see that
$$
\mathcal{B}=\{\bf{u_{1}},\ldots,\bf{u_{k}},\bf{u_{k+1}},\bf{v_{k+1}},
\ldots,\bf{u_{m}},\bf{v_{m}}\}
$$
with $n=2m-k$ is a basis of $\mathbb{R}^{n}$ such that $\mathbb{R}^{n}=E^{s}\oplus E^{u}$, where $E^{u}$ and $E^{s}$ are the \textit{stable} and \textit{unstable} 
spaces defined by
\begin{displaymath}
E^{s} = \text{Span}\{\bf{u_{j}},\bf{v_{j}} \mid \textnormal{$a_{j}<0$}\} \quad
\textnormal{and} \quad
\textnormal{$E^{u}$} = \text{Span}\{\bf{u_{j}},\bf{v_{j}} \mid \textnormal{$a_{j}>0$}\}.
\end{displaymath}

Let us recall that any solution $t\mapsto x(t)$ of (\ref{li-auto}) passing
through $x_{0}$ at $t=0$ can be written as $x(t)=e^{At}x_{0}$, then it follows that; see for example \cite[p.55]{Perko};
the subspaces $E^{s}$ and $E^{u}$ are invariants under the flow induced by (\ref{li-auto}), namely, if $x_{0}\in E^{s}$ (\textit{resp}. $x_{0}\in E^{u}$) then $x(t)\in E^{s}$ (\textit{resp}. $x(t)\in E^{u}$).

A consequence of the above mentioned invariance is that if $x_{0}=x_{s}+x_{u}$ with $x_{s}\in E^{s}$
and $x_{u}\in E^{u}$, then it is easy to verify that any
solution $x(t)=e^{At}x_{0}$ can be splitted as
$x(t)=x_{s}(t)+x_{u}(t)$ where
$$
x_{s}(t)=e^{At}x_{s} \in E^{s} \quad \textnormal{and} \quad
x_{u}(t)=e^{At}x_{u} \in E^{u} \quad \textnormal{for any $t\in \mathbb{R}$}.
$$

\noindent \textbf{b) Admissibility properties.} 
Another consequence of the hyperbolicity condition is the fo\-llo\-wing
result, see \textit{e.g} \cite[p.81]{Dalecki} for details:
\begin{proposition}
\label{admi}
If \eqref{li-auto} is hyperbolic then for any 
$f\in BC(\mathbb{R},\mathbb{R}^{n})$ there exists a unique
solution $t\mapsto y_{f}(t)\in  BC(\mathbb{R},\mathbb{R}^{n})$ of the inhomogeneous system
\begin{equation*}
%\label{li-auto-2} 
\dot{y}=Ay+f(t)
\end{equation*}
such that 
$$
\sup\limits_{t\in \mathbb{R}}|y_{f}(t)|\leq \alpha(A) \sup\limits_{t\in \mathbb{R}}|f(t)|,
$$
where $\alpha(A)$ is a constant dependent of $A$. 
\end{proposition}

\noindent \textbf{c) Linearization results.}
The classical linearization theorem currently known as the Hartman--Grobman Theorem is a result obtained by P. Hartman \cite{Hartman1} and D.M. Grobman \cite{Grobman} respectively. It has a distinguished place in the 
qualitative theory of ordinary differential equations and the hyperbolicity condition plays 
an essential role.

In order to state the above mentioned result, let us consider the nonlinear system
\begin{equation}
\label{NLHG00}
\dot{y}=g(y),
\end{equation}
where $g\in C^{1}(\Omega,\mathbb{R}^{n})$, $\Omega\subset \mathbb{R}^{n}$ is an open set containing the origin  such that $g(0)=0$ and $Dg(0)$ is the derivative of $g$ at the origin. The solution of (\ref{NLHG00}) passing through $y_{0}\in \Omega$ at $t=0$ will be denoted by $\phi_{t}(y_{0})$ and its local existence and uniqueness will be assumed.

\begin{proposition}[Hartman--Grobman Theorem]
If the linear system
\begin{equation}
\label{li-nona-3}
\dot{x}=Dg(0)x
\end{equation}
is hyperbolic, then there exists an homeomorphism $H\colon U\to V$
such that
\begin{itemize}
\item[a)] The sets $U$ and $V$ are open in $\mathbb{R}^{n}$ and contain the origin.
\item[b)] For any $y_{0}\in U$, there exists an open interval $I_{y_{0}}\subset \mathbb{R}$ containing zero such that
$$
H\circ \phi_{t}(y_{0}) = e^{At}H(y_{0}) \quad \textnormal{for any $t\in I_{y_{0}}$}.
$$
\end{itemize}
\end{proposition}

The Hartman--Grobman homeomorphism can be seen as a \textit{dictionary} mapping solutions of the non linear system (\ref{NLHG00}) into solutions of the linear one
(\ref{li-nona-3}) and viceversa for any $t\in I_{y_{0}}$ as follows:

\begin{itemize}
\item[i)] If $t\mapsto \phi_{t}(y_{0})$ is a solution of (\ref{NLHG00}) passing trough $y_{0}$ at $t=0$, then $t\mapsto H(\phi_{t}(y_{0}))$ is solution of  (\ref{li-nona-3}) passing through $H(y_{0})$ at $t=0$.

\item[ii)] If $t\mapsto e^{At}z_{0}$ is solution of (\ref{li-nona-3}) passing through $z_{0}$ at $t=0$, then
$H^{-1}(e^{At}z_{0})$ is a solution of (\ref{NLHG00}) passing trough $H^{-1}(z_{0})$ at $t=0$.

\item[iii)] A direct consequence of i) and ii) combined with the uniqueness of solutions is that $H(0)=0$.
\end{itemize}

\bigskip 
We point out that the three above problems cannot be directly generalized when $A$ is not constant, since the eigenvalues does not provide neither qualitative nor accurate information about the system, which prevents the naive extension 
of the hyperbolicity property to the nonautonomous framework.

\section{Nonautonomous hyperbolicity}
We stress that it does not exists an ubiquitous definition of \textit{hyperbolicity} in the nonautonomous setting. In particular, the above topics (splitting, admissibility and linearization)
cannot be addressed by using the eigenvalues of a time varying matrix.

Despite these nuances, the concept of \textit{dichotomy} plays a relevant role, being the exponential dichotomy \cite{Cop,Lin2,Mitropolski} introduced by O. Perron in the 30's the most relevant \cite{Perron}.

\subsection{Exponential Dichotomy} Let us consider the nonautonomous system of ordinary differential equations 
\begin{equation}
\label{LinealC2}
x'=A(t)x,
\end{equation}
where $A(t)$ is a locally integrable matrix.

\begin{definition}
\label{DICE}
\textit{The system \eqref{LinealC2} has an \textbf{exponential dichotomy} on the interval $J\subseteq \mathbb{R}$ if, given a fundamental matrix $X(t)$, there exist positive numbers $K>0$, $\alpha > 0$ and a projection $P^2=P$ such that}
\begin{equation}
\label{eq:2.2}
\left\{\begin{array}{rccr}
||X(t)PX^{-1}(s)||&\leq K e^{-\alpha(t-s)} & \textnormal{if} & t\geq s,\quad t,s\in J\\
||X(t)(I-P)X^{-1}(s)|| &\leq Ke^{-\alpha(s-t)} & \textnormal{if} & t \leq s, \quad t,s \in J.
\end{array}\right.
\end{equation}
\end{definition}

 It is worth to emphasize that the constants $K$ and $\alpha$ are known as the constants of the exponential dichotomy. Nevertheless, we point out that the above notation is not univocal and several references consider the pairs $(K_{1},\alpha_{1})$ and $(K_{2},\alpha_{2})$ in the first and second inequalities respectively.

In addition, we will usually consider the unbounded intervals 
$$
J=(-\infty, 0]=\mathbb{R}_0^{-}, \quad J=[0,+\infty)=\mathbb{R}_0^{+} \quad \textnormal{or} \quad J=\mathbb{R}.
$$

If the exponential dichotomy is verified on $\mathbb{R}$, it is common to say that the linear system
has a dichotomy on the full line. Similarly, if the exponential dichotomy is verified
on $\mathbb{R}_{0}^{+}$ or $\mathbb{R}_{0}^{-}$, it is respectively said that
the system has a dichotomy on the right or left half line.

If the above matrix norm is unitary, that is 
$||I||=1$, the limit case $P=I$ and $t=s$ allows to consider $K\geq 1$ in Definition \ref{DICE}.

In order to illustrate this property, we will revisit the example from L. Markus and H. Yamabe \cite{MY} stated 
in the previous chapter:
\begin{example}
\label{MYDUG}
The system
$$
\dot{x}=\left( \begin{array}{cc}
-1 + \frac{3}{2} \cos^2(t) & 1 - \frac{3}{2} \cos(t) \sin(t)\\
-1 - \frac{3}{2} \cos(t) \sin(t) & -1 + \frac{3}{2} \sin^2(t)
\end{array}
\right )x
$$
has a normalized fundamental matrix
\begin{displaymath}
X(t)= \left[\begin{array}{cc}
e^{\frac{t}{2}} \cos(t) & e^{-t}\sin(t)\\
-e^{\frac{t}{2}} \sin(t) & e^{-t} \cos(t)
\end{array}\right].
\end{displaymath}

It can be verified that the system has exponential dichotomy on $\mathbb{R}$ with projection
$P = \left ( \begin{array}{cc}
0&0\\
0&1
\end{array}
\right)$ 
and constants $K = 1$ and $\alpha = 1/2$.
\end{example}

\begin{remark}
%\label{R1ED}
{Note that the identity matrix $P=I$ and the null matrix $P=0$
are projections, which are called as \textit{trivial projections}. If
$P$ is a non trivial projection in $\mathbb{R}^n$, it will be useful to recall that \cite[p. 210]{Hof}:}
\begin{enumerate}
\item $0<\text{Rank}\,P<n$,
\item Any $\xi\in\mathbb{R}^n$ has the decomposition $\xi =P\xi+(I-P)\xi$,
\item $v\in\text{Im}\,P$ if and only if $Pv=v$,
\item $\mathbb{R}^n=\ker \,P\oplus\text{Im}\, P$.
\end{enumerate}
\end{remark}

\begin{remark}
\label{R1ED-bis}
In linear algebra, a projector $P$ is associated to a decomposition 
of $\mathbb{R}^{n}$ as a direct sum $\mathbb{R}^{n}=U\oplus W$ such that $U$
is invariant under $P$. This arise the following question: If the property of exponential dichotomy is verified in $J$, which is the subspace of  $\mathbb{R}^{n}$ invariant under the projector $P$? We will back on this question in the next section and the next chapter.
\end{remark}

A careful reading of Definition \ref{DICE} shows that the projection $P$ is strongly associated
to a fundamental matrix $X(t)$ such that (\ref{eq:2.2}) is verified. Note that if $Y(t)$ is 
another fundamental matrix, there exists a change of basis matrix $V$ such 
that $Y(t)=X(t)V$ and we refer to the statement a) of Remark \ref{R2} in Chapter 1 for details. Now,
we rewrite the equation (\ref{eq:2.2}) as follows 
\begin{displaymath}
\left\{\begin{array}{rccr}
||X(t)VV^{-1}PVV^{-1}X^{-1}(s)||&\leq K e^{-\alpha(t-s)} & \textnormal{if} & t\geq s,\quad t,s\in J\\
||X(t)VV^{-1}[I-P]VV^{-1}X^{-1}(s)|| &\leq Ke^{-\alpha(s-t)} & \textnormal{if} & t\leq s, \quad t,s \in J.
\end{array}\right.
\end{displaymath}

Moreover, as $V$ is a nonsingular matrix, it follows that $\tilde{P}=VPV^{-1}$ is also a projection and $I-\tilde{P}=V[I-P]V^{-1}$. Then, the above inequality is equivalent to 
\begin{equation}
\label{DICEPRI}
\left\{\begin{array}{rccr}
||Y(t)\tilde{P}Y^{-1}(s)||&\leq K e^{-\alpha(t-s)} & \textnormal{if} & t\geq s,\quad t,s\in J\\
||Y(t)[I-\tilde{P}]Y^{-1}(s)|| &\leq Ke^{-\alpha(s-t)} & \textnormal{if} & t\leq s, \quad t,s \in J.
\end{array}\right.
\end{equation}

In summary, the above estimates allow us to state the following remark:
\begin{remark}
\label{ChoB}
Assume that the linear system (\ref{LinealC2}) has an exponential dichotomy on $J$ with constants $K>0$ and $\alpha>0$. If $X(t)$ and $Y(t)$ are fundamental matrices of (\ref{LinealC2}) the inequalities
(\ref{eq:2.2}) and (\ref{DICEPRI}) are equivalent. Namely, there exists a nonsingular matrix $V$ such that 
$$
Y(t)=X(t)V \quad \textnormal{and} \quad \tilde{P}=VPV^{-1}
$$
and, consequently, the projections are not unique but are similar.
\end{remark}

\begin{remark}
\label{R1ED+}
{As any projection $P$ is idempotent, it can be proved that the eigenvalues of any non trivial projection $P$ in $\mathbb{R}^{n}$ are either $0$ or $1$, which implies its dia\-go\-na\-lizability. Then, for any projection $P$ there exists a nonsingular matrix $V$ such that}
\begin{displaymath}
V^{-1}PV=P_{r}=\left[\begin{array}{cc}
I_{r} & 0 \\
0 & 0
\end{array}\right] \quad \textnormal{and} \quad
V^{-1}[I-P]V=I-P_{r}=\left[\begin{array}{cc}
0 & 0 \\
0 & I_{r-1}
\end{array}\right]
\end{displaymath}
{where $r$ is the rank of $P$. Moreover, it follows (see \textit{e.g.} \cite[p.67]{Bellman3}) that $P$ is semi definite positive. From now on, $P_{r}$ which will be called as the \textbf{canonical projection}.}
\end{remark}

A useful consequence of Remarks \ref{ChoB} and \ref{R1ED+} is the following result, which will play a key role in future chapters:
\begin{lemma}
\label{DEPCa}
If the system \eqref{LinealC2} has an exponential dichotomy on the interval $J\subseteq \mathbb{R}$ constants $K>0$, $\alpha > 0$, then there exist a fundamental matrix $Y(t)$ and a canonic projection $P_{r}$ such that
\begin{displaymath}
\left\{\begin{array}{rccr}
||Y(t)P_{r}Y^{-1}(s)||&\leq K e^{-\alpha(t-s)} & \textnormal{if} & t\geq s,\quad t,s\in J\\
||Y(t)[I-P_{r}]Y^{-1}(s)|| &\leq Ke^{-\alpha(s-t)} & \textnormal{if} & t \leq s, \quad t,s \in J.
\end{array}\right.
\end{displaymath}
\end{lemma}

The next result is only an alternative characterization of the exponential
dichotomy property. Nevertheless, its intermediate computations and estimations will be useful in the future.

\begin{proposition}
\label{triada}
Let $X(t)$ be a fundamental matrix of the linear system \eqref{LinealC2}. This system has an exponential dichotomy on $J$ with positive constants $K_{1},K_{2},\alpha$ and a projection $P^2=P$ if and only if for any $\xi \in \mathbb{R}^n$ it follows that:

\begin{equation}\label{eq:2.3}
\left\{\begin{array}{rcll}
| X(t)P\xi\vert &\leq & K_{1} e^{-\alpha(t-s)}\vert X(s)P\xi | & \textnormal{\textit{if} $t\geq s$,} \\
| X(t)(I-P)\xi\vert &\leq & K_{2}e^{-\alpha(s-t)}\vert X(s)(I-P)\xi | & \textnormal{\textit{if} $s\geq t$,}\\
||X(t)PX^{-1}(t)|| &\leq & M & \textnormal{\textit{for any} $t$,}
\end{array}\right.
\end{equation}
where $t,s\in J$.
\end{proposition}
\begin{proof}
Firstly, we will assume that (\ref{LinealC2}) has an exponential dichotomy and (\ref{eq:2.2}) is satisfied. Notice that if $t\geq s$, the first inequality of (\ref{eq:2.2}) allow us to deduce:
\begin{displaymath}
\begin{array}{rcl}
|X(t)P\xi|&=&|X(t)PX^{-1}(s)X(s)P\xi|\\
&\leq & ||X(t)PX^{-1}(s)||\, |X(s)P\xi|\\
&\leq &  K_{1}e^{-\alpha(t-s)}|X(s)P\xi|,
\end{array}
\end{displaymath}
and the first inequality of (\ref{eq:2.3}) follows with $K_{1}=K$. The second inequality can be proved
in a similar way with $K_{2}=K$ while the third inequality follows directly (with $M=K$) by considering $t=s$ in
(\ref{eq:2.2}).

Secondly, assume that the inequalities (\ref{eq:2.3}) are verified. In the case $t\geq s$ we can consider $\xi=X^{-1}(s)\eta$
with $\eta\neq 0$. The first inequality of (\ref{eq:2.3}) implies that
\begin{displaymath}
\begin{array}{rcl}
|X(t)P\xi|=|X(t)PX^{-1}(s)\eta| &\leq & K_{1}e^{-\alpha(t-s)}|X(s)PX^{-1}(s)\eta| \\\\
&\leq & MK_{1}e^{-\alpha(t-s)}|\eta|, 
\end{array}
\end{displaymath}
then for any $t\geq s$, we can deduce that:
\begin{displaymath}
\sup\limits_{\eta\neq 0}\frac{|X(t)PX^{-1}(s)\eta|}{|\eta|}=
||X(t)PX^{-1}(s)||\leq MK_{1}e^{-\alpha(t-s)}. 
\end{displaymath}

Similarly, if $t\leq s$ and considering $\xi=X^{-1}(s)\eta$,
the second inequality of (\ref{eq:2.3}) implies that
\begin{displaymath}
\begin{array}{rcl}
|X(t)[I-P]X^{-1}(s)\eta| &\leq & K_{2}e^{-\alpha(s-t)}|X(s)[I-P]X^{-1}(s)\eta| \\\\
&\leq & MK_{2}e^{-\alpha(t-s)}|\eta|, 
\end{array}
\end{displaymath}
and for any $t\leq s$, we can deduce that:
\begin{displaymath}
\sup\limits_{\eta\neq 0}\frac{|X(t)[I-P]X^{-1}(s)\eta|}{|\eta|}=
||X(t)[I-P]X^{-1}(s)||\leq MK_{2}e^{-\alpha(s-t)},  
\end{displaymath}
then the exponential dichotomy is verified with projector $P$ and the constants $(K,\alpha)$
with $K=\max\{MK_{1},MK_{2}\}$. 
\end{proof}

The above result is stated without proof in \cite[p.11]{Cop}, to fill this gap we follow the lines of \cite[Prop. 3.4]{BFS}.
As pointed out in \cite{Cop}, the third inequality of (\ref{eq:2.3}) is not implied
by the previous ones. Indeed it provides an example of linear system
uniquely satisfying the first two inequalities of (\ref{eq:2.3}) but not having an exponential dichotomy. This example is described in the section of exercises.

Nevertheless, if the property of bounded growth is verified, 
then the third inequality of (\ref{eq:2.3}) is implied by the previous ones as stated by
the following result:
\begin{lemma}\label{C12-C3}
If the linear system \eqref{LinealC2} has a uniform bounded growth on $J$ and there exists a non null projector
$P$ and positive constants $K_{1},K_{2}$ and $\alpha$ such that the first two inequalities of \eqref{eq:2.3} are verified, then
there exists $M>0$ such that the third one is also verified.
\end{lemma}

\begin{proof}
Firstly, let us define the positive maps for any $t\in J$:
\begin{equation*}
%\label{PM1}
\sigma(t)=||X(t)PX^{-1}(t)||  \quad \textnormal{and} \quad \rho(t)=||X(t)[I-P]X^{-1}(t)||.
\end{equation*}

By using triangular inequality, it can be proved that 
\begin{displaymath}
\rho(t)-\sigma(t)\leq ||I|| \quad \textnormal{and} \quad 
\sigma(t)-\rho(t)\leq ||I||,
\end{displaymath}
which implies that
\begin{equation}
\label{PM2}
|\sigma(t)-\rho(t)|\leq ||I|| \quad \textnormal{for any $t\in J$}.
\end{equation}

Now, let us consider $T>0$ and $\eta \neq 0$. By using the first and second inequalities of (\ref{eq:2.3}) with $\xi=X^{-1}(t)\eta$, it follows that
\begin{displaymath}
\begin{array}{rcl}
|X(t+T)PX^{-1}(t)\eta| &\leq & \displaystyle K_{1}e^{-\alpha T}|X(t)PX^{-1}(t)\eta|,
\end{array}
\end{displaymath}
and
\begin{displaymath}
\begin{array}{rcl}
|X(t)[I-P]X^{-1}(t)\eta| &\leq & 
\displaystyle K_{2}e^{-\alpha T}|X(t+T)[I-P]X^{-1}(t)\eta|,
\end{array}
\end{displaymath}
which leads to 
\begin{equation*}
%\label{PM3}
||X(t + T)PX^{-1}(t)||\leq K_{1}e^{-\alpha T}\sigma(t),
\end{equation*}
and
\begin{equation*}
%\label{PM4}
K_{2}^{-1}e^{\alpha T}\rho(t) \leq ||X(t+T)[I-P]X^{-1}(t)||.
\end{equation*}

Let us construct the strictly increasing map 
$$
T\mapsto\gamma(T)=K_{2}^{-1}e^{\alpha T}-K_{1}e^{-\alpha T}
$$
and choose $T$ big enough such that $\gamma(T)>0$. Now, notice that the above estimations imply:
\begin{displaymath}
\begin{array}{rcl}
\gamma(T) &\leq &  ||\rho^{-1}(t)X(t+T)[I-P]X^{-1}(t)||-||\sigma^{-1}(t)X(t+T)PX^{-1}(t)|| \\\\
&\leq & ||\rho^{-1}(t)X(t+T)[I-P]X^{-1}(t)+\sigma^{-1}(t)X(t+T)PX^{-1}(t)|| \\\\
&\leq & ||\rho^{-1}(t)X(t+T,t)\Phi(t)[I-P]X^{-1}(t)+\sigma^{-1}(t)X(t+T,t)X(t)PX^{-1}(t)|| \\\\
&\leq & ||X(t+T,t)||\, ||\rho^{-1}(t)X(t)[I-P]X^{-1}(t)+\sigma^{-1}(t)X(t)PX^{-1}(t)||.
\end{array}
\end{displaymath}

Moreover, we know by hypothesis that (\ref{LinealC2}) has 
a bounded growth on $J$. That is, according to Definition \ref{BGBD-C} from Chapter 1, there exist $K_{0}\geq 1$
and $\beta\geq 0$ such that:
\begin{displaymath}
||X(t+T,t)||\leq K_{0}e^{\beta T}.
\end{displaymath}
Now, by the above inequalities combined with (\ref{PM2}). we have that:
\begin{displaymath}
\begin{array}{rcl}
\gamma(T)K_{0}^{-1}e^{-\beta T} &\leq & ||\rho^{-1}(t)X(t)[I-P]X^{-1}(t)+\sigma^{-1}(t)X(t)PX^{-1}(t)|| \\\\
& = & ||\rho^{-1}(t)I+(\sigma^{-1}(t)-\rho^{-1}(t))X(t)PX^{-1}(t)|| \\\\
& \leq &  \rho^{-1}(t)||I||+|\sigma^{-1}(t)-\rho^{-1}(t)|\sigma(t) \\\\
& = &  \rho^{-1}(t)\left(||I||+|\rho(t)-\sigma(t)|\right) \\\\
&\leq & 2||I||\rho^{-1}(t),
\end{array}
\end{displaymath}
and we have that
$$
\rho(t)\leq \frac{2||I||K_{0}e^{\beta T}}{\gamma(T)}.
$$
Finally by (\ref{PM2}) we have that
$$
||X(t)PX^{-1}(t)|| \leq ||I||\left( 1+  \frac{2K_{0}e^{\beta T}}{\gamma(T)}\right),
$$
and the Lemma follows since the above estimation is valid for any $t\in J$.
\end{proof}

\subsection{Distinguished examples of exponential dichotomy}
\begin{lemma}
\label{ED-Aut-Ex}
If $A(t)=A\in M_{n}(\mathbb{C})$ is a constant matrix and any eigenvalue $\lambda\in\mathbb{C}$ of $A$ has non zero real part, then the autonomous system $\dot{x}=Ax$ has an exponential dichotomy on $\mathbb{R}$.
\end{lemma}

\begin{proof}
Firstly, we will assume that $A$ is a diagonalizable matrix with eigenvalues $\lambda_{j}=\alpha_{j}+i\beta_{j}\in \mathbb{C}$ with $j=1\ldots,n$ whose real parts satisfy
\begin{equation}
\label{ine+DE}
\alpha_{1}\leq \ldots\leq\alpha_{k}<0
<\alpha_{k+1}\ldots\leq \alpha_{n}.
\end{equation}

Then, there exists an invertible matrix $Q$ such that 
\begin{displaymath}
Q^{-1}AQ=D=\left(\begin{array}{cccc}
\lambda_{1} &        &     \\
            &\ddots  &     \\
            &        & \lambda_{n} 
            \end{array}\right),
\end{displaymath}
and $y=Q^{-1}x$ transforms $\dot{x}=Ax$ into the diagonal system
\begin{equation}
\label{diagonal-ED-ex}
\dot{y}=Dy,
\end{equation}
which has the following fundamental matrix
\begin{displaymath}
Y(t)=\left[\begin{array}{cc}
Y_{1}(t) & 0\\
0 & Y_{2}(t)
\end{array}\right],      
\end{displaymath}
where the matrices $Y_{1}(t)\in M_{k}(\mathbb{C})$ and $Y_{2}(t)\in M_{n-k}(\mathbb{C})$ are respectively defined by
\begin{displaymath}
Y_{1}(t)=\left[\begin{array}{ccc}
e^{\lambda_{1}t} &        &     \\
            &\ddots  &     \\
&  & e^{\lambda_{k}t}
            \end{array}\right]
            \quad \textnormal{and} \quad
  Y_{2}(t)=\left[\begin{array}{ccc}
e^{\lambda_{k+1}t} &        &     \\
            &\ddots  &     \\
            &        & e^{\lambda_{n}t}
            \end{array}\right].          
\end{displaymath}

Let $I_{k}$ be the identity matrix of $M_{k}(\mathbb{C})$ and consider the canonical projection
\begin{displaymath}
P_{k}=\left[\begin{array}{cc}
I_{k} & 0\\
0 & 0
\end{array}\right].      
\end{displaymath}

By using the matrix norm $||\cdot||_{1}$ combined with (\ref{ine+DE}) we can see that if $t\geq s$ then
\begin{displaymath}
\begin{array}{rcl}
||Y(t)P_{k}Y^{-1}(s)||_{1}&=&\max\{e^{-|\alpha_{1}|(t-s)},\ldots,e^{-|\alpha_{k}|(t-s)}\}\\
&=& e^{-|\alpha_{k}|(t-s)} \quad \textnormal{if $t\geq s \quad t,s\in \mathbb{R}$}.
\end{array}
\end{displaymath}

Similarly, if $t\leq s$ we can see that
\begin{displaymath}
\begin{array}{rcl}
||Y(t)[I-P_{k}]Y^{-1}(s)||_{1}&=&\max\{e^{\alpha_{k+1}(t-s)},\ldots,e^{\alpha_{n}(t-s)}\}\\
&=&
\max\{e^{-\alpha_{k+1}(s-t)},\ldots,e^{-\alpha_{n}(s-t)}\}\\
&=& e^{-|\alpha_{k+1}|(s-t)} \quad \textnormal{if $s\geq t \quad t,s\in \mathbb{R}$}.
\end{array}
\end{displaymath}

Let $\alpha=\min\{|\alpha_{k}|,\alpha_{k+1}\}$ and note that
\begin{displaymath}
\left\{\begin{array}{rccr}
||Y(t)P_{k}Y^{-1}(s)||_{1}&\leq  e^{-\alpha(t-s)} & \textnormal{if} & t\geq s,\quad t,s\in \mathbb{R}\\
||Y(t)(I-P_{k})Y^{-1}(s)||_{1} &\leq e^{-\alpha(s-t)} & \textnormal{if} & t \leq s, \quad t,s \in \mathbb{R},
\end{array}\right.
\end{displaymath}
then we can conclude that the diagonal system (\ref{diagonal-ED-ex}) has an exponential dichotomy on $\mathbb{R}$.

It is an easy exercise to show that $X(t)=QY(t)$ is a fundamental matrix of $\dot{x}=Ax$. Then we can deduce
\begin{displaymath}
\begin{array}{rcl}
||X(t)P_{k}X^{-1}(s)||_{1}&=&||QY(t)P_{k}Y^{-1}(s)Q^{-1}||_{1}\\\\
&\leq & ||Q||_{1} \,||Q^{-1}||_{1}\,||X(t)P_kX^{-1}(s)||_{1}\\\\
&\leq & ||Q||_{1}||Q^{-1}||_{1}e^{-\alpha(t-s)} \quad t\geq s,
\end{array}
\end{displaymath}
and similarly
\begin{displaymath}
\begin{array}{rcl}
||X(t)[I-P_{k}]X^{-1}(s)||_{1}
&\leq & ||Q||_{1}||Q^{-1}||_{1}e^{-\alpha(s-t)} \quad s\geq t,
\end{array}
\end{displaymath}
which can be summarized as
\begin{displaymath}
\left\{\begin{array}{rccr}
||X(t)P_{k}X^{-1}(s)||_{1}&\leq  ||Q||_{1}||Q^{-1}||_{1}e^{-\alpha(t-s)} & \textnormal{if} & t\geq s,\quad t,s\in \mathbb{R}\\
||X(t)(I-P_{k})X^{-1}(s)||_{1} &\leq ||Q||_{1}||Q^{-1}||_{1}e^{-\alpha(s-t)} & \textnormal{if} & t \leq s, \quad t,s \in \mathbb{R}
\end{array}\right.
\end{displaymath}
and the exponential dichotomy on $\mathbb{R}$ follows.

When the matrix is not diagonalizable we refer the reader to \cite[pp.214--215]{CK} where the proof follows a less constructive approach which is based in the Riesz's projector.

\end{proof}

The method used to proof the above Lemma (in the diagonalizable case) illustrates a particular case of the following result:
\begin{theorem}
\label{KSDE}
If the linear system \eqref{LinealC2} has an exponential dichotomy on $J$ and is generally kinematically similar to
\begin{equation}
\label{LinealC2B}
y'=B(t)y \quad \textnormal{with $t\in J$,}
\end{equation}
then \eqref{LinealC2B} has an exponential dichotomy on $J$ with the same projection $P$ and rate $\alpha$.
\end{theorem}

\begin{proof}
If (\ref{LinealC2}) and (\ref{LinealC2B}) are generally kinematically similar by the transformation $Q(t)$. It is an easy exercise to prove that if $X(t)$ is the fundamental matrix of (\ref{LinealC2}), then $Y(t)=Q^{-1}(t)X(t)$ is the fundamental matrix of (\ref{LinealC2B}).

Now, let us assume that
$t\geq s$, $t,s\in J$ and $P$ is the projection of Definition \ref{DICE}. By using the boundedness of $Q$ and $Q^{-1}$ we can see that
\begin{displaymath}
\begin{array}{rcl}
||Y(t)PY^{-1}(s)||&=&||Q^{-1}(t)X(t)PX^{-1}(s)Q(s)||\\\\
&\leq & ||Q(s)|| \,||Q^{-1}(t)||\,||X(t)PX^{-1}(s)||\\\\
&\leq & ||Q||_{\infty}||Q^{-1}||_{\infty}Ke^{-\alpha(t-s)} \quad t\geq s.
\end{array}
\end{displaymath}

The inequality for the case $t\leq s$ can be proved similarly.
\end{proof}

\begin{corollary}
If $t\mapsto A(t)$ is continuous and  $A(t+\omega)=A(t)$ for any $t\in \mathbb{R}$, then the system $\dot{x}=A(t)x$ has an exponential dichotomy on $\mathbb{R}$ if the Floquet multipliers are not in the unit circle.
\end{corollary}

\begin{proof} 
Without loss of generality, let us assume that the fundamental matrix is normalized, \textit{i.e.}, $X(0)=I$. Then by Floquet's theory we know that 
$X(t)=Q(t)e^{Dt}$ where  $t\mapsto Q(t)$ is $\omega$--periodic and $D\in\mathbb{C}^{n\times n}$ satisfies
$$
D=\frac{1}{\omega}\ln X(\omega).
$$

As the Floquet's multipliers are the eigenvalues of $X(\omega)$, it is a technical exercise of matrix functions to prove that the eigenvalues of $D$ cannot be purely imaginary and by Lemma \ref{ED-Aut-Ex} we have that the autonomous system
$$
\dot{y}=Dy
$$
has an exponential dichotomy on $\mathbb{R}$.

Finally, by Theorem \ref{Teo-Floquet} from Chapter 1, we know that the above system is classically kinematically similar to $\dot{x}=A(t)x$, the Theorem \ref{KSDE} implies that the $\omega$--periodic system
has an exponential dichotomy on $\mathbb{R}$ and the result follows.
\end{proof}

\subsection{An alternative characterization for the exponential dichotomy}

\begin{theorem}
\label{TCA}
Let $X(t)$ be a fundamental matrix of the linear system \eqref{LinealC2}. This system has an exponential dichotomy on $J$ with projection 
$P(\cdot)$ and positive constants $K$,$\alpha$ if and only if 
\begin{equation}
\label{AltChar}
\left\{\begin{array}{rccr}
||X(t,s)P(s)||&\leq K e^{-\alpha(t-s)} & \textnormal{if} & t\geq s,\quad t,s\in J\\
||X(t,s)[I-P(s)]|| &\leq Ke^{-\alpha(s-t)} & \textnormal{if} & t \leq s, \quad t,s \in J,
\end{array}\right.
\end{equation}
where $t\mapsto P(t)$ is also a projection for any $t\in J$ such that
\begin{equation}
\label{invariancia} 
P(t)X(t,s)=X(t,s)P(s)
\end{equation}
and the rank of $P(t)$ and $P$ are identical for any $t\in J$.
\end{theorem}

\begin{proof}
$(\Rightarrow$ Let us assume that the linear system (\ref{LinealC2}) has an exponential dichotomy on
$J$ described by (\ref{eq:2.2}). Now, we construct the matrix
$$
P(t)=X(t)PX^{-1}(t) \quad \textnormal{for any $t\in J$}.
$$

Firstly, it is easy to verify that $P(t)$ is a projection for any $t\in J$, moreover
\begin{displaymath}
\begin{array}{rcl}
P(t)X(t,s)&=& X(t)PX^{-1}(s)\\\\
&=& X(t,s)X(s)PX^{-1}(s)\\\\
&=& X(t,s)P(s),
\end{array}
\end{displaymath}
and (\ref{invariancia}) is satisfied. 

A direct byproduct of the above identities are
$$
X(t)PX^{-1}(s)=X(t,s)P(s)
\quad \textnormal{and} \quad 
 X(t)[I-P]X^{-1}(s)=X(t,s)[I-P(s)],
$$ 
then, the inequalities (\ref{eq:2.2}) imply (\ref{AltChar}).

Finally, a direct consequence of (\ref{invariancia}) is that if
$\{\xi_{1},\xi_{2},\ldots,\xi_{d}\}$ is a basis of $\ker P(s)$ then $\{X(t,s)\xi_{i}\}_{i=1}^{d}$ is a basis of $\ker P(t)$, this is implies
$\text{Rank} P(t)$ is invariant for any $t\in J$ and consequently is equal to
$\text{Rank}\, P(0)=\text{Rank}\, X(0)PX^{-1}(0)=\text{Rank} P$ since $X(0)$ is invertible.

\medskip

\noindent $\Leftarrow)$ On the other hand, let us assume that the transition matrix of the linear system
(\ref{LinealC2}) is such that the inequalities (\ref{AltChar}) are verified and the projections $P(\cdot)$ verify (\ref{invariancia}). 

Without loss of generality, we will consider a normalized fundamental matrix $X(t)$, then we can see that
\begin{displaymath}
\begin{array}{rcl}
X(t,s)P(s)&=&X(t,0)X(0,s)P(s)\\\\
&=&X(t,0)P(0)X(0,s)\\\\
&=&X(t)P(0)X^{-1}(s),
\end{array}
\end{displaymath}
while the identity $X(t,s)[I-P(s)]=X(t)[I-P(0)]X^{-1}(s)$ can be deduced in a similar
way and the inequalities (\ref{AltChar}) imply (\ref{eq:2.2}) with projection $P(0)$. 
\end{proof}

\begin{remark}
As $t\mapsto \text{Rank} P(t)$ is constant for any $t\in J$, it is said in the literature
that $P(\cdot)$ satisfying (\ref{invariancia}) is an \textit{invariant projection}.
\end{remark}

\begin{remark}
\label{UNIpro}
The above theorem allows to introduce the \textit{uniquess of projector problem}:
Given a fundamental matrix $X(t)$ of a linear system (\ref{LinealC2})
having an exponential dichotomy on $J$, the Theorem \ref{TCA} only ensures
the existence of a projector $P(\cdot)$. Is this
projector unique? The answer is dependent of the interval $J$.
\end{remark}

In spite that the above characterization of exponential dichotomy by variable projections is known,
there exists few detailed explanations in the literature. Among them we highlight \cite[Prop.7.4]{CLR} and \cite{Chu}. We stress that, in \cite{CLR}, the exponential dichotomy with constant projections is called as \textit{Coppel's dichotomy}.

This alternative characterization of the exponential dichotomy allows to carry out a simple proof
of the following result:
\begin{lemma}
\label{M-alpha}
If the linear system \eqref{LinealC2} has the property of bounded growth $\&$ decay on $J$:
$$
||X(t,s)||\leq K_{0}e^{M|t-s|} \quad \textnormal{where $J=[t_{0},+\infty)$}
$$
and moreover
has an exponential dichotomy on the interval $J$ with constants $K$ and $\alpha$ then it follows that $\alpha \leq M$.
\end{lemma}

\begin{proof}
By using properties of matrix norms verifying $||I||=1$ combined with the invariance property (\ref{invariancia}), we can deduce that
\begin{displaymath}
\begin{array}{rcl}
1 & = & ||X(t,s)X(s,t)||\\\\
& = & ||X(t,s)[P(s)+I-P(s)]X(s,t)||\\\\
&\leq &  ||X(t,s)P(s)||\, ||X(s,t)||+||X(t,s)||\,||X(s,t)[I-P(t)]||.
\end{array}
\end{displaymath}

Now, without loss of generality, we will assume that $t\geq s$. By using the property of exponential dichotomy combined with the bounded growth growth,
we can deduce that for any $t\geq s$ with $t,s\in J$ it follows that: 
\begin{displaymath}
\begin{array}{rcl}
1 &\leq &  K_{0}e^{M(t-s)}||X(t,s)P(s)||+K_{0}e^{M(t-s)}\,||X(s,t)[I-P(t)]||\\\\
&\leq & KK_{0}e^{M(t-s)}e^{-\alpha(t-s)}+KK_{0}e^{M(t-s)}e^{-\alpha(t-s)}
\end{array}
\end{displaymath}
and we conclude that
$$
1\leq 2KK_{0}e^{(M-\alpha)(t-s)} \quad \textnormal{for any $t\geq s$ with $t,s\in J$}.
$$

We have either $M\geq \alpha$ or $M<\alpha$, if the last inequality is verified, by letting $t\to +\infty$ we will have that
$1\leq 0$ obtaining a contradiction and the result follows.
\end{proof}

In spite that the above result is rather intuitive, it is difficult to find a formal proof in the literature. In some cases, as for example \cite[p.823]{Shi}, this result is stated without proof.

Finally, we point out that this alternative characterization allow us to revisit Theorem \ref{KSDE} which states that the kinematical
similarity preserves the exponential dichotomy with the same projection and rate.

\begin{theorem}
%\label{KSDE-bis}
If the linear system \eqref{LinealC2} is generally kinematically similar to
\begin{equation}
\label{LinealC2B-bis}
y'=B(t)y \quad \textnormal{with $t\in J$},
\end{equation}
via the transformation $y=Q^{-1}(t)x$ and has an exponential dichotomy on $J$ with constants $(K,\alpha)$ and 
invariant projector $P(t)$, then \eqref{LinealC2B-bis} has an exponential dichotomy on $J$ with 
projections  $\tilde{P}(t)=Q^{-1}(t)P(t)Q(t)$ and constants $(K||Q||_{\infty}||Q^{-1}||_{\infty},\alpha)$.
\end{theorem}

\begin{proof}
Notice that $\tilde{P}^{2}(t)=\tilde{P}(t)$ can be deduced straightforwardly. Moreover, if $X(t)$ is a fundamental matrix of (\ref{LinealC2}) then we know that $Y(t)=Q^{-1}(t)X(t)$ is a fundamental matrix
of (\ref{LinealC2B-bis}) and the identity $X(t,s)=Q(t)Y(t,s)Q^{-1}(s)$ is fulfilled. 

By using the above identity combined with the invariance of the projector $P(t)$
described by (\ref{invariancia}), we can deduce that
\begin{displaymath}
\begin{array}{rcl}
\tilde{P}(t)Y(t,s) & = & Q^{-1}(t)P(t)Q(t)Y(t,s)Q^{-1}(s)Q(s) \\
                  &  = & Q^{-1}(t)P(t)X(t,s)Q(s) \\
                  & =  & Q^{-1}(t)X(t,s)P(s)Q(s) \\
                  & =  & Y(t,s)Q^{-1}(s)P(s)Q(s) \\
                  & = &  Y(t,s)\tilde{P}(s),
\end{array}
\end{displaymath}
and the invariance of $\tilde{P}(s)$ follows. 

If $t\geq s$, the above estimations together with the first inequality of (\ref{AltChar})
allow us to deduce that
\begin{displaymath}
\begin{array}{rcl}
||Y(t,s)\tilde{P}(s)|| & = & ||Y(t,s)Q(s)P(s)Q^{-1}(s)|| \\
                       & = & ||Q(s)X(t,s)P(s)Q^{-1}(s)|| \\ 
                       & \leq & K||Q||_{\infty}||Q^{-1}||_{\infty}||X(t,s)P(s)|| \\
                       & \leq & K||Q||_{\infty}||Q^{-1}||_{\infty}e^{-\alpha(t-s)}
\end{array}
\end{displaymath}
whereas
$$
||Y(t,s)[I-\tilde{P}(s)]||\leq K||Q||_{\infty}||Q^{-1}||_{\infty}e^{-\alpha(s-t)}
$$
can be deduced in a similar way when $t\leq s$.
\end{proof}

\section{Exponential dichotomy on the full line}

In this section we will study the Definition \ref{DICE} of exponential dichotomy
considering $J=\mathbb{R}$, this case will allow to take a look at two consequences of the hyperbolicity, namely, the properties of splitting of solutions and admissibility from a nonautonomous point of view.  

\subsection{Splitting of solutions}

\begin{theorem}
\label{prop-split}
If $\dot{x}=A(t)x$ has an exponential dichotomy on $\mathbb{R}$ with projector $P(\cdot)$ and positive constants $K$ and $\alpha$, then any solution
$t\mapsto x(t,t_{0},\xi):=X(t,t_{0})\xi$ passing through $\xi$ at $t=t_{0}$
can be splitted as follows: 
\begin{equation}
\label{splitting}
x(t,t_{0},\xi)=\underbrace{X(t,t_{0})P(t_{0})\xi}_{:=x^{+}(t,t_{0},\xi)} + \underbrace{X(t,t_{0})[I-P(t_{0})]\xi}_{:=x^{-}(t,t_{0},\xi)}
\end{equation}
such that for any $t\geq t_{0}$ it follows that
\begin{subequations}
  \begin{empheq}{align}
    & |x^{+}(t,t_{0},\xi)|\leq Ke^{-\alpha (t-t_{0})}|P(t_{0})\xi| , \label{DE+I1}\\
    & \frac{|Q(t_{0})\xi|}{K}e^{\alpha(t-t_{0})}\leq |x^{-}(t,t_{0},\xi)|,\label{DE+I2}
  \end{empheq}
\end{subequations}
where $Q(t_{0})=I-P(t_{0})$. Morevoer, for any $t \leq t_{0}$ it follows that
\begin{subequations}
  \begin{empheq}{align}
    & \frac{|P(t_{0})\xi|}{K}e^{\alpha(t_{0}-t)}\leq |x^{+}(t,t_{0},\xi)| , \label{DE-I1}\\
    & |x^{-}(t,t_{0},\xi)|\leq Ke^{\alpha(t-t_{0})}|Q(t_{0})\xi|.\label{DE-I2}
  \end{empheq}
\end{subequations}
\end{theorem}

\begin{proof}
Firstly note that by using $\xi=P(t_{0})\xi+[I-P(t_{0})]\xi$, the decomposition
$x(t,t_{0},\xi)=x^{+}(t,t_{0},\xi)+x^{-}(t,t_{0},\xi)$ is straightforwardly deduced.

Let us assume that $t\geq t_{0}$. Then, by using the matrix norm definition combined with $P(t_{0})\xi=P^{2}(t_{0})\xi $ and the first inequality of (\ref{AltChar}) with $s=t_{0}$ and $J=\mathbb{R}$, it follows that
\begin{displaymath}
|x^{+}(t,t_{0},\xi)|\leq ||X(t,t_{0})P(t_{0})||\,|P(t_{0})\xi| \leq Ke^{-\alpha(t-t_{0})}|P(t_{0})\xi|,
\end{displaymath}
and (\ref{DE+I1}) is deduced. On the other hand, as $Q(\cdot)$ is an invariant projection, notice that
\begin{displaymath}
\begin{array}{rcl}
Q(t_{0})\xi&=&X(t_{0},t)X(t,t_{0})Q(t_{0})\xi \\\\
 &=&X(t_{0},t)X(t,t_{0})Q(t_{0})Q(t_{0})\xi \\\\
 &=&X(t_{0},t)Q(t)X(t,t_{0})Q(t_{0})\xi\\\\
 &=&X(t_{0},t)Q(t)x^{-}(t,t_{0},\xi)\\\\
\end{array}
\end{displaymath}

Now, by using the matrix norm definition and the second inequality if (\ref{AltChar})
with $(t_{0},t)$ instead of $(t,s)$, we have that
\begin{displaymath}
|Q(t_{0})\xi|\leq ||X(t_{0},t)Q(t)||\, |x^{-}(t,t_{0},\xi)|\leq Ke^{-\alpha(t-t_{0})}|x^{-}(t,t_{0},\xi)|
\end{displaymath}
and (\ref{DE+I2}) can be easily deduced.

The proof of the inequalities (\ref{DE-I1})--(\ref{DE-I2})
can be made by following the lines of the previous one and is left for the reader.
\end{proof}

\begin{remark}
\label{ext-split}
Despite the above theorem considered an exponential dichotomy on $\mathbb{R}$ as assumption. A careful reading of the proof will show that a similar splitting of solutions can be obtained if we consider an exponential dichotomy on
$\mathbb{R}_{0}^{+}$ or $\mathbb{R}_{0}^{-}$.
\end{remark}

As Theorem \ref{prop-split} ensures the decomposition of any nontrivial solution $t\mapsto x(t,t_{0},\xi):=X(t,t_{0})\xi$ of a system (\ref{LinealC2}) having the property of exponential 
dichotomy on the real line, we will see some consequences:

\noindent $\bullet$ The solutions 
$$
x^{+}(t,t_{0},\xi)=X(t,t_{0})P(t_{0})\xi
\quad \textnormal{and} \quad  x^{-}(t,t_{0},\xi)=X(t,t_{0})Q(t_{0})\xi
$$
have an asymptotic behavior dependent of the splitting of any initial condition: observe that $t\mapsto x^{+}(t,t_{0},\xi)$ has an exponential decay to the origin (with rate $\alpha$) when $t\to +\infty$, while $t\mapsto x^{-}(t,t_{0},\xi)$ has an exponential growth (with rate $\alpha$)
when $t\mapsto +\infty$. This \textit{dichotomic} asymptotic behavior with exponential rate motivates the use of the name exponential dichotomy.

\noindent $\bullet$ The above behavior is reversed when $t\to -\infty$.

\noindent $\bullet$ The exponential growth or decay of any solution --at rate $\alpha$-- above mentioned is dependent of the difference $t-t_{0}$ but it is independent of the initial time $t_{0}$. Due to this fact, several authors use the concept of \textit{uniform exponential dichotomy} in order to emphasize the uniformity with respect to the initial time.

\noindent $\bullet$ Last but not least, an important consequence of 
Theorem \ref{prop-split} is the following result:
\begin{corollary}
\label{u-bounded}
If the system $\dot{x}=A(t)x$ has an exponential dichotomy on $\mathbb{R}$, then the unique solution bounded on $\mathbb{R}$ is the trivial one.
\end{corollary}

\begin{proof}
Let $t\mapsto x(t,t_{0},\xi)$ be a nontrivial solution of the linear system. Then, by the above comments, we have that  $t\mapsto x^{+}(t,t_{0},\xi)$ and $t\mapsto x^{-}(t,t_{0},\xi)$
must be unbounded at $-\infty$ and $+\infty$ respectively, which implies that any nontrivial solution cannot be bounded on $\mathbb{R}$.
\end{proof}

The above mentioned splitting of the solutions of a linear system $\dot{x}=A(t)x$ having an exponential dichotomy on $\mathbb{R}$ also provides a decomposition of the initial conditions as direct sum of
two subspaces which is reminiscent to the description of the hyperbolicity condition in the 
autonomous case. In fact, for any initial time $t_{0}$, we can define the $\mathbb{R}$--vector subspaces $\mathcal{V}(t_{0})$
and $\mathcal{W}(t_{0})$ as follows:
\begin{displaymath}
\mathcal{V}(t_{0})=\left\{\xi\in \mathbb{R}^{n}\colon t\mapsto |X(t,t_{0})\xi| \quad \textnormal{is bounded in $[t_{0},+\infty)$}\right\}
\end{displaymath}
and
\begin{displaymath}
\mathcal{W}(t_{0})=\left\{\xi\in \mathbb{R}^{n}\colon t\mapsto |X(t,t_{0})\xi| \quad \textnormal{is bounded in $(-\infty,t_{0}]$}\right\}.
\end{displaymath}

The following result describes the relation between these subspaces with the projector $P(t_{0})$:
\begin{lemma}
\label{subpespacios-t_0}
Let $X(t)$ be a fundamental matrix of \eqref{LinealC2}. If this system has an exponential dichotomy on $\mathbb{R}$ with projection $P(\cdot)$ and positive constants $K$ and $\alpha$, then the above defined subspaces  $\mathcal{V}(t_{0})$ and $\mathcal{W}(t_{0})$ verify
\begin{displaymath}
\textnormal{\text{Im}}P(t_{0})=\mathcal{V}(t_{0}), \quad \ker P(t_{0})=\mathcal{W}(t_{0}) \quad \textnormal{and} \quad \mathbb{R}^{n}=\mathcal{V}(t_{0})\oplus \mathcal{W}(t_{0}). 
\end{displaymath}

Moreover, the corresponding subspaces $\mathcal{V}(t)$ 
and $\mathcal{W}(t)$ are well defined for any $t\in \mathbb{R}$ and verify
\begin{displaymath}
\dim \mathcal{V}(t)=\dim\mathcal{V}(t_{0}) \quad \textnormal{and} \quad
\dim \mathcal{W}(t)=\dim\mathcal{W}(t_{0}) \quad \textnormal{for any $t\in \mathbb{R}$}.
\end{displaymath}
\end{lemma}
\begin{proof}
Let $\xi \in \text{Im}P(t_{0})$, then $P(t_{0})\xi=\xi$ and assume that $t\geq t_{0}$. By using inequality (\ref{DE+I1}), we have that
$$
|X(t,t_{0})\xi|=|X(t,t_{0})P(t_{0})\xi| \leq Ke^{-\alpha(t-t_{0})}|\xi|\leq K|\xi|,
$$
and the boundedness of $t\mapsto X(t,t_{0})\xi$ on $[t_{0},+\infty)$ follows, then $\xi\in \mathcal{V}(t_{0})$ and $\text{Im}P(t_{0})\subseteq \mathcal{V}(t_{0})$ is verified.

Now let $\xi \in \mathcal{V}(t_{0})$ and assume that $t\geq t_{0}$.
The proof of $\xi \in \text{Im}P(t_{0})$ will be made by contradiction. Indeed, if $\xi \notin \text{Im}P(t_{0})$, we can deduce the property $[I-P(t_{0})]\xi\neq 0$ and, by using triangle's inequality, we have that:
\begin{displaymath}
|X(t,t_{0})[I-P(t_{0})]\xi|\leq |X(t,t_{0})\xi|+|X(t,t_{0})P(t_{0})\xi|.
\end{displaymath}

Note that the left term is unbounded on $[t_{0},+\infty)$ since (\ref{DE+I2}), while the right terms are bounded on $[t_{0},+\infty)$, the first term is bounded by hypothesis whereas the second one is bounded due to (\ref{DE+I1}), obtaining a contradiction. Then $\mathcal{V}(t_{0})\subseteq \text{Im}P(t_{0})$ and the identity of both sets follows.

\medskip
In order to prove that $\ker P(t_{0})=\mathcal{W}(t_{0})$, let $\xi \in \ker P(t_{0})$, then $\xi=[ I-P(t_{0})]\xi$. As we are assuming that $t\leq t_{0}$ the inequality (\ref{DE-I1}) leads to
$$
|X(t,t_{0})\xi|=|X(t,t_{0})[I-P(t_{0})]\xi| \leq Ke^{-\alpha(t_{0}-t)}|\xi|\leq K|\xi|,
$$
then the boundedness of $t\mapsto X(t,t_{0})\xi$ on $(-\infty,t_{0}]$ follows and $\xi\in \mathcal{W}(t_{0})$ which implies that $\ker P(t_{0})\subseteq \mathcal{W}(t_{0})$.

Now let $\xi \in \mathcal{W}(t_{0})$ and assume that $t\leq t_{0}$. As before, the proof of $\xi \in \ker P(t_{0})$ will be made by contradiction. Indeed, if $\xi \notin \ker P(t_{0})$, it follows that the property $P(t_{0})\xi\neq 0$ and, by using triangle's inequality, we have that:
\begin{displaymath}
|X(t,t_{0})P(t_{0})\xi| \leq |X(t,t_{0})\xi|+|X(t,t_{0})[I-P(t_{0})]\xi|
\end{displaymath}

Note that the left term is unbounded on $(-\infty,t_{0}]$ since (\ref{DE-I1}), while the right terms are bounded on $(-\infty,t_{0}]$, the first since $\xi\in \mathcal{W}(t_{0})$ and the second for (\ref{DE-I2}), obtaining a contradiction. Then $\mathcal{W}(t_{0})\subseteq \text{Im}P(t_{0})$ and the identity of both sets follows.

\medskip

Note that $\mathcal{V}(t_{0})\cap \mathcal{W}(t_{0})=\{0\}$. In fact, if $\xi \in \mathcal{V}(t_{0})\cap \mathcal{W}(t_{0})$, we will have that $t\mapsto X(t,t_{0})\xi$ is solution of the linear system and is  bounded on $\mathbb{R}$, then Corollary \ref{u-bounded} imply that $\xi=0$. This fact, combined with the Rank--Nullity Theorem imply that $\mathbb{R}^{n}=\mathcal{V}(t_{0})\oplus \mathcal{W}(t_{0})$.

Finally, we know by Theorem \ref{TCA} that $\dim \textnormal{Im}P(t)$ is constant for
any $t\in \mathbb{R}$, then it follows that that the dimension of the subspaces $\mathcal{V}(t_{0})$ and $\mathcal{W}(t_{0})$ is not dependent of the initial time $t_{0}$.
\end{proof}

By using Theorem \ref{TCA} we can see that as similar result can be obtained when considering constant
projectors
\begin{lemma}
\label{subpespacios-cte}
Let $X(t)$ be a fundamental matrix of \eqref{LinealC2}. If this system has an exponential dichotomy on $\mathbb{R}$ with projector $P$ and positive constants $K$ and $\alpha$, then if follows that:
\begin{displaymath}
\textnormal{\text{Im}}P=\mathcal{V}, \quad \ker P=\mathcal{W}\quad \textnormal{and} \quad \mathbb{R}^{n}=\mathcal{V}\oplus \mathcal{W},
\end{displaymath}
where 
\begin{displaymath}
\mathcal{V}=\left\{\xi\in \mathbb{R}^{n}\colon t\mapsto |X(t,0)\xi| \quad \textnormal{is bounded in $[0,+\infty)$}\right\}
\end{displaymath}
and
\begin{displaymath}
\mathcal{W}=\left\{\xi\in \mathbb{R}^{n}\colon t\mapsto |X(t,0)\xi| \quad \textnormal{is bounded in $(-\infty,0]$}\right\}.
\end{displaymath}
\end{lemma}

The above result provides the answer to a question stated in Remark \ref{R1ED-bis}.

\subsection{Admissibility}

The property of exponential dichotomy on the full line will allow us to
generalize the Proposition \ref{admi} to the nonautonomous framework. In order to do that,
we need to introduce the Green function associated to the dichotomy:
\begin{definition}
\label{Green2}
Let $X(t)$ be fundamental matrix of \eqref{LinealC2}. If this system has an exponential dichotomy on $\mathbb{R}$ with projector $P(\cdot)$, its \textbf{Green function} is defined by
\begin{displaymath}
\mathcal{G}(t,s)=\left\{\begin{array}{rcl}
X(t)PX^{-1}(s)&\quad \textnormal{if} \quad & t\geq s,\\
-X(t)(I-P)X^{-1}(s)& \quad \textnormal{if} \quad & t \leq s.\\
\end{array}\right.
\end{displaymath}
\end{definition}

Obviously, if we work with the alternative characterization of exponential
dichotomy stated in subsection 2.2, the Green function is
\begin{displaymath}
\mathcal{G}(t,s)=\left\{\begin{array}{rcl}
X(t,s)P(s)&\quad \textnormal{if} \quad & t\geq s,\\
-X(t,s)[I-P(s)]& \quad \textnormal{if} \quad & t \leq s.\\
\end{array}\right.
\end{displaymath}

A remarkable consequence of Corollary \ref{u-bounded}, and consequently, of exponential
dichotomy on $\mathbb{R}$ is the following result:
\begin{proposition}
\label{boundeness}
If the linear system \textnormal{(\ref{LinealC2})} has an exponential dichotomy on $\mathbb{R}$
with constants $K$,$\alpha$ and projector $P(\cdot)$, then for each bounded continous function $t\mapsto g(t)$, the inhomogeneous system
\begin{equation}
\label{perturbado}
z'=A(t)z+g(t)
\end{equation}
has a unique solution $x_{g}^{*}\in BC(\mathbb{R},\mathbb{R}^{n})$ defined by
$$
x_{g}^{*}(t)=\int_{-\infty}^{\infty}\mathcal{G}(t,s)g(s)\,ds.
$$

Moreover, the map $g\mapsto x_{g}^{*}$ is Lipschitz and verifies:
$$
|x_{g}^{*}|_{\infty}\leq \frac{2K}{\alpha}|g|_{\infty}.
$$
\end{proposition}

\begin{proof}

\noindent\textit{Step 1:} It is easy to prove that
\begin{displaymath}
x_{g}^{*}(t)=\int_{-\infty}^{\infty}\mathcal{G}(t,s)g(s)\,ds
\end{displaymath}
is a bounded solution of (\ref{perturbado}). Moreover, given a pair of functions
$g_{1},g_{2}\in B(\mathbb{R},\mathbb{R}^{n})$, a careful reading of Definition \ref{Green2} combined with the property of exponential dichotomy on $\mathbb{R}$ shows that 
$$
\begin{array}{rcl}
|x_{g_1}^*(t) - x_{g_2}^*(t)| & \leq & \displaystyle \frac{2 K}{ \alpha} |g_1(t) - g_2(t)|_{\infty}.
\end{array}
$$

\noindent\textit{Step 2:} We will prove that $x_{g}^{*}(t)$ is the unique bounded solution of (\ref{perturbado}). Indeed, let $t\mapsto x(t)$
be a bounded solution. By variation of parameters, we have that
\begin{displaymath}
\begin{array}{rcl}
x(t)&=&\displaystyle X(t,0)x(0)+\int_{0}^{t}X(t,s)g(s)\,ds
\end{array}
\end{displaymath}
which can be written as
\begin{displaymath}
\begin{array}{rcl}
x(t)&=&\displaystyle X(t,0)x(0)+\int_{0}^{t}\{X(t)PX^{-1}(s)g(s)+X(t)(I-P)X^{-1}(s)g(s)\}\,ds\\\\
&=&\displaystyle X(t,0)x(0)-\underbrace{\int_{-\infty}^{0}X(t)PX^{-1}(s)g(s)\,ds}_{:=x_{1}}+\int_{-\infty}^{t}X(t)PX^{-1}(s)g(s)\,ds\\\\
& &\displaystyle +\underbrace{\int_{0}^{\infty}X(t)(I-P)X^{-1}(s)g(s)\,ds}_{:=x_{2}}-\int_{t}^{\infty}X(t)(I-P)X^{-1}(s)g(s)\,ds
\end{array}
\end{displaymath}
and we can see that the bounded solution $t\mapsto x(t)$ can be written as follows
\begin{displaymath}
x(t)=X(t,0)\{x(0)-x_{1}+x_{2}\}+x_{g}^{*}(t).
\end{displaymath}

As $t\mapsto x_{g}^{*}(t)$ is a bounded solution of the inhomogeneous system (\ref{perturbado}), we have that
$t\mapsto x(t)-x_{g}^{*}(t)$ is a bounded solution of the linear system (\ref{LinealC2}). Finally, Corollary \ref{u-bounded} implies that $x(t)=x_{g }^{*}(t)$ and the uniqueness follows.
\end{proof}

In spite of the ubiquitousness and well knowness of this result in the literature,
it is surprising to find few proofs with the exception with the skectched proof in the Coppel's monography \cite{Cop}. The above proof follow the ideas from the paper of X. Chen and Y. Xia \cite{Chen}, which is remarkable by its simplicity.

The Proposition \ref{boundeness} has a plethora of consequences
and applications. In particular, the previous admissibility result can be
adapted for closed subspaces of the Banach space $(BC(\mathbb{R},\mathbb{R}^{n}),|\cdot|_{\infty})$
as the almost periodic and $\omega$--periodic functions. We will consider here the case when $f$ is $\omega$--periodic 
and will also consider the case of bounded nonlinear Lipschitz perturbations.

\begin{corollary}
\label{b-periodic}
If the linear and continuous $\omega$--periodic system \textnormal{(\ref{LinealC2})} has an exponential dichotomy on $\mathbb{R}$, then for each continuous $\omega$--periodic function $t\mapsto g(t)$, the 
system
\begin{equation}
\label{perturbado-per}
z'=A(t)z+g(t)
\end{equation}
has a unique continuous $\omega$--periodic solution $t\mapsto x^{*}(t)$ satisfying
\begin{equation}
\label{ITH}
x^{*}(0)=[I-X(\omega,0)]^{-1}\int_{0}^{\omega}X(\omega,s)g(s)\,ds=
\int_{-\infty}^{+\infty}\mathcal{G}(0,s)g(s)\,ds.
\end{equation}
\end{corollary}

\begin{proof}
The proof will use the property of exponential dichotomy from two perspectives. A first one
recalls that the exponential dichotomy on $\mathbb{R}$ is equivalent to the fact that
the Floquet's multipliers of  the linear and continuous $\omega$--periodic system (\ref{LinealC2}) cannot be in the unit circle. Then, by using Remark \ref{EPS} from Chapter 1, we have that the fundamental matrix $X(t)$ is such that $X(\omega,0)-I$ is non singular and the vector
\begin{equation}
\label{THPER}
x^{*}(0)=[I-X(\omega,0)]^{-1}\int_{0}^{\omega}X(\omega,s)g(s)\,ds
\end{equation}
is well defined. Now, a direct consequence of the above identity is that $t\mapsto x(t)=X(t,0)x(0)$ is a continuous $\omega$--periodic solution of (\ref{perturbado-per}) if and only if $x(0)=x^{*}(0)$. In fact, notice that
$$
x(\omega)=X(\omega,0)x(0)+\int_{0}^{\omega}X(\omega,s)g(s)\,ds=x(0)
$$
is equivalent to
$$
x(0)=[I-X(\omega,0)]^{-1}\int_{0}^{\omega}X(\omega,s)g(s)\,ds,
$$
and the existence of a continuous $\omega$--periodic solution $t\mapsto x^{*}(t)$ is ensured.

As the continuous $\omega$--periodic functions are bounded on $\mathbb{R}$, we will follow a second perspective of the exponential dichotomy property: the Proposition \ref{boundeness} implies that the unique bounded solution is defined by the Green function, which must be identical to the $\omega$--periodic solution obtained previously, that is
$$
t\mapsto x^{*}(t)=\int_{-\infty}^{\infty}\mathcal{G}(t,s)g(s)\,ds=X(t,0)x^{*}(0)+\int_{0}^{t}X(t,s)g(s)\,ds,
$$
and the identity (\ref{ITH}) is obtained easily by evaluating at $t=0$ and using (\ref{THPER}).
\end{proof}

The proof of the above result was carried out considering the identity
$x(0)=x(\omega)$ but the proof could also be deduced by considering $x(0)=x(-\omega)$, which leads to the identity
\begin{equation}
\label{IPSP}
[I-X(\omega,0)]^{-1}\int_{0}^{\omega}X(\omega,s)g(s)\,ds=[I-X(-\omega,0)]^{-1}\int_{0}^{-\omega}\hspace{-0.3cm}X(-\omega,s)\,ds.
\end{equation}

In addition, the identity (\ref{ITH}) has been obtained as consequence
of uniqueness of the bounded solution. Nevertheless, it can be deduced directly
by Proposition \ref{boundeness} applied to the $\omega$--periodic case. In fact, notice that:
\begin{equation}
\label{ICPP}
\begin{array}{rcl}
x^{*}(0)&=& \displaystyle\int_{-\infty}^{+\infty}\mathcal{G}(0,s)g(s)\,ds\\\\ 
&=& \displaystyle
\int_{-\infty}^{0}\hspace{-0.2cm}X(0,s)P(s)g(s)\,ds-\int_{0}^{\infty}\hspace{-0.2cm}X(0,s)[I-P(s)]g(s)\,ds \\\\
&=& \displaystyle
P(0)\int_{-\infty}^{0}\hspace{-0.2cm}X(0,s)g(s)\,ds-[I-P(0)]\int_{0}^{\infty}\hspace{-0.2cm}X(0,s)g(s)\,ds.
\end{array}
\end{equation}

In addition, by using a natural change of variables combined with by the identity (\ref{TF1}) from Chapter 1 we can deduce that
\begin{displaymath}
\begin{array}{rcl}
\displaystyle \int_{0}^{\infty}X(0,s)g(s)\,ds & = &\displaystyle \int_{0}^{\omega}X(0,s)g(s)\,ds+ \int_{\omega}^{\infty}X(0,s)g(s)\,ds\\\\
& = &\displaystyle\int_{0}^{\omega}X(0,s)g(s)\,ds+
\int_{0}^{\infty}X(0,s+\omega)g(s)\,ds\\\\
& = &\displaystyle\int_{0}^{\omega}X(0,s)g(s)\,ds+
X(0,\omega)\int_{0}^{\infty}X(\omega,s+\omega)g(s)\,ds\\\\
& = &\displaystyle\int_{0}^{\omega}X(0,s)g(s)\,ds+
X(0,\omega)\int_{0}^{\infty}X(0,s)g(s)\,ds.
\end{array}
\end{displaymath}

We multiply the above identity by $X(\omega,0)$ and easily verify that
$$
\int_{0}^{\infty}X(0,s)g(s)=[X(\omega,0)-I]^{-1}\int_{0}^{\omega}X(\omega,s)g(s)\,ds.
$$
and we deduce that
$$
-[I-P(0)]\int_{0}^{\infty}X(0,s)g(s)=[I-P(0)][I-X(\omega,0)]^{-1}\int_{0}^{\omega}X(\omega,s)g(s)\,ds.
$$

Similarly as in the previous case, by using a  change of variables combined with the identity (\ref{TF1}) from Chapter 1 we can deduce that
\begin{displaymath}
\begin{array}{rcl}
\displaystyle \int_{-\infty}^{0}X(0,s)g(s)\,ds & = &\displaystyle \int_{-\omega}^{0}X(0,s)g(s)\,ds+
X(0,-\omega)\int_{0}^{\infty}X(0,s)g(s)\,ds.
\end{array}
\end{displaymath}

By multipliying by $X(-\omega,0)$ and using (\ref{IPSP}) we can deduce that
$$
P(0)\int_{-\infty}^{0}X(0,s)g(s)\,ds=P(0)[I-X(\omega,0)]^{-1}\int_{0}^{\omega}X(\omega,s)g(s)\,ds
$$
and (\ref{ICPP}) combined with the identities for the integral terms $\int_{-\infty}^{0}X(0,s)g(s)\,ds$ and $-[I-P(0)]\int_{0}^{\infty}X(0,s)g(s)$ obtained above leads to
\begin{displaymath}
\begin{array}{rcl}
\displaystyle\int_{-\infty}^{+\infty}\mathcal{G}(0,s)g(s)\,ds 
&=& \displaystyle [I-X(\omega,0)]^{-1}\int_{0}^{\omega}X(\omega,s)g(s)\,ds.
\end{array}
\end{displaymath}

%we can deduce an interesting property for the projections:
%\begin{lemma}
%If the $\omega$--periodic and continuous system \textnormal{(\ref{LinealC2})} has an exponential dichotomy on $\mathbb{R}$ with constants $K$,$\alpha$ and projections $P(\cdot)$, then the projections satisfy $P(t+\omega)=P(t)$ for any $t\in \mathbb{R}$.
%\end{lemma}

%\begin{proof}
%Let $P(\cdot)$ be family of projections associated to the exponential dichotomy.
%By using the property (\ref{invariancia}) we can see that
%$$
%X(t+\omega,\omega)P(\omega)=P(t+\omega)X(t+\omega,\omega).
%$$

%By using the identity (\ref{TF1}) from Chapter 1, it follows that the above identity
%is equivalent to
%$$
%X(t,0)P(\omega)=P(t+\omega)X(t,0),
%$$
%and we can see that $X(t,0)P()$
%\end{proof}

\begin{proposition}
\label{bounded2}
If the system \textnormal{(\ref{LinealC2})} has an exponential dichotomy on $\mathbb{R}$ with constants $K$,$\alpha$ and projection $P$. Moreover, if the nonlinear system:
\begin{equation}
\label{perturbado2}
w'=A(t)w+h(t,w)
\end{equation}
has a continuous perturbation 
$t\mapsto h(t,w)$ satisfying: 
\begin{equation}
\label{HLHB}
|h(t,w_{1})-h(t,w_{2})|\leq \gamma|w_{1}-w_{2}| \quad \textnormal{and} \quad |h(t,w)|\leq \mu
\end{equation}
for any $t\in \mathbb{R}$ and $w,w_{1},w_{2}\in \mathbb{R}^{n}$ where $\gamma$ is such that
$$
2K\gamma<\alpha,
$$
then the system \eqref{perturbado2}
has a unique bounded function $t\mapsto w^{*}(t)$ satisfying the fixed point property:
\begin{equation}
\label{FP}
w^{*}(t)=\int_{-\infty}^{\infty}\mathcal{G}(t,s)h(s,w^{*}(s))\,ds.
\end{equation}
\end{proposition}
\begin{proof}
Let $\varphi_{1}$ in the Banach space $(BC(\mathbb{R},\mathbb{R}^{n}),|\cdot|_{\infty})$ and construct a sequence $\{\varphi_{k}\}_{k}\subset BC(\mathbb{R},\mathbb{R}^{n})$, where $t\mapsto \varphi_{k}(t)$ is the unique
bounded solution of the inhomogeneous system
\begin{displaymath}
w'=A(t)w+h(t,\varphi_{k-1}(t))  \quad \textnormal{for any} \quad k\geq 2.
\end{displaymath}

In fact, notice that $t\mapsto h(t,\varphi_{k-1}(t))\in BC(\mathbb{R},\mathbb{R}^{n})$; in consequence, Proposition \ref{boundeness} implies the existence of the sequence $\{\varphi_{k}\}_{k}$, which ve\-ri\-fies the recursivity:
\begin{displaymath}
\varphi_{k}(t)=\int_{\infty}^{\infty}\mathcal{G}(t,s)h(s,\varphi_{k-1}(s))\,ds \quad \textnormal{for any $t\in \mathbb{R}$}. 
\end{displaymath}

Finally, notice that as $h(t,z)$ satisfies (\ref{HLHB}), it follows that
\begin{displaymath}
\begin{array}{rcl}
|\varphi_{k}(t)-\varphi_{k-1}(t)|& \leq & \displaystyle \int_{-\infty}^{\infty}|\mathcal{G}(t,s)|\,|h(s,\varphi_{k-1}(s))-h(s,\varphi_{k-2}(s))|\,ds\\\\
                                 & \leq &  \displaystyle \int_{-\infty}^{\infty} Ke^{-\alpha|t-s|}\gamma |\varphi_{k-1}-\varphi_{k-2}|_{\infty}\,ds\\\\
                                 & \leq &  \displaystyle\frac{2K\gamma}{\alpha}  |\varphi_{k-1}-\varphi_{k-2}|_{\infty}
\end{array}
\end{displaymath}
and we have that $\{\varphi_{k}\}_{k}$ is a Cauchy sequence. As $(BC(\mathbb{R},\mathbb{R}^{n}),|\cdot|_{\infty})$ is a Banach space, we have that $\varphi_{k}$
converges uniformly to (\ref{FP}).
\end{proof}

We will finish this section with two comments about the property of exponential dichotomy on $\mathbb{R}$.

\noindent $\bullet$ Firstly, a direct consequence of Definition \ref{DICE}
is that if the linear system (\ref{LinealC2}) has an exponential dichotomy on the interval $J$ then this property is also verified on any interval $\tilde{J}\subset J$. In consequence, the property of exponential dichotomy on $\mathbb{R}$ implies this property on $\mathbb{R}_{0}^{-}$ and $\mathbb{R}_{0}^{+}$.

\noindent $\bullet$ Secondly, it is important to emphasize that 
the linear system (\ref{LinealC2}) can have simultaneously an exponential dichotomy on $\mathbb{R}_{0}^{-}$ and $\mathbb{R}_{0}^{+}$ but this not implies necessarily the existence of an exponential dichotomy 
on $\mathbb{R}$. An interesting question is to find complementary conditions ensuring the
exponential dichotomy on $\mathbb{R}$ when the exponential dichotomy on
the half lines is simultaneously verified. In order 
to address this problem it will be useful to carry out a study of the exponential
dichotomy on the half lines, which is the topic of the next chapter.

\section{Comments and References}

\noindent \textbf{1)} A nice description of the exponential dichotomy on $\mathbb{R}$ as
a property emulating the hyperbolicity and its consequences is given in the article 
of Cs\'asz\'ar and Kov\'acs \cite{CK}. We also
highlight the interesting survey from Elaydi and Hajek \cite{Elaydi}. 

%It is important to highlight some ideas previous to Perron's seminal work \cite{Perron}, which pointed out related ideas as the works of Hadamard \cite{Hadamard}.
%\textcolor{red}{Desarrollarlo mas}
\medskip

\noindent \textbf{2)} The property of exponential dichotomy on $J$ can be generalized as the existence of an invariant projection $P(\cdot)$, a pair of positive constants $(\alpha,K)$ and an increasing homeomorphism $h\colon (a_{0},+\infty)\to (0,+\infty)$ 
such that the transition matrix $X(t,s)$ of (\ref{LinealC2}) satisfies:
    \begin{equation}
    \label{GDE}
    \left\{\begin{array}{rcl}
    ||X(t,s)P(s)||\leq  K\left(h(t)h^{-1}(s)\right)^{-\alpha} &\textnormal{if} & t\geq s\geq a_{0}, \\\\
    ||X(t,s)[I-P(s)]||\leq K\left(h(s)h^{-1}(t)\right)^{-\alpha} &\textnormal{if} & s\geq t \geq a_{0}.
    \end{array}\right.
    \end{equation}

The above property is called $h$--dichotomy has been introduced by Naulin and Pinto in \cite{Naulin-0,Naulin-a}, it encompasses
the exponential dichotomy and the \textit{generalized exponential dichotomy} introduced by Martin \cite{Martin}
and studied in depth by Jiang \cite{Jiang}, Xia \cite{Chen,Wang-XZ} and its collaborators. The $h$--dichotomy has been
recently revisited in \cite{Elorreaga} where the expansions and contractions are studied from a group theory
approach.

A generalization of the $h$--dichotomy has been done by replacing the map $h$ in the second inequality of (\ref{GDE})
by another increasing homeomorphism $k\colon (a_{0},+\infty)\to (0,+\infty)$ which has been introduced by Naulin
and Pinto in \cite{Naulin-a,Pinto-89}. 

The case considering a map $s\mapsto K(s)$ for any $s\in J$ instead of a constant corresponds to 
the \textit{nonuniform $h$--dichotomy}, which has been deeply
studied for the case $h(t)=e^{t}$ by Barreira, Valls and Dragi\v{c}evi\'c in \cite{Barreira1,Barreira2} and references therein.

We point out the existence of a wide family of dichotomies which can be seen as
generalizations of the previous ones and we refer the reader to \cite[Table 1]{ZFY}
for a detailed summary.

\medskip

\noindent \textbf{3)} The above description of generalized dichotomies allows 
a classification of asymptotic stabilities in the nonautonomous case
if we consider $J=[t_{0},+\infty)$ and the projection $P(s)=I$, the 
above construction becomes the nonuniform contraction
    \begin{displaymath}
    \begin{array}{rcl}
    ||X(t,s)||\leq  K(s)\left(h(t)\right)^{-\alpha} &\textnormal{for} & t\geq s>t_{0},
    \end{array}
    \end{displaymath}
where  $K\colon [0,+\infty[\to (0,+\infty)$, $\alpha>0$.

The most distinguished and studied case is
the \textit{uniform asymptotic stability} 
or \textit{uniform exponential stability}, which corresponds
to $K(s)=K>0$ for any $s\geq t_{0}$ and $h(t)=e^{t}$.

The \textit{nonuniform asymptotic stability} corresponds to 
$K(s)=Ke^{\varepsilon |s|}$ for any $s\geq t_{0}$ 
with $K,\varepsilon>0$ and $h(t)=e^{t}$.

The \textit{uniform $h$--stability} corresponds to $K(s)=K>0$ and $h$ satisfying the above properties, which has been introduced by Naulin
and Pinto in \cite{Naulin-a,Pinto-89}. This stability has been revisited recently from a group theory perspective by Pe\~na and Rivera--Villagr\'an in  \cite{JFP}. 

\medskip

\noindent \textbf{4)} To the best of our knowledge, the notion of admissibility has been introduced by Massera and Sch\"affer \cite{Massera}. 

\begin{definition}
Given two spaces of functions $\mathcal{B}$
and $\mathcal{D}$, consisting in mappings from $J\subseteq \mathbb{R}\to \mathbb{R}^{n}$, the pair $(\mathcal{B},\mathcal{D})$ is admissible for the linear system $\dot{y}=A(t)y+g(t)$ if for each $t\mapsto g(t)\in \mathcal{B}$ there exists a solution 
$t\mapsto y(t)\in \mathcal{D}$.
\end{definition}

The property of admissibility is strongly related to the exponential dichotomy 
on $\mathbb{R}$. In fact, Proposition \ref{boundeness} and Corollary \ref{b-periodic} can be seen as results of $(\mathcal{B},\mathcal{B})$--admissibility for $\mathcal{B}=BC(\mathbb{R},\mathbb{R}^{n})$ and $\mathcal{B}=C_{\omega}(\mathbb{R},\mathbb{R}^{n})$ provided that
$a_{ij}$ and $g$ are elements of $\mathcal{B}$ and $\dot{x}=A(t)x$
 has an exponential dichotomy on $\mathbb{R}$. There exists generalizations to the Banach spaces of almost periodic functions \cite[Ch.7]{Fink}, pseudo almost periodic functions \cite{AA,Zhang-PAP} and remotely almost periodic functions \cite{Maulen}. 

\medskip 
\noindent \textbf{5)} There exists numerical methods oriented to detect the exponential dichotomy on the real line and we refer to the works carried out by L. Dieci, C. Elia and E. Van Vleck in \cite{Dieci2010,Dieci2011}.
\section{Exercises}

\begin{itemize}
    \item[1.-] Prove the Remark \ref{R1ED+} in detail.
    \item[2.-] Prove that if $P$ is a projection and $V$ is a nonsingular matrix, then $\tilde{P}=V^{-1}PV$ is also a projection. 
    \item[3.-] Prove that if $P$ is a nontrivial projection, then its eigenvalues are $0$ and $1$.
    \item[4.-] Prove that if $\{P_{k}\}_{k\in \mathbb{N}}$ is a sequence of projections convergent to $Q$
    then $Q$ is also a projection.
    \item[5.-] If the linear system (\ref{LinealC2}) has an exponential dichotomy on $J$ with projection
    $P$ then prove that its adjoint has an exponential dichotomy with projection $I-P^{T}$. 
    \item[6.-] In \cite[p.12]{Cop}, the author consider the linear system
\begin{displaymath}
\begin{array}{rcl}
\dot{x}_{1}&=&-x_{1}+e^{2t}x_{2} \\
\dot{x}_{2}&=&x_{2},
\end{array}
\end{displaymath}
whose fundamental matrix is 
\begin{displaymath}
X(t)=\left[\begin{array}{cc}
e^{-t} & (e^{3t}-e^{-t})/4  \\
0      &   e^{t}
\end{array}\right].
\end{displaymath}
\begin{itemize}
\item[6.1)] Prove that the above system has not the bounded growth property.
\item[6.2)] Prove that the first two inequalities of (\ref{eq:2.3}) are verified
with the projector
\begin{displaymath}
P=\left[\begin{array}{cc}
1 & 0  \\
0 & 0
\end{array}\right].
\end{displaymath}
\item[6.3)] Prove that $||X(t)PX^{-1}(t)||$ is not bounded. 
\end{itemize}
    
    \item[7.-] Prove the inequalities (\ref{DE-I1})--(\ref{DE-I2}) from Theorem \ref{prop-split}.
    \item[8.-] If the linear system (\ref{LinealC2}) is antisymmetric for any $t\in \mathbb{R}$ then prove that cannot have an exponential dichotomy on $\mathbb{R}$.
    \item[9.-] Let $[0,+\infty)\ni t\mapsto a(t)$ a bounded and locally integrable function. Prove that if
    $$
    \limsup\limits_{t-s\to +\infty}\frac{1}{t-s}\int_{s}^{t}a(\tau)\,d\tau<-\alpha<0,
    $$
    then the scalar differential equation $\dot{x}=a(t)x$ has an exponential dichotomy on $\mathbb{R}^{+}$ with projection $p=1$.
     \item[10.-] Similarly, if $[0,+\infty)\ni t\mapsto a(t)$ is a bounded and locally integrable function. Prove that if
    $$
    \liminf\limits_{t-s\to +\infty}\frac{1}{t-s}\int_{s}^{t}a(\tau)\,d\tau>\alpha>0,
    $$
    then the scalar equation $\dot{x}=a(t)x$ has an exponential dichotomy on $\mathbb{R}^{+}$ with projection $p=0$.
    \item[11.-] Prove that the following linear scalar differential equations have not an exponential dichotomy on $\mathbb{R}$: 
    $$
    \dot{x}=\frac{1}{1+t^{2}}x, \quad \dot{x}=e^{-t^{2}}x \quad \textnormal{and} \quad
    \dot{x}=a(t)x
    $$
    with $a(t)>0$ for any $t\in \mathbb{R}$ such that $\displaystyle\int_{\mathbb{R}}a(t)\,dt<\infty$.
    \item[12.-] The average of a $\omega$--periodic and continuous function $f$
    is defined by
    $$
    \mathcal{M}(f):=\frac{1}{\omega}\int_{0}^{\omega}f(t)\,dt.
    $$
    
    Prove that the scalar linear equation $\dot{x}=f(t)x$ has an exponential dichotomy on $\mathbb{R}$
    if and only if $\mathcal{M}(f)\neq 0$. Determine the relation between the corresponding projection and the sign of $\mathcal{M}(f)$.
    
    \item[13.-] Let us consider the linear system $\dot{x}=A(t)x$ where
    $t\mapsto A(t)$ is continuous and $\omega$--periodic. If the system has an exponential dichotomy on $[0,+\infty)$, can we say that the system has an exponential dichotomy on $\mathbb{R}$ with the same projection and constant $\alpha$?

    \item[14.-]  Provide examples showing that Corollary \ref{u-bounded} cannot be extended for systems
    having an exponential dichotomy on $[0,+\infty)$.

    \item[15.-] Prove Lemma \ref{subpespacios-cte}.
    
    \item[16.-] Let us consider the Proposition \ref{boundeness} under the more specific assumption that the linear system \textnormal{(\ref{LinealC2})} has an exponential dichotomy on $\mathbb{R}$ with the identity as projection. Then for each bounded function $t\mapsto g(t)$, prove that the  
system (\ref{perturbado}) has a unique solution $t\mapsto x_{g}^{*}(t)$ bounded on $\mathbb{R}$ such that any other solution $t\mapsto x(t)$ satisfies 
$$
\lim\limits_{t\to +\infty}[x(t)-x^{*}(t)]=0.
$$
     \item[17.-] Similarly, let us consider the Proposition \ref{boundeness} under the more specific assumption that the linear system \textnormal{(\ref{LinealC2})} has an exponential dichotomy on $\mathbb{R}$ with the null projection $P=0$. Then for each bounded function $t\mapsto g(t)$, prove that the  
system (\ref{perturbado}) has a unique solution $t\mapsto x_{g}^{*}(t)$ bounded on $\mathbb{R}$ such that any other solution $t\mapsto x(t)$ satisfies 
$$
\lim\limits_{t\to -\infty}[x(t)-x^{*}(t)]=0.
$$

    \item[18.-] Let $t\mapsto x(t)$ a solution of $\dot{x}=A(t)x$ and define
$$
x_{1}(t)=X(t)PX^{-1}(t)x(t) \quad \textnormal{and} \quad x_{2}(t)=X(t)(I-P)X^{-1}(t)x(t),
$$
where $P$ is a projection. Prove that $x_{1}(t)$ and $x_{2}(t)$ satisfy the following properties:

\medskip
\noindent a) $X(t)PX^{-1}(t)x_{1}(t)=x_{1}(t)$, 

\noindent b) $X(t)(I-P)X^{-1}(t)x_{2}(t)=x_{2}(t)$,

\noindent c) $x(t)=X(t)PX^{-1}(s)x_{1}(s)+X(t)(I-P)X^{-1}(s)x_{2}(s)$,

\noindent d) $x(t)=x_{1}(t)+x_{2}(t)$.
\
    \item[19.-] In the context of Proposition \ref{bounded2}, can we obtain additional information of the fixed point in the cases where the linear part has exponential dichotomy with trivial projections?. 
    
    \item[20.-] Is it possible to extend Proposition \ref{bounded2} and to prove the existence
of $\omega$--periodic solutions of (\ref{perturbado2}) when $A(t)$ is $\omega$--periodic and continuous
and $t\mapsto h(t,w)$ is $\omega$--periodic for any $w \in \mathbb{R}^{n}$?

    \item[21.-] Let us consider a linear system $\dot{x}=A(t)x$ having exponential dichotomy on $[0,+\infty)$ with nontrivial projection $P$. Prove that if $t\mapsto A(t)$
    is bounded on $[0,+\infty)$ then it follows that
    \begin{displaymath}
    X(t)PX^{-1}(t)=I+\int_{0}^{+\infty}\mathcal{G}(t,s)A(s)\,ds,
    \end{displaymath}
where $X(t)$ is a fundamental matrix and $\mathcal{G}(t,s)$ is the Green's function introduced in Definition  \ref{Green2} restriced to $t,s\geq 0$.
\end{itemize}

\chapter{Exponential dichotomy on the  half lines}

\section{Preliminaries}
In this chapter we will consider the linear system
\begin{equation}
\label{LinealCap29}    
\dot{x}=A(t)x
\end{equation}
by focusing our study on the properties of exponential dichotomy on the half lines $\mathbb{R}_{0}^{-}$ or $\mathbb{R}_{0}^{+}$, which are more subtle compared with the case $J=\mathbb{R}$. For example the Proposition \ref{boundeness} from Chapter 2 and its consequences are no longer verified since any solution of the linear system (\ref{LinealCap29}) with exponential dichotomy on $\mathbb{R}_{0}^{+}$ and projection $P=I$ will be bounded on the half line. Nevertheless, as we will see, the restriction to a half line will allow alternative and sharper characterizations of the dichotomy property. 

\section{Basic properties and alternative characterizations}
As in the previous chapter, we can directly deduce by the definition of exponential dichotomy on $J$ that, if the linear system (\ref{LinealCap29}) has an exponential dichotomy on the interval $J$, then it also has this property on any interval $\tilde{J}\subset J$. In consequence, the property of exponential dichotomy on $\mathbb{R}_{0}^{+}$ implies the
dichotomy property on $[t_{0},+\infty)$ for any $t_{0}\geq 0$. The following result from W. Coppel \cite[p.13]{Cop} provides a converse statement
\begin{lemma}
\label{completar}
If the linear system \eqref{LinealCap29} has an exponential dichotomy on $[t_{0},+\infty)$ with $t_{0}>0$ then it also has an exponential dichotomy on $\mathbb{R}_{0}^{+}$ and preserve the same projection $P$
and constant $\alpha>0$.
\end{lemma}

\begin{proof}
Without loss of generality, we will consider a unitary norm. Then,
note that, given a unitary norm the number
$$
N= \exp\left(\int_{0}^{t_{0}}||A(u)||\,du\right)> 1 
$$
is well defined. Now, let $t$ and $s$ such that $0\leq s,t\leq t_{0}$ and 
notice that
\begin{equation}
\label{C2A}
||X(t)X^{-1}(s)||\leq N \quad \textnormal{for $0\leq s,t\leq t_{0}$}.
\end{equation}

We will verify (\ref{C2A}) when $s<t$ but the other case can be proved in a similar way. In fact, by using the inequality (\ref{BGO-I}) from Chapter 1 we have that
for any $s\leq t$ with $0,\leq s,t\leq t_{0}$ it follows that
$$
||X(t)X^{-1}(s)||\leq \exp\left(\int_{s}^{t}|A(u)|\,du\right) \leq \exp\left(\int_{0}^{t_{0}}|A(u)|\,du\right)=N
$$
and (\ref{C2A}) is verified.

By using (\ref{C2A}) combined with the property of exponential dichotomy on $[t_{0},+\infty)$ we can deduce the following estimations:

\noindent $\bullet$ When $0\leq s \leq t_{0} \leq t$, it follows that:
\begin{displaymath}
\begin{array}{rcl}
||X(t)PX^{-1}(s)|| & = & ||X(t)PX^{-1}(t_{0})X(t_{0})X^{-1}(s)|| \\\\
                   & \leq & ||X(t)PX^{-1}(t_{0})||\,||X(t_{0})X^{-1}(s)|| \\\\ 
                   & = & N ||X(t)PX^{-1}(t_{0})|| \\\\
                   & \leq & N Ke^{-\alpha(t-t_{0})}  \\\\
                   & \leq & N Ke^{\alpha t_{0}}e^{-\alpha(t-s)}.
\end{array}
\end{displaymath}

\noindent $\bullet$ When $0\leq s \leq t \leq t_{0}$, it follows that
\begin{displaymath}
\begin{array}{rcl}
||X(t)PX^{-1}(s)|| & = & ||X(t)PX^{-1}(t_{0})X(t_{0})X^{-1}(s)|| \\\\
                   & \leq & ||X(t)PX^{-1}(t_{0})||\,||X(t_{0})X^{-1}(s)|| \\\\ 
                   & = & N ||X(t)PX^{-1}(t_{0})|| \\\\
                   & = & N ||X(t)X^{-1}(t_{0})X(t_{0})PX^{-1}(t_{0})|| \\\\
                   & \leq  & N^{2} ||X(t_{0})PX^{-1}(t_{0})|| \\\\
                   & \leq & N^{2} K = N^{2} K e^{\alpha t_{0}}e^{-\alpha t_{0}}  \\\\
%                   & \leq & N^{2} K = N^{2} K e^{\alpha t_{0}}e^{-\alpha t_{0}}e^{\alpha s}\\\\
                   & \leq & N^{2}e^{\alpha t_{0}} Ke^{-\alpha(t-s)}.
\end{array}
\end{displaymath}

Let $K_{1}=\max\{N^{2}e^{\alpha t_{0}} K,N K\}$, then we have that
\begin{displaymath}
||X(t)PX^{-1}(s)||\leq K_{1}e^{-\alpha (t-s)} \quad \textnormal{for any $0\leq s \leq t$},
\end{displaymath}
while the other inequality can be proved similarly and the Lemma follows.
\end{proof}

The next result is symmetric to the previous one and its proof is given to the reader:
\begin{lemma}
\label{completar2}
If the linear system \eqref{LinealCap29} has an exponential dichotomy on $(-\infty,-t_{0}]$ with $t_{0}>0$ then it also has an exponential dichotomy on $\mathbb{R}_{0}^{-}$ and preserve the same projection $P$
and constant $\alpha>0$.
\end{lemma}

\subsection{An equation having exponential dichotomy on the semiaxes but not in $\mathbb{R}$}
The following example proposed by K.J. Palmer in \cite{Palmer2006}  provides a simple case of a linear system having an exponential dichotomy on the
two semiaxes $(-\infty,0]$ and $[0,+\infty)$ but not in $\mathbb{R}$: let us consider the scalar equation
\begin{equation}
\label{schned}
x'=a(t)x,
\end{equation}
where $a\colon \mathbb{R}\to \mathbb{R}$ is defined by
\begin{displaymath}
a(t)=\left\{\begin{array}{ccl}
1 &\textnormal{if} &  t\leq -T \\
\phi(t)  &\textnormal{if} & t\in [-T,T] \\
-1 &\textnormal{if} &  t\geq T, 
\end{array}\right.
\end{displaymath}
where $T>0$ and $\phi\colon [-T,T]\to \mathbb{R}$ is a continuous function verifying $\phi(-T)=1$ and $\phi(T)=-1$.

It is easy to see that (\ref{schned}) has an exponential dichotomy on $[T,+\infty)$ 
with $P=1$, $K=1$ and $\alpha=1$ and also has exponential dichotomy on $(-\infty,-T]$ with 
$P=0$, $K=1$ and $\alpha=1$. Now, by Lemmas \ref{completar}--\ref{completar2}
we can see that $x'=a(t)x$ has an exponential dichotomy on $(-\infty,0]$ and $[0,+\infty)$. 
Nevertheless, by using Corollary \ref{u-bounded} from Chapter 2 we have this equation has not an exponential dichotomy on $(-\infty,+\infty)$ because any solution is bounded on $\mathbb{R}$.

\subsection{Noncritical uniformity}
The property of noncritical uniformity has been introduced by N.N. Krasovskii
in \cite{Krasovski} as a necessary and sufficient condition for the existence
of Lyapunov functions of nonautonomous systems having an equilibrium at the  origin.
The following definition, tailored for linear systems, has been used without being explicitly mentioned  
by W.A. Coppel in \cite[p.14]{Coppel1} while K.J. Palmer proposes it explicitly in \cite[Def.4]{Palmer}:

\begin{definition}
The linear system \eqref{LinealCap29} is \textbf{uniformly
non\-critical} on $J\subseteq \mathbb{R}$ if there exists
$T>0$ and $\theta \in (0,1)$ such that any solution $t\mapsto x(t)$ of \eqref{LinealCap29} satisfies
\begin{equation}
\label{CEDO-palmer}
|x(t)| \leq \theta \sup\limits_{|u-t|\leq T}|x(u)| \quad \textnormal{for any $t$ such that $[t-T,t+T]\subset J$}.
\end{equation}
\end{definition}

It will be useful to revisit the above definition for specific intervals $J$. If $J=[0,+\infty)$, the
property (\ref{CEDO-palmer}) becomes
\begin{equation*}
%\label{CED0}
|x(t)| \leq \theta \sup\limits_{|u-t|\leq T}|x(u)| \quad \textnormal{for any $t\geq T$},
\end{equation*}
whereas if $J=(-\infty,0]$, the property (\ref{CEDO-palmer}) becomes
\begin{equation*}
%\label{CED0-palmer2}
|x(t)| \leq \theta \sup\limits_{|u-t|\leq T}|x(u)| \quad \textnormal{for any $t\leq -T$}.
\end{equation*}

The following result from Coppel \cite[pp.13--14]{Coppel1} states that
the noncritical uniformity is implied by the exponential dichotomy on $[0,+\infty)$:

\begin{theorem}
\label{EDINCU}
If the system \eqref{LinealC2} has an exponential dichotomy on $\mathbb{R}_{0}^{+}$ with nontrivial projector $P$
and constants $K\geq 1$ and $\alpha>0$, then 
it has the property of noncritically uniformity on $\mathbb{R}_{0}^{+}$.
\end{theorem}

\begin{proof}
Suppose that \eqref{LinealC2} has an exponential dichotomy on $[0,+\infty)$ with constants $K$ and  $\alpha$. For any solution $t\mapsto x(t)$ of (\ref{LinealC2}) let us define the auxiliary functions: 
    \begin{equation}
    \label{x1x2}
        x_1(t):= \Phi(t,s)P(s)x(s) \quad \textnormal{ and } \quad x_2(t):= \Phi(t,s)Q(s)x(s).
    \end{equation}
    Now, it can be easily verified that $x_1(t)$ and $x_2(t)$ satisfy the following properties:
%    \begin{enumerate}
%        \item[(i)] $x(t)=x_1(t)+x_2(t)$.
%        \item[(ii)] $x_1(s)=P(s)x(s)$ and $x_2(s)=Q(s)x(s)$.
%        \item[(iii)] $P(t)x_1(t)=x_1(t)$ and $Q(t)x_2(t)=x_2(t)$.
%    \end{enumerate}
\begin{subequations}
  \begin{empheq}{align}
  & x(t)=x_1(t)+x_2(t)  \label{i)}, \\
&   x_1(s)=P(s)x(s) \quad \textnormal{and} \quad x_2(s)=Q(s)x(s)  \label{ii)}, \\
& P(t)x_{1}(t)=x_{1}(t) \quad \textnormal{and} \quad Q(t)x_{2}(t)=x_{2}(t)  \label{iii)}. 
\end{empheq}
\end{subequations}
    
By the exponential dichotomy property, (\ref{x1x2}),(\ref{ii)}) and (\ref{iii)}) we can deduce the following estimations
when $t\geq s$: 
    \begin{equation*}
    \left\{\begin{array}{rcl}
       |x_1(t)| & = & |P(t)\Phi(t,s)x_1(s)|=|\Phi(t,s)P(s)x_1(s)|\leq Ke^{-\alpha(t-s)}|x_1(s)|, 
    \\\\
        |x_2(s)| &=& |\Phi(s,t)Q(t)\Phi(t,s)Q(s)x_2(s)|\leq Ke^{-\alpha(t-s)}|x_2(t)|.
    \end{array}\right.
    \end{equation*}

    Similarly, if $t\leq s$ we can deduce the estimations:
    \begin{equation*}
    \left\{\begin{array}{rcl}
       |x_2(t)| &=& |\Phi(t,s)Q(s)x(s)|=|\Phi(t,s)Q(s)x_2(s)|\leq Ke^{-\alpha(s-t)}|x_2(s)|,\\\\ 
       |x_1(s)| &=& |\Phi(s,t)P(t)\Phi(t,s)P(s)x_1(s)|\leq Ke^{-\alpha(s-t)}|x_1(t)|.
    \end{array}\right.   
    \end{equation*}
      \medskip 
    The above inequalities can be summarized as follows:
       \begin{equation*}
               |x_1(t)|\leq Ke^{-\alpha(t-s)}|x_1(s)| \quad \textnormal{and} \quad |x_2(t)|\geq K^{-1}e^{\alpha(t-s)}|x_2(s)| \,\,\, \textnormal{if} \,\, t\geq s
       \end{equation*}
        and 
       \begin{equation*}
           |x_2(t)|\leq Ke^{-\alpha(s-t)}|x_2(s)| \quad \textnormal{and}\quad |x_1(t)|\geq K^{-1}e^{\alpha(s-t)}|x_1(s)| \,\,\, \textnormal{if}\,\, t\leq s.
       \end{equation*}
       
       Now, we will assume that $|x_2(s)|\geq |x_1(s)|$, then if $t\geq s$,  the above inequalities and (\ref{i)}) imply that

       \begin{align*}
           |x(t)| &\geq |x_2(t)| - |x_1(t)|\\
           &\geq K^{-1}e^{\alpha(t-s)}|x_2(s)| - Ke^{-\alpha(t-s)}|x_1(s)|\\
           &\geq \left\{K^{-1}e^{\alpha(t-s)} - Ke^{-\alpha(t-s)}\right\}|x_2(s)|. 
       \end{align*}

       Similarly, if $|x_2(s)|<|x_1(s)|$, then for $s\geq t$,
       
       \begin{align*}
        |x(t)| &\geq |x_1(t)| - |x_2(t)|\\
        &\geq K^{-1}e^{\alpha(s-t)}|x_1(s)| - Ke^{-\alpha(s-t)}|x_2(s)|\\
        &\geq \left\{K^{-1}e^{\alpha(s-t)} - Ke^{-\alpha(s-t)}\right\}|x_1(s)|.
       \end{align*}
       
       Defining the auxiliary function
       \begin{equation*}
           \Psi(t):= K^{-1}e^{\alpha  t} - K e^{-\alpha t},
       \end{equation*}
    it follows that
    \begin{equation}\label{In-psi1}
        |x(t)|\geq \Psi(t-s)|x_2(s)| \quad \textnormal{ if } t\geq s,
    \end{equation}
    and
    \begin{equation}\label{In-psi2}
        |x(t)|\geq \Psi(s-t)|x_1(s)| \quad \textnormal{ if } t\leq s.
    \end{equation}
    For the function $\Psi$, we can deduce that
    \begin{itemize}
        \item $\Psi(0) =\dfrac{1}{K} - K \leq 0$
        \item Notice that:
    \begin{equation*}
        \Psi'(t) = K^{-1}\alpha e^{\alpha t} + K\alpha e^{-\alpha t}>0 
    \end{equation*}
    \item $\Psi(t_0)=0$, where $t_0=\frac{1}{\alpha}\ln(K) \geq 0$.
    \end{itemize}
For any $\theta\in (0,1)$ we can choose $T>\frac{1}{\alpha}\ln(K)$ such that $\Psi(T)\geq\dfrac{2}{\theta}$. Then if $t-s\geq T>0$, from \eqref{In-psi1}, we have
\begin{align*}
    |x(t)|>\Psi(T)|x_2(s)|\geq\dfrac{2}{\theta}|x_2(s)|
\end{align*}
and since that $|x_1(s)|\leq |x_2(s)|$, we get
\begin{equation*}
    |x(s)|\leq |x_1(s)| + |x_2(s)| \leq 2|x_2(s)|\leq \theta|x(t)|.
\end{equation*}
In a similar way, if $s-t \geq T>0$, from \eqref{In-psi2} and the inequality $|x_2(s)|<|x_1(s)|$, it is shown that
\begin{equation*}
    |x(t)|>\dfrac{2}{\theta}|x_1(s)|>\dfrac{1}{\theta}|x(s)|.
\end{equation*}
Thus, we can conclude that
\begin{equation*}
    |x(s)|<\theta|x(t)|, \quad \textnormal{ if } \quad |s-t|\geq T.
\end{equation*}
Next, for a fixed $s\geq T$, we choose  $t=T+s\geq s$, then
\begin{align}
    |x(s)| &\leq \theta|x(T+ s)|\nonumber\\
    &\leq \theta\sup\limits_{s\leq u\leq s+T}|x(u)|=\theta\sup\limits_{0\leq u-s\leq T}|x(u)|.\label{NCmayor}
\end{align}
For a fixed $s\geq T$, we now choose $t=s-T<s$, then 
\begin{align}
    |x(s)|&\leq \theta|x(s-T)|\nonumber\\
    &\leq \theta\sup\limits_{s-T\leq u\leq s}|x(u)| = \theta \sup\{|x(u)| : -T\leq u-s\leq 0\}.\label{NCmenor}
\end{align}
Finally, \eqref{NCmayor} and \eqref{NCmenor} imply that
\begin{equation*}
    |x(s)|\leq \theta\sup\limits_{|u-s|\leq T}|x(u)| \quad \textnormal{ for } s\geq T,
\end{equation*}
which concludes the proof.
\end{proof}

A converse result can be obtained  provided the additional property of bounded 
growth stated in Chapter 1.

\begin{theorem}
\label{NCU+BG}
If the linear system \eqref{LinealCap29} has the properties of bounded growth and noncritical uniformity
on $[0,+\infty)$ then it also has an exponential dichotomy on $[0,+\infty)$.
\end{theorem}

As the linear system (\ref{LinealCap29}) is noncritical uniform on $[0,+\infty)$, there exists constants $T>0$ and $0<\theta<1$ such that every solution $t\mapsto x(t)$ satisfy the property
\begin{equation}
\label{UNCT1}
|x(t)| \leq \theta \sup\limits_{|u-t|\leq T}|x(u)| \quad \textnormal{for every $t\geq T$}.   
\end{equation}

In addition, by considering the above constant $T>0$, the bounded growth property implies the existence of $C>1$ such that every solution $t\mapsto x(t)$ of the linear system \eqref{LinealCap29} satisfy the property 
\begin{equation}
\label{BGT1}
|x(t)|\leq C|x(s)| \quad \textnormal{for $0\leq s \leq t \leq s+T$}.
\end{equation}

The proof will be consequence of several Lemmas. Firstly, notice that any nontrivial solution $t\mapsto x(t)$ of (\ref{LinealCap29}) is either
bounded or unbounded on $[0,+\infty)$. The idea is to prove that bounded and unbounded solutions
has the same qualitative properties that the functions described in the splitting of solution given in (\ref{splitting}) from the previous chapter.

\begin{lemma}\label{Bound-sol}
Under the assumptions of Theorem \ref{NCU+BG}, there exists $K\geq 1$ and $\alpha>0$ such that any nontrivial solution $t\mapsto x(t)$ which is bounded on $[0,+\infty)$ verifies:
\begin{equation}
\label{contra}
        |x(t)|    \leq  \displaystyle  Ke^{-\alpha(t-s)}|x(s)| \quad \textnormal{for any $t\geq s \geq 0$}.
    \end{equation}
\end{lemma}    

\begin{proof}

\noindent \textit{Step 1: A first estimation.}  As (\ref{BGT1}) provides information about the solution $x(\cdot)$ on $[s,s+T]$, we will study its properties when $t\geq s+T$. In order to do that let us define the map $\mu\colon [0,+\infty)\to [0,+\infty)$ as follows: 
    \begin{equation*}
        \mu(s):=\sup\limits_{u\geq s}|x(u)|.
    \end{equation*}

Notice that $t\geq s+T$ is equivalent to
$t-T\geq s$ and also implies that $t\geq T$ since $s\geq 0$. These facts combined with (\ref{UNCT1}) imply that:  
\begin{equation}
\label{e1}
    \begin{array}{rl}
        |x(t)| &\leq \theta\sup\limits_{|u-t|\leq T}|x(u)|\\
        %&= \theta \sup\limits_{t-T\leq u \leq t+T}|x(u)|\\
        &= \theta \sup\{|x(u)| : t-T\leq u \leq t+T\}\\
        &\leq  \theta \sup\limits_{s\leq u \leq t+T}|x(u)|\\
        &\leq \theta \sup\limits_{u\geq s}|x(u)|=\theta \mu(s).
    \end{array}
\end{equation}    

    The above estimation implies that 
    \begin{equation}
    \label{e0}
        \mu(s)=\sup\limits_{s\leq u\leq s+T}|x(u)|. 
    \end{equation}
    
    Indeed, by using $\sup(A\cup B)=\max\{\sup(A),\sup(B)\}$ and the estimation (\ref{e1}),
    we can deduce that
    \begin{displaymath}
        \mu(s) =\max\left\{\sup\limits_{s\leq u\leq s+T}|x(u)|,\sup\limits_{s+T \leq u}|x(u)|\right\}\leq  \max\left\{\sup\limits_{s\leq u\leq s+T}|x(u)|,\theta \mu(s)\right\},
    \end{displaymath}
which leads to 
$$
\mu(s)\leq \sup\limits_{s\leq u\leq s+T}|x(u)|,
$$
whereas the inverse inequality is trivial and (\ref{e0}) follows.

    Therefore, by using (\ref{BGT1}),(\ref{e1}),(\ref{e0}) and recalling that $t\geq s+T$, we deduce that
    \begin{align*}
        |x(t)| &\leq \theta \mu(s)\\
        &= \theta \sup\limits_{s\leq u\leq s+T}|x(u)|\\
        &\leq \theta C |x(s)|,
    \end{align*}
    which combined with (\ref{BGT1}) implies
    \begin{equation*}
        |x(t)| \leq C|x(s)| \quad \mbox{ for } \quad 0\leq s\leq t. 
    \end{equation*}

\noindent \textit{Step 2: A useful inequality.} At this step, we will assume that
 $t\geq s+ nT$ for some $n\in \mathbb{N}$. In addition,  by recalling that $nT$ is strictly increasing, we have that $t\geq T$ and $t-T\geq s+(n-1)T$. These inequalities and (\ref{UNCT1}) imply: 
  \begin{equation}
  \label{e2}
    \begin{array}{rl}
        |x(t)| &\leq \theta \sup\limits_{|u-t|\leq T}|x(u)|\\
        &\leq \theta \sup\limits_{u\geq s+(n-1)T}|x(u)|=\theta \mu(s+(n-1)T).
    \end{array}
   \end{equation} 
    
As $t\geq s+nT$ and noticing that $n\in \mathbb{N}$ is arbitrary, the above estimation implies that 
    \begin{equation}
    \label{e3}
        \mu(s+nT) \leq \theta \mu(s+(n-1)T).
    \end{equation}

\noindent \textit{Step 3: $t\mapsto |x(t)|$ is an exponential contraction.}

At this step we will assume that $t\geq s$. Notice there always exists some $n\in \mathbb{N}$ such that $s+nT\leq t\leq s+(n+1)T$, then the estimations (\ref{e2})--(\ref{e3}) allow us to deduce that:
    \begin{displaymath}
        |x(t)|  \leq \theta \mu(s+nT)\leq \theta^2 \mu(s+(n-2)T)\leq
        \cdots \leq \theta^n\mu(s),
    \end{displaymath}
and, by using (\ref{BGT1}) and (\ref{e0}), we have that

\begin{equation}  
\label{e5}
|x(t)|\leq  \theta^n\sup\limits_{s\leq u\leq s+T}|x(u)|\leq \theta^n C|x(s)|.
    \end{equation}
    
    On the other hand, as $nT\leq t-s\leq (n+1)T$, we can see that
    \begin{equation*}
      \displaystyle  \theta^{n+1}\leq \theta^{\frac{t-s}{T}}\leq \theta^n.
    \end{equation*}

    The above estimation together with (\ref{e5}) implies:
    \begin{align*}
        |x(t)| &\leq \theta^{-1}C \theta^{n+1}|x(s)|\\
        &\leq \theta^{-1}C e^{\frac{\ln(\theta)}{T}(t-s)}|x(s)|.
    \end{align*}
    
    Taking $K=\theta^{-1}C\geq1$ and $\alpha=-\ln(\theta)/T>0$ since $\theta \in (0,1)$ and $T>0$, we obtain (\ref{contra}) and the Lemma follows.
\end{proof}

\begin{lemma}\label{Unbound-sol}
Under the assumptions of Theorem \ref{NCU+BG}, there exists $K\geq 1$, $\alpha>0$ and $T_{1}>0$ such that any solution $t\mapsto x(t)$
unbounded on $[0,+\infty)$ verifies:
 \begin{equation}
 \label{expa}
        |x(t)|\leq Ke^{-\alpha(s-t)}|x(s)| \quad \mbox{ for } \quad s\geq t\geq T_{1}>0.
        \end{equation}
\end{lemma}

\begin{proof}

\noindent \textit{Step 1: Preliminaries.}
    Let $t\mapsto x(t)$ be an unbounded solution with $|x(0)|=1$. As $x(\cdot)$ is also continuous, we can take a sequence $\{t_n\}$ with $t_n>0$ such that  
    \begin{equation}
    \label{QPS}
        |x(t_n)|=\theta^{-n}C \quad \textnormal{and} \quad |x(t)|<\theta^{-n}C \quad \mbox{ for } \quad 0\leq t< t_n,
    \end{equation}
which implies that the sequence is strictly increasing and upper unbounded. On the other hand, by the uniform bounded growth (\ref{BGT1}) with $s=0$ we have that:
    \begin{equation*}
        |x(t)|\leq C|x(0)|=C<\theta^{-1}C \quad \mbox{ for } \quad 0\leq t\leq T,
    \end{equation*}
    which allow us to see that $T<t_{1}$. 

\noindent \textit{Step 2:  $t_{n+1}\leq t_n+T$ for any $n\in \mathbb{N}$.}
    Otherwise, there exists $n_0 \in \mathbb{N}$ such that $t_{n_{0}+1}>t_{n_0}+T$, this fact combined with recalling $t_{n_0}>T$ and (\ref{UNCT1}) leads to:
    \begin{equation*}
        |x(t_{n_0})|\leq \theta\sup\limits_{|u-t_{n_0}|\leq T}|x(u)|\leq \theta\sup\limits_{0\leq u\leq t_{n_0}+ T}|x(u)|\leq  \theta\sup\limits_{0\leq u\leq t_{n_0+1}}|x(u)|.
    \end{equation*}

    In addition, a direct consequence from (\ref{QPS}) is that
    \begin{displaymath}
     \sup\limits_{0\leq u\leq t_{n_0+1}}|x(u)|=C\theta^{-(n_{0}+1)}=\theta^{-1}|x(t_{n_{0}})|.   
    \end{displaymath}

    By gathering the above estimations, we obtain that
    \begin{equation*}
        |x(t_{n_0})|\leq \theta\sup\limits_{0\leq u\leq t_{n_0+1}}|x(u)|\leq |x(t_{n_0})|,
    \end{equation*}
    which is a contradiction. 

    \medskip
    
\noindent \textit{Step 3: End of proof.}    
    Suppose that $0\leq t\leq s$ and $t_m\leq t<t_{m+1}$, $t_n\leq s< t_{n+1}$ with $1\leq m<n$. By (\ref{QPS})
    and the qualitative properties of $\{t_{n}\}_{n}$  we have that:
   \begin{equation}
   \label{SEF}
    \begin{array}{rl}
        |x(t)| &\leq \theta^{-m-1}C\\
        &=\theta^{n-m}|x(t_{n+1})|\\
        &\leq \theta^{-1}C \theta^{n-m+1}|x(s)|,
    \end{array}
   \end{equation} 
where the last estimation follows from $t_{n+1}\leq t_{n}+T$ combined with (\ref{BGT1}).

    In order to obtain a better estimation of (\ref{SEF}), we have that
    \begin{equation*}
        s-t\leq t_{n+1}-t_m\leq (t_n+T)-t_m = (t_n-t_m)+T.
    \end{equation*}
    Moreover,
    \begin{equation*}
        t_n\leq t_{n-1}+ T\leq t_{n-2}+2T\leq \cdots \leq t_m+(n-m)T,
    \end{equation*}
    which implies that 
    \begin{equation*}
        e^{s-t} 
 \leq e^{(n-m+1)T}=(e^{T})^{n-m+1}
    \end{equation*}
    and so $(s-t)/T\leq n-m+1$, which leads to:
    \begin{equation*}
        \theta^{n-m+1}\leq \theta^{\frac{s-t}{T}}.
    \end{equation*}
    Next, upon inserting the above estimation in (\ref{SEF}), we get that
    \begin{align*}
        |x(t)| &\leq \theta^{-1}C_T\theta^{\frac{s-t}{T}}|x(s)|\\
        &= \theta^{-1}C_Te^{(s-t)\frac{\ln(\theta)}{T}}|x(s)|,
    \end{align*}
then (\ref{expa}) is obtained with $K=\theta^{-1}C\geq1$, $\alpha=-\ln(\theta)/T>0$
    and $T_{1}=t_{1}$.
\end{proof}

\subsection{Uniform exponential dichotomy and uniform noncriticality: Proof of Theorem \ref{NCU+BG}} Let us consider the subspace: 
$$
V_1:=\left\{\xi\in\mathbb{R}^n : 
 \sup\limits_{t\geq 0} |X(t,0)\xi|<+\infty\right\}
 $$ 
 and let $V_2$ be a subspace of $\mathbb{R}^n$ such that $\mathbb{R}^n=V_1\oplus V_2$. For any $\xi\in V_2$, with $|\xi|=1$, let $t\mapsto x(t)=x(t,\xi)$ be the solution of \eqref{LinealCap29} with initial condition $x(0)=\xi$. Then $t\mapsto x(t,\xi)$ is unbounded and so there exists a value $t_1:=t_1(\xi)$ such that 
 $$
|x(t,\xi)|<\theta^{-1}C \quad \textnormal{for any $0\leq t<t_{1}(\xi)$} \quad \textnormal{and} \quad |x(t_1,\xi)|=\theta^{-1}C.
 $$
 
 We will show that the set $B:=\{t_1(\xi): \xi\in V_2 \quad \textnormal{ and } \quad |\xi|=1\}$ is bounded. In fact, if $B$ is unbounded, there is a sequence $\{\xi_{\nu}\}_{\nu\in\mathbb{N}}\subset V_2$ with $|\xi_{\nu}|=1$ and $t_1^{(\nu)}=t_1(\xi_{\nu})\to +\infty$ as $\nu\to +\infty$. By the compactness of the unit sphere in $V_2$, we may assume that $\xi_{\nu}\to \xi_0$ as $\nu\to +\infty$, for some unit vector $\xi_0\in V_2$. Thus, by Proposition \ref{CDCI} from Chapter 1, we have that the uniform bounded growth property implies that
 \begin{equation*}
     x(t,\xi_{\nu}) \to x(t,\xi) \quad \textnormal{ as } \quad \nu\to +\infty,
 \end{equation*}
 for all $t\geq 0$. Since $|x(t,\xi_{\nu})|<\theta^{-1}C$ for $0\leq t< t_1^{(\nu)}$, it follows that
 \begin{equation*}
     |x(t,\xi_0)|\leq \theta^{-1}C \quad \textnormal{ for } \quad 0\leq t<+\infty,
 \end{equation*}
 which is a contradiction  with the fact that $\xi_0\in V_2$. Then there is $T_1>0$ such that $t_1(\xi)\leq T_1$ for all unit vector $\xi\in V_2$. Next, for all $\xi\in V_2$, with $\xi\neq 0$, Lemma \ref{Unbound-sol} implies that 
 \begin{equation*}
     |x(t,\xi)|\leq|\xi|Ke^{-\alpha(s-t)}|x(s,\frac{\xi}{|\xi|})|=Ke^{-\alpha(s-t)}|x(s,\xi)|,
 \end{equation*}
 for $s\geq t\geq T_1>0$.
 
 Now, let $P$ be the projection from the split $\mathbb{R}^{n}=V_1\oplus V_2$ on the subspace $V_1$. Then for each $\xi\in \mathbb{R}^n$, by Lemmas  \ref{Bound-sol} and \ref{Unbound-sol}, we have that 
 \begin{equation*}
     |X(t)P\xi|\leq Ke^{-\alpha(t-s)}|X(s)P\xi| \quad \textnormal{ for } \quad t\geq s\geq T_1 
 \end{equation*}
 and
 \begin{equation*}
     |X(t)(I-P)\xi|\leq Ke^{-\alpha(s-t)}|X(s)(I-P)\xi| \quad \textnormal{ for } \quad s\geq t\geq T_1.
 \end{equation*}
 The uniform bounded growth property together with Lemma \ref{C12-C3} imply that there exists a positive constant $M$ such that 
 \begin{equation*}
     ||X(t)PX^{-1}(t)||\leq M \quad \textnormal{for any $t\geq T_{1}$}.
 \end{equation*}
 Hence, by Proposition \ref{triada} from Chapter 2 we have that \eqref{LinealCap29} has a uniform exponential dichotomy on the interval $[T_1,+\infty)$ and by Lemma \ref{completar}, it has a uniform exponential dichotomy on $[0, +\infty)$.

\section{Bounded solutions on the half lines and exponential dichotomy}
As we shown in an example at the beginning of this chapter, if the linear system (\ref{LinealCap29}) has simultaneously an exponential dichotomy on $\mathbb{R}_{0}^{-}$ and $\mathbb{R}_{0}^{+}$ this not always implies the existence of an exponential dichotomy on $\mathbb{R}$. Now, in this section we will provide a necessary and sufficient condition ensuring the exponential dichotomy on $\mathbb{R}$. In this context, we recall the following definition. 
\begin{definition}
 The \textit{index} of a $n$--dimensional linear system $\dot{x}=A(t)x$ with exponential dichotomies both on $\mathbb{R}_{0}^{+}$ and $\mathbb{R}_{0}^{-}$ with constant projections $P_{+}$ and $P_{-}$ respectively, is given by
 $$
 i(A)=\textnormal{\text{dim}\,\text{Im}}P_{+}+\textnormal{\text{dim}}\,\ker P_{-}-n.
 $$
\end{definition}

Moreover, we will consider the linear system (\ref{LinealCap29}) with fundamental matrix $X(t)$ and let us redefine the subspaces 
\begin{subequations}
  \begin{empheq}{align}
    &\mathcal{V}:=\left\{\xi \in \mathbb{R}^{n}\colon \sup\limits_{t\geq 0}|X(t,0)\xi|<+\infty\right\}, \label{EVS}\\
    & 
\mathcal{W}:=\left\{\xi \in \mathbb{R}^{n}\colon \sup\limits_{t\leq 0}|X(t,0)\xi|<+\infty\right\}. \label{EVI}
  \end{empheq}
\end{subequations}

The next two results consider different projections but its proof is
similar to the Lemma \ref{subpespacios-t_0} studied in the previous chapter.
\begin{lemma}
\label{BSPHL}
If the linear system \eqref{LinealCap29} has an exponential dichotomy on $[0,+\infty)$ 
with projector $P_{+}$ and constants $(K,\alpha)$, then
it follows that $\mathcal{V}=\textnormal{\text{Im}}P_{+}$.
\end{lemma}

\begin{lemma}
\label{BSNHL}
If the linear system \eqref{LinealCap29} has an exponential dichotomy on $(-\infty,0]$ 
with projector $P_{-}$ and constants $(K,\alpha)$, then
it follows that $\mathcal{W}=\ker P_{-}$.
\end{lemma}

The following result has been stated in \cite[Lemma 1]{Palmer2006}.

\begin{lemma}
\label{Palmeriano}
Assume that the linear system \eqref{LinealCap29} has exponential dichotomies both on $\mathbb{R}_{0}^{+}$ and $\mathbb{R}_{0}^{-}$ with projections $P_{+}$ and $P_{-}$ respectively, then the following properties are equivalent:
\begin{itemize}
\item[i)] The unique bounded solution of the lineal system \eqref{LinealCap29} is the trivial one,
\item[ii)] The projectors $P_{+}$ and $P_{-}$ satisfy the identity
\begin{equation*}
%\label{PQ=OP}
P_{+}P_{-}=P_{-}P_{+}=P_{+}.
\end{equation*}
\end{itemize}
\end{lemma}

\begin{proof}
By using lemmas (\ref{BSPHL})--(\ref{BSNHL}) it follows that:
\begin{displaymath}
\text{Im} P_{+}=\mathcal{V}\quad \textnormal{and} \quad \ker P_{-}=\mathcal{W}.
\end{displaymath}

If we assume the statement i), it can be straightforwardly proved that $\ker P_{-}\cap \text{Im} P_{+}=\{0\}$. Now, for any nonzero vector $\xi\in \ker P_{-}$ it follows that $\xi$ must be
in a complementary subspace of $\text{Im} P_{+}$, which leads to 
\begin{equation}
\label{contenzion}
\ker P_{-} \subseteq \ker P_{+},
\end{equation}
this allow us to deduce that $\ker P_{-}\oplus \text{Im} P_{+} \subset \ker P_{+}\oplus \text{Im} P_{+}=\mathbb{R}^{n}$. We can deduce the existence of a complementary subspace $V$ such that:
$$
\ker P_{-}\oplus V \oplus \text{Im} P_{+}=\mathbb{R}^{n},
$$
then $V \oplus \text{Im} P_{+}=\text{Im} P_{-}$ which leads to
\begin{equation}
\label{contenzion2}
 \text{Im} P_{+} \subseteq \text{Im} P_{-}.
\end{equation}

For any $\xi \in \mathbb{R}^{n}$ it follows that $P_{+}\xi \in \text{Im} P_{+}\subset \text{Im} P_{-}$, which implies that $P_{-}P_{+}\xi=P_{+}\xi$ and the identity $P_{-}P_{+}=P_{+}$ is verified.

Now, for any $\xi\in \mathbb{R}^{n}$ we have that $\xi=\eta+\lambda$ with
$\eta \in \ker P_{-}$ and $\lambda \in \text{Im} P_{-}$. We have that $P_{-}\xi=\lambda$
and, as $\ker P_{-}\subset \ker P_{+}$, we also have that $P_{+}\eta=0$ which allow us to conclude that $P_{+}P_{-}\xi = P_{+}\lambda =P_{+}(\eta+\lambda)=P_{+}\xi$, then $P_{+}P_{-}=P_{+}$
and the statement (ii) follows.

The converse implication is left as an exercise for the reader, we also refer to \cite[p.S176]{Palmer2006}.

\end{proof}

\begin{proposition}
\label{FullDico}
The linear system $\dot{x}=A(t)x$ has an exponential dichotomy on $\mathbb{R}$
if and only if: 
\begin{itemize}
\item[a)] The system has exponential dichotomies on $\mathbb{R}_{0}^{+}$ and $\mathbb{R}_{0}^{-}$
with projections $P_{+}$ and $P_{-}$ respectively,
\item[b)] The system has no nontrivial bounded solutions, 
\item[c)] The index of the linear system is $i(A)=0$.
\end{itemize}
 \end{proposition}

\begin{proof}
Firstly, if the linear system has an exponential dichotomy on $\mathbb{R}$ with projector $P$,
the properties a) and c) are verified straightforwardly with $P=P_{-}=P_{+}$. In addition,
the property b) follows from the Corollary \ref{u-bounded} from Chapter 2.

Secondly, let us assume that the properties a),b) and c) are satisfied. Now, by Lemma \ref{Palmeriano} we can prove that the contentions (\ref{contenzion})--(\ref{contenzion2}) cannot be strict, namely:
\begin{equation*}
%\label{esgalite-sub}
\ker P_{-} = \ker P_{+} \quad \textnormal{and} \quad  \text{Im} P_{+} = \text{Im} P_{-}.
\end{equation*}

Indeed, otherwise, if $\ker P_{-} \subset \ker P_{+}$ we will have that
$\text{dim} \ker P_{-}<\text{dim} \ker P_{+}$, which combined with $i(A)=0$ leads to
\begin{displaymath}
\begin{array}{rcl}
n &=& \text{dim} \ker P_{+} + \text{dim} \text{Im} P_{+} \\
&>&  \text{dim} \ker P_{-} + \text{dim} \text{Im} P_{+}=n,
\end{array}
\end{displaymath}
obtaining a contradiction. The identity $\text{Im} P_{+} = \text{Im} P_{-}$ can be proved in
a similar way.

Let $\xi\in \mathbb{R}^{n}$. As $P_{-}\xi\in \text{Im} P_{-}=\text{Im} P_{+}$, it follows that
$P_{+}P_{-}\xi=P_{-}\xi$. On the other hand, by using the statement ii) of Lemma \ref{Palmeriano}
we also have that $P_{+}P_{-}\xi=P_{+}\xi$, which implies that
$$
P_{-}\xi = P_{+}\xi \quad \textnormal{for any $\xi \in \mathbb{R}^{n}$},
$$
and the identity $P_{+}=P_{-}$ follows, this implies that the linear system has
an exponential dichotomy on $\mathbb{R}$ with projector $P=P_{+}=P_{-}$.

\end{proof}

\section{Backing to the projector problem}

In Remark \ref{UNIpro} we called as the \textit{projector problem} the following question: if a linear system (\ref{LinealC2}) with a fundamental matrix $X(t)$ has an exponential dichotomy on $J$ with projector $P$, is this
projector unique? 

The results of the previous section will allow to provide an answer considering $J=\mathbb{R}_{0}^{-}$, $J=\mathbb{R}_{0}^{+}$,
and $J=\mathbb{R}$. In the case of the half axes we will provide sufficient conditions ensuring the existence of other projectors whereas the uniqueness is proved for the case $J=\mathbb{R}$.

If a linear system (\ref{LinealC2}) has an exponential dichotomy on $\mathbb{R}_{0}^{-}$ with projector $P$ we 
know, as we have seen in Remark \ref{ext-split} from Chapter 2, that any solution $x(t)=\Phi(t,0)\xi$ of (\ref{LinealC2}) can be splitted 
between an exponential
expansion $t\mapsto \Phi(t,0)P\xi$ and an exponential contraction $t\mapsto \Phi(t,0)[I-P]\xi$ when $t\to -\infty$. Now,
if there exists another projector $\overline{P}$ such that $\ker P=\ker \overline{P}$, the Lemma \ref{BSNHL}
states that
$$
\ker \overline{P}=\mathcal{W}=\left\{\xi \in \mathbb{R}^{n}\colon \sup\limits_{t\leq 0}|X(t,0)\xi|<+\infty\right\}.
$$

As the solutions of (\ref{LinealC2}) also can be splitted as $t\mapsto x(t)=\Phi(t,0)\xi=\Phi(t,0)\overline{P}\xi+\Phi(t,0)[I-\overline{P}]\xi$ where $t\mapsto \Phi(t,0)[I-\overline{P}]\xi$ is an exponential contraction when $t\to -\infty$. A natural question is to determine what are additional assumptions on $\overline{P}$ in order to ensure the existence
an exponential dichotomy on $\mathbb{R}_{0}^{-}$ with projector $\overline{P}$. The following result provides an answer.

\begin{lemma}
\label{Unip-}
If the linear system \eqref{LinealCap29} has an exponential dichotomy on $\mathbb{R}_{0}^{-}$ with projection $P$ and there exist another projection $\overline{P}$ satisfying the properties 
\begin{itemize}
\item[i)] There exists $\tilde{K}>0$ such that 
\begin{equation}
\label{ProProj1}
\tilde{K}=\sup_{t\in (-\infty,0]}\Vert X(t)(I-\overline{P})X^{-1}(t)\Vert,
\end{equation}
\item[ii)] $\ker P=\ker \overline{P}$,
\end{itemize}
then the system \eqref{LinealCap29} has also an exponential dichotomy on $\mathbb{R}_{0}^{-}$ with projection $\overline{P}$.
\end{lemma}
\begin{proof}
The property $\ker P= \ker \overline{P}$ implies the following identities whose proof is left an exercise: 
$$
(I-\overline{P})=(I-P)(I-\overline{P}) \text{ and } \overline{P}=(I-P+\overline{P})P.
$$

By using the above characterization of $(I-\overline{P})$ and considering $t\leq s\leq 0$, we have that 
\begin{align*}
\Vert X(t)(I-\overline{P})X^{-1}(s)\Vert &= \Vert X(t)(I-P)(I-\overline{P})X^{-1}(s)\Vert\\
&= \Vert X(t)(I-P)X^{-1}(s)X(s)(I-\overline{P})X^{-1}(s)\Vert\\
&\leq \Vert X(t)(I-P)X^{-1}(s)\Vert \Vert X(s)(I-\overline{P})X^{-1}(s)\Vert\\
&\leq Ke^{-\alpha (s-t)}\Vert X(s)(I-\overline{P})X^{-1}(s)\Vert \\
&\leq K \tilde{K} e^{-\alpha (s-t)}.
\end{align*}

By using the above characterization of $\overline{P}$ and considering $s\leq t\leq 0$, we have that 
\begin{align*}
\Vert X(t)\overline{P}X^{-1} (s) \Vert &= \Vert X(t)(I-P+\overline{P})P X^{-1}(s)\Vert\\
&= \Vert X(t)(I-P+\overline{P})X^{-1}(t)X(t)PX^{-1}(s)\Vert\\ 
&\leq \Vert X(t)(I-P+\overline{P})X^{-1}(t)\Vert \Vert X(t)PX^{-1}(s)\Vert\\ 
&\leq Ke^{-\alpha (t-s)}\Vert X(t)(I-P+\overline{P})X^{-1}(t)\Vert\\
&\leq Ke^{-\alpha (t-s)}\{\Vert X(t)(I-P)X^{-1}(t)\Vert + \Vert X(t)\overline{P}X^{-1}(t)||\} \\
&\leq K(K+K\tilde{K}+||I||)e^{-\alpha (t-s)}
\end{align*}
where the last estimation is obtained by using
$$
||X(t)\overline{P}X^{-1}(t)||-||I||\leq ||X(t)(I-\overline{P})X^{-1}(t)||\leq K \tilde{K}
$$
and the result follows.
\end{proof}

The above result has a symmetric version when considering the exponential dichotomy on $\mathbb{R}_{0}^{+}$.

\begin{lemma}
\label{Unip+}
If the linear system \eqref{LinealCap29} has an exponential dichotomy on $\mathbb{R}_{0}^{+}$ with projection $P$ and there exist another projection $P'$ satisfying the properties 
\begin{itemize}
\item[i)] There exists $K'>0$ such that 
\begin{equation}
\label{ProProj2}
K'=\sup_{t\in [0,+\infty)}\Vert X(t)P' X^{-1}(t)\Vert,
\end{equation}
\item[ii)] $\textnormal{Im}\,P=\textnormal{Im}\,P'$,
\end{itemize}
then the system \eqref{LinealCap29} has also an exponential dichotomy on $\mathbb{R}_{0}^{+}$ with projection $P'$.
\end{lemma}
\begin{proof}
We can see that the property $\textnormal{Im}\, P=\textnormal{Im}\,P'$ implies the identities
\begin{equation}
\label{identites-projecteur}
P'P=P \quad \textnormal{and} \quad PP'=P', 
\end{equation}
which allow us to deduce:
\begin{displaymath}
P-P'=P(P-P')=(P-P')(I-P)  \quad \textnormal{and} \quad
I-P'=[I+P-P'](I-P).
\end{displaymath}

By using $P'=PP'$ and considering $0\leq s\leq t$, 
we have that 
\begin{align*}
\Vert X(t) P' X^{-1} (s) \Vert &= \Vert X(t)PP'X^{-1}(s)\Vert\\
&= \Vert X(t)PX^{-1}(s)X(s)P'X^{-1}(s)\Vert\\ 
&\leq \Vert X(t)PX^{-1}(s)\vert\vert \vert\vert X(s)P'X^{-1}(s)\Vert\\ 
&\leq KK'e^{-\alpha(t-s)}.
\end{align*}

By using the above characterization of $(I-P')$ and considering $0\leq t\leq s$, we have that 
\begin{align*}
\Vert X(t)(I-P')X^{-1}(s)\Vert &= \Vert X(t)(I+P-P')(I-P)X^{-1}(s)\Vert\\
&= \Vert X(t)(I+P-P')X^{-1}(t)X(t)(I-P)X^{-1}(s)\Vert\\
&\leq \Vert X(t)(I+P-P')X^{-1}(t)\vert \vert \vert X(t)(I-P)X^{-1}(s)\Vert\\
&\leq Ke^{-\alpha(s-t)}\Vert X(t)(I-P'+P)X^{-1}(t)\vert\vert \\
&\leq Ke^{-\alpha(s-t)}\left(\Vert X(t)[I-P']X^{-1}(t)\vert\vert+\vert\vert X(t)PX^{-1}(t)\vert\vert\right) \\
&\leq  K(||I||+K+KK')e^{-\alpha(s-t)},
\end{align*}
where the last estimation is a consequence of
\begin{displaymath}
 ||X(t) P' X^{-1}(t)||-||I|| \leq ||X(t)[I-P']X^{-1}(t)||\leq KK'   
\end{displaymath}
and the result follows.
\end{proof}

Finally, as we stated at the beginning of this section, we conclude it
with an affirmative answer to the projector problem when $J=\mathbb{R}$. 

\begin{lemma}
\label{Unicité}
If the linear system \eqref{LinealCap29} has an exponential dichotomy on $\mathbb{R}$ with projection $P$
this projection is unique.
\end{lemma}

\begin{proof}
Let us assume that there exists another projector $P'$. 
In addition, it follows that the system (\ref{LinealCap29}) also has an exponential dichotomy
on $\mathbb{R}_{0}^{-}$ and $\mathbb{R}_{0}^{+}$, then by using lemmas \ref{BSPHL} and \ref{BSNHL}
we have that 
$$
\ker P = \ker P'   \quad \textnormal{and} \quad \text{Im}\, P = \text{Im}\, P'.
$$

By Lemma \ref{Palmeriano} we have that
$PP'=P'P$ which combined with the identity (\ref{identites-projecteur}) from Lemma \ref{Unip+}
leads to $P=P'$.
\end{proof}

\section{Comments and references}

\noindent \textbf{1)} The majority of the results stated in this section
are from the Coppel's book, namely, Lemma \ref{completar}, Theorem \ref{EDINCU}
and Theorem \ref{NCU+BG}. We only added
intermediate computations and explanations in order to improve its clarity. 

\medskip

\noindent \textbf{2)} The Theorem \ref{NCU+BG} is revisited and improved
by K.J. Palmer in \cite[Th.1]{Palmer2006} where it is proved that if the system
(\ref{LinealCap29}) is continuous on $\mathbb{R}_{0}^{+}$ and has the property of
uniform bounded growth then, the properties of exponential dichotomy on $\mathbb{R}_{0}^{+}$,
uniform noncriticality and exponential expansiveness are equivalent. 

We recall that
(\ref{LinealCap29}) is exponentially expansive on $\mathbb{R}_{0}^{+}$ if there exist positive
constants $L$ and $\beta$ such that for any solution $t\mapsto x(t)$ of (\ref{LinealCap29})
and compact interval $[a,b]\subset \mathbb{R}_{0}^{+}$ it follows that:
$$
|x(t)|\leq L\left\{e^{-\beta(t-a)}|x(a)|+e^{-\beta(b-t)}|x(b)|\right\} \quad \textnormal{for any $t\in [a,b]$}.
$$

\medskip

\noindent \textbf{3)} The Lemma \ref{Palmeriano}
is from K.J. Palmer \cite{Palmer2006} and allowed to carry out an easy proof of Proposition \ref{FullDico}.

\medskip

\noindent \textbf{4)} The Lemmas \ref{Unip-} and \ref{Unip+} about the projector problem
where adapted from Kloeden $\&$ Rasmussen's book \cite[p.83]{Klo}. Notice that, by considering
the section 2.3 from Chapter 2, we can construct the projectors $\overline{P}(t)=X(t)\overline{P}X^{-1}(t)$
and $P'(t)=X(t)P'X^{-1}(t)$ such that the properties (\ref{ProProj1}) and (\ref{ProProj2}) becomes
$$
\tilde{K}=\sup_{t\in (-\infty,0]}\Vert [I-\overline{P}(t)]\Vert,
\quad \textnormal{and} \quad
K'=\sup_{t\in [0,+\infty)}\Vert P'(t)\Vert.
$$

Last but not least, Lemmas \ref{Unip-} and \ref{Unip+} allowed us to provide an easy proof
for Lemma \ref{Unicité}.

%It is important to highlight some ideas previous to Perron's seminal work \cite{Perron}, which pointed out related ideas as the works of Hadamard \cite{Hadamard}.
%\textcolor{red}{Desarrollarlo mas}
\medskip

\section{Exercises}

\begin{itemize}
\item[1.-] Prove the Lemma \ref{completar2}.
\item[2.-] Prove the Lemma \ref{BSNHL}.
\item[3.-] Prove that ii) implies i) in Lemma \ref{Palmeriano}.
\item[4.-] Under the assumptions of Lemma \ref{Unip-} prove the identities
   $$
   (I-\overline{P})=(I-P)(I-\overline{P}) \text{ and } \overline{P}=(I-P+\overline{P})P.
   $$
\item[5.-] Under the assumptions of Lemma \ref{Unip+} prove the identities
\begin{displaymath}
P-\overline{P}=P(P-\overline{P})=(P-\overline{P})(I-P)  
\end{displaymath}
and
 \begin{displaymath}
I-\overline{P}=[I+P-\overline{P}](I-P).
\end{displaymath}
\item[6.-] Prove Lemmas \ref{Unip-} and \ref{Unip+} by considering variable projectors. 
\item[7.-] Prove a symmetric version of Theorem \ref{NCU+BG}: if the linear system \eqref{LinealCap29} has the properties of bounded decay and noncritical uniformity
on $(-\infty,0]$ then it also has an exponential dichotomy on $(-\infty,0]$.
\end{itemize}

\chapter{Reducibility}
The main result of this section states that any linear system having an exponential dichotomy on $J$ with nontrivial projection is generally kinematically similar to a block diagonal system.

\section{Preliminaries}
\subsection{Definition}
The reducibility is a particular case of general kinematical similarity introduced by W. Coppel in
\cite{Coppel67} and revisited in \cite[p.38]{Cop}.

\begin{definition}
\label{RedCop}
The linear $n$--dimensional system
\begin{equation}
\label{LinealRED}
\dot{x}=A(t)x \quad \textnormal{with $t\in J\subseteq \mathbb{R}$}
\end{equation}
with $A(t)\in M_{n}(\mathbb{R})$ continuous on $J$ is \textbf{reducible} to 
\begin{equation}
\label{LinealRED2}
\dot{y}=B(t)y \quad \textnormal{with $t\in J\subseteq \mathbb{R}$}
\end{equation}
if the systems \eqref{LinealRED} and \eqref{LinealRED2} are generally kinematically 
similar on $J$. Additionally the matrix $B(t)$ has the block form:
\begin{displaymath}
B(t)=\left[\begin{array}{cc}
B_{1}(t) & 0\\\\
0        & B_{2}(t)
\end{array}\right],
\end{displaymath}
where $B_{1}(t)\in M_{n_{1}}(\mathbb{C})$ and  $B_{2}(t)\in M_{n_{2}}(\mathbb{C})$
with $n_{1}+n_{2}=n$ and $n_{1},n_{2}<n$. 
\end{definition}

In order to avoid confusion, as pointed by W. Coppel in \cite{Coppel67}, it is important to say that in a vast number of references, the property of reducibility is a different thing, namely, the kinematical similarity to an autonomous system, which was introduced by A. Lyapunov and strongly influenced the mathematical literature. In addition, W. Coppel states that Definition \ref{RedCop} agrees with the definition of reducibility in linear algebra.

\subsection{Technical results}

\begin{lemma}
\label{coppel39}
Let $P$ be an orthogonal projection (i.e., $P^{*}=P$) and $X$ be an invertible matrix, then there exists an invertible matrix $S$ such that
\begin{displaymath}
SPS^{-1}=XPX^{-1}.
\end{displaymath}

Moreover, the matrix $S$ satisfies the following properties:
\begin{displaymath}
||S||_{2}\leq  \sqrt{2}\quad \textnormal{and} \quad
||S^{-1}||_{2} \leq  \sqrt{||XPX^{-1}||_{2}^{2}+||X(I-P)X^{-1}||_{2}^{2}}.
\end{displaymath}
\end{lemma}

\begin{proof}
In this proof, we will write $||\cdot||$ instead of $||\cdot||_{2}$. Moreover, we will divide the proof in several steps.

\noindent\textit{Step 1:} The matrix
\begin{equation}
\label{U}
U=PX^{*}XP+(I-P)X^{*}X(I-P),
\end{equation}
is Hermitian and positive.

The identity $U=U^{*}$ is a direct consequence of $P=P^{*}$. In order to prove that $U$ is positive, let $\xi \in \mathbb{C}^{n}\setminus\{0\}$ and notice that
\begin{displaymath}
\begin{array}{rcl}
\xi^{*}U\xi &=& \xi^{*}PX^{*}XP\xi+\xi^{*}(I-P)X^{*}X(I-P)\xi\\\\
&=&||XP\xi||^{2}+||X(I-P)\xi||^{2}\geq 0\\\\
\end{array}
\end{displaymath}
and we will see that the inequality is strict. Indeed, otherwise if $\xi^{*}U\xi=0$, it will be equivalent to
$$
XP\xi=0 \quad \textnormal{and} \quad X(I-P)\xi=0,
$$
which implies that $XP\xi+X(I-P)\xi=X\xi=0$ and $\xi=0$ since $X$ is invertible, obtaining a contradiction.

\medskip

\noindent\textit{Step 2:} The positiveness of $U$ implies the 
existence and uniqueness of a matrix $R=R^{*}>0$ such that $R^{2}=U$, we 
refer the reader to the Appendix A and references for details.

\medskip

\noindent\textit{Step 3}: The matrices $R$ and $P$ commute.

As $P^{2}=P$ and $R^{2}=U$, by using (\ref{U}), it is easy to see that 
$$
R^{2}P=PR^{2}=PX^{*}XP.
$$
 
As $R^{2}$ and $P$ are Hermitian matrices they are diagonalizable. In addition, as they commute it is well known that they have simultaneous diagonalization (see \textit{e.g.} \cite[pp.69--72]{Horn},\cite{Rajdavi}), that is, there exists an invertible matrix $Q$ such that
\begin{displaymath}
R^{2}=Q^{-1}AQ \quad \textnormal{and} \quad
P=Q^{-1}BQ,
\end{displaymath}
where $A$ and $B$ are diagonal matrices containing the eigenvalues of $R^{2}$ and $P$ respectively.

As the diagonal matriz $A$ has only positive terms, we can consider $A=DD$, where $D$ is a diagonal matrix
$D=\textnormal{Diag}\{d_{11},\ldots,d_{nn}\}$ such that $a_{ii}=\sqrt{d_{ii}}$. Then, it follows that
$$
R^{2}=Q^{-1}DDQ=(Q^{-1}DQ)(Q^{-1}DQ)
$$
and we have that $R=Q^{-1}DQ$.

As $BD=DB$ since are diagonal matrices, we have that
$$
Q^{-1}BDQ=Q^{-1}DBQ,
$$
which is equivalent to
$$
(Q^{-1}BQ)(Q^{-1}DQ)=(Q^{-1}DQ)(Q^{-1}BQ),
$$
which is also equivalent to 
$PR=RP$ and also implies that
$$
R(I-P)=R-RP=R-PR=(I-P)R.
$$

\noindent\textit{Step 4}: Let us consider the transformation $S=XR^{-1}$. Now, note that $PR=RP$ implies:
\begin{displaymath}
\begin{array}{rcl}
SPS^{-1}&=& XR^{-1}PRX^{-1}\\\\
&=& XR^{-1}RPX^{-1}\\\\
&=&XPX^{-1},
\end{array}
\end{displaymath}
and we deduce the first identity stated in the Lemma.

\noindent\textit{Step 5}:  By using $SR=X$, $R^{*}=R$ and recalling that $PR=RP$ where $R^{2}=U$, we have that
\begin{displaymath}
\begin{array}{rcl}
R^{2} & = & PX^{*}XP+(I-P)X^{*}X(I-P)\\\\
 &=&PR^{*}S^{*}SRP+(I-P)R^{*}S^{*}SR(I-P)\\\\
 &=&PRS^{*}SRP+(I-P)RS^{*}SR(I-P)\\\\
 &=&RPS^{*}SPR+R(I-P)S^{*}S(I-P)R.
\end{array}
\end{displaymath}

We multiply the above identity by $R^{-1}$ by the right and the left, obtaining
$$
I=PS^{*}SP+(I-P)S^{*}S(I-P).
$$

Now, let $\xi \in \mathbb{R}^{n}\setminus \{0\}$ and note that
\begin{displaymath}
\begin{array}{rcl}
||S\xi||^{2} &=& ||SP\xi+S(I-P)\xi||^{2} \\\\
& \leq & \{||SP\xi||+||S(I-P)\xi||\}^{2} \\\\
& \leq & 2\{||SP\xi||^{2}+||S(I-P)\xi||^{2}\}\\\\
& \leq & 2\{\xi^{*}PS^{*}SP\xi+\xi^{*}(I-P)S^{*}S(I-P)\xi\}\\\\
& \leq & 2\xi^{*}\{PS^{*}SP+(I-P)S^{*}S(I-P)\}\xi\\\\
& \leq & 2\xi^{*}\xi=2||\xi||^{2} 
\end{array}
\end{displaymath}
then we can conclude that $||S||\leq \sqrt{2}$.

\medskip

\noindent\textit{Step 6:} Finally, as $S=XR^{-1}$, we have that $S^{-1}=RX^{-1}$ and
\begin{displaymath}
\begin{array}{rcl}
(S^{-1})^{*}S^{-1}&=&(X^{-1})^{*}R^{*}RX^{-1}\\\\
&=&(X^{-1})^{*}R^{2}X^{-1}\\\\
&=&(X^{-1})^{*}[PX^{*}XP+(I-P)X^{*}X(I-P)]X^{-1}\\\\
&=&(X^{-1})^{*}PX^{*}XPX^{-1}+(X^{-1})^{*}(I-P)X^{*}X(I-P)X^{-1}\\\\
\end{array}
\end{displaymath}
then we can deduce, similarly as in the previous step, that
$$
||S^{-1}||^{2}\leq ||XPX^{-1}||^{2}+||X(I-P)X^{-1}||^{2}
$$
and the result follows.
\end{proof}

A careful reading of the statement of Lemma \ref{coppel39} shows that given a fixed orthogonal projection $P$ and any function $t\mapsto X(t)$ of invertible matrices, there will be a collection of matrices $t\mapsto S(t)$ satisfying the above identities and estimations. This prompt the following byproducts:  

\begin{corollary}
\label{Copp39-2}
Let $J\ni t\mapsto X(t)$ be a fundamental matrix of the linear system $\dot{x}=A(t)x$ and $P$ be an orthogonal projection (i.e., $P^{*}=P$), then there exists an invertible and differentiable (with derivable inverse) matrix $J\ni t\mapsto S(t)$ such that
\begin{displaymath}
S(t)PS^{-1}(t)=X(t)PX^{-1}(t).
\end{displaymath}

Moreover, $S(t)$ satisfies the following properties:
\begin{displaymath}
||S(t)||_{2}\leq  \sqrt{2}\,\, \textnormal{and} \,\,
||S^{-1}(t)||_{2} \leq  \sqrt{||X(t)PX^{-1}(t)||_{2}^{2}+||X(t)(I-P)X^{-1}(t)||_{2}^{2}}
\end{displaymath}
and
\begin{equation}
\label{DSKS}
\dot{S}(t)=A(t)S(t)-S(t)\dot{R}(t)R^{-1}(t).
\end{equation}
\end{corollary}

\begin{proof}
The existence of a matrix valued function $t\mapsto S(t)$ satisfying the above identities and inequalities for any $t\in J$ is a consequence of Lemma \ref{coppel39}. Moreover,
by following the lines of the previous proof, we know that $S(t)=X(t)R^{-1}(t)$,
where $R^{2}(t)P=PR^{2}(t)$ and $R^{2}(t)$ defined by:
\begin{equation}
\label{R2U}
R^{2}(t)=PX^{*}(t)X(t)P+[I-P]X^{*}(t)X(t)[I-P]>0.
\end{equation}

The derivability of $X(t)$ combined with (\ref{R2U}) implies that $R^{2}(t)$ is de\-ri\-va\-ble on $J$. Now, we
will prove that $R(t)$ is derivable on $J$. Firstly, let us recall that as $P$ is hermitian it follows that 
is diagonalizable and there exists a nonsingular and constant matrix $Z$ such that $Z^{-1}PZ=D_{P}$.

On the other hand, as $R^{2}(t)P=PR^{2}(t)$, we know that $R^{2}(t)$ and $P$ have simultaneous  diagonalization, we can assume that 
$$
Z^{-1}R^{2}(t)Z=\textnormal{Diag}\{a_{1}(t),\ldots,a_{n}(t)\},
$$ 
that is, $R^{2}(t)$ is diagonalizable
by a constant basis. As we know that $R^{2}(t)$ is de\-ri\-va\-ble on $J$ it follows that $t\mapsto a_{i}(t)$ are derivable for any $i\in \{1,\ldots,n\}$.

The property $R(t)>0$ implies that
$$
R(t)=Z \,\textnormal{Diag}\{\sqrt{a_{1}(t)},\ldots,\sqrt{a_{n}(t)}\}Z^{-1}
$$
is well defined and derivable on $J$. Now, by Lemma \ref{dimf} from Chapter 1, we have that $\dot{R}^{-1}(t)=-R^{-1}(t)\dot{R}(t)R^{-1}(t)$. Then, the de\-ri\-va\-bility of $S(t)=X(t)R^{-1}(t)$ follows and the identity (\ref{DSKS})
can be deduced easily.
\end{proof}

Note that the equation (\ref{DSKS}) is related with the definitions of kinematical similarity stated
in the Chapter 1. In fact, the linear systems $\dot{x}=A(t)x$ and $\dot{y}=\dot{R}(t)R^{-1}(t)y$ will be generally kinematically similar provided that the matrix valued functions $t\mapsto S(t)$ and $t\mapsto S^{-1}(t)$ are bounded on $J$. This is the topic of the next section.

%\begin{lemma}
%The matrix $\dot{R}(t)R^{-1}(t)$ satisfies
%\begin{equation}
%    \label{desigualdad-R}
%\lambda(t)I \leq   \dot{R}(t)R^{-1}(t)+R^{-1}(t)\dot{R}(t)\leq \mu(t)I,
%\end{equation}
%where 
%$$
%\lambda(t)I \leq   A(t)+A^{*}(t)\leq \mu(t)I
%$$
%\end{lemma}

%\begin{proof}

%We derivate (\ref{R2U}) and obtain.
%\begin{displaymath}
%\begin{array}{rcl}
%R(t)\dot{R}(t)+\dot{R}(t)R(t) &=& PX^{*}(t)[A(t)+A^{*}(t)]X(t)P\\\\
%&& +[I-P]X^{*}[A(t)+A^{*}(t)]X(t)[I-P].
%\end{array}
%\end{displaymath}

%The inequality $A(t)+A^{*}(t)\leq \mu(t)I$ implies that
%\begin{displaymath}
%\begin{array}{rcl}
%R(t)\dot{R}(t)+\dot{R}(t)R(t) &\leq  &  \mu(t)PX^{*}(t)X(t)P\\\\
%&& +\mu(t)[I-P]X^{*}(t)X(t)[I-P]\\\\
%&\leq & \mu(t) R^{2}(t),
%\end{array}
%\end{displaymath}
%which combined with the inequality $\lambda(t)I\leq A(t)+A^{*}(t)$ implies
%\begin{displaymath}
%\lambda(t)R^{2}(t)\leq R(t)\dot{R}(t)+\dot{R}(t)R(t)\leq \mu(t)R^{2}(t).
%\end{displaymath}

%By multiplying the above matrix inequality for $R^{-1}(t)$ by the left and by the right we obtain that 
%\begin{displaymath}
%\lambda(t)\leq \dot{R}(t)R^{-1}(t)+R^{-1}(t)\dot{R}(t)\leq \mu(t).
%\end{displaymath}
%\end{proof}

\section{Exponential Dichotomy and Reducibility}

The main result of this section states that any linear system
having the property of exponential dichotomy is reducible. This allow us to revisit the 
fact that any solution of a linear system having an exponential dichotomy on $J$ can be decomposed in an asymptotically stable
component and an asymptotically unstable one since
the reducibility implies a decoupling in two systems having the same asymptotic behavior.

\begin{theorem}
\label{EDIR}
Let $X(t)$ be a fundamental matrix of $\dot{x}=A(t)x$, where $t\mapsto A(t)$ is continuous on $J$. Moreover, assume the existence of an or\-tho\-go\-nal and nontrivial projection $P$ with rank $r$ such that $t\mapsto ||X(t)PX^{-1}(t)||_{2}$ is
 bounded on $J$. Then, the linear system is reducible to the diagonal block system
 \begin{equation}
 \label{rDr}
 \dot{y}=\dot{R}(t)R^{-1}(t)y=\left[\begin{array}{cc}
B_{1}(t) & 0\\\\
0        & B_{2}(t)
\end{array}\right]y,
 \end{equation}
where $B_{1}(t)$ has rank $r$ while $B_{2}(t)$ has rank $n-r$ for any $t\in J$.
\end{theorem}
%\begin{lemma}
%\label{boundedS1}
%Let $X(t)$ be a fundamental matrix of (\ref{LinealRED}) and suppose that there exists an orthogonal projection $P$ such that $t\mapsto ||X(t)PX^{-1}(t)||_{2}$ is
% bounded on $J$. Then, the linear system (\ref{LinealRED}) is kinematically similar (general sense) to the block system
% \begin{equation}
% \label{rDr}
% \dot{y}=\dot{R}(t)R^{-1}(t)y,
% \end{equation}
% whose blocks submatrices have dimensions $r$ and $n-r$ respectively.
%\end{lemma}

\begin{proof}
Without loss of generality, by using Lemma \ref{DEPCa} from Chapter 2, we can consider
a fundamental matrix $X(t)$ and a canonical projection $P_{r}$ with rank $0<r<n$ as follows:
\begin{displaymath}
X(t)=\left[\begin{array}{cc}
X_{11}(t) & X_{12}(t) \\
X_{21}(t) & X_{22}(t)
\end{array}\right] \quad \textnormal{and} \quad 
P_{r}=\left[\begin{array}{cc}
I_{r} & 0 \\
0 & 0
\end{array}\right]
\end{displaymath}
with $X_{11}\in M_{r}(\mathbb{C})$,$X_{12}\in M_{r,n-r}(\mathbb{C})$,$X_{21}\in M_{n-r,r}(\mathbb{C})$ and $X_{22}\in M_{n-r}(\mathbb{C})$.

The boundedness of $t\mapsto ||X(t)P_{r}X^{-1}(t)||_{2}$ on $J$
implies the boundedness of $t\mapsto ||X(t)[I-P_{r}]X^{-1}(t)||_{2}$, then Corollary \ref{Copp39-2} implies the existence of $S(t)=X(t)R^{-1}(t)$ satisfying (\ref{DSKS}),
such that $S(t)$ and $S^{-1}(t)$ are bounded and derivable on $J$.

Secondly, it can be easily seen that the general kinematical similarity on $J$  between (\ref{LinealRED}) and (\ref{rDr}) is a consequence of (\ref{DSKS}) combined with the transformation $y=S^{-1}(t)x$. Then, the Theorem will follows if we prove that $\dot{R}(t)R^{-1}(t)$ is block diagonal.

In order to verify that $\dot{R}(t)R^{-1}(t)$ is block diagonal,
let us recall that the identity (\ref{R2U}) is considered with any orthogonal
projection $P$, in particular we can use $P_{r}$. In addition, note that 
$R^{2}(t)$ can be represented as follows:
$$
R^{2}(t)=\left[\begin{array}{cc}
X_{11}^{*}(t)X_{11}(t)+X_{21}^{*}(t)X_{21}(t) & 0 \\
0 & X_{12}^{*}(t)X_{12}(t)+X_{22}^{*}(t)X_{22}(t)
\end{array}\right],
$$
which has blocks of dimensions $r$ and $n-r$ respectively.

By using the properties $R^{2}(t)>0$ and $R(t)>0$
combined with the fact that the submatrices of $R^{2}(t)$ are also
positive (see \cite[p.13]{Householder}) it is easy to see that $R(t)$ is also block diagonal with the same dimensions that $R^{2}(t)$. Then, this block diagonal property is also preserved by $R^{-1}(t)$ and $\dot{R}(t)$ and the reducibility follows.
\end{proof}

The next result is devoted to the subsystems of (\ref{rDr}), namely
\begin{equation}
\label{subsystemes-V}
\dot{y}_{1}=B_{1}(t)y_{1} \quad \textnormal{and} \quad \dot{y}_{2}=B_{2}(t)y_{2}
\end{equation}
and describes its relation with (\ref{LinealRED}).

\begin{corollary}
\label{Redu2}
 If the linear system \eqref{LinealRED} has an exponential dichotomy on $J$ with constants $K,\alpha>0$ and projection
 $P$ having  rank equal to $r$ with $0<r<n$, then

\noindent \textnormal{a)} If $Y_{i}(t)$ is a fundamental matrix of $\dot{y}_{i}=B_{i}(t)y_{i}$ with $i=1,2$, then
there exists $\tilde{K}>0$ such that:
\begin{displaymath}
\left\{\begin{array}{rcl}
||Y_{1}(t,s)||_{2}\leq \tilde{K}e^{-\alpha(t-s)} & \textnormal{for $t\geq s$ with $t,s\in J$}\\
||Y_{2}(t,s)||_{2}\leq \tilde{K}e^{-\alpha(s-t)} & \textnormal{for $t\leq s$ with $t,s\in J$}.
\end{array}\right.
\end{displaymath}

\noindent \textnormal{b)}  The subsystems \eqref{subsystemes-V} have an exponential dichotomy on $J$ with the trivial projections $P_{1}=I_{r}\in M_{r}(\mathbb{R})$
and $P_{2}=0\in M_{n-r}(\mathbb{R})$ respectively.
\end{corollary}

\begin{proof}
As (\ref{LinealRED}) has an exponential dichotomy on $J$, we have that
$t\mapsto ||X(t)PX^{-1}(t)||$ is bounded on $J$. Then, by Theorem \ref{EDIR} we know that (\ref{LinealRED}) is reducible to (\ref{LinealRED2}), which can be written as 
\begin{equation}
\label{blocs}
\left(\begin{array}{c}
\dot{y}_{1}\\
\dot{y}_{2}
\end{array}\right)
=\left[\begin{array}{cc}
B_{1}(t) & 0\\\\
0        & B_{2}(t)
\end{array}\right]
\left(\begin{array}{c}
y_{1}\\
y_{2}
\end{array}\right).
\end{equation}

On the other hand, as the reducibility is a particular case of kinematical similarity
and, as in the proof of the previous result, we are assuming that the exponential dichotomy
of (\ref{LinealRED}) have positive constants $K,\alpha$ and a canonical projection $P_{r}$, by 
using Theorem \ref{KSDE} from Chapter 2 we have the existence of a fundamental matrix $Y(t)$ for (\ref{blocs}) and $\tilde{K}>0$ such that
\begin{displaymath}
\left\{\begin{array}{rcl}
||Y(t)P_{r}Y^{-1}(s)||_{2}\leq \tilde{K}e^{-\alpha(t-s)} & \textnormal{for $t\geq s$ with $t,s\in J$}\\
||Y(t)[I-P_{r}]Y^{-1}(s)||_{2}\leq \tilde{K}e^{-\alpha(s-t)} & \textnormal{for $t\leq s$ with $t,s\in J$},
\end{array}\right.
\end{displaymath}
where $Y(t)$ is a fundamental matrix of (\ref{blocs}) such that $Y(t)=S^{-1}(t)X(t)$. 

Let us recall that $S(t)=X(t)R^{-1}(t)$, then it follows that $Y(t)=R(t)$ which is a diagonal block matrix, as we see in the proof of Theorem \ref{EDIR}. In consequence, $Y(t)$ must have the form
\begin{displaymath}
Y(t)=\left[\begin{array}{cc}
Y_{1}(t) & 0\\\\
0        & Y_{2}(t)
\end{array}\right],
\end{displaymath}
where $Y_{1}(t)$ and $Y_{2}(t)$ are a pair of fundamental matrices of the subsystems (\ref{subsystemes-V}) respectively.

Then, by considering the stucture of the canonical projection $P_{r}$, it is easy to deduce that
\begin{displaymath}
Y(t)P_{r}Y^{-1}(s)=\left[\begin{array}{cc}
Y_{1}(t,s) & 0\\\\
0        &  0
\end{array}\right] \,\, \textnormal{and}\,\,\, Y(t)[I-P_{r}]Y^{-1}(s)=\left[\begin{array}{cc}
0 & 0\\\\
0        &  Y_{2}(t,s)
\end{array}\right]
\end{displaymath}
and we have that 
$$
||Y(t)P_{r}Y^{-1}(s)||_{2}=||Y_{1}(t,s)||_{2}
\quad \textnormal{and}\quad
||Y(t)[I-P_{r}]Y^{-1}(s)||_{2}=||Y_{2}(t,s)||_{2}
$$
where the left terms
are norms for matrices of order $n$ while the right terms are norms for matrices of orders
$r$ and $n-r$ respectively  and the statement a) is proved.

Finally, a direct consequence is that
the subsystems (\ref{subsystemes-V}) have an exponential dichotomy on $J$ with projections $P_{1}=I_{r}\in M_{r}(\mathbb{R})$ and $P_{2}=0\in M_{n-r}(\mathbb{R})$ respectively.

\end{proof}

A careful reading of the above results show us that any linear system
having exponential dichotomy with no trivial projection is reducible to 
a block system, which can be decomposed in a stable system and a
unstable one.

\section{Comments and References}

\noindent \textbf{1)} As pointed out in the subsection 1.1, there exist a different property
of reducibility introduced by Lyapunov. In fact, in \cite[pp. 241--242]{Lyapunov}, Lyapunov describes
the currently called Lyapunov transformations and says "$\ldots$\textit{The considered systems of differential equations are, by an appropriate choice of transformations, transformed in a system with constant coefficients. In this case, we will call such systems as \textbf{reducible}}$\ldots$" \footnote{Free translation of the authors.}. On other hand, we refer
to \cite{Bylov2} where the kinematical similarity to a block system is treated.

\medskip

\noindent \textbf{2)} Lemma \ref{coppel39} has been stated in \cite[Lemma 1,p.39]{Cop}, we only provide
intermediate computations and additional estimations in order to facilitate its reading.
In particular, we use the fact that $R^{2}$ and $P$ have simultaneous diagonalization
as a tool to prove that if $R^{2}$ and $P$ commute then $R$ and $P$ also commute.

\medskip

\noindent \textbf{3)}
Corollary \ref{Copp39-2} has been partially stated in \cite[Lemma 1, p.39]{Coppel67}, where
the derivability of $t\mapsto R(t)$ is proved by using a result from \cite[Lemma 3.1]{Reid}
while our proof uses the fact that $P$ and $R^{2}(t)$ have simultaneous diagonalization.

\medskip

\noindent \textbf{4)} There exist several works in the literature, such as \cite{Chu, Siegmund2,Silva,Zhang}, that mentioned   \cite[Lemma 1, p.39]{Coppel67} as a useful tool for diverse results into their respective works. We emphasize that in \cite{Siegmund2} makes a work different from W. Coppel due to the S. Siegmund deals  with the more general class of systems with locally integrable $t-$dependence rather than systems with continuous $t-$dependence.

\section{Exercises}
\begin{itemize}
\item[1.-] Let us consider the linear system (\ref{LinealRED}) with $t\mapsto A(t)$ continuous on $\mathbb{R}$ and $\omega$--periodic (\ref{LinealRED}). Prove that if (\ref{LinealRED}) has Floquet
multipliers inside an outside the unit circle then it is is reducible to an autonomous system.
\item[2.-] If (\ref{LinealRED}) is reducible to (\ref{LinealRED2}), prove that $x'=-A^{T}(t)x$ is reducible to
\begin{displaymath}
\dot{y}=-\left[\begin{array}{cc}
B_{1}^{T}(t) & 0\\\\
0        & B_{2}^{T}(t)
\end{array}\right]y
\end{displaymath}
\item[3.-] Prove that the system considered in the Example \ref{MYDUG} from Chapter 2 is reducible and obtain an explicit representation.

\end{itemize}

\chapter{Exponential Dichotomy Spectrum}

\section{Introduction}
Given a matrix $A\in M_{n}(\mathbb{R})$, let us consider the following systems:
\begin{subequations}
  \begin{empheq}{align}
    &\dot{x}=Ax, \label{ESA1}\\
    &\dot{x}=Ax+f(t), \label{ESA2}\\
    &\dot{x}=Ax+Bu(t), \label{ESA3}
  \end{empheq}
\end{subequations}
where $f\in BC(\mathbb{R},\mathbb{R}^{n})$, $B\in M_{n,p}(\mathbb{R})$
and $u\colon \mathbb{R}\to \mathbb{R}^{p}$ is a measurable function.

The eigenvalues of $A$ provides essential information about several properties of the above systems.

We know that (\ref{ESA1}) is \textit{hyperbolic} if and only if the eigenvalues
of $A$ have non zero real part. In addition, (\ref{ESA1}) is \textit{uniformly asymptotically stable}
if and only if the eigenvalues of $A$ have negative real part.

We also know by Proposition \ref{admi} from Chapter 2 that if (\ref{ESA1}) is hyperbolic,
then the system (\ref{ESA2}) has the \textit{admissibility} property, that is, for any $f\in BC(\mathbb{R},\mathbb{R}^{n})$ the system (\ref{ESA2}) has bounded 
continuous solutions.

The system (\ref{ESA3}) is \textit{controllable} at $t_{0}$ if for any nontrivial initial condition $x_{0}\in \mathbb{R}^{n}$, there exists a finite time $t_{1}>t_{0}$ and a function $u_{0}\colon [t_{0},t_{1}]\to \mathbb{R}^{p}$, such that the solution $t\mapsto x(t,t_{0},x_{0},u_{0})$ of (\ref{ESA3}) with $u=u_{0}(\cdot)$ verifies $x(t_{1},t_{0},x_{0},u_{0})=0$. A nice result \cite[Th.3.1]{Zhou} states that the system (\ref{ESA3})
is controllable if and only if for any set of $\mathcal{S}\subset \mathbb{C}$ composed by $n$ 
arbitrary points, module conjugation, there exists $F\in M_{p,n}(\mathbb{R})$ such that the eigenvalues of $A-BF$ are exactly $\mathcal{S}$.

It is well known that there no exists an eigenvalues--based approach able to replicate the above results for the nonautonomous systems
\begin{subequations}
  \begin{empheq}{align}
    &\dot{x}=A(t)x, \label{LinCap3}\\
    &\dot{x}=A(t)x+f(t), \label{lin-nh}\\
    &\dot{x}=A(t)x+B(t)u(t), \label{ESNA3}
  \end{empheq}
\end{subequations}
where $A\colon J\subseteq \mathbb{R}\to M_{n}(\mathbb{R})$ and the interval $J$ will be assumed either $\mathbb{R}$,$\mathbb{R}_{0}^{-}$ or $\mathbb{R}_{0}^{+}$.

In order to fill this essential gap, several spectral theories have been developed
for the linear system (\ref{LinCap3}) and the definition proposed by Dieci $\&$ Van Vleck in \cite[p.267]{Dieci} encompasses the essential properties of these ones when considering $J=[0,+\infty)$:
\begin{definition}
\label{TSD}
A subset $\Sigma(A)\subset \mathbb{R}$ it is a \textbf{spectrum} associated to \eqref{LinCap3} if one or more of the following properties are verified:

\begin{itemize}
\item[{\bf (P1)}] For $\Sigma(A) \cap (-\infty,0)\neq \varnothing$ and $0\notin \Sigma(A)$, there exists
a non--zero bounded solution of \eqref{LinCap3}.

\item[{\bf (P2)}] For $0\notin \Sigma(A)$, there exists a bounded solution of 
\eqref{lin-nh} for any 
$f\in BC(J,\mathbb{R})\setminus \{0\}$.

\item[{\bf (P3)}] For $\Sigma(A) \cap (-\infty,0)\neq \varnothing$ and $0\notin \Sigma(A)$, there exists
a non--zero bounded solution of \eqref{lin-nh} for any $f\in BC(J,\mathbb{R})$.
\end{itemize}
\end{definition}

We have to point out that if $A(t)\equiv A$, the set 
\begin{displaymath}
\Sigma_{EIG}(A):=\left\{\lambda\in \mathbb{R}\colon \textnormal{$\lambda$ is the real part of an eigenvalue of $A$}\right\}
\end{displaymath}
satisfies \textbf{(P1)--(P3)}. The 
property $0\notin \Sigma_{EIG}(A)$ is the \textit{hyperbolicity} condition and $A$ is a Hurwitz matrix if $\Sigma_{EIG}(A)\subset (-\infty,0)$, which is equivalent to the \textit{uniform asymptotic stability}.
 
As we have seen in the Chapter 2, the property of exponential dichotomy 
can be seen as a nonautonomous hyperbolicity, in this section we will construct a spectrum $\Sigma(A)$ based in the property of exponential dichotomy.

\section{Definitions and basic results}

\subsection{A formal definition}
Let us consider the family of linear shifted systems:
\begin{equation}
\label{ec:2.4}
\dot{x}=[A(t)-\lambda I]x
\end{equation}
with $\lambda \in \mathbb{R}$ and notice that it has the following fundamental matrix:
\begin{equation*}%\label{ec:2.5}
X_{\lambda}(t)=X(t)e^{-\lambda t},
\end{equation*}
where $X(t)$ is a fundamental matrix of (\ref{LinCap3}), then it is straightforward to introduce the next property: 
\begin{definition}
\label{ED-Per}
The system \eqref{ec:2.4} has an \textbf{exponential dichotomy} on $J$ if there exist positive numbers $K_{\lambda}$, $\alpha_{\lambda}$ and a projection $P_{\lambda}^2=P_{\lambda}$ such that the fundamental matrix of \eqref{ec:2.4} satisfies
\begin{displaymath}
\left\{\begin{array}{rl}
\Vert X(t)e^{-\lambda t}P_{\lambda}X^{-1}(s)e^{\lambda s}\Vert \leq K_{\lambda}e^{-\alpha_{\lambda}(t-s)}&\textnormal{if} \,\, t\geq s, \,\,\, \textnormal{$t,s\in J$}\\
\Vert X(t)e^{-\lambda t}(I-P_{\lambda})X^{-1}(s)e^{\lambda s}\Vert \leq K_{\lambda}e^{-\alpha_{\lambda} (s-t)} &\textnormal{if} \,\, t \leq s, \,\,\, \textnormal{$t,s\in J$}.
\end{array}\right.
\end{displaymath}
\end{definition}

The reader can deduce easily that, for any $t,s\in J$, the above property is equivalent to
\begin{equation}
\label{ec:2.7}
\left\{\begin{array}{rl}
\Vert X(t)P_{\lambda}X^{-1}(s)\Vert \leq K_{\lambda}e^{\lambda (t-s)}e^{-\alpha_{\lambda}(t-s)}&\textnormal{if}\,\, t\geq s, \\
\Vert X(t)(I-P_{\lambda})X^{-1}(s)\Vert \leq K_{\lambda}e^{\lambda (t-s)}e^{-\alpha_{\lambda} (s-t)} &\textnormal{if}\,\, t \leq s.
\end{array}\right.
\end{equation}

%\begin{lemma}
%If the system (\ref{ec:2.4}) has an exponential dichotomy on $\mathbb{R}$, then the origin is the unique solution that is bounded in $\mathbb{R}$.
%\end{lemma}

%\begin{proof}
%The proof is similar to the proof of Lemma 2.1 from Ch.2.
%\end{proof}

\begin{definition}
\label{DefinicionEspectro}
The \textbf{exponential dichotomy spectra} of \eqref{LinCap3} are defined by the sets:
\begin{displaymath}
\begin{array}{rcl}
\Sigma (A)&:=&\{\lambda\in\mathbb{R}\colon \textnormal{(\ref{ec:2.4}) has not an exponential dichotomy on $\mathbb{R}$}\},\\
\Sigma^{-} (A)&:=&\{\lambda\in\mathbb{R}\colon \textnormal{(\ref{ec:2.4}) has not an exponential dichotomy on $\mathbb{R}_{0}^-$}\},\\
\Sigma^{+}(A)&:=&\{\lambda\in\mathbb{R}\colon\textnormal{(\ref{ec:2.4}) has not an exponential dichotomy on $\mathbb{R}_{0}^+$}\}.
\end{array}
\end{displaymath}
\end{definition}

In some works, the above spectrum $\Sigma(A)$ is also known as the Sacker $\&$ Sell spectrum since 
it has been developed in 1978 by R.J. Sacker \& G. Sell in \cite{Sac} in a more general
framework of nonautonomous dynamics. As stated in \cite{Klo}, the spectra
$\Sigma^{-}(A)$ and $\Sigma^{+}(A)$ are also respectively known as \textit{past dichotomy spectrum} and \textit{future dichotomy spectrum}. On the other hand, it is important to emphasize that, in general, the above spectra are not the same.

The study of the properties of the above spectra is strongly related
with the properties of its complements, which will be defined as follows:
\begin{definition}
%\label{resolvent}
The \textbf{resolvents} associated to \eqref{LinCap3} are the sets 
$$
\rho (A)=\mathbb{R}\setminus \Sigma (A), \quad \rho^{-}(A)=\mathbb{R}\setminus \Sigma^{-}(A) \quad \textnormal{and} \quad
\rho^{+}(A)=\mathbb{R}\setminus \Sigma^{+}(A).
$$
\end{definition}

Note that if $\lambda_{0} \in \rho(A)$, then the exponential dichotomy property (\ref{ec:2.7}) is satisfied with $\lambda=\lambda_{0}$ and $J=\mathbb{R}$. This is also verified when considering
$\lambda_{0}\in \rho^{\pm}(A)$ and $J=\mathbb{R}_{0}^{\pm}$.

\begin{remark}
\label{contentions}
By using the resolvent definition, it is straightforward to verify that
\begin{equation}
\label{inc-spec}
\Sigma^{+}(A) \subseteq \Sigma(A) \quad \textnormal{and} \quad \Sigma^{-}(A)\subseteq \Sigma(A).
\end{equation}

In fact, by revisiting the equation \ref{schned} from Chapter 3, we can see that $0\notin \Sigma^{+}(a)$ whereas $0\in \Sigma(a)$.
\end{remark}

\subsection{Backwardly and forwardly bounded solutions of (\ref{ec:2.4}) }
The study of the family of shifted linear systems (\ref{ec:2.4}) is also
essential to study of the above mentioned spectra and resolvents. Then, similarly as we have done in the previous chapters,
for any $\lambda\in\mathbb{R}$ we will introduce the sets:
$$
\mathcal{S}^{\lambda}:=\left\{\xi\in\mathbb{R}^n: \sup_{t\in\mathbb{R}_{0}^+}\vert X(t,0)\xi\vert e^{- \lambda t}< +\infty \right\}
$$
and
$$
\mathcal{U}^{\lambda}:=\left\{\xi\in\mathbb{R}^n: \sup_{t\in\mathbb{R}_{0}^-}\vert X(t,0)\xi\vert e^{-\lambda t}<+\infty \right\}.
$$

\begin{remark}
\label{TE0}
As the solution of (\ref{ec:2.4}) with initial condition $x(0)=\xi$ is given by $t\mapsto X(t,0)e^{-\lambda t}\xi$. It is direct to see that $\mathcal{S}^{\lambda}$ (resp. $\mathcal{U}^{\lambda}$) is the set of initial conditions $\xi$ such that, the solutions of (\ref{ec:2.4}) passing trough $\xi$ at $t=0$ are bounded in $\mathbb{R}_{0}^+$ (resp. $\mathbb{R}_{0}^-$).
\end{remark}

\begin{remark}
\label{unstable-stable-spaces}
Notice that:
\begin{enumerate}
\item[a)] $\mathcal{S}^{\lambda}$ and $\mathcal{U}^{\lambda}$ are $\mathbb{R}$--vector subspaces of $\mathbb{R}^{n}$,
\item[b)] If $\lambda\leq\mu$, then $\mathcal{S}^{\lambda}\subseteq \mathcal{S}^{\mu}$,
\item[c)] If $\lambda\leq\mu$, then $\mathcal{U}^{\mu}\subseteq \mathcal{U}^{\lambda}$,
\item[d)] If $\lambda=0$, we recover the sets $\mathcal{V}=\mathcal{S}^{0}$ and $\mathcal{W}=\mathcal{U}^{0}$ defined by (\ref{EVS})--(\ref{EVI}) in Chapter 3. 
\end{enumerate}
\end{remark}

In case of $\lambda$ is such that the shifted linear system (\ref{ec:2.4}) has an exponential dichotomy, the following lemmas -- whose proof is left as an exercise -- describe the strong relationship between the vector spaces $\mathcal{S}^{\lambda}$ and $\mathcal{U}^{\lambda}$ with the projection $P_{\lambda}$.
\begin{lemma}
\label{carac-stable}
If the shifted linear system \eqref{ec:2.4} has an exponential dichotomy on $\mathbb{R}_{0}^{+}$
with projection $P_{\lambda}$, then $\mathcal{S}^{\lambda}=\textnormal{\text{Im}}P_{\lambda}$.
\end{lemma}

\begin{lemma}
\label{carac-unstable}
If the linear shifted system \eqref{ec:2.4} has an exponential dichotomy on $\mathbb{R}_{0}^{-}$ with projection $P_{\lambda}$, then $\mathcal{U}^{\lambda}=\ker P_{\lambda}$.
\end{lemma}

\begin{corollary}
If the linear system \eqref{ec:2.4} has an exponential dichotomy on $\mathbb{R}$ with projection $P_{\lambda}$, then 
$$
\mathcal{U}^{\lambda}=\ker P_{\lambda} \quad \textnormal{and} \quad \mathcal{S}^{\lambda}=\textnormal{\text{Im}}P_{\lambda}.
$$
\end{corollary}

If the shifted linear system (\ref{ec:2.4}) has an exponential dichotomy on $\mathbb{R}_{0}^{+}$ with projection $P_{\lambda}^{+}$ and also has an exponential dichotomy on $\mathbb{R}_{0}^{-}$ with projection $P_{\lambda}^{-}$ 
then 
$$
\mathcal{U}^{\lambda}=\ker P_{\lambda}^{+} \quad \textnormal{and} \quad \mathcal{S}^{\lambda}=\textnormal{\text{Im}}P_{\lambda}^{-}.
$$

As we have seen when addressing the projector problem in the section 4 from Chapter 3, the projectors
$P_{\lambda}^{+}$ and $P_{\lambda}^{-}$ will not be necessarily the same.

\begin{lemma}
\label{2proj}
If the shifted system \eqref{ec:2.4} has an exponential dichotomy on $\mathbb{R}_{0}^{-}$ with projection $P_{\lambda}$ and there exist another projection $\overline{P}$ satisfying the properties 
\begin{itemize}
\item[i)] There exists $K_1>0$ such that $$K_1=\sup_{t\in (-\infty,0]}\Vert X_{\lambda}(t)(I-\overline{P})X_{\lambda}^{-1}(t)\Vert,$$
\item[ii)] $\ker P_{\lambda}=\ker \overline{P}$,
\end{itemize}
then the system \eqref{ec:2.4} has also an exponential dichotomy on $\mathbb{R}_{0}^{-}$ with projection $\overline{P}$.
\end{lemma}

\begin{proof}
The proof is similar to the proof of Lemma \ref{Unip-} from Chapter 3.
\end{proof}

\begin{lemma}
\label{2proj+}
If the shifted system \eqref{ec:2.4} has an exponential dichotomy on $\mathbb{R}_{0}^{+}$ with projection $P_{\lambda}$ and there exist another projection $P'$ satisfying the properties 
\begin{itemize}
\item[i)] There exists $K_2>0$ such that 
$$
K_2=\sup_{t\in [0,+\infty)}\Vert X_{\lambda}(t)P'X_{\lambda}^{-1}(t)\Vert,
$$
\item[ii)] $\textnormal{Im} P_{\lambda}=\textnormal{Im} P'$,
\end{itemize}
then the system \eqref{ec:2.4} has also an exponential dichotomy on $\mathbb{R}_{0}^{+}$ with projection $\overline{P}$.
\end{lemma}

\begin{proof}
The proof is similar to the proof of Lemma \ref{Unip+} from Chapter 3.
\end{proof}

\section[Properties of the spectrum and the resolvent]{Properties of the exponential dichotomy spectrum and the resolvent}

\subsection{Topological Properties}
\begin{proposition}
\label{resolvent-open}
The resolvents $\rho(A)$,$\rho^{-}(A)$ and $\rho^{+}(A)$ of the linear
system \eqref{LinCap3} are open sets in $\mathbb{R}$.
\end{proposition}
\begin{proof}
The proof will be made only for $\rho(A)$ since the other resolvents
can be addressed similarly. The cases $\rho (A)=\varnothing$ and $\rho (A)=\mathbb{R}$ are direct. Now, let us assume that $\rho (A)$ is a proper subset of $\mathbb{R}$ and $\lambda\in \rho (A)$. Then, by Definition \ref{ED-Per}, we have that the shifted system $\dot{x}=[A(t)-\lambda I]x$ has an exponential dichotomy on $\mathbb{R}$, namely, there exist a projection $P_{\lambda}$ and two positive constants $K_{\lambda}$ and $\alpha_{\lambda}$ such that:
\begin{equation*}
\left\{\begin{array}{rlcr}
\Vert X(t)e^{-\lambda t}P_{\lambda}X^{-1}(s)e^{\lambda s}\Vert &\leq K_{\lambda}e^{-\alpha_{\lambda} (t-s)} &\text{if}&t\geq s\\
\Vert X(t)e^{-\lambda t}(I-P_{\lambda})X^{-1}(s)e^{\lambda s }\Vert &\leq K_{\lambda}e^{-\alpha_{\lambda} (s-t)} &\text{if}& t \leq s.
\end{array}\right.
\end{equation*}

Now, when considering $\mu\in (\lambda - \alpha_{\lambda} , \lambda + \alpha_{\lambda})$, we will prove that the linear system 
\begin{equation}
\label{MDRA2}
\dot{x}=[A(t)-\mu I]x
\end{equation}
has an exponential dichotomy on $\mathbb{R}$. Indeed, the inequality above is equivalent to 
\begin{equation*}
\left\{\begin{array}{rlcr}
\Vert X(t)P_{\lambda}X^{-1}(s)\Vert &\leq K_{\lambda}e^{-(\alpha_{\lambda}-\lambda ) (t-s)} &\text{if}&t\geq s\\
\Vert X(t)(I-P_{\lambda})X^{-1}(s)\Vert &\leq K_{\lambda}e^{-(\alpha_{\lambda} +\lambda ) (s-t)} &\text{if}& t \leq s,
\end{array}\right.
\end{equation*}
or also equivalently to
\begin{equation*}
\left\{\begin{array}{rlcr}
\Vert X(t)e^{-\mu t}P_{\lambda}X^{-1}(s)e^{\mu s}\Vert &\leq K_{\lambda}e^{-(\alpha_{\lambda} -\lambda +\mu )(t-s)} &\text{if}&t\geq s\\
\Vert X(t)e^{-\mu t}(I-P_{\lambda})X^{-1}(s)e^{\mu s }\Vert &\leq K_{\lambda}e^{-(\alpha_{\lambda} + \lambda -\mu ) (s-t)} &\text{if}& t \leq s.
\end{array}\right.
\end{equation*}

Now, let $\beta = \min \{\alpha_{\lambda} - \lambda +\mu , \alpha_{\lambda} + \lambda - \mu\}>0$ and we can deduce that 

\begin{equation*}
\left\{\begin{array}{rlcr}
\Vert X(t)e^{-\mu t}P_{\lambda}X^{-1}(s)e^{\mu s}\Vert &\leq K_{\lambda}e^{-\beta (t-s)} &\text{if}&t\geq s\\
\Vert X(t)e^{-\mu t}(I-P_{\lambda})X^{-1}(s)e^{\mu s }\Vert &\leq K_{\lambda}e^{-\beta (s-t)} &\text{if}& t \leq s,
\end{array}\right.
\end{equation*}
and we can see that the linear system (\ref{MDRA2}) has an exponential dichotomy on $\mathbb{R}$ with the same projection $P_{\lambda}$.

Finally, we can deduce that $(\lambda - \alpha_{\lambda} , \lambda + \alpha_{\lambda})\subset \rho (A)$ and the Proposition follows.
\end{proof}

A careful reading of the above proof show us that we worked with the same projection
$P_{\lambda}$, this is a deeper property of the resolvents:
\begin{corollary}
\label{open-resolvent}
For any $\lambda\in\rho (A),\rho^{-}(A),\rho^{+}(A)$, there exists $\delta_{\lambda}>0$ such that for any $\mu\in (\lambda - \delta_{\lambda}, \lambda + \delta_{\lambda})$, the system $$\dot{x}=[A(t)-\mu I]x$$ has an exponential dichotomy with the same projection $P_{\lambda}$.
\end{corollary}

A direct consequence of the definition of resolvent sets is that the spectra $\Sigma (A)$,$\Sigma^{-}(A)$ and $\Sigma^{+}(A)$ are closed sets in $\mathbb{R}$. We can obtain
a sharper result provided that the linear system has the property of bounded growth $\&$ decay
stated in the Chapter 1:
\begin{proposition}
\label{EDECS}
Under the assumption that the linear system \eqref{LinCap3} has the property of bounded growth $\&$ decay on $\mathbb{R}$ (resp. $\mathbb{R}_{0}^{-}$,$\mathbb{R}_{0}^{+}$) with
constants $K>0$ and $\alpha\geq 0$, if the spectrum $\Sigma (A)$ (resp. $\Sigma^{-}(A)$,$\Sigma^{+}(A)$) is non empty, then $\Sigma(A)$ is a compact subset of $[-\alpha,\alpha]$ when $\alpha>0$ and $\Sigma(A)=\{0\}$ when $\alpha=0$.
\end{proposition}

\begin{proof}
As before, the proof will be made only for the spectrum $\Sigma(A)$ and considering $\alpha>0$. The other ones can be obtained in a similar way.

It will be useful to note that
$\Sigma(A)\subset [-\alpha,\alpha]$ if and only
$$
(-\infty,-\alpha)\cup (\alpha,+\infty) \subset \rho(A).
$$

By Definition \ref{BGBD-C} of bounded growth $\&$ decay from Chapter 1, we know that the transition
matrix of (\ref{LinCap3}) verifies
$$
||X(t,s)||\leq Ke^{\alpha|t-s|} \quad \textnormal{for any $t,s\in \mathbb{R}$}.
$$

\noindent \textit{Case a):} If $t\geq s$, we have that
$$
||X(t,s)e^{-\lambda(t-s)}||\leq Ke^{(\alpha-\lambda)(t-s)}
$$
and, we can see that if $\lambda>\alpha$, there exists $\delta>0$ such that $\lambda = \alpha+\delta$, then we have that 
$$
\Vert X(t) e^{-\lambda t} X^{-1} (s)e^{\lambda s}\Vert \leq Ke^{-\delta (t-s)},
$$ 
which implies that for any $\lambda > \alpha$, the system $\dot{x}=[A(t)-\lambda I]x$ has an exponential dichotomy on $\mathbb{R}$ with projection $P=I$. This is equivalent to say that $(\alpha , +\infty)\subset \rho (A)$. 

\medskip

\noindent\textit{Case b):} If $t\leq s$, we have that 
\begin{displaymath}
\Vert X(t) e^{-\lambda t}X^{-1}(s)e^{\lambda s}\Vert \leq  Ke^{(\alpha+\lambda)(s-t)}
\end{displaymath}
and, we can see that if $\lambda<-\alpha$, there exists $\delta>0$ such that $\lambda =-(\alpha+\delta)$, then we can deduce that 
$$\Vert X(t) e^{-\lambda t} X^{-1} (s)e^{\lambda s}\Vert \leq Ke^{-\delta (s-t)},
$$ 
which implies that for any $\lambda < -\alpha$, the system $\dot{x}=[A(t)-\lambda I]x$ has an exponential dichotomy on $\mathbb{R}$ with projection $P=0$. This is equivalent to say that $(-\infty, -\alpha)\subset \rho (A)$.

Summarizing, we have proved that
$(-\infty,-\alpha) \cup (\alpha,+\infty) \subset \rho (A)$, and the Proposition follows.  
\end{proof}

\begin{remark}
\label{cotas-parciales}
By following the lines of the proof of the above result, we can deduce that:
\begin{itemize}
\item[a)] If \eqref{LinCap3} has the property of bounded growth on $\mathbb{R}$ (resp. $\mathbb{R}_{0}^{-}$,$\mathbb{R}_{0}^{+}$) with constants $K>0$ and $\alpha\geq 0$, if the spectrum $\Sigma (A)$ (resp. $\Sigma^{-}(A)$,$\Sigma^{+}(A)$) is non empty, then
$$
\Sigma(A) \subset (-\infty,\alpha].
$$
\item[b)] If \eqref{LinCap3} has the property of bounded decay on $\mathbb{R}$ (resp. $\mathbb{R}_{0}^{-}$,$\mathbb{R}_{0}^{+}$) with
constants $K>0$ and $\alpha\geq 0$, if the spectrum $\Sigma (A)$ (resp. $\Sigma^{-}(A)$,$\Sigma^{+}(A)$) is non empty, then 
$$
\Sigma(A) \subset [-\alpha,+\infty).
$$
\end{itemize}
\end{remark}

Notice that we explicitly assumed that the spectrum $\Sigma(A)$ is a not empty set. We will see, as a consequence of next results, that the spectra cannot be 
empty when the property of bounded growth $\&$ decay is satisfied.

In case that the exponential dichotomy spectrum is bounded, we will introduce 
the notation 
\begin{equation}
\label{Bornes-Spec}
\Sigma_{\min}=\min \Sigma(A) \quad\textnormal{and} \quad \Sigma_{\max}=\max\Sigma(A).
\end{equation}

A direct consequence of the boundedness of $\Sigma(A)$ is 
that the resolvent $\rho(A)$ is an unbounded set satisfying
$$
(-\infty,\Sigma_{\min}) \cup (\Sigma_{\max},+\infty) \subset \rho (A).
$$ 

In particular, if the spectrum $\Sigma(A)$ is bounded, the resolvent $\rho(A)$ is not connected and has two unbounded connected components, namely, $(-\infty,\Sigma_{\min})$  and $(\Sigma_{\max},+\infty)$. The characterization of the bounded connected components
will be carried out in the next subsection.

\subsection{Basic properties of the resolvents}
In the last subsection we have seen that:

\medskip

\noindent $\bullet$ The Corollary \ref{open-resolvent} states that for any $\lambda\in\rho (A),\rho^{-}(A),\rho^{+}(A)$, there exists a neighborhood $V_{\lambda}$ of $\lambda$
such that for any $\mu\in V_{\lambda}$, the linear system $\dot{x}=[A(t)-\mu\,I]x$ has an exponential dichotomy with the same projection.

\noindent $\bullet$ If the linear system $\dot{x}=A(t)x$ has a bounded exponential dichotomy spectrum, we can deduce that the resolvents $\rho(A)$ and $\rho^{\pm}(A)$ are neither connected nor bounded. 

In this subsection we will explore the consequences of above properties in order to characterize
the spectra in terms of the connected components of its resolvents.
\begin{proposition} 
\label{reso-o}
If $\Sigma (A)\neq \varnothing$  and $[\gamma_1, \gamma_2]\subset\rho (A)$, then the shifted systems $$\dot{x}=[A(t)-\gamma_1 I]x\text{ and } \dot{x}=[A(t)-\gamma_2 I]x$$ have an exponential dichotomy on
$\mathbb{R}$ with projections $P_{\gamma_{1}}$ and $P_{\gamma_{2}}$, which have the same rank.
\end{proposition}

\begin{proof}
The proof will be made in several steps: 

\noindent \textit{Step 1:} As $[\gamma_{1},\gamma_{2}]\subset \rho(A)$, it follows by definition that shifted
systems $\dot{x}=[A(t)-\gamma_{i}I]x$ have an exponential dichotomy on
$\mathbb{R}$ with projections $P_{\gamma_{i}}$ with $i=1,2$.

\medskip

\noindent \textit{Step 2: A covering of $[\gamma_{1},\gamma_{1}].$}
By Corollary \ref{resolvent-open} we know that for any $\lambda\in[\gamma_1, \gamma_2]\subset\rho (A)$, there exists $\delta_\lambda >0$ such that $(\lambda - \delta_{\lambda}, \lambda + \delta_{\lambda})\subset\rho (A)$. This implies that the family of the open sets $\{(\lambda - \delta_{\lambda}, \lambda + \delta_{\lambda}) : \lambda \in [\gamma_1, \gamma_2]\}$ is an open covering of $[\gamma_1, \gamma_2]$.

By using the compactness of $[\gamma_1, \gamma_2]$, we can choose a finite covering $$
[\gamma_1 , \gamma_2]\subseteq \bigcup_{i=1}^{m}(\lambda_i -\delta_{\lambda_i}, \lambda_i+\delta_{\lambda_i}),
$$ 
such that $\lambda_1 <\lambda_2<\ldots <\lambda_m$ and
$$
(\lambda_i -\delta_{\lambda_i}, \lambda_i+\delta_{\lambda_i})\cap (\lambda_{i+1} -\delta_{\lambda_{i+1}}, \lambda_{i+1}+\delta_{\lambda_{i+1}})\neq\varnothing
$$ 
for any $i\in \{1,\ldots,m-1\}$, where
$$
\gamma_{1}\in (\lambda_{1}-\delta_{\lambda_{1}},\lambda_{1}+\delta_{\lambda_{1}}) \quad \textnormal{and} \quad
\gamma_{2}\in (\lambda_{m}-\delta_{\lambda_{m}},\lambda_{m}+\delta_{\lambda_{m}}).
$$

\medskip 

\noindent \textit{Step 3: Study of the covering's first interval.} Without loss of generality, we will assume that
$\lambda_{1}-\delta_{\lambda_{1}}<\gamma_{1}<\lambda_{1}$. By statement b) of Remark \ref{unstable-stable-spaces} we have that
\begin{displaymath}
\mathcal{S}^{\mu}\subseteq \mathcal{S}^{\gamma_{1}}\subseteq \mathcal{S}^{\lambda_{1}} \quad
\textnormal{for any $\lambda_{1}-\delta_{\lambda_{1}}<\mu<\lambda_{1}$}.
\end{displaymath}

As the exponential dichotomy on $\mathbb{R}$ implies the exponential
dichotomy on $\mathbb{R}_{0}^{+}$, we can use the Lemma \ref{carac-stable}, which leads to:
\begin{displaymath}
\dim \textnormal{Im}P_{\mu}\leq \dim \textnormal{Im}P_{\gamma_{1}}\leq \dim \textnormal{Im}P_{\lambda_{1}} \quad
\textnormal{for any $\lambda_{1}-\delta_{\lambda_{1}}<\mu<\lambda_{1}$}.
\end{displaymath}

By Corollary \ref{open-resolvent} we know that $P_{\mu}=P_{\lambda_{1}}$, which implies
that $\dim\textnormal{Im}P_{\gamma_{1}}=\dim \textnormal{Im}P_{\lambda_{1}}$, or equivalently,
$P_{\lambda_{1}}$ and $P_{\gamma_{1}}$ have the same rank.

\medskip 

\noindent \textit{Step 4: Study of the covering's last interval.} By following
the lines of the previous step, it can be proved easily that $P_{\lambda_{m}}$ and $P_{\gamma_{2}}$ have the same rank.

\medskip

\noindent \textit{Step 5: End of proof.} Finally, when choosing 
$$
\mu \in 
(\lambda_i -\delta_{\lambda_i}, \lambda_i+\delta_{\lambda_i})\cap (\lambda_{i+1} -\delta_{\lambda_{i+1}}, \lambda_{i+1}+\delta_{\lambda_{i+1}}),
$$
we have that $\dot{x}=[A(t)-\mu I]x$ has an exponential dichotomy on $\mathbb{R}$ with
projections $P_{\lambda_{i}}$ and $P_{\lambda_{i+1}}$. Then, similarly as before, we conclude
that $\text{Rank}\, P_{\lambda_{i}}= \text{Rank}\, P_{\lambda_{i+1}}$ for any
$i\in \{1,\ldots,m-1\}$ and it follows that
$$
\text{Rank}\, P_{\gamma_{1}}= \text{Rank}\, P_{\lambda_{1}}=\ldots =\text{Rank}\, P_{\lambda_{m}}=
\text{Rank} P_{\gamma_{2}},
$$
and the result follows.
\end{proof}

\begin{remark}
%\label{BFP}
If $\Sigma^{+}(A)$ and $\Sigma^{-}(A)$ are non--empty and the property $[\gamma_1, \gamma_2]\subset\rho (A)$ is verified, then the shifted systems $$\dot{x}=[A(t)-\gamma_1 I]x\text{ and } \dot{x}=[A(t)-\gamma_2 I]x$$ also have an exponential dichotomy on
$\mathbb{R}$ with projections $P_{\gamma_{1}}$ and $P_{\gamma_{2}}$ having the same rank. Indeed:
\begin{itemize}
\item[a)] The proof for $\Sigma^{+}(A)$ is identical since Lemma \ref{carac-stable} still can be used.
\item[b)] The proof for $\Sigma^{-}(A)$ changes at the step 3, where it can be deduced that
\begin{displaymath}
\mathcal{U}^{\lambda_{1}}\subseteq \mathcal{U}^{\gamma_{1}}\subseteq \mathcal{U}^{\mu} \quad
\textnormal{for any $\lambda_{1}-\delta_{\lambda_{1}}<\mu<\lambda_{1}$}.
\end{displaymath}

Now, by using Lemma \ref{carac-unstable}, it follows in a similar ay that $\dim \ker P_{\lambda_{1}}=\dim \ker P_{\gamma_{1}}$, and by the Rank--Nullity Theorem we can see that
$P_{\lambda_{1}}$ and $P_{\gamma_{1}}$ have the same rank. The rest of the proof follows as before.
\end{itemize}
\end{remark}

The next result can be seen as a converse of the previous one

\begin{proposition}
\label{reso-o2}
If $\gamma_1 , \gamma_2\in\rho (A)$ are such that 
$$
\gamma_1<\gamma_2
\quad \textnormal{with} \quad
\textnormal{\text{Rank}}(P_{\gamma_1})=\textnormal{\text{Rank}}(P_{\gamma_2}),
$$
then $[\gamma_1, \gamma_2]\subset \rho(A)$.
\end{proposition}

\begin{proof}
The proof will be made in several steps:\medskip

\noindent\textit{Step 1:} Note that $\mathcal{U}^{\gamma_2}\subseteq \mathcal{U}^{\gamma_1}$ is a consequence of $\gamma_1 <\gamma_2$ combined with Remark \ref{unstable-stable-spaces}. Now, the identity $\textnormal{Rank}(P_{\gamma_1})=\textnormal{Rank}(P_{\gamma_2})$ implies that $\dim \ker(P_{\gamma_1})=\dim \ker(P_{\gamma_2})$ and Lemma \ref{carac-unstable} says that $\dim(\mathcal{U}^{\gamma_1})=\dim(\mathcal{U}^{\gamma_2})$.

As $\mathcal{U}^{\gamma_2}$ is a subspace of $\mathcal{U}^{\gamma_1}$ having the same dimension, it follows that $\mathcal{U}^{\gamma_2}=\mathcal{U}^{\gamma_1}$ and by using again Lemma \ref{carac-unstable}
we have that $\ker P_{\gamma_{1}}=\ker P_{\gamma_{2}}$.

\medskip

\noindent\textit{Step 2:} As $\gamma_2\in\rho (A)$, let us recall that the system  $$\dot{x}=[A(t)-\gamma_2 I]x$$ has an exponential dichotomy on $\mathbb{R}_{0}^{-}$, namely, there exist
a projection $P_{\gamma_{2}}$ and a pair
of positive constants $K_{2}$ and $\alpha_{2}$ such that:
\begin{equation}\label{ec:4.1}
\left\{\begin{array}{rlcr}
\Vert X(t)P_{\gamma_2}X^{-1}(s)\Vert &\leq K_2e^{\gamma_2 (t-s)}e^{-\alpha_2 (t-s)} &\text{if}&t \geq s\\
\Vert X(t)(I-P_{\gamma_2})X^{-1}(s)\Vert &\leq K_2e^{\gamma_2 (t-s)}e^{-\alpha_2 (s-t)} &\text{if}&  t \leq s.
\end{array}\right.
\end{equation}

\noindent\textit{Step 3:} At step 1 we verified that $\ker(P_{\gamma_1})=\ker(P_{\gamma_2})$. Now, by using Lemma \ref{2proj}
with $P_{\gamma_{2}}=\overline{P}$ and $P_{\lambda}=P_{\gamma_{1}}$, we can see that 
$$\dot{x}=[A(t)-\gamma_1 I]x$$ has also an exponential dichotomy on $\mathbb{R}_{0}^{-}$ with projection $P_{\gamma_2}$:

\begin{equation}\label{ec:4.2}
\left\{\begin{array}{rlcr}
\Vert X(t)P_{\gamma_2}X^{-1}(s)\Vert &\leq K_1e^{\gamma_1 (t-s)}e^{-\alpha_1 (t-s)} &\text{if}&t\geq s\\
\Vert X(t)(I-P_{\gamma_2})X^{-1}(s)\Vert &\leq K_1e^{\gamma_1 (t-s)}e^{-\alpha_1 (s-t)} &\text{if}& t \leq s.
\end{array}\right.
\end{equation}

\noindent\textit{Step 4:} If $\gamma\in[\gamma_1, \gamma_2]$ it is easy to verify that 
$$
e^{\gamma_1 (t-s)}\leq e^{\gamma (t-s)} \text{ if } t\geq s
$$
and
$$
e^{\gamma_2 (t-s)}\leq e^{\gamma (t-s)} \text{ if } s\geq t.
$$

We will pair the first inequality of (\ref{ec:4.2}) and the second inequality of (\ref{ec:4.1}) by considering $K=\max \{K_1, K_2\}$ and $\alpha =\min \{\alpha_1, \alpha_2\}$:

\begin{equation*}
\left\{\begin{array}{rlcr}
\Vert X(t)P_{\gamma_2}X^{-1}(s)\Vert &\leq Ke^{\gamma (t-s)}e^{-\alpha (t-s)} &\text{if}&t\geq s\\
\Vert X(t)(I-P_{\gamma_2})X^{-1}(s)\Vert &\leq Ke^{\gamma (t-s)}e^{-\alpha (s-t)} &\text{if}& t \leq s,
\end{array}\right.
\end{equation*}
and we have that $\dot{x}=[A(t)-\gamma I]x$ has an exponential dichotomy on $\mathbb{R}_{0}^{-}$ with projection $P_{\gamma_{2}}$ which implies that $\gamma \in \rho^{-}(A)$ and $[\gamma_1, \gamma_2]\subset \rho^{-}(A)$.

\medskip
\noindent \textit{Step 5:} Similarly as done in step 1, it can be proved that $\mathcal{S}^{\gamma_{1}}=\mathcal{S}^{\gamma_{2}}$ and the Lemma \ref{carac-stable} implies the identity $\textnormal{Im} P_{\gamma_{1}}=\textnormal{Im}P_{\gamma_{2}}$. Then, by using 
Lemma \ref{2proj+} we can replicate the steps 2, 3, and 4 and deduce that, for any $\gamma \in [\gamma_{1},\gamma_{2}]$, the shifted system has an exponential dichotomy on $\mathbb{R}_{0}^{+}$
with the same projector $P_{\gamma_{2}}$.

\medskip 

\noindent  \textit{Step 6:} We have proved that, for any $\gamma\in [\gamma_{1},\gamma_{2}]$, the shifted system $\dot{x}=[A(t)-\gamma I]x$ has an exponential dichotomy on the both half lines
with the same projection $P_{\gamma_{2}}$. Then by using Proposition \ref{FullDico} from Chapter 3,
the exponential dichotomy is verified on $\mathbb{R}$ and the result follows.
 \end{proof}

\begin{remark}
An attentive reading of the above result shows that the result is also true when considering
$\rho^{-}(A)$ and $\rho^{+}(A)$.
\end{remark}

\begin{corollary}
If the linear system \eqref{LinCap3} has the property of bounded growth $\&$ decay on 
$\mathbb{R}$ (resp. $\mathbb{R}_{0}^{-}$,$\mathbb{R}_{0}^{+}$) with constants $K>0$
and $\alpha \geq 0$ then its corresponding
spectrum $\Sigma(A)$ (resp. $\Sigma^{-}(A)$,$\Sigma^{+}(A)$) is non empty.
\end{corollary}

\begin{proof}
We will only consider the case $\Sigma(A)$ since the other cases can be addressed
in a similar way. The proof will be made by contradiction. In fact, if we assume $\Sigma(A)=\varnothing$, it follows
that $\rho(A)=\mathbb{R}$. Then, by Proposition \ref{reso-o}  we have that for any $\gamma_{1}<\gamma_{2}$ it follows that  $\textnormal{Rank}P_{\gamma_{1}}=\textnormal{Rank}P_{\gamma_{2}}$ since
$[\gamma_{1},\gamma_{2}]\subset \rho(A)$.

On the other hand, by observing the proof of Proposition \ref{EDECS} and considering constants $\gamma_{1}<-\alpha$
and $\gamma_{2}>\alpha$, we will have that $P_{\gamma_{1}}=0$ and $P_{\gamma_{2}}=I$, which leads to 
$\textnormal{Rank}P_{\gamma_{1}}=0<n=\textnormal{Rank}P_{\gamma_{2}}$, obtaining a contradiction.
\end{proof}

In case that the exponential dichotomy spectrum $\Sigma(A)$ has the bounded growth $\&$ decay property, the following result provides
information about the unbounded connected subsets on $\rho(A)$, namely, 
$(-\infty,\Sigma_{\min})$ and $(\Sigma_{\max},+\infty)$, where $\Sigma_{\min}$
and $\Sigma_{\max}$ are the upper and lower bounds of the spectrum described in (\ref{Bornes-Spec}).

\begin{theorem}
\label{espectro-proj}
If the linear system \eqref{LinCap3} has the property of bounded growth $\&$ decay on $\mathbb{R}$ (resp. $\mathbb{R}_{0}^{+}$,$\mathbb{R}_{0}^{-}$), it follows that:
\begin{enumerate}
\item[i)] For any $\lambda < \Sigma_{\min}$, the shifted system $\dot{x}=[A(t)-\lambda I]x$ has an exponential dichotomy on $\mathbb{R}$ (resp. $\mathbb{R}_{0}^{+}$,$\mathbb{R}_{0}^{-}$) with projection $P_{\lambda}=0$. That is, there exists $\delta >0$ such that 
\begin{equation}
\label{BGTP1}
\Vert X(t) e^{-\lambda t}X^{-1}(s)e^{\lambda s}\Vert \leq Ke^{-\delta (s-t)}, \text{ for any } t\leq s.
\end{equation}

\item[ii)] For any $\lambda > \Sigma_{\max}$, the shifted system $\dot{x}=[A(t)-\lambda I]x$ has an exponential dichotomy on $\mathbb{R}$ with projection $P_{\lambda}=I$. That is, there exists $\delta >0$ such that 
\begin{equation*}
%\label{BGTP2}
\Vert X(t) e^{-\lambda t}X^{-1}(s)e^{\lambda s}\Vert \leq Ke^{-\delta (t-s)}, \text{ for any } t\geq s.
\end{equation*}
\end{enumerate}
\end{theorem}

\begin{proof}
We will only consider the case of $\mathbb{R}$ since the other ones can be addressed similarly.
Now, as the linear system (\ref{LinCap3}) has a bounded growth $\&$ decay on $\mathbb{R}$, by Proposition \ref{EDECS}
we have the existence of $\alpha\geq 0$ such that the spectrum verifies $\Sigma(A)\subset [-\alpha,\alpha]$ or
$\Sigma(A)=\{0\}$, as appropriate.

The proof will be made only for the case $\lambda>\Sigma_{\max}$ while the other case is left to the reader. Firstly notice that if $\lambda\geq \alpha+\varepsilon$, with $\varepsilon>0$ arbitrarily small,
the case a) from the proof of Proposition \ref{EDECS} says that (\ref{BGTP1}) is verified, that is, the shifted system $\dot{x}=[A(t)-\lambda I]x$
has an exponential dichotomy on $\mathbb{R}$ with projection $P_{\lambda}=I$.

Secondly, let $\lambda \in (\Sigma_{\max},\alpha+\varepsilon)$ and consider $\gamma>\alpha+\varepsilon$.
Notice that as $[\lambda,\gamma]\subset (\Sigma_{\max},+\infty)\subset \rho(A)$, the Proposition \ref{reso-o} implies that $P_{\lambda}$ and $P_{\gamma}=I$ have the same rank and the result follows since $P_{\lambda}=I$.
\end{proof}

The above result has a converse one:
\begin{theorem}
\label{BSIBG}
Assume that the spectrum $\Sigma(A)$ of the linear system \eqref{LinCap3} is non empty and bounded. If:
\begin{itemize}
\item[a)] For any $\lambda \in (\Sigma_{\max},+\infty)$, the shifted system $\dot{x}=[A(t)-\lambda I]x$
has an exponential dichotomy on $\mathbb{R}$ with projector $P_\lambda=I$,
\item[b)] For any $\gamma \in (-\infty,\Sigma_{\min})$, the shifted system $\dot{x}=[A(t)-\gamma I]x$
has an exponential dichotomy on $\mathbb{R}$ with projector $P_\gamma=0$,
\end{itemize}
then the linear system a bounded growth $\&$ decay on $\mathbb{R}$.
\end{theorem}

\begin{proof}
Let us consider a fixed $\lambda>\Sigma_{\max}$, then there exist positive constants $K_{\lambda}$
and $\alpha_{\lambda}$ such that
\begin{displaymath}
 ||X_{\lambda}(t,s)||\leq K_{\lambda}e^{-\alpha_{\lambda} (t-s)}  \quad \textnormal{for any $t\geq s$}.
 \end{displaymath}

By using $\alpha_{\lambda}>0$ and considering $\Delta>0$ 
the above estimation implies
\begin{displaymath}
\begin{array}{rcll}
 ||X(t,s)|| &\leq & K_{\lambda}e^{\lambda (t-s)} & \textnormal{for any $t\geq s$}, \\
 &\leq & K_{\lambda}e^{\max\{\Delta,\lambda\} |t-s|} & \textnormal{for any $t\geq s$}.
 \end{array}
\end{displaymath}

Similarly, we will consider a fixed $\gamma<\Sigma_{\min}$, then there exist positive constants $K_{\gamma}$
and $\alpha_{\gamma}$ such that
\begin{displaymath}
 ||X_{\gamma}(t,s)||\leq K_{\gamma}e^{-\alpha_{\gamma} (s-t)}  \quad \textnormal{for any $t\leq s$}.
 \end{displaymath}

As $\alpha_{\gamma}>0$, by choosing again $\Delta>0$ 
the above estimation implies
\begin{displaymath}
\begin{array}{rcll}
 ||X(t,s)|| &\leq & K_{\gamma}e^{\gamma (s-t)} & \textnormal{for any $t\leq s$}, \\
 &\leq & K_{\gamma}e^{\max\{\Delta,\gamma\} |t-s|} & \textnormal{for any $t\leq s$}.
 \end{array}
\end{displaymath}

By gathering the above estimations we conclude that
\begin{displaymath}
||X(t,s)||\leq Ke^{\beta|t-s|} \quad \textnormal{for any $t,s\in \mathbb{R}$}    
\end{displaymath}
with $K=\max\{K_{\lambda},K_{\gamma}\}$ and $\beta=\max\{\Delta,\gamma\}+\max\{\Delta,\lambda\}$,
and the result follows.
\end{proof}

\begin{remark}
Let us observe that if, in the above result, only the hypotesis a) is verified,
we can deduce the property of bounded growth on $\mathbb{R}$. Analogously,
if only b) is verified we can deduce the property of bounded decay on $\mathbb{R}$.
\end{remark}

\subsection{Structure of the resolvent and the spectrum}

We know that the resolvents $\rho(A)$,$\rho^{+}(A)$ and $\rho^{-}(A)$ are open sets.
By using the results of the previous subsection, we will prove that this resolvents 
are a finite union of open intervals.

\begin{definition}
\label{componente}
Under the assumption that $\rho(A)$ (resp. $\rho^{\pm}(A)$) is a non empty set. For any $\lambda\in\rho(A)$
(resp. $\lambda\in\rho^{\pm}(A)$), the set $C_\lambda\subset \rho(A)$ (resp. $C_\lambda\subset \rho^{\pm}(A)$) denotes the \textbf{maximal open interval} containing $\lambda$. 
\end{definition}

We can deduce several consequences of the above definition:
\begin{itemize}
    \item[(i)] As Proposition \ref{open-resolvent} ensures that $\rho(A)$ is an open set in $\mathbb{R}$, we know 
that $C_{\lambda}\neq \varnothing$;
     \item[(ii)] $C_{\lambda_{1}}\cap C_{\lambda_{2}}\neq \varnothing$ if and only if $C_{\lambda_{1}}=C_{\lambda_{2}}$;
     \item[(iii)] Let us denote $\lambda_{1}\sim \lambda_{2}$ if $C_{\lambda_{1}}=C_{\lambda_{2}}$ and note that
this defines an equivalence relation on $\mathbb{R}$;
\item[(iv)] If $\Sigma(A)$ (resp $\Sigma^{\pm}(A)$) is lowerly bounded and $\lambda <\Sigma_{\min}$, it follows that $C_{\lambda}=(-\infty,\Sigma_{\min})$;
\item[(v)] If $\Sigma(A)$ (resp $\Sigma^{\pm}(A)$) is upperly bounded and $\lambda >\Sigma_{\max}$, it follows that $C_{\lambda}=(\Sigma_{\max},+\infty)$;
\item[(vi)] If $\Sigma(A)=\varnothing$ it follows that $C_{\lambda}=(-\infty,+\infty)$ for any $\lambda\in \mathbb{R}$. 
\end{itemize}

The statement (iii) implies that $\rho(A)$ can be partitioned by the equivalence classes
associated to equivalence relation $\sim$. Given $\lambda\in \rho(A)$ is equivalence class
will be denoted by
\begin{displaymath}
[\lambda]:=\left\{\mu \in \rho(A)\colon \mu \sim \lambda\right\}    
\end{displaymath}
and the set of equivalence classes will be denoted by $\rho(A)/\sim $.

%\begin{corollary}
%\label{const-conexa}
%\textcolor{red}{Corregir!!!}   Given $\lambda \in \rho(A)$, it follows that the system
%$$
%\dot{x}=[A(t)-\mu I]x
%$$
%has an exponential dichotomy on $J=\mathbb{R},\mathbb{R}^{-},\mathbb{R}^{+}$ with the same projection %$P_{\lambda}$ for any $\mu \in C_{\lambda}$.
%\end{corollary}

%\begin{proof}
%Let us consider $\mu_{1}\in C_{\lambda}$ and $\mu _{2}\in C_{\lambda}$ with $\mu_{1}<\mu_{2}$ such %that the systems
%$$
%\dot{x}=[A(t)-\mu_{1}I]x \quad \textnormal{and} \quad
%\dot{x}=[A(t)-\mu_{2}I]x
%$$
%have exponential dichotomy on $\mathbb{R}$. In addition, as $[\mu_{1},\mu_{2}]\subset C_{\lambda}\subset \rho(A)$, Proposition \ref{reso-o} implies that the above systems have exponential dichotomy with the same projection and the result follows.
%\end{proof}

\begin{lemma}
\label{coment-rangos}
Let $\lambda_1\in\rho (A)$ and $\lambda_2\in\rho (A)$ such that $\lambda_1<\lambda_2$, then \textnormal{\text{Rank}}$(P_{\lambda_1})\leq$ \textnormal{\text{Rank}}$(P_{\lambda_2})$. Moreover, the inequality is strict if 
the respective sets $C_{\lambda_1}$ and $C_{\lambda_2}$ are disjoint.
\end{lemma}

\begin{proof}
By using statements b) and c) of Remark \ref{unstable-stable-spaces}, we have that 
$\mathcal{S}^{\lambda_{1}}\subseteq \mathcal{S}^{\lambda_{2}}$ and  $\mathcal{U}^{\lambda_2}\subseteq\mathcal{U}^{\lambda_1}$. Now, by using Lemma \ref{carac-stable}
or Lemma \ref{carac-unstable}, we 
obtain that  $\dim \textnormal{Im} P_{\lambda_{1}}\leq \dim \textnormal{Im} P_{\lambda_{2}}$ or $\dim \ker \, P_{\lambda_2}\leq \dim \ker \, P_{\lambda_1}$, which implies that $\text{Rank}(P_{\lambda_{1}})\leq \text{Rank}(P_{\lambda_{2}})$.

Firstly, under the assumption that $C_{\lambda_1}\cap C_{\lambda_2} \neq \varnothing$ 
the above stated properties of the $C_{\lambda}$ sets imply that $C_{\lambda_{1}}=C_{\lambda_{2}}$, then
we have that $[\lambda_{1},\lambda_{2}]\subset C_{\lambda_{1}}\subseteq \rho(A)$.
Then, by using Proposition \ref{reso-o} we deduce that Rank$(P_{\lambda_2})=$ Rank$(P_{\lambda_1})$.

Secondly, under the assumption that $C_{\lambda_1}\cap C_{\lambda_2}=\varnothing$, we will prove that Rank$(P_{\lambda_1})<$ Rank$(P_{\lambda_2})$. Indeed, otherwise, if Rank$(P_{\lambda_1})=$ Rank$(P_{\lambda_2})$. Hence, the Proposition \ref{reso-o2} says that $[\lambda_1, \lambda_2]\subset \rho (A)$, obtaining a contradiction with $C_{\lambda_1}\cap C_{\lambda_2}=\varnothing$.
\end{proof}

Notice that the result is still valid when considering $\rho^{+}(A)$ and $\rho^{-}(A)$. In fact, the 
resolvents $\rho^{+}(A)$ and $\rho^{-}(A)$ can be addressed  by using $\mathcal{S}^{\lambda_{1}}\subseteq \mathcal{S}^{\lambda_{2}}$ and $\mathcal{U}^{\lambda_{2}}\subseteq \mathcal{U}^{\lambda_{1}}$ respectively.

%\begin{corollary}
%If $\Sigma(A)\neq \varnothing$, there exist $N\leq n-1$ such that 
%the resolvent $\rho(A)$ has the form
%$$
%\rho(A)=\bigcup\limits_{i=0}^{N+1}I_{i},
%$$
%where $I_{i}$ are open and bounded intervals for $i=1,\ldots,N$ while
%$$
%I_{0}=(-\infty,\min\Sigma(A)) \quad \textnormal{and} \quad I_{N+1}=(\max\Sigma(A),+\infty).
%$$
%\end{corollary}

\begin{theorem}
\label{descripcion-reso}
If the linear system $\dot{x}=A(t)x$ 
has the property of bounded growth $\&$ decay on $\mathbb{R}$ (resp. $\mathbb{R}_{0}^{-}$,$\mathbb{R}_{0}^{+}$), then the resolvent $\rho(A)$ is the union of --at most-- $n+1$ open intervals, where  $n$ is the dimension of the system.
\end{theorem}

\begin{proof}
By using Definition \ref{componente}, we note that
$$
\rho(A)=\bigcup\limits_{\lambda\in \rho(A)}C_{\lambda}=\bigcup\limits_{[\lambda]\in \rho(A)/\sim}C_{\lambda}.
$$

Let us denote $C_{[\gamma]}=(-\infty,\Sigma_{\min})$ and $C_{[\mu]}=(\Sigma_{\max},+\infty)$ for any $\gamma<\Sigma_{\min}$ and $\mu>\Sigma_{\max}$.

This allows to deduce several properties of $\rho(A)$. Firstly, let us consider the function 
\begin{displaymath}
\begin{array}{rcl}
\phi &\colon & \rho(A) \to \{0,1,\ldots,n\}\\
     &       &    \lambda \mapsto \phi(\lambda)= \textnormal{\text{Rank}} \, P_{\lambda}
     \end{array}
\end{displaymath}
and, by Lemma \ref{coment-rangos} we have that $\phi$ is nondecreasing and constant at each set $C_{[\lambda]}\subset \rho(A)$.

By Theorem \ref{espectro-proj} combined with the local properties of $\phi$, we can see that
\begin{itemize}
    \item[a)] The statement i) from Theorem \ref{espectro-proj} implies that $\phi(\gamma)=0$ for any $\gamma <\Sigma_{\min}$,
    \item[b)] the statement ii) from Theorem \ref{espectro-proj} implies that $\phi(\mu)=n$ for any $\mu >\Sigma_{\max}$,
    \item[c)] by a) and b), we know that $\phi^{-1}(n)$ and $\phi^{-1}(0)$ are nonempty sets. On the other hand, for any $k\in \{1,\ldots,n-1\}$ either $\phi^{-1}(k)=\varnothing$ or there exists an equivalence class $[\lambda]\in \rho(A)/\sim$ such that 
    $$
    \phi^{-1}(k)=C_{\lambda},
    $$
    which is unique due to Lemma \ref{coment-rangos},
    \item[d)]  by Lemma \ref{coment-rangos} we have that if $\lambda_{1}<\lambda_{2}$ and $C_{\lambda_{1}}\cap C_{\lambda_{2}}=\varnothing$ then $\phi(\lambda_{1})<\phi(\lambda_{2})$.
\end{itemize}

By summarizing this properties, we can conclude that 
$$
\rho(A)=\bigcup\limits_{i=0}^{n}\phi^{-1}(i),
$$
that is, $\rho(A)$ is the union of at most $n+1$ open intervals.

When considering the cases of bounded growth $\&$ decay on $\mathbb{R}_{0}^{-}$
and $\mathbb{R}_{0}^{+}$, the proof can be addressed in a similar way and 
it is left to the reader.
\end{proof}

The sets $\phi^{-1}(i)$ constructed in the above proof, are open intervals
which usually are named \textit{spectral gaps} in the literature.

\begin{theorem}[Spectral Theorem -- Special case]
\label{TeoEsp1}
If the linear system \eqref{LinCap3} has the property of bounded growth $\&$ decay
on $\mathbb{R}$, then
$$
\Sigma (A)=[a_1, b_1]\cup[a_2,b_2]\cup\ldots\cup[a_{\ell -1},b_{\ell -1}]\cup[a_{\ell} , b_ {\ell}],$$ where $a_1\leq b_1<a_2\leq b_2<\ldots <a_{\ell}\leq b_{\ell}$ with $1\leq \ell\leq n$.

\end{theorem}

\begin{proof}
By Theorem \ref{descripcion-reso}, we know that
$\rho(A)$ is a union of at most $n+1$ open intervals defined by:
$$
\rho(A)=\bigcup\limits_{i=0}^{n}\phi^{-1}(i),
$$
where $\phi^{-1}(i)$ is either $\varnothing$ or an open interval. Then, as $\Sigma(A)=\rho(A)^{c}$ it follows easily that 
$\Sigma(A)$ is the union of at most $n$ closed intervals. 

Firstly, we will assume that  $\phi^{-1}(i)\neq \varnothing$ for all $i\in \{0,\ldots,n\}$. Then, we have that $\Sigma(A)$ is the union of $n$ closed sets:
$$
\Sigma(A)=\bigcup\limits_{i=1}^{n}[a_{i},b_{i}] \quad \textnormal{where $[a_{i},b_{i}]=[\sup\phi^{-1}(i-1),\inf\phi^{-1}(i)]$}.
$$

Secondly, we will assume that  $\phi^{-1}(i) = \varnothing$ for at least  one index $i\in \{0,1,\ldots,n\}$. Then, the number of indexes $i\in \{0,1,\ldots,n\}$ such that $\phi^{-1}(i)\neq \varnothing$
will be $\ell<n$. In consequence, the number of indexes $i\in \{0,1,\ldots,n\}$ such that $\phi^{-1}(i)=\varnothing$ will be $k\geq 1$ with $k+\ell=n+1$. Then we have that $\Sigma(A)$ is the union of $\ell<n$ closed sets: 
$$
\Sigma(A)=\bigcup\limits_{i=1}^{\ell}[a_{i},b_{i}].
$$

\end{proof}

The above result deserves some remarks:

\medskip
\noindent a) For $i=1,\ldots,\ell$, the intervals $[a_{i},b_{i}]$ are called \textit{spectral intervals}.

\medskip

\noindent b) We will see in the next examples that some spectral intervals could be degenerate, that is, $a_{i}=b_{i}$
since the identity $\sup\phi^{-1}(i-1)=\inf\phi^{-1}(i)$ is possible.

\medskip

\noindent c) When $\ell=n$, it is said  that the system has the \textit{full spectrum} property.

\subsection{Illustrative examples}
In order to describe the above results, we will see 
some examples.

\textit{First Example:} Let us consider the scalar linear di\-ffe\-rential equation
\begin{equation}
\label{ejemplo-escalar}
\dot{x}=a(t)x \quad \textnormal{with} \quad a(t)=\frac{1}{1+t^{2}},
\end{equation}
and note that $X(t)=e^{\arctan(t)}$ is a fundamental matrix. As $X(t)$ is bounded on $\mathbb{R}$, the Corollary \ref{u-bounded} from Chapter 2 implies that the above equation has not an exponential dichotomy on $\mathbb{R}$ and consequently $0\in \Sigma(a)$. Moreover, as $|a(t)|\leq 1$ for any $t\in \mathbb{R}$, we know by Corollary
\ref{Cota-BG1} from Chapter 1 that the equation (\ref{ejemplo-escalar}) has the property of 
bounded growth $\&$ decay with constants $K=\alpha=1$ and the 
Proposition \ref{EDECS} implies that $\Sigma(a)\subset [-1,1]$. 

In order to study
\begin{equation}
\label{ejemplo-escalar2}
\dot{x}=[a(t)-\lambda]x \quad \textnormal{with} \quad a(t)=\frac{1}{1+t^{2}},
\end{equation}
let us recall that $X_{\lambda}(t)=e^{\arctan(t)-\lambda\, t}$. Then, for any $\lambda>0$ and $t\geq s$ it follows that
\begin{displaymath}
\begin{array}{rcl}
|X_{\lambda}(t)X_{\lambda}^{-1}(s)| &=& e^{\arctan(t)}e^{-\arctan(s)}e^{-\lambda(t-s)}
\leq  e^{\pi}e^{-\lambda(t-s)},
\end{array}
\end{displaymath}
then (\ref{ejemplo-escalar2}) has an exponential dichotomy on $\mathbb{R}$ with projection $P_{\lambda}=1$ for any $\lambda>0$, which implies that $(0,\infty)\subset \rho(a)$.

Similarly, it can be proved that for any $\lambda<0$ and $t\leq s$, then
$$
|X_{\lambda}(t)X_{\lambda}^{-1}(s)|\leq e^{\pi}e^{-|\lambda|(s-t)}
$$
and (\ref{ejemplo-escalar2}) has an exponential dichotomy on $\mathbb{R}$ with projection $P_{\lambda}=0$. This implies that $(-\infty,0)\subset \rho(A)$ and we can conclude that
$$
\rho(a)=(-\infty,0)\cup (0,+\infty) \quad \textnormal{and} \quad \Sigma(a)=\{0\}.
$$

In addition, by Remark \ref{contentions}, we know that $\Sigma^{+}(a)\subseteq \Sigma(a)=\{0\}$ and
it can be proved that $\Sigma^{+}(a)=\{0\}$. Indeed, otherwise, we will have that $0\in \rho^{+}(a)$ and (\ref{ejemplo-escalar}) would have and exponential dichotomy on $\mathbb{R}_{0}^{+}$. Then,
by Theorem \ref{prop-split} and Remark \ref{ext-split} from Chapter 2 it follows that any nontrivial solution $x(t)=x(0)e^{\arctan(t)}$ of (\ref{ejemplo-escalar}) would have either exponential growth or exponential decay, obtaining a contradiction.

\medskip

\textit{Second Example:} Let us consider the autonomous system
\begin{equation*}
%\label{ex-auto}
\dot{x}=Ax \quad \textnormal{where} \quad A=\left[\begin{array}{ccc}
1 & 0 & 0 \\
0 & 1 & 0 \\
0 & 0 & -1
\end{array}\right].
\end{equation*}

On the one hand, we know that the exponential dichotomy spectrum $\Sigma(A)$ is composed by the numbers $\lambda\in \mathbb{R}$ such that $\dot{x}=[A-\lambda I]x$ has not an exponential dichotomy. On the other hand, by Lemma \ref{ED-Aut-Ex} from Chapter 2 it follows that $\dot{x}=[A-\lambda I]x$ has an exponential dichotomy if and only if the eigenvalues of $A-\lambda I$ have nonzero real part. Finally,
as $\Sigma(A-\lambda I)=\{1-\lambda,-1-\lambda\}$ we can deduce that
$$
\Sigma(A)=\{-1,1\} \quad \textnormal{and} \quad 
\rho(A)=(-\infty,-1)\cup (-1,1)\cup (1,+\infty).
$$

In order to study the function $\phi \colon \rho(A)\to \{0,1,2,3\}$
\begin{itemize}
  \item[a)] If $\lambda >1$. We can prove that  $\dot{x}=[A-\lambda I]x$
 has an exponential dichotomy with projection $P=I$. Then, $\phi(\lambda)=3$
 for any $\lambda>1$.
\item[b)] If $\lambda <-1$. We can prove that  $\dot{x}=[A-\lambda I]x$
 has an exponential dichotomy with projection $P=0$. Then, $\phi(\lambda)=0$
 for any $\lambda<-1$.
    \item[c)] If $\lambda \in (-1,1)$, it is easy to verify that the system
    $\dot{x}=[A-\lambda I]x$ has an exponential dichotomy with the projection
    $$
    P=\left[\begin{array}{ccc}
0 & 0 & 0 \\
0 & 0 & 0 \\
0 & 0 & 1
\end{array}\right].
    $$
    and $\phi(\lambda)=1$ for any $\lambda\in (-1,1)$.
\end{itemize}    

Then we can conclude that 
$$
\phi^{-1}(0)=(-\infty,-1), \quad \phi^{-1}(1)=(-1,1),
\quad \phi^{-1}(2)=\varnothing \quad \textnormal{and} \quad
\phi^{-1}(3)=(1,\infty).
$$

\textit{Third Example:} We will verify that the spectrum of the scalar equation
$$
\dot{x}=a(t)x \quad \textnormal{with $a(t)=t$}
$$
verifies $\Sigma(a)=\mathbb{R}$. In fact, note that if there exists some $\lambda \in \mathbb{R}$ such 
that $\lambda \in \rho(a)$ then
$$
\dot{x}=(t-\lambda)x
$$
has an exponential dichotomy with projection either $P_{\lambda}=1$ or $P_{\lambda}=0$.

As any solution of the perturbed equation is of type
$x_{\lambda}(t)=x_{\lambda}(0)e^{\frac{(t-\lambda)^{2}}{2}}$,
which is not exponentially decreasing to zero when $t\to +\infty$,
we can see that the perturbed equation cannot have an exponential
dichotomy with projection $P_{\lambda}=1$. 

On the other hand, if the system has an exponential dichotomy
with the null projection $P_{\lambda}=0$, we will have the existence
of $K>0$ and $\alpha>0$ such that
\begin{displaymath}
\begin{array}{rcl}
x_{\lambda}(t,s)&=&\displaystyle e^{\frac{(t-\lambda)^{2}-(s-\lambda)^{2}}{2}}
=e^{\frac{(t-s)(t+s-2\lambda)}{2}}\leq Ke^{-\alpha(s-t)} \quad \textnormal{for $t\leq s$},
\end{array}
\end{displaymath}
which is equivalent to
$$
(t-s)(t+s-\lambda)\leq 2\ln(K)-\alpha(s-t) \quad \textnormal{for $t<s$} 
$$
or 
$$
-(t+s-\lambda)\leq \frac{2\ln(K)}{s-t}-\alpha \quad \textnormal{for $t<s$}. 
$$

Now, letting $t\to -\infty$ we will have that $+\infty<-\alpha$, obtaining
a contradiction and $\lambda \notin \rho(a)$ for any $\lambda \in \mathbb{R}$.

\medskip

The last example will be written as a Lemma due to its importance.
\begin{lemma}
%\label{ed-exsu}
Let $a\colon \mathbb{R}_{0}^{+}\to \mathbb{R}$ be a bounded continuous function such that
\begin{displaymath}
\beta^{-}(a):=\liminf\limits_{L,s\to +\infty}\frac{1}{L}\int_{s}^{s+L}a(r)\,dr\quad \textnormal{and} \quad \beta^{+}(a):=\limsup\limits_{L,s\to +\infty}\frac{1}{L}\int_{s}^{s+L}a(r)\,dr, 
\end{displaymath}
then spectrum of the scalar equation
$$
\dot{x}=a(t)x
$$
is described by $\Sigma^{+}(a)=[\beta^{-}(a),\beta^{+}(a)]$.
\end{lemma}

\begin{proof}
See Proposition \ref{SPEC} from Appendix A.
\end{proof}

\subsection{The Spectral Theorem}
An essential assumption of Theorem \ref{TeoEsp1} was that the system (\ref{LinCap3}) has a property of bounded growth $\&$ decay
which implies the boundedness of $\Sigma(A)$. If we drop this assumption by using the Remark \ref{cotas-parciales}
and following the lines of the proof of Theorem \ref{TeoEsp1} we can prove the next result:
\begin{theorem}
\label{TeoSpec2}
Assume that the exponential dichotomy spectrum of linear system \eqref{LinCap3} is non empty, then
\begin{itemize}
\item[a)] If \eqref{LinCap3} has the property of bounded growth on $\mathbb{R}$ then
\begin{displaymath}
\Sigma(A)=\left\{\begin{array}{lc}
& (-\infty,b_{1})\cup [a_{2},b_{2}] \cup \cdots \cup [a_{\ell},b_{\ell}] \\
&\textnormal{or} \\
& \hspace{0.5cm} [a_{1},b_{1}]\cup [a_{2},b_{2}] \cup \cdots \cup [a_{\ell},b_{\ell}].
\end{array}\right.
\end{displaymath}
\item[b)] If \eqref{LinCap3} has the property of bounded decay on $\mathbb{R}$ then
\begin{displaymath}
\Sigma(A)=\left\{\begin{array}{lc}
& \hspace{0.4cm} [a_{1},b_{1}]\cup [a_{2},b_{2}] \cup \cdots \cup [a_{\ell},+\infty) \\
&\textnormal{or} \\
&[a_{1},b_{1}]\cup [a_{2},b_{2}] \cup \cdots \cup [a_{\ell},b_{\ell}].
\end{array}\right.
\end{displaymath}
\item[c)] If \eqref{LinCap3} has neither the property of bounded growth nor bounded decay on $\mathbb{R}$ then
\begin{displaymath}
\Sigma(A)=\left\{\begin{array}{lc}
& \hspace{-0.5cm} (-\infty,b_{1})\cup [a_{2},b_{2}] \cup \cdots \cup [a_{\ell},b_{\ell}], \\\\
& \hspace{0.35cm} [a_{1},b_{1}]\cup [a_{2},b_{2}] \cup \cdots \cup [a_{\ell},+\infty), \\\\
& (-\infty,b_{1}]\cup [a_{2},b_{2}] \cup \cdots \cup [a_{\ell},+\infty), \\
&\textnormal{or} \\
&[a_{1},b_{1}]\cup [a_{2},b_{2}] \cup \cdots \cup [a_{\ell},b_{\ell}].
\end{array}\right.
\end{displaymath}
\end{itemize}
\end{theorem}

\begin{proof}
See the list of exercises.   
\end{proof}

\section{Spectral Manifolds and Block Diagonalization}
The main goal of this subsection is to prove that any linear system
$\dot{x}=A(t)x$ having a bounded growth $\&$ decay on $\mathbb{R}$ (resp. $\mathbb{R}_{0}^{+}$,$\mathbb{R}_{0}^{-}$) is generally kinematically similar to a block diagonal system
\begin{displaymath}
\dot{y}=\left[\begin{array}{cccc}
B_{1}(t) &  &  & \\
 & B_{2}(t) &  & \\
 & & \ddots  &\\
 &  & & B_{\ell}(t)
\end{array}\right]y.
\end{displaymath}

In addition, notice that the above system can be seen as $\ell$ uncoupled subsystems
$\dot{z}=B_{i}(t)z$ for any $i\in \{1,\ldots,\ell\}$. In this context it will be also proved that
$\Sigma(B_{i})$ is the $i$--th spectral interval of $\Sigma(A)$.

\subsection{Spectral manifolds}
Before starting to work on the spectrum, let us recall that if $V_{1}$ and $V_{2}$ are vector subspaces of $\mathbb{R}^{n}$, we denote as uual
$$
V_{1}+V_{2}:=\left\{z\in \mathbb{R}^{n}\colon z=v_{1}+v_{2} \quad \textnormal{where $v_{1}\in V_{1}$ and $v_{2}\in V_{2}$}\right\}.
$$

In addition, the following Lemma will be useful to deduce interesting properties of the spectrum:

\begin{lemma}
Let $E, F$ and $G$ be subspaces of $\mathbb{R}^n$ with $G\subset E$, then $$E\cap (F+G)=(E\cap F)+G.$$
\end{lemma}

\begin{proof}
Firstly, we will prove that $(E\cap F)+G\subseteq E\cap (F+G):$

If $x\in (E\cap F)+G$, then 
$$
x=z+g, \text{ where } z\in E\cap F \text{ and } g \in G.
$$ 

We have to prove that $x\in F+G$ and $x\in E$. Notice that $x\in F+G$ follows directly from the above identity
since $z\in E\cap F\subset F$ and $g\in G$. Now, as $g\in G\subset E$, $z\in E\cap F\subset E$, and $E$ is a subspace, we have that $x=z+g\in E$ and we conclude that $x\in E\cap (F+G)$.

\bigskip

Secondly, we will prove that $E\cap (F+G)\subseteq (E\cap F) +G$:

If $x\in E\cap (F+G)$, it follows that $x\in E$ and $x=f+g$ where $f\in F$ and $g\in G$, which implies that
$x-g=f\in F$. As $g\in G\subset E$ and $x\in E$, we also have $x-g\in E\cap F$. This fact combined with $g\in G$ implies that $x\in (E\cap F)+G$ and the Lemma follows. 
\end{proof}

\begin{theorem}
\label{subesp-inv}
Let us assume that the linear system $\dot{x}=A(t)x$ has an exponential dichotomy spectrum 
described by
$$
\Sigma (A)=[a_1, b_1]\cup [a_2, b_2]\cup \ldots\cup [a_{\ell}, b_{\ell}],
$$ 
where $a_{1}\leq b_1 < a_2 \leq b_2 <\ldots <b_{\ell -1} < a_{\ell} \leq b_{\ell}$ and $\ell \leq n$.

Let us consider the trivial projections $P_{\gamma_0}=0$ and $ P_{\gamma_{\ell}}=I$, where 
$\gamma_{0}\in (-\infty,a_{1})$ and $\gamma_{\ell}=(b_{\ell},+\infty)$. Moreover, for
 $i\in \{1, \ldots , \ell -1\}$, choose $\gamma_i\in (b_i, a_{i+1})$ and projections $P_{\gamma_i}$. Then, the sets $$\mathcal{W}_i=\textnormal{Im}(P_{\gamma_i})\cap \ker(P_{\gamma_{i-1}})$$ are subspaces of $\mathbb{R}^{n}$ such that $$\mathcal{W}_1\oplus \cdots \oplus \mathcal{W}_{\ell}=\mathbb{R}^n.$$
\end{theorem}

\begin{proof}
Firstly, we will prove that $$
\mathcal{W}_i\cap \mathcal{W}_j=0 \text{ for any } 1\leq i \neq j\leq \ell\leq n.
$$

As $\gamma_{0}<\gamma_1 <\gamma_2<\cdots <\gamma_{\ell}$ the statement c) from Remark \ref{unstable-stable-spaces} says that 
$$
\mathcal{U}^{\gamma_{\ell}}\subset \cdots\subset \mathcal{U}^{\gamma_2}\subset \mathcal{U}^{\gamma_1}\subset \mathcal{U}^{\gamma_0},
$$ 
and Lemma \ref{carac-unstable} implies 
$$
\ker (P_{\gamma_{\ell}})\subset \cdots \subset \ker (P_{\gamma_2}) \subset \ker (P_{\gamma_1}). 
$$

As $i\neq j$, we have thar either $i\geq j+1$ or $i\leq j-1$. Without loss of generality, let us assume that $i\leq j-1$. Notice that 
$$
\mathcal{W}_i=\text{Im}(P_{\gamma_i})\cap \ker (P_{\gamma_{i-1}})\subseteq \text{Im}(P_{\gamma_i})
$$ 
and $\ker(P_{\gamma_{j-1}})\subset\ker(P_{\gamma_i})$, then we can deduce
$$
\mathcal{W}_j=\text{Im}(P_{\gamma_j})\cap\ker(P_{\gamma_{j-1}}) \subseteq \ker(P_{\gamma_{j-1}})\subseteq\ker(P_{\gamma_i}).
$$ 

The above properties imply that
$$
\mathcal{W}_{i}\cap \mathcal{W}_{j}\subseteq \ker(P_{\gamma_{i}})\cap \textnormal{Im}(P_{\gamma_{i}})=0.
$$

Secondly, we will prove that
$$
\mathcal{W}_1+\cdots + \mathcal{W}_{\ell}=\mathbb{R}^n.
$$

By using the fact that $\ker(P_{\gamma_0})=\mathbb{R}^n$ we note that 
\begin{align*}
\mathbb{R}^n &= \text{Im}(P_{\gamma_1})+\ker(P_{\gamma_1})\\
&= [\text{Im}(P_{\gamma_1})\cap \ker(P_{\gamma_0})]+\ker(P_{\gamma_1}) \\
&= \mathcal{W}_1+\ker(P_{\gamma_1})\\
&= \mathcal{W}_1 +\{\ker P_{\gamma_1}\cap [\text{Im}(P_{\gamma_2})+\ker(P_{\gamma_2})]\}.
\end{align*}

By using Lemma 4.1 with $E=\ker(P_{\gamma_1}), F=\text{Im}(P_{\gamma_2}) $ and $G=\ker(P_{\gamma_2})$ combined with $G=\ker(P_{\gamma_2})\subset \ker(P_{\gamma_1})=E$, it follows that
\begin{align*}
\mathbb{R}^n&=\mathcal{W}_1+\{[\ker(P_{\gamma_1})\cap\text{Im}(P_{\gamma_2})]+\ker(P_{\gamma_2})\}\\
&=\mathcal{W}_1+\mathcal{W}_2+\ker(P_{\gamma_2})\\
&= \mathcal{W}_1+\mathcal{W}_2 + \ker(P_{\gamma_2})\cap[\text{Im}(P_{\gamma_3})+\ker(P_{\gamma_3})]
\end{align*}
and the identity $\mathcal{W}_1+\ldots +\mathcal{W}_{\ell}=\mathbb{R}^n$ can be proved recursively.
\end{proof}

\begin{remark}
\label{Spec+Dico}
Under the assumptions of Theorem \ref{subesp-inv} and supposing that $\dot{x}=A(t)x$ has an exponential dichotomy on $\mathbb{R}$, it follows by definition that $0\notin \Sigma(A)$ and consequently $0$ must be an element of some spectral gap. By the Theorem \ref{espectro-proj}, we will consider three cases:

\noindent a) If $0\in (b_{\ell},+\infty)$, or equivalently $\Sigma(A)\subset (-\infty,0)$, then the linear system has an exponential dichotomy with the trivial projection $P=I$ and
$$
||X(t,s)||\leq Ke^{-\alpha(t-s)} \quad \textnormal{with $t\geq s$}.
$$

\noindent b) If $0 \in (-\infty,a_{1})$, or equivalently $\Sigma(A)\subset (0,+\infty)$, then the linear system has an exponential dichotomy with the trivial projection $P=0$ and
$$
||X(t,s)||\leq Ke^{-\alpha(s-t)} \quad \textnormal{with $s\geq t$}.
$$

\noindent c) If  $0\in (b_{i-1},a_{i})$ for some $i=2,\ldots,\ell$, or equivalently
$$
\bigcup\limits_{j=1}^{i-1}[a_{j},b_{j}]\subset (-\infty,0)  \quad \textnormal{and} \quad
\bigcup\limits_{j=i}^{\ell}[a_{j},b_{j}]\subset (0,\infty), 
$$
then the linear system has an exponential dichotomy  with a nontrivial projection $P$ having \text{Rank}$(P)=\text{Rank}(P_{\gamma})$ for any 
$\gamma \in (b_{i-1},a_{i})$. Moreover, there exists
$i-1$ spectral intervals to the left of $0$ and $\ell-i+1$ spectral intervals to the right.
\end{remark}

\subsection{A Block Diagonalization Result}

The following Lemmas will provide useful tools for the block diagonalization
result.

\begin{lemma}
%\label{INVTRA}
If the system $\dot{x}=A(t)x$ has bounded spectrum
$$
\Sigma (A)=[a_1, b_1]\cup[a_2,b_2]\cup\cdots\cup[a_{\ell -1}, b_{\ell -1}]\cup [a_{\ell}, b_{\ell}],$$ 
and $\lambda \notin \Sigma(A)$, then it follows that
$$
\Sigma (A-\lambda I)=[a_1-\lambda , b_1-\lambda ]\cup [a_2-\lambda , b_2-\lambda ]\cup\ldots\cup[a_{\ell -1} -\lambda , a_{\ell -1}-\lambda ]\cup[a_{\ell}-\lambda , b_{\ell}-\lambda ].
$$
\end{lemma}

\begin{proof}
By definition of spectrum, we know that 
\begin{displaymath}
\begin{array}{rcl}
\Sigma (A-\lambda I)&=&\left\{\mu\in\mathbb{R} : \dot{x}=[A(t)-(\lambda +\mu ) I]x \text{ has not an exponential dichotomy on $\mathbb{R}$} \right\}\\\\
&=&\left\{\mu\in\mathbb{R} :\lambda +\mu  \in \Sigma(A) \right\}.
\end{array}
\end{displaymath}

Notice that 
\begin{align*}
\mu\in\Sigma (A-\lambda I) &\Leftrightarrow (\mu +\lambda)\in\Sigma (A)\\
&\Leftrightarrow a_j\leq\mu +\lambda \leq b_j \text{ for some}\, j=1, \ldots, \ell\\
&\Leftrightarrow a_j -\lambda \leq\mu\leq b_j-\lambda,
\end{align*}
and the results follows.
\end{proof}

\begin{lemma}
\label{similarite1}
If the systems
\begin{equation}
\label{KSa}
\dot{x}=A(t)x \quad \textnormal{and} \quad 
\dot{y}=B(t)y
\end{equation}
are generally kinematically similar, then
the systems
\begin{equation}
\label{KSb}
\dot{x}=[A(t)-\lambda I]x \quad \textnormal{and} \quad 
\dot{y}=[B(t)-\lambda I]y
\end{equation}
are also generally kinematically similar for any $\lambda\in\mathbb{R}$.
\end{lemma}

\begin{proof}
As the systems (\ref{KSa}) are generally kinematically similar, by Definition
\ref{wsimilarite} from Chapter 1 there exists a matrix $Q(t)$ which verifies the matrix differential equation $\dot{Q}(t)=A(t)Q(t)-Q(t)B(t)$ and the change of variables $y=Q^{-1}(t)x$ transforms the first system of (\ref{KSa}) into the second one.

Now, notice that
\begin{displaymath}
\begin{array}{rcl}
\dot{Q}(t)&=& A(t)Q(t)-\lambda Q(t)I-Q(t)B(t)+\lambda Q(t)I\\
&=&\{A(t)-\lambda I\}Q(t)-Q(t)\{B(t)-\lambda I\}
\end{array}
\end{displaymath}
and we can easily verify that the change of variable $y=Q^{-1}(t)x$ transforms the first system of (\ref{KSb}) into the second one.
\end{proof}

\begin{lemma}
\label{PKSSPEC}
If the systems \eqref{KSa} are generally kinematically similar, then its spectra are the same, namely, $\Sigma (A)=\Sigma (B)$.
\end{lemma}
\begin{proof}
Firstly, by hypothesis combined with Lemma \ref{similarite1} we know that the systems (\ref{KSb}) are generally kinematically similar. Secondly, as the kinematical similarity is a symmetric relation, 
the Theorem \ref{KSDE} from Chapter 2 implies that $\lambda\in\rho (A)$ if and only if $\lambda\in\rho (B)$, which is equivalent to $\Sigma (A)=\Sigma (B)$ and the Theorem follows.
\end{proof}

If $\dot{x}=A(t)x$ has an exponential dichotomy 
with non trivial projection, from Corollary \ref{Redu2} from Chapter 4,
we know that this linear system is generally kinematically similar to
\begin{equation}
\label{bloque}
\left(\begin{array}{c}
\dot{y}_{1}\\
\dot{y}_{2}
\end{array}\right)
=\left[\begin{array}{cc}
B_{1}(t) & 0\\\\
0        & B_{2}(t)
\end{array}\right]
\left(\begin{array}{c}
y_{1}\\
y_{2}
\end{array}\right),
\end{equation}
where $B_{1}(t)\in M_{n_{1}}(\mathbb{R})$, $B_{2}(t)\in M_{n_{2}}(\mathbb{R})$
and $n_{1}+n_{2}=n$. Moreover, Lemma \ref{PKSSPEC} implies that both systems have the same exponential
dichotomy spectrum, namely
\begin{displaymath}
\Sigma(A)=\Sigma\left(\begin{array}{cc}
B_{1}(t) & 0\\
0        & B_{2}(t)
\end{array}\right).
\end{displaymath}

The next result provides a sharper characterization of
$\Sigma(A)$ in terms of the spectra $\Sigma(B_{1})$ and $\Sigma(B_{2})$
of the subsystems
\begin{equation}
\label{sub-diag-FF}
\dot{y}_{1}=B_{1}(t)y_{1} \quad \textnormal{and} \quad \dot{y}_{2}=B_{2}(t)y_{2}.
\end{equation}

\begin{lemma}
\label{dec-spec-22}
If $\dot{x}=A(t)x$ has bounded spectrum
$$
\Sigma(A)=[a_{1},b_{1}]\cup [a_{2},b_{2}]\cup \ldots \cup [a_{\ell},b_{\ell}]
$$
and also has an exponential dichotomy with non trivial projection, then it follows that:
\begin{itemize}
\item[a)] The exponential dichotomy spectrum $\Sigma(A)$ and the exponential dichotomy spectrum of the system \eqref{bloque} verify
\begin{equation}
\label{Id-Spect-T}
\Sigma(A)=\Sigma\left(\left[\begin{array}{cc}
B_{1} & 0\\
0        & B_{2}
\end{array}\right]\right)=\Sigma(B_{1})\cup \Sigma(B_{2});
\end{equation}
\item[b)] There exists a spectral gap $(b_{j-1},a_{j})\subset \rho(A)$ 
containing $0$  such that 
$$
\Sigma(B_{1})=[a_{1},b_{1}]\cup \cdots \cup [a_{j-1},b_{j-1}] \quad \textnormal{and} \quad
\Sigma(B_{2})=[a_{j},b_{j}]\cup \cdots \cup [a_{\ell},b_{\ell}].
$$
\end{itemize}
\end{lemma}
\begin{proof}
As we stated above, the first identity of (\ref{Id-Spect-T}) follows from Lemma \ref{PKSSPEC} combined with the reducibility property. Then, we only 
need to prove that
\begin{displaymath}
\Sigma\left(\left[\begin{array}{cc}
B_{1} & 0\\
0        & B_{2}
\end{array}\right]\right)=\Sigma(B_{1})\cup \Sigma(B_{2}).
\end{displaymath}

Firstly, the property
\begin{displaymath}
\Sigma\left(\left[\begin{array}{cc}
B_{1} & 0\\
0        & B_{2}
\end{array}\right]\right) \subset \Sigma(B_{1})\cup \Sigma(B_{2}).
\end{displaymath}
will be proved by contradiction. In fact, if we assume that
$$
\lambda \in \Sigma\left(\left[\begin{array}{cc}
B_{1} & 0\\
0        & B_{2}
\end{array}\right]\right)  \quad \textnormal{but} \quad \lambda \notin  \Sigma(B_{1})\cup \Sigma(B_{2}).
$$
we will have that
\begin{displaymath}
\begin{array}{rcl}
 \lambda \notin  \Sigma(B_{1})\cup \Sigma(B_{2}) &\Rightarrow & 
 \lambda \in  [\Sigma(B_{1})\cup \Sigma(B_{2})]^{c} \\\\
  &\Rightarrow & 
 \lambda \in  [\Sigma(B_{1})]^{c}\cap [\Sigma(B_{2})]^{c} \\\\
  &\Rightarrow & 
 \lambda \in  \rho(B_{1})\cap \rho(B_{2})
\end{array}
\end{displaymath}
which implies that $\lambda$ is such that the linear systems
$$
\dot{y}_{1}=[B_{1}(t)-\lambda I_{n_{1}}]y_{1} \quad \textnormal{and} \quad \dot{y}_{2}=[B_{2}(t)-\lambda I_{n_{2}}]y_{2}
$$
have an exponential dichotomy on $\mathbb{R}$ with positive constants $K_{i}$, $\alpha_{i}$
($i=1,2$) and
projections $P_{1}\in M_{n_{1}}(\mathbb{R})$
and $P_{2}\in M_{n_{2}}(\mathbb{R})$ ($n_{1}+n_{2}=n$) respectively. Namely
\begin{equation}
\label{EDB10}
\left\{\begin{array}{rcl}
||Y_{1}(t)e^{-\lambda t}P_{1}Y_{1}^{-1}(s)e^{\lambda s}||\leq K_{1}e^{-\alpha_{1}(t-s)} & \textnormal{for $t\geq s$} \,\\
||Y_{1}(t)e^{-\lambda t}[I_{n_{1}}-P_{1}]Y_{1}^{-1}(s)e^{\lambda s}||\leq K_{1}e^{-\alpha_{1}(s-t)} & \textnormal{for $s\geq t$} ,
\end{array}\right.
\end{equation}
and
\begin{equation}
\label{EDB20}
\left\{\begin{array}{rcl}
||Y_{2}(t)e^{-\lambda t}P_{2}Y_{2}^{-1}(s)e^{\lambda s}||\leq K_{2}e^{-\alpha_{2}(t-s)} & \textnormal{for $t\geq s$}\\
||Y_{2}(t)e^{-\lambda t}[I_{n_{2}}-P_{2}]Y_{2}^{-1}(s)e^{\lambda s}||\leq K_{2}e^{-\alpha_{2}(s-t)} & \textnormal{for $s\geq t$},
\end{array}\right.
\end{equation}
where $Y_{1}$ and $Y_{2}$ are the respective fundamental matrices of 
(\ref{sub-diag-FF}). Now, let us define
$$
Y_{\lambda}(t)=\left[\begin{array}{cc}
Y_{1}(t)e^{-\lambda t} & 0\\
0        & Y_{2}(t)e^{-\lambda t}
\end{array}\right] \quad \textnormal{and} \quad P_{0}=\left[\begin{array}{cc}
P_{1} & 0\\
0        & P_{2}
\end{array}\right]. 
$$

By following an approach similar to the carried out in
the proof of Corollary \ref{Redu2} from Chapter 4, it can be verified that $Y_{\lambda}(t)$ is a fundamental matrix
of 
\begin{equation}
\label{bloqueL}
\left(\begin{array}{c}
\dot{y}_{1}\\
\dot{y}_{2}
\end{array}\right)
=\left[\begin{array}{cc}
B_{1}(t)-\lambda I_{n_{1}}& 0\\\\
0        & B_{2}(t)-\lambda I_{n_{2}}
\end{array}\right]
\left(\begin{array}{c}
y_{1}\\
y_{2}
\end{array}\right).
\end{equation}
Morever, by using (\ref{EDB10})--(\ref{EDB20}) it can be deduced
that the above system has an exponential dichotomy with projection
$P_{0}$ and positive constants $K=\max\{K_{1},K_{2}\}$ and $\alpha=\min\{\alpha_{1},\alpha_{2}\}$, which implies that
$$
\lambda \notin \Sigma\left(\left[\begin{array}{cc}
B_{1} & 0\\
0        & B_{2}
\end{array}\right]\right),
$$
obtaining a contradiction and the first contention follows.

Secondly, to prove that
\begin{displaymath}
\Sigma(B_{1})\cup \Sigma(B_{2}) \subset \Sigma\left(\left[\begin{array}{cc}
B_{1} & 0\\
0        & B_{2}
\end{array}\right]\right),
\end{displaymath}
we will prove the equivalent property
\begin{displaymath}
\rho(A)=\rho\left(\left[\begin{array}{cc}
B_{1} & 0\\
0        & B_{2}
\end{array}\right]\right) \subset \rho(B_{1})\cap \rho(B_{2}).
\end{displaymath}

Now, if $\lambda \in \rho(A)$, it must be an element of some spectral gap.
Without loss of generality, we will assume that
$$
\lambda \in (b_{i-1},a_{i})\subset \rho(A)=\rho\left(\left[\begin{array}{cc}
B_{1} & 0\\
0        & B_{2}
\end{array}\right]\right) ,
$$
and we have that (\ref{bloqueL}) has an exponential dichotomy with 
non trivial projection since $(b_{i-1}-\lambda,a_{i}-\lambda)$ is a spectral
gap of (\ref{bloqueL}) containing $0$. Then, we can we can apply Corollary
\ref{Redu2} from Chapter 4 to system (\ref{bloqueL}) with $\lambda \in ]b_{i-1},a_{i}[$
and we will obtain that
$$
\dot{y}_{1}=[B_{1}(t)-\lambda I_{n_{1}}]y_{1} \quad \textnormal{and} \quad \dot{y}_{2}=[B_{2}(t)-\lambda I_{n_{2}}]y_{2}
$$
have an exponential dichotomy, then $\lambda \in \rho(B_{1})\cap \rho(B_{2})$
and the second contention follows, as well as the statement a).

To prove the statement b), as we know that $\dot{x}=A(t)x$ has
an exponential dichotomy with non trivial projection, then by Remark \ref{Spec+Dico}-c) it follows that
there exists $j<\ell$ such that $0\in (b_{j-1},a_{j})$ and then $\Sigma(A)$ verifies
$$
[a_{1},b_{1}]\cup \cdots \cup [a_{j-1},b_{j-1}] \subset (-\infty,0).
$$

Moreover, we can use the statement b) from Corollary \ref{Redu2} of Chapter 4, which implies that $\dot{y}_{1}=B_{1}(t)y_{1}$ and  $\dot{y}_{2}=B_{2}(t)y_{2}$
has an exponential dichotomy with trivial projections $I_{n_{1}}$ and $0_{n_{2}}$ respectively, then by Remark \ref{Spec+Dico}-a) and  \ref{Spec+Dico}-b) we can see that $\Sigma(B_{1})\subset (-\infty,0)$ and $\Sigma(B_{2})\subset (0,+\infty)$.

Finally, as $\Sigma(A)=\Sigma(B_{1})\cup \Sigma(B_{2})$, it follows that
$$
[a_{1},b_{1}]\cup \cdots \cup [a_{j-1},b_{j-1}]=\Sigma(B_{1}) \subset (-\infty,0)
$$
and the result follows.
\end{proof}

\begin{theorem}
\label{TDB}
For any linear system \eqref{LA} with bounded 
spectrum of $\ell$ intervals $\{[a_{i},b_{i}]\}_{i=1}^{\ell}$, then there exist $\ell$ matrices $B_i(t)$ where $\Sigma (B_i)=[a_i, b_i]$ such that \eqref{LA} is generally kinematically similar to the system
\begin{equation*}%\label{ec:7.2}
\dot{x}=\left[\begin{array}{cccc}
B_{1}(t) &  &  & \\
 & B_{2}(t) &  & \\
 & & \ddots  &\\
 &  & & B_{\ell}(t)
\end{array}\right]x.
\end{equation*}
\end{theorem}

\begin{proof}
The proof will be divided in several steps:

\noindent\textit{Step 1:} Let us choose $\lambda_{\ell} \in (b_{\ell -1}, a_{\ell})$, then, the linear system 
\begin{equation}\label{ec:7.3}
\dot{x}=[A(t)-\lambda_{\ell} I]x
\end{equation}
has an exponential dichotomy with projection $P_{\lambda_{\ell}}$ with Rank$(P_{\lambda_{\ell}})=m_{\ell}<n$.

Moreover, as $0\in (b_{\ell-1}-\lambda_{\ell},a_{\ell}-\lambda_{\ell})$, Lemma 
\ref{dec-spec-22} applied to the system (\ref{ec:7.3}) implies that it is reducible to 
\begin{equation}\label{ec:7.4}
\dot{y}(t)=\left[\begin{matrix}
B_{1}^{*}(t) & 0\\
0& B_{\ell}^{*}(t)
\end{matrix}\right] y(t),
\end{equation}
where $B_{1}^{*}\in\mathbb{R}^{m_{\ell}\times m_{\ell}}$ and $B_{\ell}^{*}\in\mathbb{R}^{(n-m_{\ell})\times (n-m_{\ell})}$ are such that
$$
\Sigma (A-\lambda_{\ell}I)=\bigcup_{i=1}^{\ell} \left[a_i-\lambda_{\ell}, b_i-\lambda_{\ell}\right]=\Sigma (B_{1}^{*})\cup\Sigma (B_{\ell}^{*}),
$$ 
and verify
\begin{displaymath}
\Sigma(B_{1}^{*})=\bigcup\limits_{i=1}^{\ell-1}\Big[a_{i}-\lambda_{\ell},b_{i}-\lambda_{\ell}\Big] \quad \textnormal{and} \quad
\Sigma(B_{\ell}^{*})=\Big[a_{\ell}-\lambda_{\ell},b_{\ell}-\lambda_{\ell}\Big].
\end{displaymath}

As we know that (\ref{ec:7.3}) is reducible to (\ref{ec:7.4}), by 
Lemma \ref{similarite1} we also can deduce that the linear system $\dot{x}=A(t)x$ is reducible to
\begin{displaymath}
\dot{y}=\left[\begin{matrix}
B_{1}^{*}(t)+\lambda_{\ell}I_{m_{\ell}} & 0\\
0& B_{\ell}^{*}(t)+\lambda_{\ell}I_{n-m_{\ell}}
\end{matrix}\right]y.
\end{displaymath}

Now we define $B_{\ell}(t):= B_{\ell}^{*}(t)+\lambda_{\ell}I_{n-m_{\ell}}$ and use again Lemma \ref{dec-spec-22} to obtain that
\begin{displaymath}
\Sigma(A)=\Sigma(B_{1}^{*}(t)+\lambda_{\ell}I_{m_{\ell}})\cup \Sigma(B_{\ell}(t)),
\end{displaymath}
where
\begin{displaymath}
\Sigma(B_{\ell})=[a_{\ell},b_{\ell}] \quad \textnormal{and} \quad
\Sigma(B_{1}^{*}+\gamma_{\ell}I_{m_{\ell}})=\bigcup\limits_{i=1}^{\ell-1}[a_{i},b_{i}]=\Sigma(\hat{B}_{1}).
\end{displaymath}

Summarizing, we have proved that the linear system $\dot{x}=A(t)x$ is reducible to
\begin{equation}\label{step1-ks-diag}
\dot{y}(t)=\left[\begin{matrix}
\hat{B}_{1}(t) & 0\\
0& B_{\ell}(t)
\end{matrix}\right] y(t),
\end{equation}
by a transformation $Q_{\ell}(t)\in M_{n}(\mathbb{R})$.

\noindent \textit{Step 2:} Note that (\ref{step1-ks-diag}) can be decoupled as
\begin{displaymath}
\begin{array}{rcl}
\dot{y} & = & \hat{B}_{1}(t)y\\
\dot{z} & = & B_{\ell}(t)z
\end{array}
\end{displaymath}
where
$$
\Sigma(\hat{B}_{1})=[a_{1},b_{1}]\cup\ldots \cup [a_{\ell-2},b_{\ell-2}]\cup [a_{\ell-1},b_{\ell-1}].
$$

Now, the subsystem $\dot{y}=\hat{B}_{1}(t)y$ can be treated as in the previous step
by choosing $\lambda_{\ell-1}\in (b_{\ell-2},a_{\ell-1})$ or, equivalently, $0\in  (b_{\ell-2}-\lambda_{\ell-1},a_{\ell-1}-\lambda_{\ell-1})$.

By following the lines of Step 1,
we can prove that $\dot{y}=\hat{B}_{1}y$ is reducible to
$$
\dot{y}=\left[\begin{matrix}
\hat{B}_{2}(t) & 0\\
0& B_{\ell-1}(t)
\end{matrix}\right]y,
$$
by a transformation $\tilde{Q}_{\ell-1}(t)\in M_{m_{\ell}}(\mathbb{R})$
where 
$$
\Sigma(\hat{B}_{2})=\bigcup\limits_{i=1}^{\ell-2}[a_{i},b_{i}] \quad \textnormal{and} \quad 
\Sigma(B_{\ell-1})=[a_{\ell-1},b_{\ell-2}]
$$
and we can see that $\dot{x}=A(t)x$ is kinematically similar to
\begin{displaymath}
\dot{y}=\left[\begin{array}{ccc}
\hat{B}_{2}(t) & &  \\
 & B_{\ell-1}(t) & \\
 & &  B_{\ell}(t)
\end{array}\right]
\end{displaymath}
by a composition
\begin{displaymath}
Q_{\ell-1}(t)=\left[\begin{matrix}
\widetilde{Q}_{\ell-1}(t) & 0\\
0& I_{n-m_{\ell}}
\end{matrix}\right]\circ Q_{\ell}(t).
\end{displaymath}

\noindent \textit{Step 3:} The rest of the proof can be achieved recursively.

\end{proof}

\section{Comments and References}

\noindent \textbf{1)} The exponential dichotomy spectrum is also called as the \textit{Sacker $\&$ Sell's spectrum}
due to the seminal 1978 article of Sacker and Sell \cite{Sac} which constructs a spectrum for linear skew--products dynamical systems encompassing those arising from linear nonautonomous differential and difference equations. A step ahead is made at the beginning of the century with the works of Aulbach and Siegmund which the construct a spectral theory focused
on linear nonautonomous differential and difference equations \cite{Aul,Sie,Siegmund-2002}, providing a solid basis for a vast research effort. An abridged but well written version of this topic can be founded on \cite[Ch.5]{Klo}.

\medskip

\noindent \textbf{2)} The sections 2 and 3 of this chapter have been writing emulating  the structure and proofs presented in \cite{Klo}. Moreover, the step 2 in the proof of Proposition \ref{reso-o} follows an idea from \cite{Lin}.

\medskip

\noindent \textbf{3)} At the beginning of this chapter we made references to several
\textit{spectral theories}. We refer the reader to  for \textit{Lyapunov spectrum} and \textit{Bohl spectrum}
which are studied in \cite[Ch.1]{Doan}.

\medskip

\noindent \textbf{4)} The assignability problem for the control system (\ref{ESNA3})
has been studied in relation with the exponential dichotomy spectra and we refer
the reader to \cite{Anh,Anh2}.

\noindent \textbf{5)} The exponential dichotomy spectrum for the block system 
\begin{equation*}
%\label{bloque-diag}
\left(\begin{array}{c}
\dot{y}_{1}\\
\dot{y}_{2}
\end{array}\right)
=\left[\begin{array}{cc}
B_{1}(t) & C(t)\\\\
0        & B_{2}(t)
\end{array}\right]
\left(\begin{array}{c}
y_{1}\\
y_{2}
\end{array}\right),
\end{equation*}
has been studied in depth by F. Battelli and K.J. Palmer in \cite{Battelli}, where
it was proved that, in general, 
\begin{displaymath}
   \Sigma\left(\left[\begin{array}{cc}
B_{1}(t) & C(t)\\\\
0        & B_{2}(t)
\end{array}\right]\right) \neq \Sigma(B_{1})\cup \Sigma(B_{2}).
\end{displaymath}

On the other hand, it was also proved that
\begin{displaymath}
   \Sigma^{+}\left(\left[\begin{array}{cc}
B_{1}(t) & C(t)\\\\
0        & B_{2}(t)
\end{array}\right]\right) = \Sigma^{+}(B_{1})\cup \Sigma^{+}(B_{2}),
\end{displaymath}
provided that $C(t)$ is bounded on $\mathbb{R}_{+}$.

\medskip

\noindent \textbf{6)} The property of \textit{diagonal significance}, which has been coined by C. P\"otzsche
in \cite{Potzsche-2016},  refers to
the identity between the spectra of a block system and the union of diagonal spectra.  Some new results about
diagonal significance have been extended for the nonuniform exponential dichotomy spectrum in \cite{CHR}.

\medskip

\noindent \textbf{7)} To the best of our knowledge, the first proof of Theorem \ref{TDB}
has been carried out by S. Siegmund in \cite[Th.3.2]{Siegmund2}. Nevertheless, we emphasize that our proof
has several differences which are based on our results from Chapter 4.

\section{Exercises}

\begin{itemize}
    \item[1-.] Prove that $\Sigma_{EIG}(A)$ satisfies properties \textbf{(P1)--(P3)} of Definition \ref{TSD}.
    \item[2.-] Prove the property (\ref{inc-spec}).
    \item[3.-] Prove that the equation (\ref{schned}) from Chapter 3 verifies the following property: $0\notin \Sigma^{+}(a)$
but $0\in \Sigma(a)$.
    \item[4.-] Prove the properties stated in the Remark \ref{TE0}. In addition
    prove that if $\dot{x}=[A(t)-\lambda I]x$ has an exponential dichotomy on $\mathbb{R}$ with
     projection $P_{\lambda}$ then $\mathcal{U}^{\lambda}\oplus \mathcal{S}^{\lambda}=\mathbb{R}^{n}$.
    \item[5.-] Prove Remark \ref{cotas-parciales} 
    \item[6.-] In the context of Definition \ref{componente}, prove that $C_{\lambda_{1}}\cap C_{\lambda_{2}}\neq \varnothing$ if and only if $C_{\lambda_{1}}=C_{\lambda_{2}}$.

    \item[7.-] Prove the property (\ref{BGTP1}) from Theorem \ref{espectro-proj}.

\item[8.-] Prove a version of Theorem \ref{BSIBG} for $\Sigma^{\pm}$ and the
property of bounded growth growth on $\mathbb{R}_{0}^{+}$ and $\mathbb{R}_{0}^{-}$.
   
   \item[9.-]  Prove that $\lambda_{1}\sim \lambda_{2}$ if $C_{\lambda_{1}}=C_{\lambda_{2}}$ defines an equivalence relation on $\mathbb{R}$. 
   
\item[10.-] Generalize the example of the subsection 3.3 by considering a function
    $f\colon \mathbb{R}\to \mathbb{R}$ such that is bounded and integrable on $\mathbb{R}$.

\item[11.-] Prove Theorem \ref{TeoSpec2}.
\end{itemize}

\chapter{Global linearization}

As we stated at the beginning of the Chapter 2, the problem of local linearization, in the framework of autonomous equations, was initially considered by
D. Grobman \cite{Grobman} and P. Hartman \cite[Theorem I]{Hartman1}; later a specific type of those equations regarded in the previous mentioned works, namely a linear and a quasi-linear system, were studied by C. Pugh \cite{Pugh} in order to achieve a global homeomorphism between the solutions of the systems.  In this chapter, we will study
the generalization of the Pugh's result to the nonautonomous case carried out by K.J. Palmer \cite{Palmer} under the assumption that the linear part has an exponential dichotomy.

\section{Introduction}
Let us consider the linear system
\begin{equation}
\label{lineal-efes}
\begin{array}{lcl}
\dot{x}&=&A(t)x,
\end{array}
\end{equation}
and the family of quasilinear systems
\begin{equation} 
\label{sistema1}
\begin{array}{lcl}
\dot{y}&=&A(t)y+f(t,y), 
\end{array}
\end{equation}
where $t\mapsto A(t)\in M_{n}(\mathbb{R})$ and  
$f\colon \mathbb{R}\times\mathbb{R}^{n}\to \mathbb{R}^{n}$ has suitable properties. In this chapter, we will revise several results devoted to the global linealization of (\ref{sistema1}).

The concept of topological equivalence for nonautonomous systems has been introduced by K.J. Palmer in \cite[pp.754--755]{Palmer}:
\begin{definition}
\label{TopEq}
The systems \textnormal{(\ref{lineal-efes})} and \textnormal{(\ref{sistema1})} are \textbf{topologically equivalent}
if there exists a function $H\colon \mathbb{R}\times \mathbb{R}^{n}\to \mathbb{R}^{n}$ with the properties
\begin{itemize}
\item[(i)] For each fixed $t\in \mathbb{R}$, $u\mapsto H(t,u)$ is an homeomorphism of $\mathbb{R}^{n}$, with inverse 
$u\mapsto G(t,u)$,
\item[(ii)] $G(t,u)-u$ is bounded in $\mathbb{R}\times \mathbb{R}^{n}$,
\item[(iii)] If $t\mapsto x(t)$ is a solution of \textnormal{(\ref{lineal-efes})}, then $t\mapsto H[t,x(t)]$ is a solution 
of \textnormal{(\ref{sistema1})},
\item[(iv)] If $t\mapsto y(t)$ is a solution of \textnormal{(\ref{sistema1})}, then $t\mapsto G[t,y(t)]$ is a solution 
of \textnormal{(\ref{lineal-efes})}.

\end{itemize}
\end{definition}

The concept of strongly topologically equivalence was introduced by Shi and Xiong \cite{Shi}, who realized that, in several
examples of topological e\-qui\-va\-lence, the maps $u\mapsto H(t,u)$ and $u\mapsto G(t,u)$ could have properties sharper than continuity.

\begin{definition}
%\label{StrTopEq}
The systems \textnormal{(\ref{lineal-efes})} and \textnormal{(\ref{sistema1})} are strongly topologically equivalent
if they are topologically equivalent and $H$ and $G$ are uniformly continuous for all $t \in \mathbb{R}.$
\end{definition}

\section{Palmer's Homeomorphism}
\begin{theorem}\cite{Palmer}
\label{intermedio}
Assume that \eqref{lineal-efes} has an exponential dichotomy on $\mathbb{R}$ with constants $K\geq 1$, $\alpha>0$ and projection $P$. Moreover, we will assume that
\begin{itemize}
 \item[\textbf{(A1)}] The matrix $A(\cdot)$ of the linear system is bounded, namely
 $$
 \sup\limits_{t\in \mathbb{R}}||A(t)||=M,
 $$
 \item[\textbf{(A2)}] there exists a positive constant $\mu$  such that
 \begin{displaymath}
|f(t,y)|\leq \mu \quad \textnormal{for any} \quad (t,y)\in \mathbb{R}\times\mathbb{R}^{n}, 
\end{displaymath} 
 \item[\textbf{(A3)}] there exists a positive constant $\gamma$ such that 
\begin{displaymath}
|f(t,y)-f(t,y')| \leq \gamma|y-y'| \quad \textnormal{for any $t\in \mathbb{R}$ and $(y,y')\in \mathbb{R}^{n}$}.
\end{displaymath}
\end{itemize}

If the following inequality is satisfied
\begin{equation}
\label{FPT}
2K\gamma<\alpha,
\end{equation}
then \textnormal{(\ref{lineal-efes})} and \textnormal{(\ref{sistema1})} are strongly topologically equivalent.
\end{theorem}

The proof of this Theorem will be a consequence of several intermediate results.

\subsection{Continuity with respect to initial conditions}
The unique solution of (\ref{lineal-efes}) passing through $\xi$ at $t=\tau$ will be denoted by $t\mapsto x(t,\tau,\xi)$. Similarly,
the unique solution of (\ref{sistema1}) passing through $\eta$ at $t=\tau$ will be denoted by $t\mapsto y(t,\tau,\eta)$.

\begin{proposition}
\label{T1}
Under the assumptions of Theorem \ref{intermedio}, given two solutions $t\mapsto y(t,\tau,\eta)$ and $t\mapsto y(t,\tau,\eta')$ of \textnormal{(\ref{sistema1})} it follows that
\begin{equation}
\label{CRCI}
|y(t,\tau,\eta)-y(t,\tau,\eta')|\leq |\eta-\eta'|e^{(M+\gamma)|t-\tau|}.
\end{equation}
\end{proposition}

\begin{proof}
Firstly, let us assume that $t>\tau$. By integrating the system, we have that
\begin{displaymath}
\begin{array}{rcl}
|y(t,\tau,\eta)-y(t,\tau,\eta')|&=& \displaystyle |\eta-\eta'|+\int_{\tau}^{t}|A(s)y(s,\tau,\eta))-A(s)y(s,\tau,\eta'))|\,ds  \\\\
&&
\displaystyle +\int_{\tau}^{t}|f(s,y(s,\tau,\eta))-f(s,y(s,\tau,\eta'))|\,ds.
\end{array}
\end{displaymath}

Now, by \textbf{(A1)} and \textbf{(A3)} we can deduce that

\begin{displaymath}
|y(t,\tau,\eta)-y(t,\tau,\eta')|\leq |\eta-\eta'|+(M+\gamma) \int_{\tau}^{t}|y(s,\tau,\eta)-y(s,\tau,\eta')|\,ds
\end{displaymath}
and (\ref{CRCI}) is a consequence from Gronwall's Lemma.

Now, if $t<\tau$, we can use the integral representation
\begin{displaymath}
\begin{array}{rcl}
y(t,\tau,\eta)-y(t,\tau,\eta') &=& \displaystyle \eta-\eta'-\int_{\tau}^{t}\{A(s)y(s,\tau,\eta)-A(s)y(s,\tau,\eta')\}\,ds\\\\
&&
\displaystyle -\int_{\tau}^{t}\{f(s,y(s,\tau,\eta))-f(s,y(s,\tau,\eta'))\}\,ds,
\end{array}
\end{displaymath}
and it follows that
\begin{displaymath}
\begin{array}{rcl}
|y(t,\tau,\eta)-y(t,\tau,\eta')|&\leq & \displaystyle|\eta-\eta'|\\\\
&&
\displaystyle +(M+\gamma)\int_{\tau}^{t}\,|y(s,\tau,\eta))-y(s,\tau,\eta')|\,ds
\end{array}
\end{displaymath}
and the rest of the proof can be made in a similar way.
\end{proof}

The following result is a direct consequence of the bounded growth growth property
\begin{proposition}
\label{T2}
Given to solutions $t\mapsto x(t,\tau,\xi)$ and $t\mapsto x(t,\tau,\xi')$ of \textnormal{(\ref{lineal-efes})} it follows that
\begin{equation*}
%\label{CRCI2}
|x(t,\tau,\xi)-x(t,\tau,\xi')|\leq |\xi-\xi'|e^{M|t-\tau|}.
\end{equation*}
\end{proposition}
\begin{proof}
See Exercises.
\end{proof}

\subsection{Some intermediate Lemmas}

The following result are classical tools developed
by K.J. Palmer in his global linearization result.

\begin{lemma}
\label{lemme-0}
For any solution $t\mapsto y(t,\tau,\eta)$ of \textnormal{(\ref{sistema1})} passing
through $\eta$ at $t=\tau$, there exists a unique bounded solution $t\mapsto \chi(t;(\tau,\eta))$ of
\begin{equation}
\label{auxiliar} 
\begin{array}{lcl}
\dot{z}(t)&=&A(t)z(t)-f(t,y(t,\tau,\eta)).
\end{array}
\end{equation}
\end{lemma}

\begin{proof}
By using Proposition \ref{boundeness} from Chapter 2 with $g(t)=-f(t,y(t,\tau,\eta))$, we have that  
\begin{equation}
\label{admi-cap6}
\chi(t;(\tau,\eta))=-\int_{-\infty}^{\infty}\mathcal{G}(t,s)f(s,y(s,\tau,\eta))\,ds
\end{equation}
is the unique bounded solution of (\ref{auxiliar}). In addition, \textbf{(A2)} implies 
that $|\chi(t;(\tau,\eta))|\leq 2K \mu \alpha^{-1}$.
\end{proof}

\begin{remark}
\label{uniqueness}
By uniqueness of solutions of (\ref{sistema1}), we know that
\begin{displaymath}
y\big(t,r,y(r,\tau,\eta)\big)=y\big(t,\tau,\eta\big) \quad \textnormal{for any $r\in \mathbb{R}$},
\end{displaymath} 
this fact implies that system (\ref{auxiliar}) can be written as
\begin{displaymath}
\begin{array}{lcl}
\dot{z}(t)&=&A(t)z(t)-f(t,y(t,r,y(r,\tau,\eta)))
\end{array}
\end{displaymath}
and Lemma \ref{lemme-0} implies the identity
\begin{equation}
\label{uniqueness1}
\chi(t;(\tau,\xi))=\chi(t;(r,x(r,\tau,\xi))) \quad \textnormal{for any $r\in \mathbb{R}$}.
\end{equation} 
\end{remark}

\begin{lemma}
%\label{lemme-0-b}
For any solution $t\mapsto x(t,\tau,\xi)$ of \textnormal{(\ref{lineal-efes})} passing through $\xi$ at $t=\tau$, there 
exists a unique bounded solution $t\mapsto \vartheta(t;(\tau,\xi))$ of
\begin{equation}
\label{auxiliar2} 
\begin{array}{lcl}
\dot{w}&=&A(t)w+f(t,x(t,\tau,\xi)+w). 
\end{array}
\end{equation}
\end{lemma}

\begin{proof}
By using Proposition \ref{bounded2} from Chapter 2 with $h(t,w)=f(t,x(t,\tau,\xi)+w)$, we know that there exists
a unique bounded solution
\begin{equation}
\label{fijo}
\begin{array}{l}
\vartheta(t;(\tau,\xi))=\displaystyle \int_{-\infty}^{+\infty}\mathcal{G}(t,s)f(s,x(s,\tau,\xi)+\vartheta(s;(\tau,\xi)))\,ds,
\end{array}
\end{equation}
where $\mathcal{G}$ is the Green function stated in Definition \ref{Green2} from Chapter 2.
\end{proof}

\begin{remark}
%\label{uniquenessB}
Similarly as in Remark \ref{uniqueness}, the reader can verify the identity
\begin{equation}
\label{uniqueness1B}
\vartheta(t;(\tau,\xi))=\vartheta(t;(r,x(r,\tau,\xi))) \quad \textnormal{for any $r\in \mathbb{R}$}.
\end{equation} 
%In addition, it will be useful to nothe that this fixed point can be deduced as a limit of the succesive approximations
%\begin{displaymath}
%\vartheta_{m+1}(t;(\tau,\nu))=\int_{\mathbb{R}}\widetilde{G}(t,s)f(s,y(s,\tau,\nu)+\vartheta_{m}(s;(\tau,\nu)),y(\gamma(s),\tau,\nu)+\vartheta_{m}(\gamma(s);(\tau,%\nu)))\,ds
%\end{displaymath}
%with $\vartheta_{m}(t;(\tau,\nu))=0$.
\end{remark}

\begin{lemma}
\label{lema2}
There exists a function 
$H\colon \mathbb{R}\times \mathbb{R}^{n}\to \mathbb{R}^{n}$, satisfying:
\begin{enumerate}
 \item[(i)] $H(t,x)-x$ is bounded in $\mathbb{R}\times \mathbb{R}^{n}$,
 \item[(ii)] For any solution $t\mapsto x(t,\tau,\xi)$ of \textnormal{(\ref{lineal-efes})}, we have that 
 $t\mapsto H[t,x(t,\tau,\xi)]$ is a  solution of the nonlinear system \textnormal{(\ref{sistema1})} verifying 
\begin{subequations}
  \begin{empheq}{align}
     |H[t,x(t,\tau,\xi)]-x(t,\tau,\xi)|\leq 2\mu K \alpha^{-1} \label{cotas2}, \\
      H[t,x(t,\tau,\xi)]=y(t,\tau,H(\tau,\xi)).
       \label{unicite-HP1}
  \end{empheq}
\end{subequations}
\end{enumerate}
\end{lemma}

\begin{proof}
Let us define $H\colon \mathbb{R}\times \mathbb{R}^{n}\to \mathbb{R}^{n}$ as follows:
\begin{equation}
\label{DefH}
\begin{array}{rcl}
H(t,\xi)&=&\xi+\vartheta(t;(t,\xi)) \\\\
&=&\displaystyle \xi + \int_{-\infty}^{\infty}\mathcal{G}(t,s)f(s,x(s,t,\xi)+\vartheta(s;(t,\xi)))\,ds.
\end{array}
\end{equation}

By using  (\ref{fijo}) and (\ref{uniqueness1B}), we can deduce the alternative characterizations
\begin{equation}
\label{3-0}
H[t,x(t,\tau,\xi)]=\left\{
\begin{array}{l}
x(t,\tau,\xi)+\vartheta(t;(t,x(t,\tau,\xi)))\\\\
x(t,\tau,\xi)+\vartheta(t;(\tau,\xi))\\\\
x(t,\tau,\xi)+ \displaystyle \int_{-\infty}^{+\infty}\mathcal{G}(t,s)f(s,\underbrace{x(s,\tau,\xi)+\vartheta(s;(\tau,\xi))}_{=H[s,x(s,\tau,\xi)]})\,ds.             
\end{array}\right.
\end{equation}

It will be useful to describe $H[t,x(t,\tau,\xi)]$ as follows
\begin{equation}
\label{HomeL1}
\begin{array}{rcl}
H[t,x(t,\tau,\xi)]&=&\displaystyle x(t,\tau,\xi)+\int_{-\infty}^{\infty}\mathcal{G}(t,s)f(s,H[s,x(s,\tau,\xi)])\,ds,
\end{array}
\end{equation}
and the property (\ref{cotas2}) can be deduced by property \textbf{(A2)} followed by a standard estimation of the Green's function combined with the exponential dichotomy constants.

On the other hand, as $t\mapsto x(t,\tau,\xi)$ is solution of (\ref{lineal-efes}) and $t\mapsto \vartheta(t;(\tau,\xi))$ is solution
of (\ref{auxiliar2}), the alternative characterizations for $H[t,x(t,\tau,\xi)]$ allow us to deduce that 
\begin{displaymath}
\begin{array}{rcl}
\frac{d}{dt}H[t,x(t,\tau,\xi)] &=& A(t)x(t,\tau,\xi)+ \frac{d}{dt} \vartheta(t;(\tau,\xi))\\\\
&=& A(t)x(t,\tau,\xi)+A(t) \vartheta(t;(\tau,\xi))+f(t,x(t,\tau,\xi)+ \vartheta(t;(\tau,\xi))) \\\\
&=&  A(t)\{x(t,\tau,\xi)+\vartheta(t;(\tau,\xi))\}+f(t,x(t,\tau,\xi)+ \vartheta(t;(\tau,\xi))) \\\\
&=& A(t)H[t,x(t,\tau,\xi)]+f(t,H[t,x(t,\tau,\xi)]).
\end{array}
\end{displaymath}

Then, $t\mapsto H[t,x(t,\tau,\xi)]$ is a solution of (\ref{sistema1}) passing through $H(\tau,\xi)$ at $t=\tau$, and the identity 
(\ref{unicite-HP1}) follows by the uniqueness of solutions.
\end{proof}

\begin{lemma}
\label{lema1}
There exists a  function 
$G\colon \mathbb{R}\times \mathbb{R}^{n}\to \mathbb{R}^{n}$, satisfying:
\begin{enumerate}
 \item[(i)] $G(t,y)-y$ is bounded in $\mathbb{R}\times \mathbb{R}^{n}$,
 \item[(ii)] For any solution $t\mapsto y(t,\tau,\eta)$ of \textnormal{(\ref{sistema1})}, then $t\mapsto G[t,y(t,\tau,\eta)]$ 
is a solution of \textnormal{(\ref{lineal-efes})} satisfying
\begin{subequations}
  \begin{empheq}{align}
     |G[t,y(t,\tau,\eta)]-y(t,\tau,\eta)|\leq 2\mu K\alpha^{-1} \label{cota1}, \\
      G[t,y(t,\tau,\eta)]=x(t,\tau,G(\tau,\eta)).
       \label{unicite-HP2}
  \end{empheq}
\end{subequations}
\end{enumerate}
\end{lemma}

\begin{proof}

Let us define the 
function $G\colon \mathbb{R}\times\mathbb{R}^{n}\to \mathbb{R}^{n}$
as follows
\begin{equation}
\label{homeo-00}
\begin{array}{rcl}
G(t,\eta)&=&\eta+\chi(t;(t,\eta))\\\\
&=&\displaystyle \eta-\int_{-\infty}^{\infty}\mathcal{G}(t,s)f(s,y(s,t,\eta))\,ds,
\end{array}
\end{equation}
and \textbf{(A2)} combined with (\ref{admi-cap6}) implies $|G(t,\eta)-\eta|\leq 2\mu K\alpha^{-1}$.

By replacing $(t,\eta)$ by $(t,y(t,\tau,\eta))$ in (\ref{homeo-00}) and using (\ref{uniqueness1}), we can deduce the following alternative characterizations
\begin{equation}
\label{homeosofit}
G[t,y(t,\tau,\eta)]=\left\{\begin{array}{l}
y(t,\tau,\eta)+\chi(t;(t,y(t,\tau,\eta)))\\\\
y(t,\tau,\eta)+\chi(t;(\tau,\eta))\\\\
\displaystyle y(t,\tau,\eta)-\int_{-\infty}^{\infty}\mathcal{G}(t,s)f(s,y(s,\tau,\eta))\,ds.
\end{array}\right.
\end{equation}

As in the previous result, the property (\ref{cota1}) can be deduced by property \textbf{(A2)} followed by a standard estimation of the Green's function combined with the exponential dichotomy constants.

Finally, as $t\mapsto y(t,\tau,\eta)$ is solution of (\ref{sistema1}) and $t\mapsto \chi(t;(\tau,\eta))$ is solution
of (\ref{auxiliar}), the above characterizations of
$G[t,y(t,\tau,\eta)]$
allow us to deduce that
\begin{displaymath}
\begin{array}{rcl}
\frac{d}{dt}G[t,y(t,\tau,\eta)]&=&\displaystyle A(t)y(t,\tau,\eta)+f(t,y(t,\tau,\eta))+\frac{d}{dt}\chi(t;(\tau,\eta))\\\\
&=&\displaystyle A(t)\{y(t,\tau,\eta)+\chi(t;(\tau,\eta))\}\\\\
&=& A(t)G[t,y(t,\tau,\eta)],
\end{array}
\end{displaymath}
and we verify that $t\mapsto G[t,y(t,\tau,\eta)]$ is solution of (\ref{lineal-efes}) passing through by $G(\tau,\eta)$ at $t=\tau$ and (\ref{unicite-HP2}) is verified.
\end{proof}

\begin{lemma}
\label{composition}
For any 
solution $t\mapsto y(t,\tau,\eta)$ of \textnormal{(\ref{sistema1})} with fixed $t$, it follows that
\begin{displaymath}
H[t,G[t,y(t,\tau,\eta)]]=y(t,\tau,\eta).
\end{displaymath}
\end{lemma}

\begin{proof}
By using (\ref{HomeL1}) combined with (\ref{unicite-HP2}) 
we have that
\begin{displaymath}
\begin{array}{rcl}
H[t,G[t,y(t,\tau,\eta)]]&=& H[t,x(t,\tau,G(\tau,\eta))]\\\\
&=&\displaystyle x(t,\tau,G(\tau,\eta))+\int_{\mathbb{R}}\mathcal{G}(t,s)f(s,H[s,x(s,\tau,G(s,\eta)]])\,ds\\\\
&=&\displaystyle G[t,y(t,\tau,\eta)]+\int_{\mathbb{R}}\mathcal{G}(t,s)f(s,H[s,G[s,y(s,\tau,\eta)]])\,ds.
\end{array}
\end{displaymath}

The third characterization of (\ref{homeosofit}) and the identity above implies that
\begin{displaymath}
H[t,G[t,y(t,\tau,\eta)]]-y(t,\tau,\eta)=\int_{\mathbb{R}}\mathcal{G}(t,s)\{f(s,H[s,G[s,y(s,\tau,\eta)]])-f(s,y(s,\tau,\eta))\}\,ds.
\end{displaymath}

By \textbf{(A3)} combined with standard estimations of the Green's function allow to deduce that

\begin{displaymath}
|H[t,G[t,y(t,\tau,\eta)]]-y(t,\tau,\eta)|\leq \frac{2K\gamma}{\alpha}\sup\limits_{t\in \mathbb{R}}|H[t,G[t,\tau,\eta)]-y(t,\tau,\eta)|.
\end{displaymath}

The Lemma is a direct consequence of the identity
$$
\sup\limits_{t\in \mathbb{R}}|H[\cdot,G[\cdot,\tau,\eta)]-y(\cdot,\tau,\eta)|=0.
$$

In fact, the above term cannot be positive since it will implies a contradiction with (\ref{FPT}).
\end{proof}

\begin{lemma}
\label{composition2}
For any solution $t\mapsto x(t,\tau,\xi)$ of \textnormal{(\ref{lineal-efes})}, it follows that
\begin{displaymath}
G[t,H[t,x(t,\tau,\xi)]]=x(t,\tau,\xi).
\end{displaymath}
\end{lemma}
\begin{proof}
Let us observe that  (\ref{unicite-HP1}), (\ref{HomeL1})  and (\ref{homeo-00}) imply that
\begin{displaymath}
\begin{array}{rcl}
G[t,H[t,x(t,\tau,\xi)]]&=&\displaystyle H[t,x(t,\tau,\xi)]-\int_{\mathbb{R}}\mathcal{G}(t,s)f(s,y(s,t,H[t,x(t,\tau,\xi)])\,ds\\\\
&=&\displaystyle 
H[t,x(t,\tau,\xi)]-\int_{\mathbb{R}}\mathcal{G}(t,s)f(s,y(s,t,y(t,\tau,H(\tau,\xi))))\,ds\\\\
&=&\displaystyle 
H[t,x(t,\tau,\xi)]-\int_{\mathbb{R}}\mathcal{G}(t,s)f(s,y(s,\tau,H(\tau,\xi)))\,ds\\\\
&=&
x(t,\tau,\xi)\\\\
&&\displaystyle +\int_{\mathbb{R}}\mathcal{G}(t,s)\{f(s,H[s,x(s,\tau,\xi)])-f(s,y(s,\tau,H(\tau,\xi)))\}\,ds\\\\
&=&x(t,\tau,\xi),
\end{array}
\end{displaymath}
and the result follows.
\end{proof}

The reader can notice (see also Definition \ref{TopEq}) that the notation $H[\cdot,\cdot]$ and $G[\cdot,\cdot]$
is reserved to the case when $H$ and $G$ are respectively defined on solutions of (\ref{sistema1}) and (\ref{lineal-efes}).

\begin{lemma}
\label{invertibilidad}
For any fixed $t$ and any pair $(\xi,\eta)\in \mathbb{R}^{n}\times\mathbb{R}^{n}$, it follows that
\begin{equation}
\label{inve1}
G(t,H(t,\xi))=\xi 
\end{equation}
and
\begin{equation}
\label{inve2}
H(t,G(t,\eta))=\eta.
\end{equation}
\end{lemma}
\begin{proof}
By using Lemma \ref{composition}, we have that
\begin{displaymath}
G[t,H[t,x(t,\tau,\xi)]]=x(t,\tau,\xi) \quad \textnormal{for any} \quad t\in \mathbb{R}.
\end{displaymath}

Now, if we consider the particular case  $\tau=t$, we obtain (\ref{inve1}). The identity (\ref{inve2}) can be deduced
similarly.
\end{proof}

\begin{remark}
%\label{tofin}
Notice that the maps $\xi  \mapsto H(t,\xi)$ and $\eta \mapsto G(t,\eta)$ satisfy properties
(ii) and (iii) of Definition \ref{TopEq}, which is a consequence of Lemmas \ref{lema1}--\ref{composition}. In addition,
Lemma \ref{invertibilidad} says that $u\mapsto G(t,u)=H^{-1}(t,u)$ for any $t\in \mathbb{R}$. In consequence, the last step is to prove
the uniform continuity of the maps, which will be made in the next two sections. 
\end{remark}

\subsection{End of the Proof: Continuity of the maps $H$ and $G$}
We follow the approach developed in the nice paper of Shi and Xiong \cite{Shi}.

\begin{lemma}
%\label{end-a1} 
The map $\eta\mapsto G(t,\eta)=\eta+\chi(t;(t,\eta))$ is uniformly continuous for any $t\in \mathbb{R}$.
\end{lemma}

\begin{proof}
As the identity is uniformly continuous, we only need to prove that the map $\eta\to \chi(t;(t,\eta))$ is uniformly continuous.

Let $\eta$ and $\eta'$ be two initial conditions of (\ref{sistema1}). Notice that (\ref{homeo-00}) allows to say that
\begin{equation*}
%\label{decoupage}
\begin{array}{rcl}
\chi(t;(t,\eta))-\chi(t;(t,\eta'))&=&\displaystyle -\int_{-\infty}^{t}\mathcal{G}(t,s)\big\{f(s,y(s,t,\eta))-f(s,y(s,t,\eta'))\big\}\,ds\\\\
&&\displaystyle +\int_{t}^{\infty}\mathcal{G}(t,s)\big\{f(s,y(s,t,\eta))-f(s,y(s,t,\eta'))\big\}\,ds\\\\
&=&-I_{1}+I_{2}.
\end{array}
\end{equation*}

Now, we divide $I_{1}$ and $I_{2}$ as follows:
\begin{displaymath}
I_{1}=\int_{-\infty}^{t-L}+\int_{t-L}^{t}=I_{11}+I_{12}
\quad
\textnormal{and}
\quad
I_{2}=\int_{t}^{t+L}+\int_{t+L}^{\infty}=I_{21}+I_{22},
\end{displaymath}
where $L$ is a positive constant.

By using \textbf{(A2)} combined with exponential dichotomy, we can see that the integrals 
$I_{11}$ and $I_{22}$ are always finite since 
\begin{displaymath}
|I_{11}|\leq 2K\mu \int_{-\infty}^{t-L}e^{-\alpha(t-s)}\,ds=\frac{2K\mu}{\alpha}e^{-\alpha L} 
\end{displaymath}
and
\begin{displaymath}
|I_{22}|\leq 2K\mu \int_{t+L}^{\infty}e^{-\alpha(s-t)}\,ds = \frac{2K\mu}{\alpha}e^{-\alpha L}. 
\end{displaymath}

Now, by \textbf{(A3)}, we have that 
\begin{displaymath}
\begin{array}{rcl}
|I_{12}|&\leq & \displaystyle \int_{t-L}^{t}Ke^{-\alpha(t-s)}\gamma |y(s,t,\eta)-y(s,t,\eta')|\,ds\\\\
& \leq  & \displaystyle \int_{0}^{L}K e^{-\alpha u}\gamma |y(t-u,t,\eta)-y(t-u,t,\eta')|\,du.\\\\
\end{array}
\end{displaymath}

On the other hand, by Proposition \ref{T1}, we have that
\begin{displaymath}
0\leq |y(t-u,t,\eta)-y(t-u,t,\eta')|\leq |\eta-\eta'|e^{(M+\gamma) u} \quad \textnormal{for any} \quad u\in [0,L].
\end{displaymath}

Upon inserting this inequality in the previous estimation of $|I_{12}|$, we obtain that
\begin{displaymath}
\begin{array}{rcl}
|I_{12}|&\leq  & \displaystyle \int_{0}^{L} Ke^{-\alpha u}\gamma |y(t-u,t,\eta)-y(t-u,t,\eta')|\,du\\\\
&\leq & \displaystyle K\gamma |\eta-\eta'|\int_{0}^{L}e^{-\alpha u}e^{(M+\gamma)u}\,du\\
&\leq & \displaystyle K\gamma |\eta-\eta'| e^{(M+\gamma)L}\int_{0}^{L}e^{-\alpha u}\,du\\
&= & \displaystyle\frac{K\gamma e^{(M+\gamma)L}}{\alpha}(1-e^{-\alpha L})|\eta-\eta'|
\end{array}
\end{displaymath}

The reader can deduce that the inequalities above implies
\begin{equation}
\label{estimacion1}
|I_{12}|\leq D|\eta-\eta'| \quad \textnormal{with} \quad D=\frac{K\gamma e^{(M+\gamma)L}}{\alpha}(1-e^{-\alpha L}).
\end{equation}

Analogously, we can deduce that 
\begin{equation}
\label{estimacion2}
|I_{21}|\leq D|\eta-\eta'|.
\end{equation}

For any $\varepsilon>0$, we can choose 
\begin{displaymath}
L \geq \displaystyle \frac{1}{\alpha}\ln\left(\frac{8K\mu}{\alpha \varepsilon}\right),
\end{displaymath}
which implies that
$|I_{11}|+|I_{22}|<\varepsilon/2$. By using this fact combined with (\ref{estimacion1})--(\ref{estimacion2}), we obtain 
that
\begin{displaymath}
\forall \varepsilon>0 \hspace{0.1cm}\exists \delta=\frac{\varepsilon}{4D}>0 \,\, \textnormal{such that} \,\,
|\eta-\eta'|<\delta \Rightarrow |\chi(t;(t,\eta))-\chi(t;(t,\eta'))|<\varepsilon
\end{displaymath}
and the uniform continuity follows. 
\end{proof}

\begin{lemma}
%\label{end-a1-2} 
The map $\xi\mapsto H(t,\xi)=\xi+\vartheta(t;(t,\xi))$ is uniformly continuous for any $t\in \mathbb{R}$.
\end{lemma}

\begin{proof}
 We only need to prove that the map $\xi\mapsto \vartheta(t;(t,\xi))$ is uniformly continuous. In order to prove that,
let $\xi$ and $\xi'$ be two initial conditions of (\ref{lineal-efes}) and define the differences
$$
\Delta \vartheta :=\vartheta(t;(t,\xi))-\vartheta(t;(t,\xi')),
$$
and
$$
\Delta f:=f(s,x(s,t,\xi)+\vartheta(s;(t,\xi)))-f(s,x(s,t,\xi')+\vartheta(s;(t,\xi'))).
$$

By using (\ref{fijo}), we can see that $\Delta \vartheta$ can be written as follows:
\begin{equation*}
%\label{decoupage2}
\begin{array}{rcl}
\Delta \vartheta&=&\displaystyle 
\underbrace{\int_{-\infty}^{t-\tilde{L}}\mathcal{G}(t,s)\Delta f\,ds}_{:=J_{11}}+
\underbrace{\int_{t-\tilde{L}}^{t}\mathcal{G}(t,s)\Delta f\,ds}_{:=J_{12}}\\\\
&&+ \displaystyle
\underbrace{\int_{t}^{t+\tilde{L}}\mathcal{G}(t,s)\Delta f\,ds}_{:=J_{21}}+
\underbrace{\int_{t+\tilde{L}}^{+\infty}\mathcal{G}(t,s)\Delta f\,ds}_{:=J_{22}}.
\end{array}
\end{equation*}

By \textbf{(A2)} and the exponential dichotomy property, it is straightforward to verify that
\begin{displaymath}
|J_{11}|\leq \frac{2K\mu}{\alpha}e^{-\alpha \tilde{L}} \quad \textnormal{and} \quad
|J_{22}|\leq \frac{2K\mu}{\alpha}e^{-\alpha \tilde{L}}. 
\end{displaymath}

Let us define
\begin{equation*}
%\label{Ninfty}
||\vartheta(\cdot;(t,\xi))-\vartheta(\cdot;(t,\xi'))||_{\infty}=\sup\limits_{s\in \mathbb{R}}|\vartheta(s;(t,\xi))-\vartheta(s;(t,\xi'))|,
\end{equation*}
and notice that \textbf{(A3)} implies:
\begin{displaymath}
\begin{array}{rcl}
|J_{12}|&\leq & \displaystyle \frac{K\gamma}{\alpha}||\vartheta(\cdot;(t,\xi))-\vartheta(\cdot;(t,\xi'))||_{\infty}\\\\
        & &\displaystyle +K\gamma\int_{t-\tilde{L}}^{t}e^{-\alpha(t-s)}|x(s,t,\xi)-x(s,t,\xi')|\,ds\\\\   
        &\leq & \displaystyle \frac{K\gamma}{\alpha}||\vartheta(\cdot;(t,\xi))-\vartheta(\cdot;(t,\xi'))||_{\infty}\\\\
        & &\displaystyle +K\gamma\int_{0}^{\tilde{L}}e^{-\alpha u}|x(t-u,t,\xi)-x(t-u,t,\xi')|\,du.\\\\
\end{array}
\end{displaymath}

By using Proposition \ref{T2}, we know that
\begin{displaymath}
|x(t-u,t,\xi)-x(t-u,t,\xi')|\leq |\xi-\xi'|e^{Mu} \quad \textnormal{for any} \quad u\in [0,\tilde{L}].
\end{displaymath}

Upon inserting this inequality in $|J_{12}|$ , we have that
\begin{displaymath}
\begin{array}{rcl}
|J_{12}|&\leq &  \displaystyle \frac{K\gamma}{\alpha}||\vartheta(\cdot;(t,\xi))-\vartheta(\cdot;(t,\xi'))||_{\infty}+K\gamma|\xi-\xi'|\int_{0}^{\tilde{L}}e^{-\alpha u}e^{M\tilde{L}}\,du\\\\
        &\leq & \displaystyle \frac{K\gamma}{\alpha}||\vartheta(\cdot;(t,\xi))-\vartheta(\cdot;(t,\xi'))||_{\infty}+
\frac{K\gamma}{\alpha}e^{M\tilde{L}}|\xi-\xi'|
\end{array}
\end{displaymath}
and the reader can deduce that 
\begin{displaymath}
|J_{12}|\leq \frac{K\gamma}{\alpha}||\vartheta(\cdot;(t,\xi))-\vartheta(\cdot;(t,\xi'))||_{\infty}+\tilde{D}|\xi-\xi'| \quad \textnormal{with} \quad
\tilde{D}=\frac{K\gamma}{\alpha}e^{M\tilde{L}}.
\end{displaymath}

The next inequality can be proved in a similar way
\begin{displaymath}
|J_{21}|\leq \frac{K\gamma}{\alpha}||\vartheta(\cdot;(t,\xi))-\vartheta(\cdot;(t,\xi'))||_{\infty}+\tilde{D}|\xi-\xi'|. 
\end{displaymath}

By using the inequalities stated above combined with (\ref{FPT}), he have
\begin{displaymath}
\begin{array}{rcl}
|\vartheta(t;(t,\xi))-\vartheta(t;(t,\xi'))|&\leq& |J_{11}|+|J_{22}|+|J_{12}|+|J_{21}|\\\\
&\leq & \displaystyle \frac{4K \mu}{\alpha}e^{-\alpha \tilde{L}}+2\tilde{D}|\xi-\xi'|\\\\
& &\displaystyle  + \frac{2K\gamma}{\alpha}||\vartheta(\cdot;(t,\xi))-\vartheta(\cdot;(t,\xi'))||_{\infty}.
\end{array}
\end{displaymath}

Note that
\begin{displaymath}
\begin{array}{rcl}
||\vartheta(\cdot;(t,\xi))-\vartheta(\cdot;(t,\xi'))||_{\infty}&\leq & \displaystyle 
\frac{4K \mu}{\alpha}e^{-\alpha \tilde{L}}+2\tilde{D}|\xi-\xi'|\\\\
& &\displaystyle  + \frac{2K\gamma}{\alpha}||\vartheta(\cdot;(t,\xi))-\vartheta(\cdot;(t,\xi'))||_{\infty}.
\end{array}
\end{displaymath}

By defining $\Gamma^{*}:=\frac{2K\gamma}{\alpha}<1$,
and using the inequality
\begin{displaymath}
|\vartheta(t;(t,\xi))-\vartheta(t;(t,\xi'))|\leq ||\vartheta(\cdot;(t,\xi))-\vartheta(\cdot;(t,\xi'))||_{\infty},   
\end{displaymath}
we easily deduce that
\begin{displaymath}
\begin{array}{rcl}
|\vartheta(t;(t,\xi))-\vartheta(t;(t,\xi'))|&\leq & \displaystyle \frac{4K \mu e^{-\alpha \tilde{L}}}{\alpha(1-\Gamma^{*})}+\frac{2\tilde{D}}{1-\Gamma^{*}}|\xi-\xi'|.
\end{array}
\end{displaymath}

Finally, for any $\varepsilon>0$, we can choose 
\begin{displaymath}
\tilde{L} \geq \displaystyle \frac{1}{\alpha}\ln\Big(\frac{8K\mu}{\alpha \varepsilon (1-\Gamma^{*})}\Big),
\end{displaymath}
which implies that
$\frac{4K\rho^{*} \mu}{\alpha(1-\Gamma^{*})}e^{-\alpha \tilde{L}} <\varepsilon/2$. By using this fact, we obtain 
that
\begin{displaymath}
\forall\varepsilon>0 \,\,\exists \delta=\frac{\varepsilon}{4\tilde{D}(1-\Gamma^{*})}>0 \,\, \textnormal{such that} \,\,
|\xi-\xi'|<\delta \Rightarrow |\vartheta(t;(t,\xi))-\vartheta(t;(t,\xi'))|<\varepsilon
\end{displaymath}
and the uniform continuity follows. 
\end{proof}

\section{Comments and References}

\noindent \textbf{1)} The proof of Theorem \ref{intermedio} is inspired in the global linearization result from K.J. Palmer \cite{Palmer},
where construction of the maps $H$ and $G$ is carried out in a nicely way but its continuity with respect to $t$ is only sketched. Furthermore, the uniform continuity with respect to this time variable has been proved by 
using the boundedness of $A(t)$ and adapting the proof of a global linearization from J. Shi and K. Xiong \cite{Shi} carried out for a more general
family. 

On the other hand, we emphasize that our demonstration has made an effort to be clear in several intricate details of Lemmas \ref{lema2} and
\ref{lema1}. We also highlight
the triple characterization of the maps $t\mapsto H[t,x(t,\tau,\xi)]$ and $t\mapsto G[t,y(t,\tau,\eta)]$
described in (\ref{3-0}) and (\ref{homeosofit}), respectively.

\medskip

\noindent \textbf{2)} The methods and ideas developed by K.J. Palmer in \cite{Palmer} have been
extended for several families of equations: difference equations \cite{Reinfelds97}, piecewise constant delay equations \cite{Papas}, impulsive equations \cite{LFP} and time--scales equations \cite{Hilger,Potzsche-2008}. We also refer to \cite{Reinfelds} for dichotomies beyond the exponential.

\medskip

\noindent \textbf{3)} The global linearization result of Palmer was applied by R. Naulin \cite{Naulin}
in order to study asymptotic properties of the quasilinear system
$$
\dot{x}=Ax+f(t,x)
$$
where $A$ has purely imaginary eigenvalues and the perturbation verifies
$$
\lim\limits_{x\to 0}f(t,x)=0 \quad \textnormal{and} \quad \lim\limits_{x\to 0}Df(t,x)=0 
$$
uniformly with respect to $t$.

\medskip

\noindent \textbf{4)} The boundedness of the nonlinearity played a key role in the global linearization
carried out in this chapter and we point out the existence of different approaches for global linearization 
which drop this assumption. In this context, a first result in this line is obtained by F. Lin in \cite{Lin2} who showed that if the linear system has an exponential dichotomy with $P = I$ on $\mathbb{R},$ then the linearization is of class $C^0.$ In order to obtain his result Lin uses the concept of \textit{almost reducibility}, \textit{i.e}, the linear system can be written as a linear diagonal system perturbed by a bounded linear term where the diagonal part is contained in the spectrum of the exponential dichotomy (see Ch.5), and it uses the concept of \textit{crossing times} with respect to the unit sphere. A second approach is carried out by I. Huerta in \cite{Huerta} which generalizes the previous work of Lin to a nonuniform framework, that is, it is constructed a linearization of class $C^0$ in two steps: the first one considers to write the linear system on $\mathbb{R}^+$ as in Lin's case where the diagonal part lies on the spectrum of the nonuniform exponential dichotomy (see \cite{Chu} and \cite{Zhang} for details abouth this spectrum), and the second step is devoted to construct a suitable Lyapunov function that plays the role of crossing time with respect to the unit sphere. Recently, M. Wu and Y. Xia in \cite{WX} generalize the previous works considering projections different from trivial ones.

\section{Exercises}

\begin{itemize}
   \item[1.-] If $t\mapsto y(t,\tau,\eta)$ is the solution of (\ref{sistema1}) passing trough $\eta$
   at $t=\tau$, prove that
   $$
   G[t,y(t,\tau,\eta)]=X(t,\tau)G(\tau,\eta).
   $$

\item[2.-] If $t\mapsto x(t,\tau,\xi)$ and $t\mapsto y(t,\tau,\eta)$ 
are respectively solutions of (\ref{lineal-efes}) and (\ref{sistema1})
passing respectively trough $\xi$ and $\eta$ at $t=\tau$, use Lemmas \ref{lema2} and \ref{composition2} 
to prove that
\begin{displaymath}
 x(t,\tau,\xi)=y(t,\tau,H(\tau,\xi)) +\vartheta(t;(\tau,H(\tau,\xi))).  
\end{displaymath}

\item[3.-] If $t\mapsto x(t,\tau,\xi)$ and $t\mapsto y(t,\tau,\eta)$ 
are respectively solutions of (\ref{lineal-efes}) and (\ref{sistema1})
passing respectively trough $\xi$ and $\eta$ at $t=\tau$, use Lemmas \ref{lema1} and \ref{composition} 
to prove that
\begin{displaymath}
 y(t,\tau,\eta)=x(t,\tau,G(\tau,\eta)) +\chi(t;(\tau,G(\tau,\eta))).  
\end{displaymath}

\item[4.-] Use exercises 2 and 3 to prove that 
\begin{displaymath}
 \vartheta(t;(\tau,G(\tau,\eta)))+\chi(t;(\tau,\eta))=0 \quad \textnormal{for any $\eta\in \mathbb{R}^{n}$}.    
\end{displaymath}

\item[5.-] Use exercises 2 and 3 to prove that 
\begin{displaymath}
 \vartheta(t;(\tau,\xi)))+\chi(t;(\tau,H(\tau,\xi)))=0 \quad \textnormal{for any $\xi\in \mathbb{R}^{n}$}.    
\end{displaymath}
   
\item[6.-] Use the identities (\ref{inve1})--(\ref{inve2}) together with
the Definitions of $H$ and $G$ given respectively by (\ref{DefH}) and (\ref{homeo-00}) to prove that:
\begin{displaymath}
\left\{\begin{array}{rcl}
G(t,\eta) &=& \eta -\chi(t;(t,G(t,\eta))) \quad \textnormal{for any $\eta\in \mathbb{R}^{n}$},\\
H(t,\xi) &=& \xi - \vartheta(t;(H(t,\xi))) \quad \textnormal{for any $\xi\in \mathbb{R}^{n}$}.
\end{array}\right.
\end{displaymath}

\item[7.-] Let us consider the quasilinear systems
\begin{subequations}
  \begin{empheq}{align}
  \dot{x}  = A(t)x+f_{1}(t,x)    \label{se1}, \\
 \hspace{-0.5cm}  \dot{y} = A(t)y+f_{2}(t,y) \label{se2},
  \end{empheq}
\end{subequations}
where the linear part has exponential dichotomy on $\mathbb{R}$ with constants $K$ and $\alpha$
and also satisfies \textbf{(A1)} whereas both nonlinearities verifies \textbf{(A2)}--\textbf{(A3)}
with $\mu$ and $\alpha$. Prove that if $2K\gamma<\alpha$ then there exists a map $H\colon \mathbb{R}\times \mathbb{R}^{n}\to \mathbb{R}^{n}$ such that for any solution $t\mapsto x(t)$ of \eqref{se1} it follows that $t\mapsto H[t,x(t)]$ is solution
of \eqref{se2} and viceversa.

\item[8.-] If we assume that $t\mapsto A(t)$ and $t\mapsto f(t,x)$ are continuous and $\omega$--periodic,
prove that the exists a unique $\omega$--periodic map $t\mapsto \chi(t;(\tau,\eta))$ in Lemma \ref{lemme-0}
by considering a suitable function $y(t,\tau,\eta)$. \textit{Hint}: See Corollary \ref{b-periodic} from Chapter 2.

\item[9.-] Let us consider the nonlinear control system
\begin{displaymath}
 \dot{x}=A(t)x+B(t)u(t)+f(t,u(t)),   
\end{displaymath}
where $A\in M_{nn}(\mathbb{R})$, $B\in M_{np}(\mathbb{R})$, $u\in \mathbb{R}^{p}$ are continuous 
and $f\colon \mathbb{R}\times \mathbb{R}^{p}\to \mathbb{R}^{n}$ verifies
$$
|f(t,u)|\leq \mu \quad \textnormal{and} \quad
|f(t,u)-f(t,u')|\leq \gamma |u-u'|
$$
for any $(t,u,u')\in \mathbb{R}\times \mathbb{R}^{p}\times \mathbb{R}^{p}$. 

If there exists a bounded continuous map $t\mapsto F(t)\in M_{pn}(t)$ such that the linear system $\dot{x}=[A(t)+B(t)F(t)]x$ has an exponential dichotomy
on $\mathbb{R}$, find conditions on $||F(t)||_{\infty}$ and the exponential dichotomy constants ensuring
the topological equivalence with the system
\begin{displaymath}
 \dot{x}=[A(t)+B(t)F(t)]x+f(t,F(t)x),  
\end{displaymath}
which arises from the feedback control $u(t)=F(t)x$. 

\item[10.-]Give more details for Remark \ref{uniqueness}.

\item[11.-] Prove in detail Proposition \ref{T2}.

\item[12.-] Verify carefully identity \eqref{uniqueness1B}.
    
\item[13.-] Give details in order to obtain \eqref{inve2}.

\item[14.-] Deduce identities \eqref{estimacion1} and \eqref{estimacion2}.

\item[15.-] Give explicit examples that verify the hypothesis of Theorem \ref{intermedio}.

\end{itemize}

\chapter{Differentiability of Global Linearization}

\section{Motivation} The previous chapter was devoted to construct a global homeomorphism
relating the solutions of the linear system
\begin{equation}
\label{linF}
\dot{x}=A(t)x
\end{equation}
with the solutions of its quasilinear perturbation
\begin{equation}
\label{nolinF}
\dot{y}=A(t)y+f(t,y)
\end{equation}
under the following assumptions:
\begin{itemize}
 \item[\textbf{(A1)}] The matrix $A(\cdot)$ of the linear system is bounded, namely
 $$
 \sup\limits_{t\in \mathbb{R}}||A(t)||=M,
 $$
 \item[\textbf{(A2)}] there exists a positive constant $\mu$  such that
 \begin{displaymath}
|f(t,y)|\leq \mu \quad \textnormal{for any} \quad (t,y)\in \mathbb{R}\times\mathbb{R}^{n}, 
\end{displaymath} 
 \item[\textbf{(A3)}] there exists a positive constant $\gamma$ such that
\begin{displaymath}
|f(t,y)-f(t,y')| \leq \gamma|y-y'| \quad \textnormal{for any $t\in \mathbb{R}$ and $(y,y')\in \mathbb{R}^{n}$}.
\end{displaymath}
\end{itemize}

The smoothness of the homeomorphism $\eta \mapsto G(t,\eta)$ described by (\ref{homeo-00}) from Chapter 6
has started to be studied a few years ago in a series of articles. Nevertheless, an additional and more restrictive assumption
was considered:
\begin{itemize}
    \item[\textbf{(A4)}] The linear system \textnormal{(\ref{linF})} has an exponential dichotomy with projection $P = I$
    and constants $K$ and $\alpha$ such that
\begin{displaymath}
\frac{2K}{\alpha}<\gamma.    
\end{displaymath}
    
\end{itemize}

A first approach in order to establish the differentiability of this homeomorphism was done in \cite{CRJDE} where it was proved that; under suitable differentiability conditions for $y\mapsto f(t,y)$; the map $G$ is $C^2.$ After, in \cite{CMR} it was obtained that $G \in C^{r}$ with $r \geq 1$. This last work only considered homeomorphisms $\eta\mapsto G(t,\eta)$ for $t\geq 0$.

\section{First Approach on the whole line}

The main goal of this section is to study the differentiability properties
of the homeomorphism $\eta\mapsto G(t,\eta)$ for any $t\in \mathbb{R}$. To achieve this objective,
it will be interesting to consider the map $\xi\mapsto y(t,t_{0},\xi)$
and its properties. Indeed, if $y\mapsto f(t,y)$ is $C^{1}$ for any $t\in \mathbb{R}$, it is well known (see \textit{e.g.} \cite[Chap. 2]{Coddington}) that
$\partial y(t,t_{0},\xi)/\partial \xi=y_{\xi}(t,t_{0},\xi)$ satisfies the matrix differential equation
\begin{equation}
\label{MDE1}
\left\{
\begin{array}{rcl}
\displaystyle \frac{d}{dt}\frac{\partial y}{\partial \xi}(t,t_{0},\xi)&=& \displaystyle\{A(t)+Df(t,y(t,t_{0},\xi))\}\frac{\partial y}{\partial \xi}(t,t_{0},\xi),\\\\
\displaystyle \frac{\partial y}{\partial \xi}(t_{0},t_{0},\xi)&=&I.
\end{array}\right.
\end{equation}

As $\xi=(\xi_{1},\ldots,\xi_{n})$, the above equation can be written as a set of $n$ systems of
ordinary differential equations
\begin{equation*}
%\label{MDE1-vec}
\left\{
\begin{array}{rcl}
\displaystyle \frac{d}{dt}\frac{\partial y}{\partial \xi_{i}}(t,t_{0},\xi)&=& \displaystyle\{A(t)+Df(t,y(t,t_{0},\xi))\}\frac{\partial y}{\partial \xi_{i}}(t,t_{0},\xi),\\\\
\displaystyle \frac{\partial y}{\partial \xi_{i}}(t_{0},t_{0},\xi)&=& e_{i},
\end{array}\right.
\end{equation*}
for all $i=1,\ldots,n$.

Moreover, it is proved that (see \textit{e.g.}, Theorem 4.1 from \cite[Ch.V]{Hartman}) if $y\mapsto f(t,y)$ is $C^{r}$ 
with $r>1$ for any $t\in \mathbb{R}$, then the map
$\xi \mapsto y(t,t_{0},\xi)$ is also $C^{r}$. In particular, if $f$ is $C^{2}$, we can verify that
the second derivatives $\partial^{2}y(t,t_{0},\xi)/\partial \xi_{j}\partial\xi_{i}$  are solutions
of the system of differential equations
\begin{equation}
\label{MDE15}
\left\{
\begin{array}{rcl}
\displaystyle \frac{d}{dt}\frac{\partial^{2}y}{\partial\xi_{j}\partial\xi_{i}}&=&\displaystyle
\{A(t)+Df(t,y)\}\frac{\partial^{2}y}{\partial\xi_{j}\partial\xi_{i}}
+D^{2}f(t,x)\frac{\partial y}{\partial\xi_{j}}\frac{\partial y}{\partial\xi_{i}} \\\\
\displaystyle \frac{\partial^{2} y}{\partial\xi_{j}\partial\xi_{i}}  &=&0,
\end{array}\right.
\end{equation}
with $y:=y(t,t_{0},\xi)$, for any $i,j=1,\ldots,n$.

\medskip
Now, let us introduce the following conditions on the nonlinearities:
\begin{enumerate}

\item[\textbf{(D1)}] $y\mapsto f(t,y)$ is
$C^{2}$  and, for any fixed $t$, its first derivative is such that
\begin{displaymath}
\int_{-\infty}^{t}\,\sup\limits_{\xi \in \mathbb{R}^{n}}||X(t,r)Df(r,y(r,0,\xi))X(r,t)||\,dr<1.
\end{displaymath}
\item[\textbf{(D2)}] For any fixed $\in \mathbb{R}$ and any initial condition $\xi$,
it follows that 
\[
\int_{-\infty}^{t}\textnormal{Tr}\, Df(r,y(r,t,\xi)\,dr<+\infty.
\]

\item[\textbf{(D3)}] For any fixed $t\in \mathbb{R}^{n}$ and $i,j=1,\ldots,n$, the following limit exists
\begin{equation*}
%\label{D2}
\lim\limits_{s\to -\infty}\frac{\partial Z_{\eta}(s)}{\partial \eta_{j}}e_{i},
\end{equation*}
where $\eta=(\eta_{1},\ldots,\eta_{n})$, $e_{i}$ is the $i$--th component of the canonical
basis of $\mathbb{R}^{n}$ and $Z_{\eta}(s)$ is a fundamental matrix of the $\eta$--parameter dependent system
\begin{equation}
\label{L2}
\dot{z}=F(s,\eta)z
\end{equation}
satisfying $Z_{\eta}(t)=I$, where $F(r,\eta)$ is defined as follows
\begin{equation}
\label{L1}
F(r,\eta)=X(t,r)Df(r,y(r,t,\eta))X(r,t).
\end{equation}
\end{enumerate}

\begin{remark}
%\label{about-hyp}
We will see that the construction of the homeomorphism $G$ considers the behavior of
$y(t,0,\xi)$ for any $t\in (-\infty,\infty)$. In particular, to prove that $G$ is a $C^{2}$ preserving orientation diffeomorphism
will require to know the behavior on $(-\infty,t]$. Indeed, notice that:

\textbf{(D1)} is a technical assumption, introduced to ensure that the homeomorphism $y\mapsto G(t,y)$ is $C^{1}$. In fact,
by using (\ref{unicite-HP2}) from Chapter 6 we have that $G[s,t_{0},y(s,t_{0},\xi)]=x(s,t_{0},G(t_{0},\xi))$ is a solution
of (\ref{linF}) passing through $G(t_{0},\xi)$ at $s=t_{0}$. In addition, by using (\ref{DE-I1}) from Chapter 2 and recalling that $P(\cdot)=I$, it follows that $|G[s,y(s,t_{0},\xi)]|\to +\infty$ when $s\to -\infty$ since
\begin{displaymath}
G[s,t_{0},y(s,t_{0},\xi)]=|x(s,t_{0},G(t_{0},\xi))|\geq \frac{|G(t_{0},\xi)|}{K}e^{\alpha(t_{0}-s)},    
\end{displaymath}
this fact combined with (\ref{cota1}) from Chapter 6 and triangle's inequality implies that $|y(s,t_{0},\xi)|\to +\infty$. In consequence,
\textbf{(D1)} suggest that the asymptotic behavior of $Df(s,y)$ (when $|y|\to +\infty$ and $s\to -\infty$) must ensure integrability.

%Moreover, let us note that \textbf{(D1)} appears in some results about asymptotical equivalence (see \emph{e.g.}, \cite{Akhmet},\cite{Rab}).

\textbf{(D2)} is introduced in order to assure that $G$ is a preserving orientation diffeomorphism.
We emphasize that this asumption is related to Liouville's formula.

\textbf{(D3)} is introduced to ensure that $G$ is a $C^{2}$ diffeomorphism.
\end{remark}

\begin{theorem}
\label{anexo-palmer}
If \eqref{linF} and \eqref{nolinF} satisfy \textnormal{\textbf{(A1)--(A4)}}  and \textnormal{\textbf{(D1)--(D2)}}, then, for
any fixed $t$, the function $y\mapsto G(t,y)$ is a preserving orientation diffeomorphism.
\end{theorem}

\begin{proof}

The proof is decomposed in several steps.

\noindent\textit{Step 1:} Differentiability of the map $\eta\mapsto G(t,\eta)$. By using the fact that $y\mapsto f(t,y)$ is $C^{1}$ for any $t\in \mathbb{R}$ combined with (\ref{MDE1}) and (\ref{MT1}) from Chapter 1
we can deduce that:
\begin{equation}
\label{delicat}
\begin{array}{rcl}
\displaystyle \frac{\partial G(t,\eta)}{\partial \eta}&=& \displaystyle I-\int_{-\infty}^{t}X(t,s)Df(s,y(s,t,\eta))\frac{\partial y(s,t,\eta)}{\partial \eta}\,ds\\\\
\displaystyle &=&\displaystyle  I-\int_{-\infty}^{t}\frac{d}{ds}\Big\{X(t,s)\frac{\partial y(s,t,\eta)}{\partial \eta}\Big\}\,ds\\\\
\displaystyle &=&\displaystyle \lim\limits_{s\to -\infty}X(t,s)\frac{\partial y(s,t,\eta)}{\partial \eta}.
\end{array}
\end{equation}

In consequence, the differentiability of $\eta\mapsto G(t,\eta)$ follows if and only if the limit above exists.

\medskip

\noindent\textit{Step 2:} The limit (\ref{delicat}) is well defined. By (\ref{MDE1}), we have verified that 
$s\mapsto \frac{\partial y}{\partial \eta}(s,t,\eta)$ is solution of the equation:
\begin{equation}
\label{perturbacion1}
\left\{\begin{array}{rcl}
\dot{Y}(s)&=&F(s,\eta)Y(s)\\
Y(t)&=&I.
\end{array}\right.
\end{equation}

%On the other hand, we can deduce that $\partial \phi(s,0,\xi)/\partial \xi$ is solution
%of the integral matrix equation:
%$$
%\frac{\partial \phi(s,0,\xi)}{\partial \xi}=\Psi(s,0)-\int_{s}^{0}\Psi(s,r)Df(r,\phi(r,0,\xi))\frac{\partial \phi(r,0,\xi)}{\partial \xi}\,dr,
%$$
%which is equivalent to
%\begin{displaymath}
%\begin{array}{rcl}
%\displaystyle \Psi(0,s)\frac{\partial \phi(s,0,\xi)}{\partial \xi}&=&\displaystyle I-\int_{s}^{0}\Psi(0,r)Df(r,\phi(r,0,\xi))\frac{\partial \phi(r,0,\xi)}{\partial \xi}\,dr\\\\
%                                                    &=&\displaystyle I-\int_{s}^{0}\Psi(0,r)Df(r,\phi(r,0,\xi))\Psi(r,0)\Psi(0,r)\frac{\partial \phi(r,0,\xi)}{\partial \xi}\,dr.
%\end{array}
%\end{displaymath}

By (\ref{MT1}) from Chapter 1, the identity (\ref{L1})  and (\ref{perturbacion1}), the reader can
verify that $s\mapsto Z_{\eta}(s)=X(t,s)\frac{\displaystyle \partial y(s,t,\eta)}{\displaystyle \partial \eta}$
is solution of the matrix differential equation
\begin{equation}
\label{modif1}
\left\{\begin{array}{rcl}
\displaystyle \dot{U}(s)&=&\big\{X(t,s)Df(s,y(s,t,\eta))X(s,t)\big\}U(s) \\\\
U(t)&=&I.
\end{array}\right.
\end{equation}

By using (\ref{bgpg}) from Chapter 1, we have the following estimation for the transition matrix
associated to $Z_{\eta}$
$$
||Z_{\eta}(s,t)||\leq \exp\Big(\Big|\int_{t}^{s}X(t,r)Df(r,y(r,t,\eta))X(r,t) \,dr\Big|\Big)
$$
and \textbf{(D1)} implies that (\ref{delicat}) is well defined.

\medskip

\noindent \textit{Step 3:} $\eta \mapsto G(t,\eta)$ is a preserving orientation diffeomorphism.
We already proved that $\eta \mapsto G(t,\eta)$ is $C^{1}$ for any $t\in \mathbb{R}$ and that
\begin{displaymath}
\frac{\partial G(t,\eta)}{\partial \eta}=\lim\limits_{s\to -\infty}X(t,s)\frac{\partial y(s,t,\eta)}{\partial \eta}=
\lim\limits_{s\to -\infty}Z_{\eta}(s),
\end{displaymath}
where $Z_{\eta}(\cdot)$ is a fundamental matrix of \eqref{modif1}, namely
\[
\dot{Z}(s)=F(s,\eta)Z(s),\qquad Z(t)=I.
\]
By Liouville's formula (see Theorem \ref{liouville} from Chapter 1) applied to \eqref{modif1} we have, for every $s\le t$,
\[
\det Z_{\eta}(s)=\exp\left(\int_{t}^{s}\textnormal{Tr}\, F(r,\eta)\,dr\right)
=\exp\left(-\int_{s}^{t}\textnormal{Tr}\, F(r,\eta)\,dr\right)>0.
\]
Since $F(r,\eta)=X(t,r)Df(r,y(r,t,\eta)X(r,t)$ and $X(r,t)=X(t,r)^{-1}$,
by using the cyclicity of the trace, we obtain the invariance of the trace under conjugation,
\[
\textnormal{Tr}\, F(r,\eta)=\textnormal{Tr}\, Df(r,y(r,t,\eta)).
\]
Therefore, for any fixed $t\in \mathbb{R}$ we have that
\[
\frac{\partial G(t,\eta)}{\partial \eta}=\lim_{s\to-\infty}\det Z_{\eta}(s)
=\exp\left(-\int_{-\infty}^{t}\textnormal{Tr}\, Df(r,y(r,t,\eta)\,dr\right).
\]
Hence, under assumption \textbf{(D2)}, we that $\det\frac{\partial G(t,\eta)}{\partial \eta}>0$ and $\frac{\partial G(t,\eta)}{\partial \eta}$ is invertible for any fixed $t\in \mathbb{R}$.
Since $\eta \mapsto G(t,\eta)$ is a homeomorphism, the inverse function theorem implies that $\eta \mapsto G(t,\eta)$ is a global
$C^{1}$ diffeomorphism preserving orientation.
\end{proof}

%\begin{remark}
%As $t\mapsto G[t,y(t)]$ is solution of (\ref{lineal-efes}), the uniqueness of the solution implies that
%\begin{equation}
%\label{TH}
%G[t,y(t,0,\xi)]=X(t,0)G[0,\xi].
%\end{equation}
% \end{remark}

\begin{remark}
%\label{equiv-asympt}
The matrix differential equation (\ref{perturbacion1}) can be seen as a perturbation of the matrix equation
\begin{equation}
\label{Lin-one}
\left\{\begin{array}{rcl}
\dot{W}(s)&=&A(s)W(s)\\
W(t)&=&I
\end{array}\right.
\end{equation}
related to (\ref{lineal-efes}). In addition, (\ref{Lin-one}) has a solution $s\mapsto W(s)=X(s,t)$ for any 
fixed $t\in \mathbb{R}$. Moreover, notice that $X(t,s)W(s)=I$ while Theorem \ref{anexo-palmer} says
that $s\mapsto X(t,s)\frac{\partial y(s,t,\eta)}{\partial \eta}$ exists when $s\to -\infty$. This fact prompt us that the asymptotic behavior of
(\ref{perturbacion1}) and (\ref{Lin-one}) when $s\to -\infty$ has some relation weaker than asymptotic equivalence. Indeed,
in \cite{Akhmet} it is proved that \textbf{(D1)} is a necessary condition for asymptotic equivalence between a linear system
and a linear perturbation.
\end{remark}

\begin{theorem}
\label{SegDer}
If \eqref{linF} and \eqref{nolinF} satisfy \textnormal{\textbf{(A1)--(A4)}}  and \textnormal{\textbf{(D1)--(D3)}} then, for
any fixed $t$, the function $\eta\mapsto G(t,\eta)$ is a $C^{2}$ preserving orientation diffeomorphism. 
\end{theorem}

\begin{proof}
Let us denote $\eta=(\eta_{1},\ldots,\eta_{n})$. As in the previous result, the proof will
be decomposed in several steps:\\

\noindent \textit{Step 1:} About $\partial^{2}G(t,\eta)/\partial \eta_{j}\partial \eta_{i}$.\\
\noindent For any $i,j\in \{1,\ldots,n\}$, we can verify that
\begin{displaymath}
\begin{array}{rcl}
\displaystyle \frac{\partial^{2}G}{\partial \eta_{j}\partial \eta_{i}}(t,\eta)&=& \displaystyle -\int_{-\infty}^{t}\hspace{-0.3cm}X(t,s)D^{2}f(s,y(s,t,\eta))
\frac{\partial y(s,t,\eta)}{\partial \eta_{j}}\frac{\partial y(s,t,\eta)}{\partial \eta_{i}}\,ds\\\\
&&\displaystyle -\int_{-\infty}^{t}X(t,s)Df(s,y(s,t,\eta))\frac{\partial^{2}y(s,t,\eta)}{\partial \eta_{j}\partial 
\eta_{i}}\,ds.
\end{array}
\end{displaymath}

Now, by using (\ref{MDE15}) and (\ref{MT1}) from Chapter 1, the reader can see that
\begin{equation}
\label{uruguay1}
\begin{array}{rcl}
 \displaystyle  \frac{\partial^{2}G}{\partial \eta_{j}\partial \eta_{i}}(t,\eta) &=&\displaystyle -\int_{-\infty}^{t}\frac{d}{ds}\Big\{X(t,s)\frac{\partial^{2}y(s,t,\eta)}{\partial \eta_{j}\partial \eta_{i}}\Big\}\,ds\\\\
                                &=&\displaystyle \lim\limits_{s\to -\infty}X(t,s)\frac{\partial^{2}y(s,t,\eta)}{\partial y_{j}\partial y_{i}}
\end{array}
\end{equation}
and the existence of $\partial^{2}G(t,\eta)/\partial \eta_{j}\partial \eta_{i}$ follows if and only if the limit above exists.

\medskip

\noindent\textit{Step 2:}  $\partial^{2}G(t,\eta)/\partial \eta_{j}\partial \eta_{i}$ is well defined.\\

By using (\ref{MDE1}) and (\ref{MT1}) from Chapter 1, we can see that $s\mapsto X(t,s)\partial y(s,t,\eta)/\partial \eta_{i}$ is a
solution of (\ref{L2}) passing through $e_{i}$ at $s=t$. In consequence, we can deduce that
$$
X(t,s)\frac{\partial y(s,t,\eta)}{\partial \eta_{i}}=Z_{\eta}(s)e_{i}
$$
and
$$
X(t,s)\frac{\partial^{2} y(s,t,\eta)}{\partial \eta_{j}\partial \eta_{i}}=\frac{\partial Z_{\eta}(s)}{\partial \eta_{j}}e_{i}.
$$

By \textbf{(D3)}, the last identity has limit when $s\to -\infty$ and
$\partial^{2}G(t,\eta)/\partial \eta_{j}\partial \eta_{i}$ is well defined and continuous with respect to $\eta$
for any fixed $t\in \mathbb{R}$.
\end{proof}

\section{Second approach on the half--line}

A careful reading of the previous section shows that the first and second derivatives of
$\eta \mapsto G(t,\eta)$ for any $t\in \mathbb{R}$ are related to the existence of limits when $t\to -\infty$, which are ensured 
by rather restrictive assumptions as \textbf{(D1)} and \textbf{(D3)}. In order to address this issue, notice that
theses assumptions can be dropped if we restrict the parameter $t\in [0,+\infty)$. Nevertheless, this prompt us to 
tailor the topological equivalence Definition \ref{TopEq} from Chapter 6 for any $t\in J\subset \mathbb{R}$ as follows:

\begin{definition}
\label{defnueva}
The systems \textnormal{(\ref{lineal-efes})} and \textnormal{(\ref{sistema1})} are $J$--continuously topologically equivalent if 
there exists a function $H\colon J \times \mathbb{R}^{n}\to \mathbb{R}^{n}$ with the properties
\begin{itemize}
\item[(i)] If $t\mapsto x(t)$ is a solution of \textnormal{(\ref{linF})}, then $t\mapsto H[t,x(t)]$ is a solution 
of \textnormal{(\ref{nolinF})},
\item[(ii)] $H(t,u)-u$ is bounded in $J \times \mathbb{R}^{n}$,
\item[(iii)]  For each fixed $t\in J$, $u\mapsto H(t,u)$ is an homeomorphism of $\mathbb{R}^{n}$,
\item[(iv)] $H$ is continuous in any $(t,u)\in J\times \mathbb{R}^{n}$.
\end{itemize}
In addition, the function $u\mapsto G(t,u)=H^{-1}(t,u)$ has properties \textnormal{(ii)--(iv)} and maps solutions of \textnormal{(\ref{nolinF})} into solutions of \textnormal{(\ref{linF})}.
\end{definition}

The above definition encompasses Definition \ref{TopEq} since $J=\mathbb{R}$ is only a particular case.
Furthermore, observe that the continuity of $H$ and $G$ in any $(t,u)\in J\times \mathbb{R}^{n}$ has not been considered in Definition \ref{TopEq}. However, a careful
review of the topological equivalence literature shows that this property has a disparate treatment. Indeed, it has  been considered by Palmer \cite[p.754]{Palmer} and omitted by
Shi \textit{et al.} \cite[p.814]{Shi}. Furthermore, in order to simplify the text, we will refer to $H$ and $G$ as the \textit{Palmer's homeomorphisms}.

To the best of our knowledge, the differentiability of the Palmer's homeomorphisms $u\mapsto H(t,u)$ and $u\mapsto G(t,u)$ in for any $t\in J$ has not been studied before \cite{CMR}, which prompt us to introduce the following definition:

\begin{definition}
\label{TopEqCr}
The systems \textnormal{(\ref{linF})} and \textnormal{(\ref{nolinF})} are $C^{r}$ continuously topologically equivalent on $J$ if: 
\begin{itemize}
\item[(i)] The systems are $J-$continuously topologically equivalent,
\item[(ii)] For any fixed $t\in J$; the map $u\mapsto H(t,u)$ is a $C^{r}$--diffeomorphism of $\mathbb{R}^{n}$,
with $r\geq 1$,
\item[(iii)] The partial derivatives of $H$ and $G$ up to order $r$ with respect to $u$ are continuous functions of $(t,u)\in J\times \mathbb{R}^{n}$. 
\end{itemize}
\end{definition}

In this subsection, we work on $J = \mathbb{R}^+$ and we are interested to show that the systems \textnormal{(\ref{linF})} and \textnormal{(\ref{nolinF})} are $C^{r}$ continuously topologically equivalent on $\mathbb{R}^+.$ For this purpose, we require the properties  \textnormal{\textbf{(A1)--(A4)}} to be satisfied for $t \in \mathbb{R}^+$ rather than $t \in \mathbb{R}.$

The following Theorem has a proof similar to the linearization result studied in the Chapter 6. Nevertheless,
the restriction of $t$ to the interval $[0,+\infty)$ induces some subtleties that deserve to be treated in detail. In particular, a noteworthy consequence is that all the solutions of (\ref{linF}) and (\ref{nolinF}) are bounded on $[0,+\infty)$.

\begin{theorem}
\label{teonuevo}
If the systems \eqref{linF} and \eqref{nolinF} satisfy \textnormal{\textbf{(A1)--(A4)}} for any $t\in [0,+\infty)$, then they are $\mathbb{R}^+-$continuously topologically equivalent. 
\end{theorem}

\begin{proof}
    In order to make a more readable proof, we will decompose it in several steps. Namely, the step 1 defines two auxiliary systems whose solutions are used in the step 2 to construct the maps $H$ and $G$. To prove that these maps establish a topological equivalence, the properties (i)--(ii) from Definition \ref{defnueva} are verified in the step 3, while the uniform continuity is proved in the steps 4 and 5. Finally the continuity of $H$ in any $(t_0, \xi_0) \in \mathbb{R}^{+} \times \mathbb{R}^n$ is proved in the step 6.

\medskip

\noindent\textit{Step 1: Preliminaries.} We will consider the initial value problems
\begin{equation}
\label{pivote2}
\left\{\begin{array}{rcl}
\dot{w}&=&A(t)w-f(t,y(t,\tau,\eta))\\
w(0)&=& 0,
\end{array}\right.
\end{equation}
and
\begin{equation}
\label{pivote1}
\left\{\begin{array}{rcl}
\dot{z}&=&A(t)z+f(t,x(t,\tau,\xi)+z)\\
z(0)&=& 0.
\end{array}\right.
\end{equation}

By using the variation of parameters formula we have that
\begin{equation}
\label{w-star}
w^{*}(t;(\tau,\eta))=-\int_{0}^{t} X(t,s)f(s,y(s,\tau,\eta))\,ds
\end{equation}
is the unique solution of (\ref{pivote2}). Now, in order to study (\ref{pivote1}) let $BC(\mathbb{R}^+, \mathbb{R}^n)$ be the Banach space of bounded continuous functions with the supremum norm. Now,  for any pair $(\tau, \xi) \in \mathbb{R}^+ \times \mathbb{R}^n,$ by using \textbf{(A2)} we define the  operator $\Gamma_{(\tau, \xi)} \colon BC(\mathbb{R}^+, \mathbb{R}^n) \to BC(\mathbb{R}^+, \mathbb{R}^n) $ as follows 
\begin{equation*}
\phi(t) \mapsto \Gamma_{(\tau, \xi)} \phi(t) := \displaystyle \int_0^t X(t,s) f(s,x(s,\tau,\xi) + \phi) \, ds. 
\end{equation*}

Since $2\gamma K/\alpha<1$ it is easy to see by {\bf{(A3)}} that the operator $\Gamma_{(\tau, \xi)} $ is a contraction and by the Banach fixed point theorem it follows that
\begin{displaymath}
z^{*}(t;(\tau,\xi))=\int_{0}^{t} X(t,s)f(s,x(s,\tau,\xi)+z^{*}(s;(\tau,\xi))) \, ds
\end{displaymath}
is the unique solution of (\ref{pivote1}).

On the other hand, by uniqueness of solutions it can be proved that 
\begin{equation}
\label{identity1}
z^{*}(t;(\tau,\xi))=z^{*}(t;(r,x(r,\tau,\xi))) \quad \textnormal{for any $r\geq 0$},   
\end{equation}
and
\begin{equation*}
%\label{identity2}
w^{*}(t;(\tau,\eta))=w^{*}(t;(r,y(r,\tau,\eta))) \quad \textnormal{for any $r\geq 0$}.    
\end{equation*}

\noindent\textit{Step 2: Construction of the maps $H$ and $G$.} For any $t\geq 0$ we define the maps $H(t,\cdot)\colon \mathbb{R}^{n}\to \mathbb{R}^{n}$ and
$G(t,\cdot)\colon \mathbb{R}^{n}\to \mathbb{R}^{n}$ as follows:
\begin{displaymath}
\begin{array}{rcl}
H(t,\xi)&:=&\displaystyle  \xi+\int_{0}^{t} X(t,s)f(s,x(s,t,\xi)+z^{*}(s;(t,\xi)))\,ds \\\\
&=&\xi + z^{*}(t;(t,\xi)),
\end{array}
\end{displaymath}
and
\begin{equation*}
%\label{Homeo-G}
\begin{array}{rcl}
G(t,\eta)&:=&\displaystyle \eta -\int_{0}^{t} X(t,s)f(s,y(s,t,\eta))\,ds \\\\
&=&\eta+w^{*}(t;(t,\eta)).
\end{array}
\end{equation*}

Let $t\mapsto x(t,\tau,\xi)$ be a solution of (\ref{linF}) passing trough $\xi$ at $t=\tau$. By using (\ref{identity1}), we can verify that
\begin{displaymath}
\begin{array}{rcl}
H[t,x(t,\tau,\xi)]&=&\displaystyle  x(t,\tau,\xi)+\int_{0}^{t}X(t,s)f(s,x(s,t,x(t,\tau,\xi)+z^{*}(s;(t,x(t,\tau,\xi))))\,ds \\\\
&=& \displaystyle x(t,\tau,\xi)+\int_{0}^{t}X(t,s)f(s,x(s,\tau,\xi)+z^{*}(s;(\tau,\xi)))\,ds \\\\
&=&x(t,\tau,\xi)+z^{*}(t;(\tau,\xi)).
\end{array}
\end{displaymath}

\noindent\textit{Step 3: $H$ and $G$ satisfy properties \textnormal{(i)--(ii)} of Definition \ref{defnueva}.}
By (\ref{linF}) and (\ref{pivote1}) combined with the above equality, we have that
$$
\begin{array}{rcl}
\displaystyle \frac{\partial }{\partial t} H[t,x(t,\tau,\xi)] & = & \displaystyle \frac{\partial }{\partial t} x(t, \tau, \xi) + \frac{\partial }{\partial t} z^*(t; (\tau, \xi))\\\\
& = & A(t)x(t,\tau,\xi) + A(t)z^*(t; (\tau, \xi)) + f(t,H[t,x(t,\tau,\xi)])\\\\
& = & A(t)H[t,x(t,\tau,\xi)] + f(t,H[t,x(t,\tau,\xi)]),
\end{array}
$$
then $t\mapsto H[t,x(t,\tau,\xi)]$ is solution of (\ref{sistema1}) passing trough 
$H(\tau,\xi)$ at $t=\tau$. As consequence of uniqueness of solution we obtain
\begin{equation}\label{conj1}
H[t, x(t,\tau, \xi)] = y(t,\tau, H(\tau,\xi)),
\end{equation}
similarly, it can be proved that $t\mapsto G[t,y(t,\tau,\eta)]$ is solution of (\ref{lineal-efes}) passing through $G(\tau, \eta)$
at $t=\tau$ and 
\begin{equation*}%\label{conj2}
G[t, y(t,\tau, \eta)] = x(t,\tau, G(\tau,\eta)) = X(t, \tau)G(\tau,\eta),
\end{equation*}
and the property (i) follows.
Secondly, by using (\textbf{A2}) and (\textbf{A4})  it follows that
\begin{displaymath}
|H(t,\xi)-\xi|\leq  K\mu\int_{0}^{t}e^{-\alpha(t-s)}\,ds \leq \displaystyle \frac{K\mu}{\alpha} 
\end{displaymath}
for any $t\geq 0$. A similar inequality can be obtained for
$|G(t,\eta)-\eta|$ and the property (ii) is verified.

\medskip

\noindent\textit{Step 4: $H$ is bijective for any $t\geq 0$.} We will 
first show the identity $H(t, G(t, \eta)) = \eta$ for any $t\geq 0$. Indeed, 
$$
\begin{array}{rcl}
H[t,G[t,y(t,\tau,\eta)]] & = & G[t,y(t,\tau,\eta)] \\\\ & &  + \displaystyle \int_0^t X(t,s) f(s,x(s,t,G[t,y(t,\tau,\eta)])+z^*(s;(t,G[t,y(t,\tau,\eta)]))) \, ds\\\\
& = &  y(t,\tau, \eta)  -\displaystyle \int_0^t  X(t,s)f(s,y(s,\tau, \eta)) \, ds\\\\
&&+\displaystyle \int_0^t X(t,s) f(s,x(s,t,G[t,y(t,\tau,\eta)])+z^*(s;(t,G[t,y(t,\tau,\eta)]))) \, ds.
\end{array}
$$
Let $\omega(t) = | H[t, G[t,y(t,\tau,\eta)]] - y(t,\tau, \eta)| .$ Hence by using {\bf{(A2)}} and {\bf{(A3)}} we have that
$$
\begin{array}{ll}
 \omega(t) & =  \left |\displaystyle \int_0^t X(t,s) \{f(s,x(s,t,G[t,y(t,\tau,\eta)])+z^*(s;(t,G[t,y(t,\tau,\eta)]))) -f(s,y(s,\tau, \eta))\} \, ds \right |\\\\
&\leq  K \gamma  \displaystyle \int_0^t e^{-\alpha(t-s)} |\{x(s,t,G[t,y(t,\tau,\eta)])+z^*(s;(t,G[t,y(t,\tau,\eta)])) -y(s,\tau, \eta)\}| \, ds.
\end{array}
$$

The boundedness of all the solutions of systems \eqref{linF} and \eqref{nolinF} on $\mathbb{R}^{+}$
combined with $s\mapsto z^*(s;(t,G[t,y(t,\tau,\eta)])) \in BC(\mathbb{R}^{+},\mathbb{R}^{n})$ implies that $\sup\limits_{t\geq 0}\{\omega(t)\}<+\infty$.

\medskip

In addition, by using again the definition of $\xi\mapsto H(s,\xi)$ we can deduce that
$$
x(s,t,G[t,y(t,\tau,\eta)])+z^*(s;(t,G[t,y(t,\tau,\eta)])) = H[s,x(s,t,G[t,y(t,\tau,\eta)])],
$$
which leads to
\begin{displaymath}
\begin{array}{rcl}
\omega(t) &\leq &  K \gamma  \displaystyle \int_0^t e^{-\alpha(t-s)} |\{H[s,x(s,t,G[t,y(t,\tau,\eta)])] -y(s,\tau, \eta)\}| \, ds \\\\
& = &  K \gamma  \displaystyle \int_0^t e^{-\alpha(t-s)} |\{H[s,G[s,y(s,\tau,\eta)]] -y(s,\tau, \eta)\}| \, ds, 
\end{array}
\end{displaymath}
where the last step follows from the identities
$$
x(s,t,G[t,y(t,\tau,\eta)]) = x(s, \tau, G(\tau, \eta)) = G[s,y(s,\tau, \eta)].
$$

Therefore, we obtain
$$\omega(t) \leq K \gamma  \int_0^t e^{-\alpha(t-s)} \omega(s) \, ds \leq \frac{K \gamma}{\alpha} \displaystyle \sup_{s \in \mathbb{R}^+} \{\omega(s)\} \quad \rm{for \,\, all} \quad t \geq 0.
$$

Now, we take the supremum on the left side above and due to $K \gamma / \alpha < 1$ it follows that $\omega(t) = 0$ for any $t \geq 0.$ In particular, when we take $t = \tau$ we obtain  $H(\tau, G(\tau, \eta)) = \eta.$

\medskip

Next, we will prove that $G(t, H(t, \xi)) = \xi.$ In fact, due to (\ref{conj1}) we have that
$$
\begin{array}{rcl}
G[t,H[t,x(t,\tau,\xi)]] & = & H[t,x(t,\tau,\xi)] \\\\ & &  - \displaystyle \int_0^t X(t,s) f(s,y(s,t,H[t,y(x,\tau,\xi)])) \, ds\\\\
&= &  x(t,\tau, \xi)  \\\\
&& \displaystyle +\int_0^t  X(t,s) \{f(s,H[s,x(s,\tau,\xi)]) - f(s,y(s,\tau,H(\tau,\xi)))\} \, ds\\\\
& = &  x(t,\tau, \xi),
\end{array}
$$
and taking $t=\tau$ leads to $G(\tau,H(\tau,\xi))=\xi$. In consequence, for any $t\geq 0$, $\xi \mapsto H(t,\xi)$ is a bijection and $\eta \mapsto G(t,\eta)$ is its inverse. 

\medskip

\noindent\textit{Step 5: $H$ and $G$  are uniformly continuous for any fixed $t$.} Firstly, we prove that $\eta \mapsto G(t,\eta)$ is uniformly  continuous. 

\medskip

By using the property \textbf{(A1)} combined with Corollary \ref{Cota-BG1} from Chapter 1, we can see that
$||X(t,s)||\leq ||I||e^{M|t-s|}$ for any $t,s\geq 0$. Then Lemma \ref{M-alpha} from Chapter 2 says that $\alpha \leq M$.
Now, we construct the auxiliary functions $\theta,\theta_{0}\colon [0,+\infty)\to [0,+\infty)$ defined by
\begin{displaymath}
\theta(t)=1+ K\gamma \left(\frac{e^{(M+\gamma-\alpha)t}-1}{M+\gamma-\alpha}\right) \quad \textnormal{and} \quad
\theta_{0}(t)=\left\{\begin{array}{lcl}
K\gamma t & \textnormal{if} & \alpha=M,\\\\
K\gamma\left(\frac{e^{(M-\alpha)t}-1}{M-\alpha}\right)  &\textnormal{if} & \alpha<M.
\end{array}\right.
\end{displaymath}

Now, given $\varepsilon>0$, let us define the constants
\begin{equation}
\label{auxiliares}
L(\varepsilon)=\frac{1}{\alpha}\ln\left(\frac{2\mu K}{\alpha \varepsilon}\right), \,\, \theta_{0}^{*}=\max\limits_{t\in [0,L(\varepsilon)]}\theta_{0}(t) \,\, \textnormal{and} \,\, \theta^{*}=\max\limits_{t\in [0,L(\varepsilon)]}\theta(t).
\end{equation}

We will prove the uniform continuity of $G$ by considering two cases: 

\medskip

\noindent {Case i)} $t\in [0,L(\varepsilon)]$. By {\bf{(A2)}} and  {\bf{(A4)}} we can deduce that
\begin{equation}
\label{estimation}
|G(t,\eta) - G(t,\bar{\eta})|  \leq   |\eta - \bar{\eta} |+ K \gamma \displaystyle \int_0^t e^{-\alpha (t-s)} |y(s,t,\eta) - y(s,t,\bar{\eta})| \, ds.
\end{equation}

Now, by  {\bf{(A1)}},{\bf{(A3)}} and Proposition \ref{T1} from Chapter 6 we know that
for any $0\leq s \leq t$:
\begin{equation}
\label{ContCondIni}
 |y(s,t,\eta) - y(s,t,\bar{\eta})| \leq |\eta - \bar{\eta}| e^{(M+\gamma)(t-s)}.
\end{equation}

Upon inserting (\ref{ContCondIni}) in (\ref{estimation}), we obtain that
$$
\begin{array}{rcl}
|G(t,\eta) - G(t,\bar{\eta})| & \leq&  \left (1 +  K\gamma e^{(M + \gamma-\alpha) t} \displaystyle \int_0^t e^{-(M+\gamma-\alpha) s} \right )|\eta - \bar{\eta} | \\\\
& = & \displaystyle    \left (1 + K\gamma \left\{\frac{e^{(M+ \gamma-\alpha)t}-1}{M + \gamma-\alpha}\right\} \right )|\eta - \bar{\eta} |\\\\
& \leq & \theta(t)|\eta - \bar{\eta}| \leq \theta^{*}|\eta - \bar{\eta} |.
\end{array}
$$

\noindent Case ii) $t>L(\varepsilon)$. By {\bf{(A2)}}--{\bf{(A4)}}, we have that
\begin{equation}
\label{estimacion3}
\begin{array}{rcl}
|G(t,\eta)-G(t,\bar{\eta})| & \leq & \displaystyle |\eta-\bar{\eta}|+\mu K\int_{0}^{t-L(\varepsilon)}e^{-\alpha(t-s)}\,ds \\\\
&&\displaystyle +K\gamma\int_{t-L(\varepsilon)}^{t}e^{-\alpha(t-s)}|y(s,t,\eta)-y(s,t,\bar{\eta})|\,ds\\\\
& \leq & \displaystyle |\eta-\bar{\eta}|+\frac{\mu K}{\alpha}e^{-\alpha L(\varepsilon)}\\\\
&& \displaystyle + K\gamma \int_{0}^{L(\varepsilon)}e^{-\alpha u}|y(t-u,t,\eta)-y(t-u,t,\bar{\eta})|\,du.
\end{array}
\end{equation}

As in case i), the inequality (\ref{ContCondIni})
implies that
\begin{displaymath}
\begin{array}{rcl}
\displaystyle K\gamma \int_{0}^{L(\varepsilon)}e^{-\alpha u}|y(t-u,t,\eta)-y(t-u,t,\bar{\eta})|\,du 
& \leq & \displaystyle K\gamma \int_{0}^{L(\varepsilon)}e^{(M+\gamma-\alpha)u}|\eta-\bar{\eta}|\,du\\\\
& = & \displaystyle K\gamma \left\{\frac{e^{(M+\gamma-\alpha)L(\varepsilon)}-1}{M+\gamma-\alpha}\right\}|\eta-\bar{\eta}|.
\end{array}
\end{displaymath}

Upon inserting the above inequality in (\ref{estimacion3}) and using (\ref{auxiliares}), we have that
\begin{displaymath}
\begin{array}{rcl}
|G(t,\eta)-G(t,\bar{\eta})| &\leq & \displaystyle \left(1+ K\gamma \left\{\frac{e^{(M+\gamma-\alpha)L(\varepsilon)}-1}{M+\gamma-\alpha}\right\}\right)|\eta-\bar{\eta}|+\frac{\mu K}{\alpha}e^{-\alpha L(\varepsilon)}\\\\
&\leq & \displaystyle \theta^{*}|\eta-\bar{\eta}|+\frac{\varepsilon}{2}.
\end{array}
\end{displaymath}

Summarizing, given $\varepsilon>0$, there exists $L(\varepsilon)>0$ and $\theta^{*}>0$ such that:
\begin{displaymath}
|G(t,\eta)-G(t,\bar{\eta})|\leq \left\{\begin{array}{lcl}
\displaystyle \theta^{*}|\eta-\bar{\eta}| &\textnormal{if}& t\in [0,L(\varepsilon)]\\\\
\displaystyle \theta^{*}|\eta-\bar{\eta}|+\frac{\varepsilon}{2} &\textnormal{if}& t>L(\varepsilon),
\end{array}\right.
\end{displaymath}
then it follows that
$$
\forall \varepsilon>0 \,\exists
\delta(\varepsilon):=\frac{\varepsilon}{2\theta^{*}}\quad \textnormal{such that} \quad
|\eta-\bar{\eta}|<\delta \Rightarrow |G(t,\eta)-G(t,\bar{\eta})|<\varepsilon
$$
and the uniform continuity of $G$ follows.

\medskip

Finally, we will prove that $\xi \mapsto H(t,\xi)$ is uniformly continuous for any $t\geq 0$. As the identity is uniformly continuous, we will only prove that the map $\xi \mapsto z^{*}(t;(t,\xi))$ is uniformly continuous.

Note that the fixed point $z^{*}(t;(t,\xi))$ can be seen as the uniform limit on $\mathbb{R}^{+}$ of a sequence $z_{j}^{*}(t;(t,\xi))$ defined recursively as follows:
\begin{displaymath}
\left\{\begin{array}{rcl}
z_{j+1}^{*}(t;(t,\xi))&=& \displaystyle \int_{0}^{t}X(t,s)f(s,x(s,t,\xi)+z_{j}^{*}(s;(t,\xi))) \, ds \quad \textnormal{for any $j\geq 1$},\\\\
z_{0}^{*}(t;(t,\xi)) &=&     0.
\end{array}\right.
\end{displaymath}

The uniform continuity of each map $\xi \mapsto z_{j}^{*}(t;(t,\xi))$ will be proved inductively by following the lines of \cite{Jiang, Shi}. First, it is clear that 
$\xi\mapsto z_{0}^{*}(t;(t,\xi))$ verify this property. Secondly, we will assume  the inductive hypothesis 
\begin{displaymath}
\forall\varepsilon>0\,\exists\delta_{j}(\varepsilon)>0 \,\,\textnormal{s.t.} \, |\xi-\bar{\xi}|<\delta_{j} \Rightarrow |z_{j}^{*}(t;(t,\xi))-z_{j}^{*}(t;(t,\bar{\xi}))|<\varepsilon \,\, \textnormal{for any $t\geq 0$}.
\end{displaymath}

For the step $j+1$ and given $\varepsilon>0$, we will only consider $\alpha<M$ since the case $\alpha=M$ can be carried out easily. We will use the constants $L(\varepsilon)$ and $\theta_{0}^{*}$ defined in (\ref{auxiliares}) and introduce the notation 
$$
\Delta_{j}(t,\xi,\bar{\xi})=z_{j}^{*}(t;(t,\xi))-z_{j}^{*}(t;(t,\bar{\xi})).
$$

As before, we will distinguish the cases $t\in [0,L(\varepsilon)]$ and $t>L(\varepsilon)$. First, for $t\in [0,L(\varepsilon)]$ we use {\bf{(A1)}} combined with the estimation 
\begin{equation}
\label{linealCI}
|x(s,t,\xi) - x(s,t,\bar{\xi})| \leq |\xi - \bar{\xi}| e^{M|t-s|},
\end{equation}
and we can
verify that
\begin{displaymath}
\begin{array}{rcl}
|\Delta_{j+1}(t,\xi,\bar{\xi})| &\leq & \displaystyle K\gamma e^{-\alpha t}
\int_{0}^{t}e^{\alpha s}\{|x(s,t,\xi)-x(s,t,\bar{\xi})|+|\Delta_{j}(s,\xi,\bar{\xi})|\}\,ds \\\\
&\leq & \displaystyle K\gamma e^{-\alpha t}
\int_{0}^{t}e^{\alpha s}\{|\xi-\bar{\xi}|e^{M(t-s)}+||\Delta_{j}(\cdot,\xi,\bar{\xi})||_{\infty}\}\,ds\\\\
&\leq & \displaystyle  K\gamma\left\{\frac{e^{(M-\alpha)t}-1}{M-\alpha}\right\}|\xi-\bar{\xi}|+\frac{K\gamma}{\alpha}||\Delta_{j}(\cdot,\xi,\bar{\xi})||_{\infty}\\\\
&\leq& \displaystyle \theta_{0}^{*}|\xi-\bar{\xi}|+\frac{K\gamma}{\alpha}||\Delta_{j}(\cdot,\xi,\bar{\xi})||_{\infty},
\end{array}
\end{displaymath}
where $||\Delta_{j}(\cdot,\xi,\bar{\xi})||_{\infty}=\sup\limits_{t\geq 0}|\Delta_j(t,\xi,\bar{\xi})|$.

On the other hand, when $t>L(\varepsilon)$, we use \textbf{(A2)}, \textbf{(A4)} combined with the boundedness of $f$ in $[0,t-L(\varepsilon)]$ and Lipschitzness in $[t-L(\varepsilon),t)$ to deduce that \begin{displaymath}
\begin{array}{rcl}
|\Delta_{j+1}(t,\xi,\bar{\xi})|
&\leq & \displaystyle K\mu\int_{0}^{t-L(\varepsilon)}e^{-\alpha(t-s)}\,ds \\\\
&&\displaystyle + K\gamma
\int_{t-L(\varepsilon)}^{t}e^{-\alpha(t-s) }\{|x(s,t,\xi)-x(s,t,\bar{\xi})|+|\Delta_{j}(s,\xi,\bar{\xi})|\}\,ds.
\end{array}
\end{displaymath}

By using {\bf{(A4)}} combined with $u=t-s$, (\ref{auxiliares}) and (\ref{linealCI}), we have that
\begin{displaymath}
\begin{array}{rcl}
|\Delta_{j+1}(t,\xi,\bar{\xi})| &\leq& \displaystyle \frac{K\mu}{\alpha}e^{-\alpha L(\varepsilon)}
+\frac{K\gamma}{\alpha}||\Delta_{j}(\cdot,\xi,\bar{\xi})||_{\infty} \\\\
&&
\displaystyle +K\gamma\int_{0}^{L(\varepsilon)}e^{-\alpha u }|x(t-u,t,\xi)-x(t-u,t,\bar{\xi})|\,du \\\\
&\leq &\displaystyle \frac{K\mu}{\alpha}e^{-\alpha L(\varepsilon)}+\frac{K\gamma}{\alpha}||\Delta_{j}(\cdot,\xi,\bar{\xi})||_{\infty} +K\gamma |\xi-\bar{\xi}|\int_{0}^{L(\varepsilon)}e^{(M-\alpha) u }\,du \\\\
&\leq& \displaystyle
\frac{\varepsilon}{2}+\frac{K\gamma}{\alpha}||\Delta_{j}(\cdot,\xi,\bar{\xi})||_{\infty}+K\gamma \left\{\frac{e^{(M-\alpha)L(\varepsilon)}-1}{M-\alpha}\right\}|\xi-\bar{\xi}|\\\\
&\leq& \displaystyle 
\frac{\varepsilon}{2}+\frac{K\gamma}{\alpha}||\Delta_{j}(\cdot,\xi,\bar{\xi})||_{\infty} + \theta_{0}^{*}|\xi-\bar{\xi}|.
\end{array}
\end{displaymath}

Summarizing, for any $t\geq 0$ it follows that
\begin{displaymath}
|\Delta_{j+1}(t,\xi,\bar{\xi})|\leq 
\left\{\begin{array}{lcl}
\displaystyle 
 \theta_{0}^{*}|\xi-\bar{\xi}|+\frac{K\gamma}{\alpha} ||\Delta_{j}(\cdot,\xi,\bar{\xi})||_{\infty}& \textnormal{if}& t\in [0,L(\varepsilon)]\\\\ \displaystyle\frac{\varepsilon}{2}+\frac{K\gamma}{\alpha}||\Delta_{j}(\cdot,\xi,\bar{\xi})||_{\infty} + \theta_{0}^{*}|\xi-\bar{\xi}| &\textnormal{if}& t>L(\varepsilon).
\end{array}\right.
\end{displaymath}

Now, for any $\varepsilon>0$ there exists $L(\varepsilon)>0$, $\theta_{0}^{*}>0$ and
\begin{displaymath}
\delta_{j+1}(\varepsilon)=
\min\left\{\delta_{j}(\varepsilon/2),\frac{\varepsilon}{\theta_{0}^{*}}\left(1-\frac{K\gamma}{\alpha}\right)\right\}
\end{displaymath}
such that for any $t\geq 0$, we have
\begin{displaymath}
\forall\varepsilon>0\,\exists\delta_{j+1}(\varepsilon)>0 \,\,\textnormal{s.t.} \, |\xi-\bar{\xi}|<\delta_{j+1} \Rightarrow |z_{j+1}^{*}(t;(t,\xi))-z_{j+1}^{*}(t;(t,\bar{\xi}))|<\varepsilon.
\end{displaymath}
and the uniform continuity of $\xi \mapsto z_{j}^{*}(t;(t,\xi)) $ follows for any $j\in \mathbb{N}$.

In order to finish our proof, we choose $N\in \mathbb{N}$ such that for any $j>N$ it follows that
\begin{displaymath}
||z^{*}(\cdot;(\cdot,\xi))-z_{j}^{*}(\cdot;(\cdot,\xi))||_{\infty}<\varepsilon \quad \textnormal{for any} \quad  \xi\in \mathbb{R}^{n},
\end{displaymath} 
and therefore, if $|\xi-\bar{\xi}|<\delta_{j}$ with $j>N$, it is true that
\begin{displaymath}
\begin{array}{rcl}
|z^{*}(t;(t,\xi))-z^{*}(t;(t,\bar{\xi}))|&\leq & |z^{*}(t;(t,\xi))-z_{j}^{*}(t;(t,\xi))|+\Delta_{j}(t,\xi,\bar{\xi})\\\\
&&+|z^{*}(t;(t,\bar{\xi}))-z_{j}^{*}(t;(t,\bar{\xi}))|<3\varepsilon,
\end{array}
\end{displaymath}
and the uniform continuity of $\xi\mapsto z^{*}(t;(t,\xi))$ and $\xi\mapsto H(t,\xi)$ follows for any fixed $t\geq 0$.

\medskip

\noindent\textit{Step 6: $H$ and $G$  are  continuous on $\mathbb{R}^{+} \times \mathbb{R}^n.$} By using the continuity of the solutions with respect to the initial time and initial conditions  \cite[Ch.V]{Hartman}, we note that for any $\varepsilon_{1}>0$ there exists $\delta_{1}(t_{0},\xi_{0},\varepsilon_{1})>0$ such that
\begin{equation}
\label{crt1}
|x(s,t,\xi)-x(s,t_{0},\xi_{0})|< \varepsilon_{1}  \quad\textnormal{when} \quad |t-t_{0}|+|\xi-\xi_{0}|<\delta_{1}  
\end{equation}
for any $t$ and $s$ in a compact interval of $\mathbb{R}^{+}$ containing $t_{0}$.

On the other hand, by using the continuity of the solutions with respect of the parameters \cite[Ch.V]{Hartman} combined with a Palmer's result \cite[Lemma 1]{Palmer} restricted to $\mathbb{R}^{+}$, we know that for any $\varepsilon_{2}>0$ there exists $\delta_{2}(t_{0},\xi_{0},\varepsilon_{2})>0$ such that
\begin{equation}
\label{crt2}
|z^{*}(s;(t,\xi))-z^{*}(s;(t_{0},\xi_{0}))|<\varepsilon_{2} \quad\textnormal{when} \quad |t-t_{0}|+|\xi-\xi_{0}|<\delta_{2}, 
\end{equation}
for any $t$ and $s$ in a compact interval of $\mathbb{R}^{+}$ containing $t_{0}$.

Additionally, we know that for any $\varepsilon_{3}>0$ there exists $\delta_{3}(\varepsilon_{3},t_{0})$ such that
\begin{equation}
\label{cmf}
||X(t,s)-X(t_{0},s)|| < \varepsilon_{3} \quad \textnormal{when} \quad |t-t_{0}|<\delta_{3} 
\end{equation}
for any $t$ and $s$ in a compact interval of $\mathbb{R}^{+}$ containing $t_{0}$.

From now on, we will assume that $t$,$s$ and $t_{0}$ are in a compact interval $I\subset \mathbb{R}^{+}$ and we denote
\begin{displaymath}
\omega_{1}(s,t,t_{0},\xi,\xi_{0})=f(s,x(s,t,\xi)+z^{*}(s;(t,\xi)))-f(s,x(s,t_{0},\xi_{0})+z^{*}(s;(t_{0},\xi_{0}))).
\end{displaymath}

Without loss of generality, we will assume that $t>t_{0}$. Now, notice that 
\begin{displaymath}
\begin{array}{rcl}
H(t,\xi)-H(t_{0},\xi_{0}) &=& \displaystyle \xi-\xi_{0}+\int_{0}^{t}X(t,s)f(s,x(s,t,\xi)+z^{*}(s;(t,\xi)))\,ds \\\\
&& \displaystyle - \int_{0}^{t_{0}}X(t_{0},s)f(s,x(s,t_{0},\xi_{0})+z^{*}(s;(t_{0},\xi_{0})))\,ds \\\\
&=& \xi-\xi_{0}\\\\
&& \displaystyle +\int_{0}^{t_{0}}\{X(t,s)-X(t_{0},s)\}f(s,x(s,t,\xi)+z^{*}(s;(t,\xi)))\,ds \\\\
&&
\displaystyle +\int_{0}^{t_{0}}X(t_{0},s)\omega_{1}(s,t,t_{0},\xi,\xi_{0}) \,ds\\\\
&&
\displaystyle +\int_{t_{0}}^{t}X(t,s)f(s,x(s,t,\xi)+z^{*}(s;(t,\xi)))\,ds.
\end{array}
\end{displaymath}

Let $C=\max\{||X(u,s)|| \colon u,s\in I\}$. By using \textbf{(A3)} combined with (\ref{crt1}),(\ref{crt2}) and (\ref{cmf}), we deduce that
\begin{displaymath}
\begin{array}{rcl}
|H(t,\xi)-H(t_{0},\xi_{0})| & \leq & \displaystyle |\xi-\xi_{0}|+\mu \int_{0}^{t_{0}}||X(t,s)-X(t_{0},s)||\,ds \\\\
&&+\displaystyle \mu\int_{t_{0}}^{t}||X(t,s)||\,ds\\\\
&&
\displaystyle +\int_{0}^{t_{0}}||X(t_{0},s)||\, |\omega_{1}(s,t,t_{0},\xi,\xi_{0})|\,ds\\\\
&\leq&  |\xi-\xi_{0}|+\mu t_{0}\varepsilon_{3} +C|t-t_{0}|\mu+C\gamma (\varepsilon_{1}+\varepsilon_{2})t_{0},
\end{array}
\end{displaymath}
and we conclude that $H$ is continuous in any $(t_{0},\xi_{0})\in \mathbb{R}^{+}\times \mathbb{R}^{n}$.

Finally, the result follows by verifying that $G$ is continuous in any $(t_{0},\eta_{0})\in \mathbb{R}^{+}\times \mathbb{R}^{n}$, which can be proved 
in a similar way.
\end{proof}

In order to establish the main result of this subsection, we introduce the following additional condition:
\begin{itemize}
    \item[\textbf{(D4)}] The function $f(t,x)$ and its derivatives with respect to $x$ up to order $r$--th are continuous functions of  $(t,x)$. 
\end{itemize}

A direct consequence of the above property (see \textit{e.g.} \cite[Chap.2]{Coddington}) is that 
$\partial y(t,\tau,\eta)/\partial \eta$ satisfies the matrix differential equation (\ref{MDE1}).

The following result provides sufficient conditions ensuring that the maps $\xi \mapsto H(t,\xi)$ and $\eta \mapsto G(t,\eta)$ satisfy the properties stated in Definition \ref{TopEqCr}.
\begin{theorem}
%\label{TeoDif}
If {\bf{(A1)--\bf{(A4)}}} and \textnormal{\textbf{(D4)}} are satisfied, then \textnormal{(\ref{lineal-efes})} and \textnormal{(\ref{sistema1})} are $C^{r}$ continuously topologically equivalent on $\mathbb{R}^{+}$.
\end{theorem}

\begin{proof} The property (i) of Definition \ref{TopEqCr} is satisfied by Theorem \ref{teonuevo} whereas the property (ii) will be established by cases.
 
 \medskip
 
 \noindent{\bf{Case $r=1$.}} As $y\mapsto f(t,y)$ is $C^{1}$ for any $t\geq 0$, it follows that $y\mapsto Df(t,y)$ and  $\eta \mapsto \partial y/\partial \eta$ are continuous.
This allows to calculate the first partial derivatives of the map $\eta\mapsto G(t,\eta)$ for any $t\geq 0$ and any $i\in \{1,\ldots,n\}$ as follows
\begin{equation}
\label{derivada-parcial}
\frac{\partial G}{\partial \eta_{i}}(t,\eta)=e_{i}-\int_{0}^{t}X(t,s)Df(s,y(s,t,\eta))\frac{\partial y}{\partial \eta_{i}}(s,t,\eta)\,ds, 
\end{equation}
which implies that the partial derivatives exists and are continuous for any fixed $t\geq 0$, then $\eta \mapsto G(t,\eta)$ is $C^{1}$.

\medskip

By using the identity $X(t,s)A(s)=-\frac{\partial }{\partial s}X(t,s)$ combined with
(\ref{MDE1}) we can deduce that for any $t\geq 0$, the Jacobian matrix is given by
\begin{equation}
    \label{machine}
\begin{array}{rcl}
\displaystyle\frac{\partial G}{\partial\eta}(t,\eta)&=&  \displaystyle I-\int_{0}^{t}X(t,s)Df(s,y(s,t,\eta))
\frac{\partial y}{\partial\eta}(s,t,\eta)\,ds \\\\
&=&I - \displaystyle \int_{0}^{t}\frac{d}{ds}\left\{X(t,s)\frac{\partial y}{\partial \eta}(s,t,\eta)\right\}\,ds \\\\
&=&\displaystyle X(t,0)\frac{\partial y(0,t,\eta)}{\partial \eta},
\end{array}
\end{equation}
or, in a more concise manner
\begin{equation}
\label{jacobiano-G}
\frac{\partial G}{\partial \eta}(t,\eta)=X(t,0)\frac{\partial y(0,t,\eta)}{\partial \eta}.
\end{equation}

\medskip

In order to achieve our proof we will use a variation of the Hadamard's Theorem \cite[Cor. 2.1]{Plastock} which states that $\eta \mapsto G(t,\eta)\in C^{1}$ is a diffeomorphism if and only if  $\textnormal{Det} \frac{\partial G(t,\eta)}{\partial \eta} \neq 0$ 
for any $\eta\in \mathbb{R}^{n}$ and $|G(t,\eta)|\to +\infty$ when $|\eta|\to +\infty$.

\medskip Notice that the right term of (\ref{jacobiano-G}) is composed by two fundamental matrices
associated respectively to the linear systems (\ref{linF}) and (\ref{MDE1}). By applying the Liouville's formula
(see Theorem \ref{liouville} from Chapter 1 or  Theorems 7.2 and 7.3 from  \cite[Ch.1]{Coddington})
we can see that
\begin{displaymath}
\begin{array}{rcl}
\textnormal{Det}\frac{\partial G}{\partial\eta}(t,\eta)&=&\textnormal{Det} X(t,0)\, \textnormal{Det} \frac{\partial y(0,t,\eta)}{\partial \eta}\\\\
&=& \displaystyle \exp\left(\int_{0}^{t}\textnormal{Tr}A(s)\,ds\right)\exp\left(\int_{t}^{0}\textnormal{Tr}\{A(s)+Df(s,t,\eta)\}\,ds\right),
\end{array}
\end{displaymath}
and consequently $\textnormal{Det} \frac{\partial G(t,\eta)}{\partial \eta}>0$ for any $t\geq 0$. 

In addition, let us recall that
$$
G(t,\eta)=\eta + w^{*}(t;(t,\eta)),
$$
where $w^{*}(t;(t,\eta))$ is given by (\ref{w-star}), we can deduce that
$|G(t,\eta)|\to +\infty$ as $|\eta|\to +\infty$, due to $|w^{*}(t;(t,\eta))|\leq K\mu/\alpha$ for any $(t,\eta)$. Therefore,  we conclude that $\eta \mapsto G(t,\eta)$ is a global diffeomorphism for any fixed $t\geq 0$.

Finally, by using again the property \textbf{(D4)} combined with Theorem 4.1 from \cite[Ch.V]{Hartman})  we also can conclude that  $(t,\eta)\mapsto \frac{\partial y(0,t,\eta)}{\partial \eta}$ on $\mathbb{R}^{+}\times \mathbb{R}^{n}$
and by the identity (\ref{jacobiano-G}) we have that $(t,\eta)\mapsto \frac{\partial G}{\partial \eta}(t,\eta)$ is continuous on $\mathbb{R}^{+}\times \mathbb{R}^{n}$ and this case follows.
\medskip
 
\noindent{\bf{Case $r = 2$.}} Due to \textbf{(D4)}  we can verify that
the second partial derivatives $\partial^{2}y(s,\tau,\eta)/\partial \eta_{j}\partial\eta_{i}$  satisfy
 the system of differential equations
\begin{equation*}
%\label{MDE16}
\left\{
\begin{array}{rcl}
\displaystyle \frac{d}{dt}\frac{\partial^{2}y}{\partial\eta_{j}\partial\eta_{i}}&=&\displaystyle
\{A(t)+Df(t,y)\}\frac{\partial^{2}y}{\partial\eta_{j}\partial\eta_{i}}
+D^{2}f(t,y)\frac{\partial y}{\partial\eta_{j}}\frac{\partial y}{\partial\eta_{i}} \\\\
\displaystyle \frac{\partial^{2}y}{\partial\eta_{j}\partial\eta_{i}}  &=&0,
\end{array}\right.
\end{equation*}
 for any $i,j=1,\ldots,n,$ where $D^2 f$ is the formal second derivative of $f$ and  $y=y(t,\tau,\eta).$ By using (\ref{derivada-parcial}) we have
\begin{displaymath}
\begin{array}{rcl}
\displaystyle \frac{\partial^{2} G}{\partial\eta_{j}\partial \eta_{i}}(t,\eta)&=&\displaystyle -\int_{0}^{t}X(t,s)D^{2}f(s,y(s,t,\eta))\frac{\partial y}{\partial \eta_{j}}(s,t,\eta)\frac{\partial y}{\partial \eta_{i}}(s,t,\eta)\,ds\\\\
& & \displaystyle
-\int_{0}^{t}X(t,s)Df(s,y(s,t,\eta))\frac{\partial^{2} y(s,t,\eta)}{\partial\eta_{j}\partial\eta_{i}}\,ds\\\\
&=&\displaystyle -\int_{0}^{t}\frac{d}{ds}\left\{X(t,s)
\frac{\partial^{2} y(s,t,\eta)}{\partial\eta_{j}\partial\eta_{i}}\right\}\,ds\\\\
&=&\displaystyle  X(t,0)\frac{\partial^{2} y(0,t,\eta)}{\partial\eta_{j}\partial\eta_{i}}.

\end{array}
\end{displaymath}  
Thus,  the map $\eta \mapsto G(t,\eta)$ is $C^{2}$ for any fixed $t\geq 0$. The identity $\xi=G(t,H(t,\xi))$ for any fixed $t\in \mathbb{R}^{+}$, the Jacobian matrix of the identity map on $\mathbb{R}^{n}$ can be seen as
\begin{displaymath}
DG(t,H(t,\xi))DH(t,\xi)=I \quad \textnormal{for any fixed $t\in \mathbb{R}^{+}$}.
\end{displaymath}

By Case $r=1$, we have that $\eta\mapsto G(t,\eta)$ is a diffeormorphism of class $C^{1}$ for any fixed $t\in \mathbb{R}^{+}$, which implies that
\begin{equation}
\label{Jaco2}
DH(t,\xi)=[DG(t,H(t,\xi))]^{-1} \quad \textnormal{for any $t\in \mathbb{R}^{+}$}
\end{equation}
is well defined. In addition, note that $(t,\xi) \mapsto DH(t,\xi)$ is continuous since the maps $U\mapsto U^{-1}$ and $(t,\xi) \mapsto DG(t,H(t,\xi))$ are continuous for any $U\in Gl_{n}(\mathbb{R})$ and $(t,\xi)\in \mathbb{R}^{+}\times \mathbb{R}^{n}$.

Now, differentiating again with respect to the second variable, we have the formal computation
\begin{displaymath}
D^{2}G(t,H(t,\xi))DH(t,\xi)DH(t,\xi)+DG(t,H(t,\xi))D^{2}H(t,\xi)=0
\end{displaymath}
and the identity (\ref{Jaco2}) implies that
\begin{equation}\label{Jaco3}
   D^{2}H(t,\xi)=-DH(t,\xi)D^{2}G(t,H(t,\xi))DH(t,\xi)DH(t,\xi).
\end{equation}

It is easy to see that $D^{2}H(t,\xi)$ is continuous with respect to $(t, \xi)$ due to is a composition of maps that are continuous with respect to $(t, \xi).$ Therefore $\eta \mapsto G(t,\eta)$ is a global diffeomorphism of class $C^2$ for any fixed $t\geq 0.$

\medskip

\noindent{\bf{Case $r \geq 3 $.}} By using \textbf{(D4)}  we can  conclude that $\eta \mapsto y(0,t,\eta)$ is $C^{r}$and the partial derivatives
\begin{displaymath}
(t,\eta)\mapsto \frac{\partial^{|m|} y(0,t,\eta)}{\partial\eta_{1}^{m_{1}}\cdots \partial \eta_{n}^{m_{n}}}, \quad \textnormal{where $|m|=m_{1}+\ldots+m_{n}\leq r$},
\end{displaymath}
are continuous for any $(t,\eta)\in \mathbb{R}^{+}\times\mathbb{R}^{n}$. Moreover, this fact combined with (\ref{machine}) shows that the partial derivatives up to order $r$--th  of $G$ with respect to $\eta$
\begin{displaymath}
(t,\eta)\mapsto \frac{\partial^{|m|} G(t,\eta)}{\partial\eta_{1}^{m_{1}}\cdots \partial \eta_{n}^{m_{n}}}=X(t,0)\frac{\partial^{|m|} y(0,t,\eta)}{\partial\eta_{1}^{m_{1}}\cdots \partial \eta_{n}^{m_{n}}}, 
\end{displaymath}
where $|m|=m_{1}+\ldots+m_{n}\leq r,$
are continuous in $\mathbb{R}^{+}\times \mathbb{R}$. Additionally, the higher formal derivatives of $H$ up to order $r-$th and its continuity on $\mathbb{R}^{+}\times \mathbb{R}^{n}$ can be deduced in a recursive way from (\ref{Jaco2}) and (\ref{Jaco3}).

\medskip

The property (iii) of Definition \ref{TopEqCr} is immersed in the previous analysis. Therefore, the result follows.
 \end{proof}

 \section{Comments and References}

\noindent \textbf{1)} The differentiability results of the section 1 are inspired in the article \cite{CRJDE}. In general, the section 1 improves the clarity of the presentation 
and, in several intermediate steps, provides a clearer and simpler explanation. We also point out that, to the best of our knowledge, there are no previous works
devoted to the smoothness properties of the Palmer's homeomorphism.

\medskip

\noindent \textbf{2)} The original motivation to study the smoothness of the Palmer's homeomorphism is found in the framework of converse asymptotic stability results: note that
\textbf{(A1)}--\textbf{(A4)} implies that (\ref{linF}) and (\ref{nolinF}) are uniformly asymptotically stable. On the other hand, a converse result from P. Monz\'on implies the existence
of a \textit{Rantzer's density function}  $r(t,x)$ associated to (\ref{linF}) and an interesting achievement of the article \cite{CRJDE} was the construction of a Rantzer's density function
for (\ref{nolinF}) in terms of $r(t,x)$ and the Palmer's homeomorphism. We refer to Theorem 3 from \cite{CRJDE} for details.

\medskip 

\noindent \textbf{3)} The restrictive nature of hypotheses \textbf{(D1)} and \textbf{(D3)} prompted to restrict the time domain of the Palmer's homeomorphism from
$\mathbb{R}$ to $[0,+\infty)$ which was carried out in \cite{CMR} and \cite{Monzon1} and summarized in the section 2. The construction of homeomorphisms becomes much easier, while their differentiability properties are obtained with less restrictive hypotheses.

\medskip 

\noindent \textbf{4)} In the context of Section 2, and under the additional assumption of existence of an equilibrium $\bar{y}$, namely
\begin{displaymath}
A(t)\bar{y}=f(t,\bar{y})=0  \quad \textnormal{for any $t\geq 0$},   
\end{displaymath}
a byproduct of the global linearization result is \cite[Th.3.1]{CMR}:
\begin{theorem}
Assume that \textnormal{\textbf{(A1)--(A4)}} are fulfilled, then
\begin{itemize}
    \item[i)] If the system has an equilibrium $\bar{y}$, it is unique.
    \item[ii)] If $\bar{y}=0$, namely $f(t,0)=0$ for any $t\geq 0$, then:
    \begin{displaymath}
     H(t,0)=G(t,0)=0 \quad \textnormal{for any $t\geq 0$}.
     \end{displaymath}
    \item[iii)] If $\bar{y}\neq 0$, then 
\begin{displaymath}
\lim\limits_{t\to +\infty}H(t,0)=\bar{y} \quad \textnormal{and} \quad  \lim\limits_{t\to +\infty}G(t,0)=0. 
\end{displaymath}
    \item[iv)] The equilibrium $\bar{y}$ is globally uniformly asymptotically stable for the system \eqref{nolinF}.
\end{itemize}
\end{theorem}

\medskip

\noindent \textbf{5)} In \cite{CT25} the authors generalize the smooth linearization results from \cite{CMR} by dropping
the assumption \textbf{(A2)} by allowing a family of unbounded perturbations $f$.

\medskip

\noindent \textbf{6)} When considering projections different from the trivial ones, a first result in terms of the differentiability of the linearization on $\mathbb{R}^+$ was obtained by D. Dragi\v{c}evi\'c \textit{et al.} in \cite{Dragicevic2}, its approach take in account nonuniform exponential dichotomy and the spectrum associated to this dichotomy.
Secondly in \cite{Jara}, N. Jara proves, without using dichotomy spectrum, that if the linear system has a general nonuniform dichotomy and the nonlinear part has restrictive asymptotic assumptions; in compensation for avoiding the dichotomy spectrum, then the linearization is of class $C^2$ on $\mathbb{R}^+$.

\section{Exercises}

\begin{itemize}
    \item[1.-] By using the methods of Theorems \ref{anexo-palmer} and \ref{SegDer}, show that $y(t) \mapsto G[t,y(t)]$ is a $C^3$ preserving orientation diffeomorphisms.

    \item[2.-] Verify that $\eta \mapsto G(t,\eta)$ satisfies (\ref{uruguay1}).

%     \item[3.-] Establish a equality similar to \eqref{TH} for $H.$

    \item[4.-] Prove that the homeomorphism $\eta \mapsto G(t,\eta)$ verifies the property (ii) of Definition \ref{defnueva}.

     \item[5.-] Prove that $(t,\eta) \mapsto G(t,\eta)$ of Definition \ref{defnueva}, is continuous in any $(t_{0},\eta_{0})\in \mathbb{R}^{+}\times \mathbb{R}^{n}$.

\item[6.-] By using \eqref{Jaco2} and \eqref{Jaco3},
 obtain a formula for the higher formal derivatives of $H$ up to order $r-$th and prove its continuity on $\mathbb{R}^{+}\times \mathbb{R}^{n}.$

 \item[7.-] In Theorem \ref{teonuevo}, carry out the steps of the proof considering $J = \mathbb{R^-}.$

     \item[8.-] In the condition  \textbf{(A4)} change the projection $P=I$ by $P = 0.$ What can you say about all theorems of this chapter?

     \item[9.-] Give explicit examples that verify the hypothesis of all theorems of this chapter.

\end{itemize}

%%%%%%%%%%%%%%%%%%%%%%%%%%%%%%%%%%%%%%%%%%%%%%%%%%%%%%%%%%%%%%%%%%%%%%%%%%%%%%%%%%%%%%%%%%%%%%%%%%%%%

\appendix

\chapter[Exponential Dichotomy in the scalar case]{Exponential Dichotomy in the scalar case (Pablo Amster)}

A particularity of the scalar case is that the unique possible projectors are the
trivial ones.  Limit cases are always special and provides new perspectives on the 
already studied results. In this context, 
we will be focused on the study of the scalar equation:
\begin{equation}
\label{sca1}
 \dot{x}=a(t)x \quad \textnormal{where $t\in J$}
\end{equation}
and, as it has been usual on this book, $J$ can be the intervals
$J=(-\infty,0]$, $J=[0,+\infty)$ and $J=(-\infty,+\infty)$. 

\begin{definition}
\label{DiEx}
The equation \eqref{sca1} has an exponential dichotomy on the %upperly 
unbounded interval $J$ if and only if there exist constants $C\ge 0$ and $\alpha>0$ such that
\begin{equation*}
%\label{dico}
\left|\int_s^t a(r)\,dr 
\right|\ge \alpha(t-s) - C\qquad \textnormal{for any $t\geq s$ with $t,s\in J$}.
\end{equation*}
\end{definition}

By redefining the constant $C$ if necessary, it is clear that Definition \ref{DiEx} always corresponds to one of the following situations

\begin{enumerate}
\item[a)] There  exist constants $C\ge 0$ and $\alpha>0$ such that:
\begin{equation}
    \label{proj-id} 
\int_s^t a(r)\,dr  \le C - \alpha(t-s) \qquad \textnormal{for any $t\geq s$ with $t,s\in J$}.\end{equation}

\item[b)] 
There  exist constants $C\ge 0$ and $\alpha>0$ such that:
\begin{equation}
    \label{proj-null}\int_t^s a(r)\,dr  \ge  \alpha(s-t) - C \qquad \textnormal{for any $s\geq t$ with $t,s\in J$}.
    \end{equation}
\end{enumerate}

The above properties are equivalent to the  standard definition of exponential dichotomy:
\begin{itemize}
\item[a')] There exist constants $K\ge 1$ and $\alpha>0$ such that 
\begin{equation}
\label{DEPI}
 e^{\int_{s}^{t}a(r)\,dr}\leq Ke^{-\alpha (t-s)} \quad \textnormal{for any $t\geq s$ with $t,s\in J$}.
 \end{equation}
\item[b')] There exist constants $K\ge 1$ and $\alpha>0$ such that 
\begin{equation}
\label{DENP}
e^{-\int_{t}^{s}a(r)\,dr}\leq Ke^{-\alpha (s-t)} \quad \textnormal{for any $s\geq t$ with $t,s\in J$}.
\end{equation} 
\end{itemize}

In fact, inspired by the standard definition, the case a) corresponds to the exponential dichotomy  with the \textit{identity projector} while the case b) corresponds to 
the exponential dichotomy  with the \textit{null projector}.

\begin{itemize}
\item[$\bullet$] If (\ref{sca1}) has an exponential dichotomy on $J=[0,+\infty)$, the properties (\ref{DEPI})--(\ref{DENP}) 
says that any solution either converges exponentially to zero (identity projector) or diverges exponentially (null projector) as $t\to +\infty$, since (\ref{DENP}) is equivalent to $K^{-1}e^{\alpha(t-t_0)}\leq e^{\int_{t_0}^t\mathfrak{a}(r)\,dr}$ with $t\geq t_{0}\geq 0$. 
\item[$\bullet$] Similarly, if (\ref{sca1}) has an exponential dichotomy on $J=(-\infty,0]$, the properties (\ref{DEPI})--(\ref{DENP}) 
says that any solution either converges exponentially to zero (null projector) or diverges exponentially (identity projector) as $t\to -\infty$, since (\ref{DEPI}) is equivalent to $K^{-1}e^{\alpha(t_0-t)}\leq e^{\int_{t_0}^{t}\mathfrak{a}(r)dr}$ with $t\leq t_0\leq 0$. 
\item[$\bullet$] 
If (\ref{sca1}) has an exponential dichotomy on $J=(-\infty,+\infty)$ then it also has an exponential dichotomy on $(-\infty,0]$
and $[0,+\infty)$ and inherits the above asymptotic behaviors at $\pm \infty$.
\end{itemize}

As we can see, if (\ref{sca1}) has an exponential dichotomy on $J$, its solutions have 
a \textit{dichotomic} asymptotic behavior at $t\to +\infty$ or $t\to -\infty$: they are either convergent or divergent at exponential rate. As we know, this qualitative asymptotic behavior is behind the name of exponential dichotomy.

\begin{remark}
%\label{BST}
As a consequence of the previous discussion is that, if \eqref{sca1} has an exponential dichotomy on $(-\infty,+\infty)$ then
the unique solution bounded on $(-\infty,+\infty)$ is the trivial one. In fact, any other solution will be divergent
at $+\infty$ or $-\infty$. 
\end{remark}

A noteworthy fact is that a linear equation (\ref{sca1}) can have an exponential dichotomy simultaneously on $(-\infty,0]$
and $[0,+\infty)$ but does not necessarily has an exponential dichotomy on $(-\infty,+\infty)$,
we refer to Section from Chapter 3 for an illustrative example.

\subsection{Roughness, admissibility and spectra}

The next result is known as the \textit{roughness property} and shows that the set of bounded continuous functions such that an exponential dichotomy exists is open. More precisely,
\begin{proposition}
%\label{rough-a}
Assume that  \eqref{sca1} has an exponential  
dichotomy on $[0,+\infty)$ with constants $C$, $\alpha$
and let
$b\colon [0,+\infty)\to \mathbb{R}$ be a continuous function such that
\begin{displaymath}
\delta_{0}:= \limsup\limits_{t\to +\infty}|b(t)|<{\alpha}.
\end{displaymath}
Then the perturbed equation
\begin{displaymath}
\dot{x}=[a(t)+b(t)]x    
\end{displaymath}
also has an exponential dichotomy on $[0,+\infty)$ with constants $\tilde{C},\alpha_{0}$ where
$\tilde{C}>C$ and $\alpha_{0}$ is any positive constant satisfying the restriction
$\alpha_0+\delta_{0} < \alpha$.
\end{proposition}

\begin{proof}

    Fix $\alpha_0>0$ such that $\alpha_{0}< \alpha - \delta_{0}$ and $T\geq 0$ such that $|\mathfrak{e}(t)|<\alpha-\alpha_0$ for any $t\ge T$. The proof will consider
the three possible cases: $t\geq s\geq T$, $T\geq t\geq s$ and $t\geq T \geq s$.
    
\medskip 

Firstly, when $t\ge s\ge T$ it is verified that: 
    $$\left|\int_s^t [a(r)+b(r)] \,dr \right| 
\ge \alpha(t-s) - C - \int_s^t |b(r)|\, dr \ge \alpha_0(t-s) - C.
$$

Secondly, when $s\leq t\leq T$, let us define $C_0:= \int_0^{T} |\mathfrak{a}(r)+\mathfrak{e}(r)|\ dr$. It is clear that 
$$\left|\int_s^t [a(r)+b(r)] \, dr \right| 
\ge - C_0  \ge \alpha_0(t-s) - C_0 -\alpha_0\,T.
$$
Finally, for $t> T\ge s$, we have that
\begin{displaymath}
\begin{array}{rcl}
\displaystyle \left|\int_s^t [a(r)+b(r)] \,dr \right| &\ge&  \displaystyle
\left|\int_{T}^t [a(r)+b(r)] \,dr \right| - 
\int_s^{T} |a(r)+b(r)|\ dr \\
&\ge& \alpha_0(t-T) - C-C_0 \\
&\ge &\alpha_0 (t-s) - \alpha_0 T-
C-C_0,
\end{array}
\end{displaymath}
then we have that the equation has an exponential dichotomy with constants $\alpha_0$ and $\tilde C:= C+ C_0 + \alpha_0\,T$. 
\end{proof}

We stress the beautiful and remarkable simplicity of the proof of the above lemma
compared with the  general $n$--dimensional case which has been addressed in several
monographs as \cite{Barreira2},\cite[Ch.4]{Cop} and research articles as \cite{Ju,Naulin-a}.

Now, let $(\mathcal{B},\mathcal{D})$ be a pair of function spaces, as we have seen in the Chapter 2,
this pair is said
to be \textit{admissible} if for any $v\in \mathcal{D}$, the equation 
\begin{equation}
\label{sca0}
\dot{y}=a(t)y+v(t),
\end{equation}
has --at least-- one solution in the space $\mathcal{B}$. The next result describes the equivalence between the exponential dichotomy on $[0,+\infty)$ with the admissibility property 
when $\mathcal{B}=\mathcal{D}$ is the Banach space of bounded continuous functions on 
$[0,+\infty)$ with the supremum norm:

\begin{proposition} 
%\label{Prop-Cop}
Let $t\mapsto a(t)$ be a continuous function on $[0,+\infty)$. The following properties are equivalent:
\begin{itemize}
\item[a)] The equation \eqref{sca1} has an exponential dichotomy on $[0,+\infty)$.
\item[b)] For any 
bounded and continuous function $v\colon [0,+\infty) \to \mathbb{R}$, the problem
\begin{equation}
\label{inhp}
\dot{x}=a(t)x + v(t)    
\end{equation}
has 
at least one bounded and continuous solution on $[0,+\infty)$. 
More precisely, \eqref{proj-id} holds if and only if 
all solutions are bounded for each bounded $v(\cdot)$   and
\eqref{proj-null} holds if and only if 
there exists exactly one bounded solution for  each bounded $v(\cdot)$.
\item[c)] The non homogeneous equation 
\begin{equation*}
%\label{3nh}
\dot{z}=a(t)z+ 1
\end{equation*}
has at least a bounded solution on $[0,+\infty)$. 
\item[d)] There exists a bounded and continuous function $v\colon [0,+\infty)\to \mathbb{R}$ with 
$$
\liminf\limits_{t\to +\infty}|v(t)|>0
$$
such that the problem \eqref{inhp} has at least one bounded solution. 
\end{itemize}
\end{proposition}

\begin{proof}
Proof of $\textnormal{a)} \Rightarrow \textnormal{b)}$:
If (\ref{proj-id}) holds, then any solution of (\ref{sca0}) has the form:
\begin{equation*}
%\label{sol-gral}x(t)= x(t_0) e^{\int_{t_0}^{t}a(r)\,dr} + \int_{t_0}^t v(s) e^{\int_{s}^{t}a(r)\,dr}\, ds.    
\end{equation*}
The first term of the right-hand side is bounded since it converges to $0$ when $t\to +\infty$, while the second term satisfies, for $t\ge t_0$:
$$\left|\int_{t_0}^t v(s) e^{\int_{s}^{t}a(r)\,dr}\, ds\right| \le \|v\|_\infty 
\int_{t_0}^t Ke^{-\alpha(t-s)}\, ds = \frac K\alpha \|v\|_\infty \left(1 - e^{-\alpha (t-t_0)}\right),
$$
and the boundedness on $[0,+\infty)$ of any solution of (\ref{sca0}) is verified. Now, if (\ref{proj-null}) holds instead, then it is readily verified that the expression
$$
x(t)= -\int_t^{+\infty} v(s) e^{-\int_{t}^{s}a(r)\,dr}\, ds
$$
is well defined and gives an explicit description for  the unique solution of (\ref{sca0})
which is bounded on $[0,+\infty)$. 

On the other hand, observe that the statements c) and d) are particular cases of b), which implies that $\textnormal{a)} \Rightarrow \textnormal{c)}$ and $\textnormal{a)} \Rightarrow \textnormal{d)}$ are straightforwardly verified.

Proof of $\textnormal{d)} \Rightarrow \textnormal{c)}$: Let $t\mapsto x(t)$ be a bounded solution of (\ref{inhp}) such that
\begin{displaymath}
m\leq x(t)\leq M \quad \textnormal{for any $t\in [0,+\infty)$},    
\end{displaymath}

By hypothesis, we have the existence of $v_{0}>0$ such that $|\mathfrak{v}(t)|>v_{0}$
for any $t\geq t_{0}$ with $t_{0}\geq 0$. Without loss of generality, we can suppose
$t_{0}=0$.

Firstly, we will assume that $v(t)\ge v_0 >0$ for any $t\geq 0$, 
then
the above solution $t\mapsto x(t)$ also verifies:
$$
x(t)=x_0 e^{\int_{0}^{t}a(r)\,dr } + \int_{0}^t e^{\int_{s}^{t}a(r)\,dr}v(s)ds\ge x_0 e^{\int_{0}^{t}a(r)\,dr} + v_0\int_{0}^{t} e^{\int_{s}^{t}a(r)\,dr}\,ds,
$$
which allow us to deduce:
$$
\frac {x(t)}{v_0} \ge 
\frac{x_0}{v_0} e^{\int_{0}^{t}a(r)\,dr} + \int_{0}^{t} e^{\int_{s}^{t}a(r)\,dr}\,ds:=z(t).
$$

We can see that $t\mapsto z(t)$ is solution of $\dot{z}=a(t)z+1$ passing through $x_{0}/v_{0}$ at $t=0$, which also verifies $z(t)\leq M/v_{0}$ for any $t\geq 0$.

Now, let us observe that
\begin{displaymath}
\begin{array}{rcl}
z(t)&=&\displaystyle \frac{x_0}{v_0} e^{\int_{0}^{t}\mathfrak{a}(r)\,dr} + \int_{0}^{t} e^{\int_{s}^{t}\mathfrak{a}(r)\,dr}\,ds \\\\
&=&\displaystyle ||v||_{\infty}^{-1}\left(x_{0}\frac{||v||_{\infty}}{v_0} e^{\int_{0}^{t}a(r)\,dr} + \int_{0}^{t} e^{\int_{s}^{t} a(r)\,dr}||v||_{\infty}\,ds \right)\\\\
&\geq & \displaystyle ||v||_{\infty}^{-1}\left(x_{0}\ e^{\int_{0}^{t}a(r)\,dr} + \int_{0}^{t} e^{\int_{s}^{t}a(r)\,dr}v(s)\,ds \right)\\\\
&\geq & ||v||_{\infty}^{-1}x(t),
\end{array}
\end{displaymath}
and we proved that 
\begin{displaymath}
m||v||_{\infty}^{-1}\leq z(t)\leq v_{0}^{-1}M \quad \textnormal{for any $t\in [0,+\infty)$}.   \end{displaymath}

%\textcolor{red}
{Secondly, if we assume that $-v(t) \geq v_{0}>0$ for any $t\geq 0$, we multiply (\ref{inhp}) by $-1$ and make the transformations $u=-x$ and $w(t)=-v(t)$, leading to $\dot{u}=a(t)u+w(t)$, which also has a bounded solution
and the above proof can be carried out}.

Proof of $\textnormal{c)} \Rightarrow \textnormal{a)}$: Let $t\mapsto z(t)$ be a solution
of $\dot{z}=a(t)z+1$ which is bounded on $[0,+\infty)$. Firstly, a noteworthy
property of the above equation is that: if $z(0)\geq 0$, the solution remains positive
for any $t>0$ while, if $z(0)<0$ the solution either remains negative or has a unique change of sign.

Secondly, let us verify that
\begin{equation}
\label{LIA}
\displaystyle\liminf_{t\to+\infty} |z(t)| >0.
\end{equation}

In fact, this property is immediately verified when $z(\cdot)$ is positive
after some value since $t\mapsto z(t)$ is increasing at any value of $t$ such that $|z(t)|\|a\|_\infty < 1$. Now, if $z(t)<0$ for any $t\geq 0$, we have to consider two cases: 

\medskip

\noindent \textit{Case i)} $\limsup\limits_{t\to+\infty} z(t)<0$, which is equivalent to (\ref{LIA}).

\noindent \textit{Case ii)} $\limsup\limits_{t\to+\infty} z(t)=0$, which cannot be possible, because $z(\cdot)$ is increasing after some finite time and, consequently, $z(t)=\textit{o}(1)$, which leads to $\liminf\limits_{t\to+\infty} z'(t) >0$, obtaining a contradiction with the boundedness of $t\mapsto z(t)$ and (\ref{LIA}) follows.

In consequence, we can assume that $0<c\le |z(t)|\le M$ for any $t\geq T$.  In the case that $z(\cdot)$ is positive for $t>s\geq T$ it follows that:
$$C\ge \ln\left(\frac{z(t)}{z(s)}\right) = \int_s^t a(r)\, dr + \int_s^t \frac 1{z(r)}\, dr,$$
which leads to:
$$\int_s^t a(r)\, dr \le  C - \frac 1M (t-s).$$

Similarly, if $z(\cdot)$ is negative  for any $t\leq s$ we will have that
$$
\int_t^s a(r)\, dr \ge  -C + \frac 1M (s-t),
$$
then the equation (\ref{sca1}) has an exponential dichotomy on $[T,+\infty)$
and the exponential dichotomy is verified on $[0,+\infty)$ as a consequence of Lemma \ref{completar}
from Chapter 3.

\end{proof}

We stress the remarkable simplicity of the proof of the above lemma
compared with the 
general $n$--dimensional case which has been addressed from a functional analysis approach in
\cite{Cop} and variational methods in \cite{Campos}.

\medskip

The exponential dichotomy spectrum has a simpler characterization in the scalar case when
$J=[0,+\infty)$.
\begin{proposition} 
\label{SPEC}
If $t\mapsto a(t)$ is bounded and continuous on $[0,+\infty)$ then the exponential dichotomy spectrum of \eqref{sca1} is characterized as follows:
\begin{displaymath}
\Sigma^{+}(a)=[\beta^{-}(a),\beta^{+}(a)] 
\end{displaymath}
where $\beta^{-}(a)$ and $\beta^{+}(a)$ are the lower and upper Bohl exponents, which are 
respectively defined by:
\begin{displaymath}
\beta^{-}(a):=\liminf\limits_{L,s\to +\infty}\frac{1}{L}\int_{s}^{s+L}a(r)\,dr\quad \textnormal{and} \quad \beta^{+}(a):=\limsup\limits_{L,s\to +\infty}\frac{1}{L}\int_{s}^{s+L}a(r)\,dr. 
\end{displaymath}
More precisely, 
\begin{itemize}
    \item The case \eqref{proj-id} holds for $a(\cdot)-\lambda$ if and only if $\lambda < \beta^{-}(a)$.
    \item The case \eqref{proj-null} holds for $a(\cdot)-\lambda$ if and only if $\lambda > \beta^{+}(a)$.
\end{itemize}
\end{proposition}

\begin{proof}

{We will prove an equivalent result, namely:
$$
\mathbb{R}\setminus \Sigma^{+}(a)=(-\infty,\beta^{-}(a))\cup (\beta^{+}(a),+\infty).
$$}

{If $\lambda \notin \Sigma^{+}(a)$
it follows from Definition \ref{DefinicionEspectro} from Chapter 5 that $\dot{x}=[a(t)-\lambda]x$
has an exponential dichotomy on $[0,+\infty)$, that is either (\ref{proj-id}) or (\ref{proj-null}) holds for $a(\cdot)-\lambda$}. 

%\textcolor{red}
{The proof will be a consequence of the following statements:
\begin{itemize}
\item[i)] The property  (\ref{proj-null}) holds for $a(\cdot)-\lambda$
if and only if $\lambda \in (-\infty,\beta^{-}(a))$,
\item[ii)] The property  (\ref{proj-id}) holds for $a(\cdot)-\lambda$
if and only if $\lambda \in (\beta^{+}(a),+\infty)$.
\end{itemize}
}

Assume for example that (\ref{proj-null}) holds for $a(\cdot)-\lambda$, then for arbitrary $s>t$ we have
$$ \frac 1{s-t}\int_t^s a(r)\, dr \ge \lambda + \alpha - \frac C{s-t} 
$$
for some $\alpha>0$. Taking lower 
limit for $t\to+\infty$ and $L:=s-t\to +\infty$, we deduce that $\beta^-(a)\ge \lambda + \alpha>\lambda$. Conversely, for $\beta^-(a)> \lambda$ we may fix $\alpha\in (0,\beta^-(a)- \lambda)$
and $t_0, L_0>0$ such that 
$$ 
\frac 1{s-t}\int_t^s a(r)\, dr > \lambda + \alpha \quad \textnormal{with} \quad t>t_0 \quad  \textnormal{and} \quad s-t>L_0,  
$$
that is
\begin{equation}
\label{condition}
\int_t^s [a(r)-\lambda]\, dr >  \alpha(s-t) \quad \textnormal{with} \quad s-t>L_0 \quad \textnormal{and} \quad t>t_0
\end{equation}

We have to study the cases where the conditions $s-t>L_0$ and $t>t_0$ are not satisfied, namely: 
\begin{enumerate}
\item[a)] $0\le s-t\le L_0$, 
\item[b)] $t_{0}\geq s-L_{0}>t$,
\item[c)] $s-L_{0}>t_{0}\geq t$. 
\end{enumerate}

Firstly, the cases a) and b) imply that $0\le s-t\le L_0$ and $t\le s\le L_0+t_0$. Then if a) or b) are
satisfied we have that:
$$
\int_t^s [a(r)-\lambda]\, dr \ge -\|a - \lambda\|_\infty (L_0+t_0)
\ge \alpha (s-t) - (\|a - \lambda\|_\infty + \alpha)(L_0+t_0).
$$

Finally, if $t\le t_{0} < s-L_{0}$, then by using (\ref{condition}) we have
\begin{displaymath}
\begin{array}{rcl}
\displaystyle \int_t^s [a(r)-\lambda]\, dr &=& 
\displaystyle \int_{t_{0}}^{s} [a(r)-\lambda]\, dr + \int_{t}^{t_0} [a(r)-\lambda]\, dr \\
&>& \displaystyle \alpha (s-t_0) - \|a - \lambda\|_\infty t_0 \\
&\ge&  \displaystyle
\alpha (s-t) - (\|a - \lambda\|_\infty +\alpha) t_0.
\end{array}
\end{displaymath}

By gathering the above cases we have that (\ref{proj-null}) is verified with
$\alpha>0$ and $C=(\|a - \lambda\|_\infty +\alpha)t_{0}$, then the statement i) is verified.
The statement ii) can be proved in an analogous way.

\end{proof} 

The above result is usually stated without proof in the literature \cite[p.30]{APR} and \cite{Doan}.
Last but not least, we refer the reader to \cite{barabanov2001,Barabanov2015} and \cite{Dalecki} for more details about the Bohl exponents.

%\appendix 
\chapter{Matrix Norms}
\section*{B.1 Basic facts}

Given a matrix $A\in M_{n}(\mathbb{R})$, the matrix norm induced
by the euclidean vector norm $|\cdot|_{2}$ is defined as follows:
\begin{displaymath}
||A||_{2}:=\sup\limits_{x\neq 0}\frac{|Ax|_{2}}{|x|_{2}}=\sqrt{\lambda_{\max}(A^{*}A)},
\end{displaymath}
where $\lambda_{\max}(A^{*}A)$ is the largest eigenvalue of $A^{*}A$.

\section*{B.2 Positive definite matrices}

\begin{definition}
A square hermitian matrix $A=A^{*}\in M_{n}(\mathbb{C})$ is positive definite (semi--definite) if
$x^{*}Ax>0$ ($x^{*}Ax\geq 0$) for any $(n\times 1)$ vector $x\neq 0$.
\end{definition}

\begin{definition}
 A square symmetric matrix $A=A^{T}\in M_{n}(\mathbb{R})$ is positive definite (semi--definite) if
$x^{T}Ax>0$ ($x^{T}Ax\geq 0$) for any $(n\times 1)$ real vector $x\neq 0$.
\end{definition}

The notation $A>0$ ($A\geq 0$) is reserved to 
denote a positive definite (semi--definite) matrix.

The positive definite (semi--definite) matrices have interesting properties:
\begin{itemize}
    \item[1)] $A$ is positive (semi) definite if and only if all the eigenvalues of $A$ are
    positive (nonnegative).
    \item[2)] $A$ is positive (semi) definite if and only there exists a positive (semi) definite
    matrix $B\in M_{n}(\mathbb{C})$ such that $A=BB$, that is $B$ is a square root of $A$.
    \item[3)] $A$ is a real positive (semi) definite if and only there exists a real positive (semi) definite
    matrix $B\in M_{n}(\mathbb{R})$ such that $A=BB$, that is $B$ is a square root of $A$.
\end{itemize}

For properties 2) and 3) we refer the reader to Theorem 13 from \cite[Ch.9]{Hof},\cite[p.155]{Lutkepohl} and \cite[p.37]{Zhou} for details.

\chapter*{Appendix C}

The next result has been introduced (1919) by the Swedish mathematician Thomas Hakon Gr\"onwall
in \cite{Gronwall} as a tool to prove the continuity of the solution of a differential equation 
with respect to the initial conditions.

\section*{C.1 Gronwall's Lemma}
\begin{lemma}
Let $J\subset \mathbb{R}$ be an interval,$t_{0}\in J$ and $c>0$. Let $u,v\in C(J,(0,+\infty))$ such that for any $t\geq t_{0}$ with $t\in J$ it follows that;
\begin{equation}
\label{gronwall}
u(t)\leq c + \int_{t_{0}}^{t}v(s)u(s)\,ds,
\end{equation}
then for any $t\geq t_{0}$ we have that
\begin{equation}
\label{gronwall2}
u(t)\leq c\,e^{\int_{t_{0}}^{t}v(s)\,ds}.
\end{equation}
\end{lemma}

\begin{proof}
Note that the positiveness of the functions and parameters imply that inequality (\ref{gronwall}) is equivalent to
\begin{displaymath}
\displaystyle \frac{u(t)v(t)}{ c + \int_{t_{0}}^{t}v(s)u(s)\,ds}\leq v(t).
\end{displaymath}

As the left side of the above inequality is a logarithmic derivative of $t\mapsto c+\int_{t_{0}}^{t}v(s)u(s)\,ds$. Now, by integrating between $t_{0}$ and $t$ and using again (\ref{gronwall}) we obtain
\begin{displaymath}
\ln(u(t))\leq \ln\left(c+\int_{t_{0}}^{t}v(s)u(s)\,ds\right)=\ln(c)+\int_{t_{0}}^{t}v(s)\,ds
\end{displaymath}
and (\ref{gronwall2}) can be deduced easily.
\end{proof}

An important variation of Gronwall's Lemma is given by the following result extracted from \cite[p.19]{Berthelin}:

\begin{lemma}
%\label{G2}
Let $J\subset \mathbb{R}$ be an interval,$t_{0}\in J$ and $c>0$. Let $u,v\in C(J,(0,+\infty))$ such that for any $t\in J$ it follows that;
\begin{equation}
\label{gronwall3}
u(t)\leq c + \left|\int_{t_{0}}^{t}v(s)u(s)\,ds\right|,
\end{equation}
then for any $t\in J$ we have that
\begin{equation*}
%\label{gronwall4}
u(t)\leq ce^{\left|\int_{t_{0}}^{t}v(s)\,ds\right|}
\end{equation*}
\end{lemma}

\begin{proof}
If $t\geq t_{0}$, the assumption (\ref{gronwall3}) is equivalent to (\ref{gronwall}) and by Gronwall's Lemma we can deduce that
$$
u(t)\leq c\,e^{\int_{t_{0}}^{t}v(s)\,ds}= c\,e^{\left|\int_{t_{0}}^{t}v(s)\,ds\right|}.
$$

Now, let us assume that $t\in J$ with $t< t_{0}$, let us note by (\ref{gronwall3}) that
$$
0<u(t)\leq w(t):=c+\left|\int_{t_{0}}^{t}v(s)u(s)\,ds\right|=c-\int_{t_{0}}^{t}v(s)u(s)\,ds
$$

We derivate $w(t)$ and obtain that
$$
w'(t)=-v(t)u(t)\geq -v(t)w(t),
$$
that is equivalent to
$$
w'(t)+v(t)w(t)\geq 0
$$
and also to
$$
w'(t)e^{\int_{t_{0}}^{t}v(s)\,ds}+v(t)e^{\int_{t_{0}}^{t}v(s)\,ds}w(t)\geq 0,
$$
which can be written as
$$
\frac{d}{dt}\left[w(t)e^{\int_{t_{0}}^{t}v(s)\,ds} \right]\geq 0.
$$

Then, the function $t\mapsto \phi(t)=w(t)e^{\int_{t_{0}}^{t}v(s)\,ds}$ is increasing for any $t\in J$ with $t\leq t_{0}$. In consequence, as $t\leq t_{0}$ we have that $\phi(t)\leq \phi(t_{0})=w(t_{0})=c$ or
$$
w(t)e^{\int_{t_{0}}^{t}v(s)\,ds}\leq w(t_{0})=c,
$$
then
$$
w(t)\leq ce^{-\int_{t_{0}}^{t}v(s)\,ds}.
$$

As $w(t)\geq u(t)>0$ we can conclude that 
$$
u(t)\leq w(t) \leq ce^{-\int_{t_{0}}^{t}v(s)\,ds}=ce^{\left|\int_{t}^{t_{0}}v(s)\,ds\right|}=ce^{\left|-\int_{t}^{t_{0}}v(s)\,ds\right|}=ce^{\left|\int_{t_{0}}^{t}v(s)\,ds\right|}
$$
\end{proof}

\end{document}